\pdfoutput=1
\documentclass[a4paper]{amsart}
\usepackage{amsmath,amsthm,amssymb,latexsym,epic,bbm,comment,amscd}
\usepackage{thmtools,mathtools}
\usepackage{graphicx,enumerate,stmaryrd,color,cancel}
\usepackage[all,2cell]{xy}
\usepackage{anyfontsize}
\SetSymbolFont{stmry}{bold}{U}{stmry}{m}{n}
\allowdisplaybreaks[4]

\newcommand{\mone}[1][1]{{\,\text{-}#1}}

\newtheorem{theoremm}{Theorem}[section]

\declaretheorem[style=plain,name=Theorem,numberlike=theoremm]{theorem}
\declaretheorem[style=plain,name=Lemma,numberlike=theoremm]{lemma}
\declaretheorem[style=plain,name=Proposition,numberlike=theoremm]{proposition}
\declaretheorem[style=plain,name=Corollary,numberlike=theoremm]{corollary}

\declaretheorem[style=definition,name=Definition,numberlike=theorem]{definition}
\declaretheorem[style=remark,name=Example,numberlike=theorem]{example}
\declaretheorem[style=remark,name=Remark,numberlike=theorem]{remark}

\numberwithin{equation}{section}

\usepackage[all]{xy}
\usepackage[active]{srcltx}
\usepackage{enumerate}

\usepackage[table]{xcolor}
\usepackage{array}
\newcolumntype{C}{>{$}c<{$}}

\newcommand{\C}{\mathrm{C}}
\newcommand{\s}{\square_{\C}}
\font\sc=rsfs10
\newcommand{\cC}{\sc\mbox{C}\hspace{1.0pt}}
\newcommand{\cCH}{\sc\mbox{C}_{\,\mathcal{H}}\hspace{1.0pt}}

\newcommand{\cS}{\sc\mbox{S}\hspace{1.0pt}}
\newcommand{\cE}{\sc\mbox{E}\hspace{1.0pt}}

\newcommand{\cD}{\sc\mbox{D}\hspace{1.0pt}}

\newcommand{\cA}{\sc\mbox{A}\hspace{1.0pt}}

\newcommand{\cV}{\sc\mbox{V}\hspace{1.0pt}}

\newcommand{\cBM}{\sc\mbox{B}_{\mathbf{M}}\hspace{1.0pt}}

\newcommand{\cBMop}{\sc\mbox{B}^{\,\mathrm{op}}_{\mathbf{M}}\hspace{1.0pt}}
\font\scc=rsfs7
\newcommand{\ccB}{\scc\mbox{B}\hspace{1.0pt}}

\newcommand{\ccC}{\scc\mbox{C}\hspace{1.0pt}}
\newcommand{\ccCH}{\scc\mbox{C}_{\mathcal{H}}\hspace{1.0pt}}

\newcommand{\ccD}{\scc\mbox{D}\hspace{1.0pt}}
\newcommand{\ccE}{\scc\mbox{E}\hspace{1.0pt}}

\font\sccc=rsfs5
\newcommand{\cccC}{\sccc\mbox{C}\hspace{1.0pt}}
\newcommand{\cccCH}{\sccc\mbox{C}_{\mathcal{H}}\hspace{1.0pt}}
\newcommand{\cccD}{\sccc\mbox{D}\hspace{1.0pt}}

\newcommand{\cEnd}{\cE\mathrm{nd}}
\newcommand{\ccEnd}{\ccE\mathrm{nd}}

\newcommand{\B}{\sc\mbox{B}\mathrm{icom}}
\newcommand{\BB}{\sc\mbox{B}\B}
\newcommand{\Bi}{\sc\mbox{B}\mathrm{iinj}}

\newcommand{\lunit}{\upsilon^{l}}
\newcommand{\runit}{\upsilon^{r}}

\usepackage{tikz,tikz-cd}
\usepackage{adjustbox}
\tikzset{anchorbase/.style={baseline={([yshift=-0.5ex]current bounding box.center)}},
smallnodes/.style={font=\scriptsize,text height=0.75ex,text depth=0.15ex},
cstrand/.style={line width=1.5,color=black},
dstrand/.style={line width=1,color=blue,densely dotted},
}
\usetikzlibrary{decorations}
\usetikzlibrary{decorations.markings}
\usetikzlibrary{decorations.pathreplacing}
\usetikzlibrary{decorations.pathmorphing}
\tikzstyle directed=[postaction={decorate,decoration={markings,
mark=at position #1 with {\arrow[draw=black, line width=0.4mm]{>}}}}]
\tikzstyle opdirected=[postaction={decorate,decoration={markings,
mark=at position #1 with {\arrow[draw=black, line width=0.4mm]{<}}}}]
\tikzstyle marked=[postaction={decorate,decoration={markings,
mark=at position #1 with {\draw[ultra thin,black,fill=black] (0,0) circle (.08cm);}}}]
\tikzstyle box=[minimum height=0.4cm,draw,rounded corners, draw,rectangle,solid]

\newcommand*\circled[1]{\tikz[baseline=(char.base)]{
\node[shape=circle,draw,inner sep=2pt] (char) {#1};}}

\begin{document}
\title[Finitary birepresentations of finitary bicategories]
{Finitary birepresentations of finitary bicategories}

\author[M. Mackaay, V. Mazorchuk, V. Miemietz, D. Tubbenhauer and X. Zhang]
{Marco Mackaay, Volodymyr Mazorchuk, Vanessa Miemietz,\\ Daniel Tubbenhauer and Xiaoting Zhang}

\begin{abstract}
In this paper, we discuss the generalization of finitary $2$-representation theory of finitary $2$-categories to finitary birepresentation theory of finitary bicategories. In previous papers on the subject, the classification of simple transitive $2$-representations of a given $2$-category was reduced to that for certain subquotients. These reduction results were all formulated as bijections between equivalence classes of $2$-representations. In this paper, we generalize them to biequivalences between certain $2$-categories of birepresentations. Furthermore, we prove an analog of the double centralizer theorem in finitary birepresentation theory.
\end{abstract}

\maketitle
\tableofcontents

%%%%%%%%%%%%%%%%%%%%%%%%%%%%%%%%%%%%%%%%%

\section{Introduction}\label{section:intro}

%%%%%%%%%%%%%%%%%%%%%%%%%%%%%%%%%%%%%%%%%

Finitary $2$-representation theory of finitary $2$-categories, which is the categorical analog
of finite dimensional representation theory of finite dimensional algebras, has evolved
considerably in the last decade, see e.g. \cite{MM1,MM2,MM3,MM5,MM6,CM,ChMi,KMMZ,MMMT,MMMZ} and references therein.

The main reason for restricting the framework to $2$-representations of $2$-categories,
was to avoid technical difficulties which naturally arise when one considers
birepresentations of bicategories in general: by weakening the axioms, the proofs of most results require more and bigger diagrams, whose commutativity is not always easy to show. However, in our
work on the $2$-representation theory of Soergel bimodules of finite Coxeter type, it became clear that the $2$-categorical setup is really too restrictive.
Dealing with concrete examples of a certain weak categorical structure as if they were strict and justifying the oversimplification by referring to well-known general and abstract strictification theorems, becomes untenable at some point. For example, the (conjectural) classification of the so-called simple transitive $2$-representations of Soergel bimodules involves certain subquotients of Soergel bimodules which are
naturally bicategories but not $2$-categories.
Moreover, the well-known
classification of the simple transitive birepresentations of these bicategories depends on the associator in an essential way (cf. \cite[Example 7.4.10 and Corollary 7.12.20]{EGNO}).

The main purpose of this paper is therefore to discuss the generalization of some important
foundational results on finitary $2$-representation theory to finitary birepresentation theory. By discussing, we mean formulating those results carefully in the greatest possible generality (or at least as generally as we currently can) and proving them in detail whenever the proof is not straightforward and cannot be found in the literature.
A lot of the results in this paper will not surprise the experts, but we think that it is important to have the statements and their proofs, which sometimes involve quite complicated diagrams (e.g. the proof of Theorem \ref{prop0.4}), in written form somewhere in the literature. However, the paper also contains new results, as we will discuss in the next paragraphs, which are also intended as a brief and incomplete overview of birepresentation theory.

In the series of papers mentioned above, the first key tool for studying the structure of finitary $2$-representations, is the weak Jordan--H{\"o}lder theorem \cite[Subsection 3.5]{MM5}. Just like the usual Jordan--H{\"o}lder theorem in the representation theory of finite dimensional algebras, the weak Jordan--H{\"o}lder theorem shows that any finitary $2$-representation admits a filtration by $2$-subrepresentations and an associated sequence of so-called \emph{simple transitive} subquotients,
which play the role of the simples in $2$-representation theory. This sequence is an essential invariant of the $2$-representation and for that reason the main focus in $2$-representation theory has been on the problem of classifying simple transitive $2$-representations so far. For the rest of this introduction, we will refer to this problem as the \emph{Classification Problem}. Fortunately, the generalization of the weak Jordan--H{\"o}lder theorem to finitary birepresentations is straightforward.

The Classification Problem for a given finitary $2$-category $\cC$ can be subdivided into several smaller classification problems by taking advantage of the so-called \emph{cell structure} of $\cC$, which was introduced in \cite[Section 4]{MM1}. The set of isomorphism classes of indecomposable $1$-morphisms of $\cC$ is naturally endowed with three preorders, called the left, the right and the two-sided preorders, generalizing Green's relations for semigroups \cite{Gre} the well-known Kazhdan--Lusztig preorders on
the Hecke algebras of Coxeter groups \cite{KL}. Just as in Kazhdan--Lusztig theory, the associated
equivalence classes are called left, right and two-sided cells and are partially ordered.
By \cite[Subsection 3.2]{CM}, for each simple transitive $2$-representation of $\cC$,
there is a unique maximal two-sided cell of $\cC$ that is not annihilated by the
$2$-representation, called its \emph{apex}. This shows that one can address
the Classification Problem for $\cC$ ``one apex at a time''. The generalization of these results to finitary birepresentations of finitary bicategories is also straightforward.

The next trick is to reduce the Classification Problem for $\cC$
even further. For that, one has to assume that the $\cC$ is \emph{fiat}, meaning that it is endowed with a weak categorical involution satisfying certain additional conditions (for some results it is not strictly necessary for the auto-equivalence to be involutory, but that is a technicality we do not want to discuss here). In the context of tensor categories, these notions relate to rigidity/pivotal structures. The involution maps each left cell to a right cell, called its \emph{dual}, and vice-versa, and the intersection of a left cell and its dual is called a \emph{diagonal $\mathcal{H}$-cell}. For any diagonal $\mathcal{H}$-cell $\mathcal{H}$ in any two-sided cell $\mathcal{J}$, one can naturally define a subquotient $\cC_{\mathcal{H}}$ of $\cC$ which is also fiat and contains
at most two cells: the trivial one (containing the identity $1$-morphism) and $\mathcal{H}$ (i.e. $\cC_{\mathcal{H}}$ has two
cells if $\mathcal{H}$ does not contain the identity and one cell otherwise), both of which are left, right and two-sided. In
\cite[Subsection 4.2]{MMMZ}, it was shown that there is a bijection between the set of equivalence classes of simple transitive $2$-representations of $\cC$ with apex $\mathcal{J}$ and the set of equivalence classes
of simple transitive $2$-representations of $\cC_{\mathcal{H}}$ with apex $\mathcal{H}$. The generalization in this paper, which we call \emph{strong $\mathcal{H}$-cell reduction},
is two-fold: not only do we prove it in the context of finitary birepresentations of \emph{fiab} bicategories (the bicategorical analog of fiat $2$-categories), but we also formulate it as a biequivalence between two $2$-categories of simple transitive birepresentations, rather than a mere bijection between sets of equivalence classes of simple transitive birepresentations. Both the formulation of this generalization and its proof are much more involved than
the original counterparts in \cite{MMMZ}. They are the content of Theorems \ref{theorem:H-reduction1} and \ref{theorem:H-reduction2}, but require a lot of technical preparation, which is also new and starts in Subsection \ref{quotientbicat}.

A key ingredient in strong $\mathcal{H}$-reduction is the relation between the birepresentations of a given
finitary bicategory $\cC$ and the coalgebras in $\cC$, which are $1$-morphisms together with
comultiplication and counit $2$-morphisms satisfying weak versions of coassociativity and counitality. This relation, which generalizes Ostrik's results in the context of tensor
categories \cite[Theorem 3.1]{Os} (see also \cite{EGNO} and references therein), was first studied in \cite{MMMT} and \cite{MMMZ} in the context of finitary $2$-categories (a major difference with Ostrik's results being that tensor categories are abelian, whereas finitary $2$-categories are only additive), but is vastly generalized here, in Sections \ref{section:coalgebras} and \ref{section:coalgebras-and-birepresentations}.
In this case, the generalization is three-fold: as before, everything is now done in the context of birepresentations
and bicategories and the key results are now formulated in terms of biequivalences between certain bicategories rather than mere bijections between sets of equivalence classes, but
in Subsection \ref{subsection:no-abelian} we additionally show how to avoid the (injective) abelianization of $\cC$ (under an additional assumption of \emph{$\mathcal{J}$-simplicity}), which was used in \cite{MMMT} and \cite{MMMZ} to generalize Ostrik's results. This is done by using an operation on coalgebras which we call \emph{framing} and
can be seen as a categorical generalization of conjugation. The concept of framing as such is not new, see e.g. the proof of \cite[Theorem 7.12.11]{EGNO}, but our application of it to avoid abelianization is new to the best of our knowledge.

Finally, in Section \ref{section:doublecentralizer} we state and prove the
\emph{double centralizer theorem} for simple transitive birepresentations of fiab
bicategories (Theorem \ref{thm:double-centralizer}). This is the analog of
\cite[Theorem 7.12.11]{EGNO} in our context. If $\cC$ is semisimple, the two theorems and their proofs coincide, but in general there is a subtle but important difference, which we
will explain in Section \ref{section:doublecentralizer}.

\noindent \textbf{Acknowledgments.} M.~M., Vo.~Ma, D.~T. and X.~Z. thank Va.~Mi. for organizing and hosting the very interesting and very productive workshop ``Representations of monoidal categories and $2$-categories'' in 2019,
during which many of the results in this paper and the next one (still to be finished) were first discussed. The hospitality of the University of East Anglia and financial support from EPSRC are gratefully acknowledged. D.~T. thanks tikzcd for patiently enduring his illustration attempts.

M.~M. was supported in part by Funda\c{c}{\~a}o para a Ci{\^e}ncia e a Tecnologia (Portugal), projects UID/MAT/04459/2013 (Center for Mathematical Analysis, Geometry and Dynamical Systems - CAMGSD) and PTDC/MAT-PUR/31089/2017 (Higher Structures and Applications). Vo.~Ma. and X.~Z. were partially supported by the Swedish Research Council and G{\"o}ran Gustafsson Stiftelse. Va. Mi. is partially supported by EPSRC grant EP/S017216/1.

%D.~T. does not deserve to be supported.

%%%%%%%%%%%%%%%%%%%%%%%%%%%%%%%%%%%%%%%%%

\section{Finitary bicategories and birepresentations}\label{section:fin-fiab}

%%%%%%%%%%%%%%%%%%%%%%%%%%%%%%%%%%%%%%%%%

In this section we briefly discuss finitary and (quasi) fiab bicategories and finitary birepresentations.
The notions and results in bicategory theory and birepresentation theory that are logically
intertwined, are presented together.

%%%%%%%%%%%%%%%%%%%%%%%%%%%%%%%%%%%%%%%%%

\subsection{Notation}\label{subsection:notations}

%%%%%%%%%%%%%%%%%%%%%%%%%%%%%%%%%%%%%%%%%

Throughout the paper, a \emph{category} is always assumed to be essentially small, meaning that it is equivalent to a small category (i.e. one whose classes of objects and morphisms are sets), and is denoted by a letter such as $\mathcal{C}$. A \emph{bicategory} is always assumed to be essentially small as well,
meaning that it is biequivalent to a small bicategory (i.e. one whose classes of objects,
$1$-morphisms and $2$-morphisms are sets), and is denoted by a letter such as $\cC$. These assumptions, which are satisfied by all examples of our
interest, are necessary to avoid set-theoretic problems (for some more comments, see \cite[Section 1.1]{EGNO}). Precise definitions and standard terminology for bicategories, pseudofunctors etc. can be found in many sources
e.g. \cite{Be}, \cite[Section 1.0]{Lei}, \cite{ML} or \cite[Chapter I,3]{Gr}, and the reader is referred
to these texts for details. Their strict versions, i.e. when all coherers are
identities, are called \emph{$2$-categories}, \emph{$2$-functors} etc.

Let us summarize some further notation:
\begin{itemize}

\item Objects in categories (which are not morphism categories in bicategories)
are denoted by letters such as $X\in\mathcal{C}$, and morphisms by $f\in\mathcal{C}$.

\item Objects in bicategories are denoted by letters such as $\mathtt{i}\in\cC$, $1$-morphisms by those such as
$\mathrm{F}\in\cC$, $2$-morphisms by Greek letters such as $\alpha\in\cC$, and the corresponding category of morphisms from $\mathtt{i}$ to $\mathtt{j}$
is denoted by $\cC(\mathtt{i},\mathtt{j})$.

\item Identity $1$-morphisms are denoted by $\mathbbm{1}_{\mathtt{i}}$ and identity $2$-morphisms
by $\mathrm{id}_{\mathrm{F}}$, where the subscripts are
sometimes omitted.

\item We write $\mathrm{F}\mathrm{G}=\mathrm{F}\circ_{\mathsf{h}}\mathrm{G}$ for the composition of $1$-morphisms (which is always horizontal), and $\circ_{\mathsf{v}}$ and $\circ_{\mathsf{h}}$ for the vertical and horizontal compositions of $2$-morphisms, respectively. For both compositions we use the operator convention, e.g. the source and
the target of $\mathrm{F}\mathrm{G}$ are equal to the source of $\mathrm{G}$ and
the target of $\mathrm{F}$, respectively;

\item a bicategory consists of a quadruple
$\cC=(\cC,\alpha,\lunit,\runit)$, where $\alpha$ is the
associator and $\lunit$ and $\runit$ are the left and right unitors.

\end{itemize}
We simplify the notation of
$\alpha,\lunit,\runit$ by only indicating the $1$-morphisms in their subscripts, but not the objects,
e.g. for any $1$-morphisms
$\mathrm{F}\in\cC(\mathtt{i},\mathtt{j})$, $\mathrm{G}\in
\cC(\mathtt{j},\mathtt{k})$, $\mathrm{H}\in\cC(\mathtt{k},\mathtt{l})$ with $\mathtt{i},\mathtt{j},\mathtt{k},
\mathtt{l}\in\cC$, the associator and unitors give isomorphisms
\begin{gather*}
\alpha_{\mathrm{H},\mathrm{G},\mathrm{F}}
:=\alpha^{\mathtt{l}\mathtt{k}\mathtt{j}\mathtt{i}}_{\mathrm{H},\mathrm{G},\mathrm{F}}
\colon
(\mathrm{H}\mathrm{G})\mathrm{F}\xrightarrow{\cong}
\mathrm{H}(\mathrm{G}\mathrm{F}),
\\
\lunit_{\mathrm{F}}:=(\lunit_{\mathrm{F}})^{\mathtt{j}\mathtt{i}}
\colon\mathbbm{1}_{\mathtt{j}}\mathrm{F}\xrightarrow{\cong}\mathrm{F},
\quad
\runit_{\mathrm{F}}:=(\runit_{\mathrm{F}})^{\mathtt{j}\mathtt{i}}
\colon\mathrm{F}\mathbbm{1}_{\mathtt{i}}\xrightarrow{\cong}\mathrm{F}.
\end{gather*}

Recall that $\lunit_{\mathbbm{1}_{\mathtt{i}}}=\runit_{\mathbbm{1}_{\mathtt{i}}}$. In several proofs we will use the following commutative diagrams (cf. \cite[Equations (6) and (7)]{Ke}):
\begin{gather}\label{eq:0.00}
\begin{tikzcd}[ampersand replacement=\&,column sep=3em]
\mathrm{GF}\&\\
\mathbbm{1}_{\mathtt{k}}(\mathrm{GF})\ar[u,"\lunit_{\mathrm{GF}}"]
\ar[r,"\alpha_{\mathbbm{1}_{\mathtt{k}},\mathrm{G},\mathrm{F}}^{\mone}",swap]
\&(\mathbbm{1}_{\mathtt{k}}\mathrm{G})\mathrm{F}\ar[ul,"\lunit_{\mathrm{G}}\circ_{\mathsf{h}}\mathrm{id}_{\mathrm{F}}",swap]
\end{tikzcd}
,\quad
\begin{tikzcd}[ampersand replacement=\&,column sep=3em]
\mathrm{GF}\&\\
(\mathrm{GF})\mathbbm{1}_{\mathtt{i}}\ar[u,"\runit_{\mathrm{GF}}"]
\ar[r,"\alpha_{\mathrm{G},\mathrm{F},\mathbbm{1}_{\mathtt{i}}}",swap]
\& \mathrm{G}(\mathrm{F}\mathbbm{1}_{\mathtt{i}})
\ar[ul,"\mathrm{id}_{\mathrm{G}}\circ_{\mathsf{h}}\runit_{\mathrm{F}}", swap]
\end{tikzcd}
,
\end{gather}
for any $\mathrm{F}\in\cC(\mathtt{i},\mathtt{j}),\mathrm{G}\in\cC(\mathtt{j},\mathtt{k})$ and any $\mathtt{i},\mathtt{j},\mathtt{k}\in\cC$.

\begin{example}\label{example:monoidal}
A monoidal category can be identified with the endomorphism category of a
one-object bicategory, where the monoidal product is defined by the horizontal composition.
This monoidal category is strict if and only if the bicategory is a $2$-category.
\end{example}

\begin{example}\label{example:functorbicat}
Given two bicategories $\cC$ and $\cD$, the pseudofunctors between them together with the strong transformations of pseudofunctors and modifications
form a bicategory $[\cC,\cD]$.	
If $\cD$ is a $2$-category, then
one can show that $[\cC,\cD]$ is also a $2$-category, see e.g.
\cite[Chapter I,3.3]{Gr}.
\end{example}

A \emph{biideal} $\mathcal{I}$ in a bicategory $\cC$ consists of an ideal $\mathcal{I}(\mathtt{i},\mathtt{j})$ inside each $\cC(\mathtt{i},\mathtt{j})$, such that for any $2$-morphisms $\beta\in\cC(\mathtt{k},\mathtt{l}),\gamma\in\mathcal{I}(\mathtt{j},\mathtt{k}),\zeta\in\cC(\mathtt{i},\mathtt{j})$, the horizontal composition satisfies $\beta\circ_{\mathsf{h}}\gamma\circ_{\mathsf{h}}\zeta\in\mathcal{I}(\mathtt{i},\mathtt{l})$. Finally, let
$\cC^{\,\mathrm{op}}$, $\cC^{\,\mathrm{co}}$ and $\cC^{\,\mathrm{co,op}}$ denote the bicategories obtained from $\cC$ by reversing only the horizontal composition,
only the vertical composition and both compositions, respectively.

%%%%%%%%%%%%%%%%%%%%%%%%%%%%%%%%%%%%%%%%%

\subsection{Finitary and fiab bicategories}\label{subsection:fin-fiab}

%%%%%%%%%%%%%%%%%%%%%%%%%%%%%%%%%%%%%%%%%

Let $\Bbbk$ be any field.

\begin{definition}\label{definition:fincat}
A \emph{finitary} category $\mathcal{C}$ (over $\Bbbk$) is a
$\Bbbk$-linear additive idempotent split category with finitely many isomorphism classes of indecomposable objects and finite dimensional morphism spaces.
\end{definition}

\begin{definition}\label{definition:finbicat}
We say that a bicategory $\cC$ is \emph{multifinitary} if $\cC$ has finitely many objects,
for all $\mathtt{i},\mathtt{j}\in\cC$ the categories $\cC(\mathtt{i},\mathtt{j})$ are
finitary, and horizontal composition $\circ_{\mathsf{h}}$ of $2$-morphisms is $\Bbbk$-bilinear.
If additionally the identity $1$-morphism on each object is indecomposable, then $\cC$ is called
\emph{finitary}.
\end{definition}

\begin{definition}\label{definition:fiab}
A \emph{quasi (multi)fiab bicategory} is a (multi)finitary bicategory $\cC$
together with an object-preserving $\Bbbk$-linear biequivalence ${}^{\star}\colon\cC\to\cC^{\,\mathrm{co,op}}$ with the property that, for every pair $\mathtt{i},\mathtt{j}\in\cC$ and every
$\mathrm{F}\in\cC(\mathtt{i},\mathtt{j})$, there exist adjunction $2$-morphisms $\mathrm{ev}_{\mathrm{F}}\colon
\mathrm{F}\mathrm{F}^{\star}\to\mathbbm{1}_{\mathtt{j}}$ and $\mathrm{coev}_{\mathrm{F}}\colon\mathbbm{1}_{\mathtt{i}}\to\mathrm{F}^{\star}\mathrm{F}$ such that
the diagrams
\begin{gather*}
\begin{tikzcd}[ampersand replacement=\&,column sep=3em]
\ar[d,"(\runit_{\mathrm{F}})^{\mone}",swap]\mathrm{F}\ar[r,equal]
\& \mathrm{F}
\\
\ar[d,"\mathrm{id}_{\mathrm{F}}\circ_{\mathsf{h}}\mathrm{coev}_{\mathrm{F}}",swap]\mathrm{F}
\mathbbm{1}_{\mathtt{i}}
\& \mathbbm{1}_{\mathtt{j}}\mathrm{F}
\ar[u,"\lunit_{\mathrm{F}}",swap]
\\
\mathrm{F}(\mathrm{F}^{\star}\mathrm{F})\ar[r,"\alpha_{\mathrm{F},\mathrm{F}^{\star},\mathrm{F}}^{\mone}",swap]
\&
(\mathrm{F}\mathrm{F}^{\star})\mathrm{F}\ar[u,"\mathrm{ev}_{\mathrm{F}}\circ_{\mathsf{h}}\mathrm{id}_{\mathrm{F}}",swap]
\end{tikzcd}
,\quad
\begin{tikzcd}[ampersand replacement=\&,column sep=3em]
\ar[d,"(\lunit_{\mathrm{F}^{\star}})^{\mone}",swap]\mathrm{F}^{\star}\ar[r,equal] \& \mathrm{F}^{\star}
\\
\ar[d,"\mathrm{coev}_{\mathrm{F}}\circ_{\mathsf{h}}\mathrm{id}_{\mathrm{F}^{\star}}",swap]\mathbbm{1}_{\mathtt{i}}\mathrm{F}^{\star}
\& \mathrm{F}^{\star}\mathbbm{1}_{\mathtt{j}}
\ar[u,"\runit_{\mathrm{F}^{\star}}",swap]
\\
(\mathrm{F}^{\star}\mathrm{F})\mathrm{F}^{\star}
\ar[r,"\alpha_{\mathrm{F}^{\star},\mathrm{F},\mathrm{F}^{\star}}",swap]
\&
\mathrm{F}^{\star}(\mathrm{F}\mathrm{F}^{\star})
\ar[u,"\mathrm{id}_{\mathrm{F}^{\star}}\circ_{\mathsf{h}}\mathrm{ev}_{\mathrm{F}}",swap]
\end{tikzcd}
\end{gather*}
commute.

If $\mathrm{Id}_{\ccC}$ and $({}^{\star})^{2}$ are equivalent in  $[\cC,\cC]$, then $\cC$ is called \emph{(multi)fiab}.

Following \cite{MM2}, the strict version of a (quasi) (multi)fiab bicategory is called a \emph{(quasi) (multi)fiat $2$-category}.
\end{definition}

Given a quasi (multi)fiab bicategory, there is also a quasi-inverse of
${}^{\star}\colon\cC\to\cC^{\,\mathrm{co,op}}$, which is usually denoted by the same
symbol but applied to the left side of $1$- and $2$-morphisms, e.g. ${}^{\star}\mathrm{F}$ instead of $\mathrm{F}^{\star}$. There are then additional $2$-morphisms $\mathrm{ev}_{\mathrm{F}}^{\prime}\colon
{}^{\star}\mathrm{F}\mathrm{F}\to\mathbbm{1}_{\mathtt{i}}$ and $\mathrm{coev}_{\mathrm{F}}^{\prime}\colon\mathbbm{1}_{\mathtt{j}}\to\mathrm{F}({}^{\star}\mathrm{F})$. In this paper we will only use this inverse in the multifiab case, where we can identify $\mathrm{F}^{\star}$ and ${}^{\star}\mathrm{F}$ and, in particular, obtain
\begin{gather*}
\mathrm{ev}_{\mathrm{F}}^{\prime}=
\big[\mathrm{F}^{\star}\mathrm{F}\xrightarrow{\cong}\mathrm{F}^{\star}
\mathrm{F}^{\star\star}\xrightarrow{\mathrm{ev}_{\mathrm{F}^{\star}}}
\mathbbm{1}_{\mathtt{i}}\big]
,\quad
\mathrm{coev}_{\mathrm{F}}^{\prime}=\big[\mathbbm{1}_{\mathtt{j}}\xrightarrow{\mathrm{coev}_{\mathrm{F}^{\star}}}
\mathrm{F}^{\star\star}\mathrm{F}^{\star}\xrightarrow{\cong}\mathrm{F}
\mathrm{F}^{\star}\big].
\end{gather*}
These satisfy the conditions expressed by the commutative diagrams
\begin{gather*}
\begin{tikzcd}[ampersand replacement=\&,column sep=3em]
\ar[d,"(\runit_{\mathrm{F}^{\star}})^{\mone}",swap]\mathrm{F}^{\star}\ar[r,equal]
\& \mathrm{F}^{\star}
\\
\ar[d,"\mathrm{id}_{\mathrm{F}^{\star}}\circ_{\mathsf{h}}\mathrm{coev}_{\mathrm{F}}^{\prime}",swap]\mathrm{F}^{\star}
\mathbbm{1}_{\mathtt{j}}
\& \mathbbm{1}_{\mathtt{i}}\mathrm{F}^{\star}
\ar[u,"\lunit_{\mathrm{F}^{\star}}",swap]
\\
\mathrm{F}^{\star}(\mathrm{F}\mathrm{F}^{\star})\ar[r,"\alpha_{\mathrm{F}^{\star},\mathrm{F},\mathrm{F}^{\star}}^{\mone}",swap]
\&
(\mathrm{F}^{\star}\mathrm{F})\mathrm{F}^{\star}\ar[u,"\mathrm{ev}_{\mathrm{F}}^{\prime}\circ_{\mathsf{h}}\mathrm{id}_{\mathrm{F}^{\star}}",swap]
\end{tikzcd}
,\quad
\begin{tikzcd}[ampersand replacement=\&,column sep=3em]
\ar[d,"(\lunit_{\mathrm{F}})^{\mone}",swap]\mathrm{F}\ar[r,equal] \& \mathrm{F}
\\
\ar[d,"\mathrm{coev}_{\mathrm{F}}^{\prime}\circ_{\mathsf{h}}\mathrm{id}_{\mathrm{F}}",swap]\mathbbm{1}_{\mathtt{j}}\mathrm{F}
\& \mathrm{F}\mathbbm{1}_{\mathtt{i}}
\ar[u,"\runit_{\mathrm{F}}",swap]
\\
(\mathrm{F}\mathrm{F}^{\star})\mathrm{F}
\ar[r,"\alpha_{\mathrm{F},\mathrm{F}^{\star},\mathrm{F}}",swap]
\&
\mathrm{F}(\mathrm{F}^{\star}\mathrm{F})
\ar[u,"\mathrm{id}_{\mathrm{F}}\circ_{\mathsf{h}}\mathrm{ev}_{\mathrm{F}}^{\prime}",swap]
\end{tikzcd}
.
\end{gather*}
For readers familiar with string diagrams, we note that the diagrams for
$\mathrm{ev}^{\prime}_{\mathrm{F}}$ and $\mathrm{coev}^{\prime}_{\mathrm{F}}$ are obtained from the ones for $\mathrm{ev}_{\mathrm{F}}$ and $\mathrm{coev}_{\mathrm{F}}$
by inverting the orientation.

\begin{remark}
The term \emph{weakly fiat $2$-category} was introduced in \cite[Subsection 2.5]{MM5}. The corresponding notion for bicategories would be \emph{weakly fiab},
but we decided to change \emph{weakly} to \emph{quasi}, and use that terminology
from now on, to avoid confusion with the notion of weakness in bicategories.
\end{remark}

\begin{remark}\label{remark:fiab-vs-pivotal}
Quasi fiab and fiab one-object bicategories correspond to rigid and pivotal monoidal categories in the terminology of e.g. \cite{EGNO}, respectively. Multifinitary bicategories are the additive analog of multitensor categories, cf. \cite[Definition 4.1.1]{EGNO}.
\end{remark}

\begin{example}\label{example:fiab-vs-pivotal-1}
A particular class of quasi fiab one-object bicategories is that of fusion categories, which are semisimple rigid monoidal categories, e.g.
$\cV\mathrm{ect}^{\omega}(G)$, the category of $G$-graded vector spaces whose monoidal product is twisted by a $3$-cocycle on $G$, and
$U_q(\mathfrak{g})\text{-}\mathrm{mod}_{\mathrm{ss}}$, the semisimplified module category
of the quantum group associated to a complex finite dimensional semisimple Lie algebra $\mathfrak{g}$ for $q$ a root of unity.
\end{example}

\begin{example}\label{example:fiab-vs-pivotal-2}
Let $\Bbbk=\mathbb{C}$ and let $W=(W,\mathtt{S})$ be a Coxeter group
with its reflection representation. To these data one can associate the one-object bicategory of \emph{Soergel bimodules} $\cS=\cS_{\mathbb{C}}(W,\mathtt{S})$, which categorifies the Hecke algebra of $W$ such that the indecomposable $1$-morphisms $\mathrm{C}_w$
correspond to the Kazhdan--Lusztig basis elements $c_w$, for $w\in W$.
The one-object bicategory $\cS$ can be defined over the polynomial algebra, as in e.g.
\cite{EW}, or over the coinvariant algebra, as in e.g. \cite{So2}. For finite $W$,
the bicategory $\cS$ is finitary when defined over the coinvariant algebra.

Based on results in \cite{EW}, Lusztig \cite[\S 18.5]{Lu} associated with each
two-sided cell $\mathcal{J}$ of $W$ a semisimple one-object bicategory
$\cA_{\mathcal{J}}$, called the \emph{asymptotic limit} or the \emph{asymptotic bicategory},
which categorifies the direct summand of the asymptotic Hecke algebra corresponding to $\mathcal{J}$ (or, in Lusztig's terminology, the $J$-ring associated with $\mathcal{J}$).
By \cite[Section 5]{EW6}, the monoidal category $\cS$ is pivotal for any $W$,
and so is $\cA_{\mathcal{J}}$ for any $\mathcal{J}$ of $W$.

Thus, for any finite Coxeter group, Soergel bimodules over the coinvariant algebra form a
one-object fiab bicategory in our terminology.
\end{example}

\begin{example}
There are also fiab bicategories with more than one object which play an important role in
birepresentation theory. For example, the bicategory of \emph{singular Soergel bimodules} \cite{Wi} associated with a Coxeter group $W=(W,\mathtt{S})$, whose objects are indexed by the
parabolic subsets of $\mathtt{S}$. This is why we do not restrict our setup to rigid or pivotal monoidal categories, but always consider (quasi) fiab bicategories in general.
\end{example}

The \emph{abelianizations} discussed in \cite[Section 3]{MMMT}
carry over verbatim to the bicategorical setting. Indeed, a multifinitary bicategory $\cC$ admits an
injective and a projective abelianization, denoted by $\underline{\cC}$ and $\overline{\cC}$, respectively. These are such that their morphism categories are abelian and $\cC$ is biequivalent
to the $2$-full subbicategory of injective, respectively projective, $1$-morphisms. Note that all the $2$-morphisms and coheres from $\cC$ extend to $\underline{\cC}$ and $\overline{\cC}$, and we will usually not distinguish between the ones for $\cC$, or $\underline{\cC}$ and $\overline{\cC}$ to ease notation.

\begin{remark}\label{remark:rather-technical}
Even if $\cC$ is multifiab, its abelianizations need not be, but ${}^{\star}$ gives rise to an  antiequivalence between $\underline{\cC}$ and $\overline{\cC}$.
\end{remark}

Similarly, finitary birepresentations, which will be discussed in Subsection \ref{subsection:bireps}, admit
abelianizations.

%%%%%%%%%%%%%%%%%%%%%%%%%%%%%%%%%%%%%%%%%

\subsection{Finitary birepresentations}\label{subsection:bireps}

%%%%%%%%%%%%%%%%%%%%%%%%%%%%%%%%%%%%%%%%%

\begin{definition}\label{definition:fincatcat}
We let $\mathfrak{A}^{f}_{\Bbbk}$ denote the $2$-category of finitary categories,
$\Bbbk$-linear functors (recall that any $\Bbbk$-linear
functor between $\Bbbk$-linear, additive categories is
automatically additive) and natural transformations.
\end{definition}

Let $\cC$ be a finitary bicategory, defined over $\Bbbk$.

\begin{definition}\label{definition:fin-birepresentation}
A \emph{finitary (left) birepresentation of
$\cC$} is a (covariant) $\Bbbk$-linear pseudofunctor
$\mathbf{M}\colon\cC\to\mathfrak{A}^{f}_{\Bbbk}$.
\end{definition}

Concretely, a finitary birepresentation
$\mathbf{M}$ of $\cC$ associates
\begin{itemize}

\item a finitary category $\mathbf{M}(\mathtt{i})$, defined over $\Bbbk$,
to every $\mathtt{i}\in\cC$;

\item a $\Bbbk$-linear functor
$\mathbf{M}_{\mathtt{j}\mathtt{i}}\colon\cC(\mathtt{i},\mathtt{j})\to\mathfrak{A}^{f}_{\Bbbk}\big(\mathbf{M}(\mathtt{i}),\mathbf{M}(\mathtt{j})\big)$ to every pair $\mathtt{i},\mathtt{j}\in\cC$;

\item a natural isomorphism
\begin{gather*}
\begin{tikzcd}[ampersand replacement=\&]
\mathbf{M}(\mathtt{i})
\ar[rr,bend left=15,"\mathbf{M}_{\mathrm{i}\mathrm{i}}(\mathbbm{1}_{\mathtt{i}})" {name=U}]
\ar[rr,bend right=15,"\mathrm{Id}_{\mathbf{M}(\mathtt{i})}" {name=D},swap]
\ar[d,Rightarrow,shorten <= .1em,shorten >= .1em,from=U,to=D,"\iota_{\mathtt{i}}"]
\&[2em]\&
\mathbf{M}(\mathtt{i})
\end{tikzcd}
\end{gather*}
to every $\mathtt{i}\in\cC$;

\item a natural isomorphism
\begin{gather*}
\begin{tikzcd}[scale=0.5,ampersand replacement=\&]
\ar[d,"\mathbf{M}_{\mathtt{k}\mathtt{j}}\boxtimes\mathbf{M}_{\mathtt{j}\mathtt{i}}",swap]\cC(\mathtt{j},\mathtt{k})\boxtimes\cC(\mathtt{i},\mathtt{j})
\ar[r,"\circ_{\mathsf{h}}"]
\&
\cC(\mathtt{i},\mathtt{k})\ar[d,"\mathbf{M}_{\mathtt{k}\mathtt{i}}"]
\\
\mathfrak{A}^{f}_{\Bbbk}\big(\mathbf{M}(\mathtt{j}),\mathbf{M}(\mathtt{k})\big)\boxtimes
\mathfrak{A}^{f}_{\Bbbk}\big(\mathbf{M}(\mathtt{i}),\mathbf{M}(\mathtt{j})\big)
\ar[r,"\circ_{\mathsf{h}}",swap]
\ar[ur,Rightarrow,shorten >=0.5em,shorten <=0.5em,"\mu_{\mathtt{k}\mathtt{j}\mathtt{i}}"]
\&
\mathfrak{A}^{f}_{\Bbbk}\big(\mathbf{M}(\mathtt{i}),\mathbf{M}(\mathtt{k})\big)
\end{tikzcd}
\end{gather*}
to every triple $\mathtt{i},\mathtt{j},\mathtt{k}\in\cC$.
\end{itemize}

These data are such that the diagrams
\begin{gather}\label{eq:birepresentation2}
\begin{tikzcd}[scale=0.5,ampersand replacement=\&]
\ar[d,"\mathrm{id}_{\mathbf{M}_{\mathtt{j}\mathtt{i}}(\mathrm{F})}\circ_{\mathsf{h}}\iota_{\mathtt{i}}",swap]\mathbf{M}_{\mathtt{j}\mathtt{i}}
(\mathrm{F})\mathbf{M}_{\mathtt{i}\mathtt{i}}(\mathbbm{1}_{\mathtt{i}})
\ar[r,"\mu_{\mathtt{j}\mathtt{i}\mathtt{i}}^{\mathrm{F},\mathbbm{1}_{\mathtt{i}}}"]
\&
\mathbf{M}_{\mathtt{j}\mathtt{i}}(\mathrm{F}\mathbbm{1}_{\mathtt{i}})
\ar[d,"\mathbf{M}_{\mathtt{j}\mathtt{i}}(\runit_{\mathrm{F}})"]
\\
\mathbf{M}_{\mathtt{j}\mathtt{i}}(\mathrm{F})
\mathrm{Id}_{\mathbf{M}(\mathtt{i})}\ar[r,equal]
\&
\mathbf{M}_{\mathtt{j}\mathtt{i}}(\mathrm{F})
\end{tikzcd}
,
\begin{tikzcd}[ampersand replacement=\&]
\ar[d,"\iota_{\mathtt{j}}\circ_{\mathsf{h}}\mathrm{id}_{\mathbf{M}_{\mathtt{j}\mathtt{i}}(\mathrm{F})}",swap]
\mathbf{M}_{\mathtt{j}\mathtt{j}}(\mathbbm{1}_{\mathtt{j}})\mathbf{M}_{\mathtt{j}\mathtt{i}}(\mathrm{F})
\ar[r,"\mu_{\mathtt{j}\mathtt{j}\mathtt{i}}^{\mathbbm{1}_{\mathtt{j}},\mathrm{F}}"]
\& \mathbf{M}_{\mathtt{j}\mathtt{i}}(\mathbbm{1}_{\mathtt{j}}\mathrm{F})
\ar[d,"\mathbf{M}_{\mathtt{j}\mathtt{i}}(\lunit_{\mathrm{F}})"]
\\
\mathrm{Id}_{\mathbf{M}(\mathtt{j})}
\mathbf{M}_{\mathtt{j}\mathtt{i}}(\mathrm{F})\ar[r, equal]
\&
\mathbf{M}_{\mathtt{j}\mathtt{i}}(\mathrm{F})
\end{tikzcd},
\\
\label{eq:birepresentation3}
\begin{tikzcd}[ampersand replacement=\&]
\ar[d,"\mu_{\mathtt{l}\mathtt{k}\mathtt{j}}^{\mathrm{H},\mathrm{G}}\circ_{\mathsf{h}}
\mathrm{id}_{\mathbf{M}_{\mathtt{j}\mathtt{i}}(\mathrm{F})}",swap]
\big(\mathbf{M}_{\mathtt{l}\mathtt{k}}(\mathrm{H})\mathbf{M}_{\mathtt{k}\mathtt{j}}(\mathrm{G})\big)\mathbf{M}_{\mathtt{j}\mathtt{i}}(\mathrm{F})
\ar[r,equal]\& \mathbf{M}_{\mathtt{l}\mathtt{k}}(\mathrm{H})\big(
\mathbf{M}_{\mathtt{k}\mathtt{j}}(\mathrm{G})\mathbf{M}_{\mathtt{j}\mathtt{i}}(\mathrm{F})\big)
\ar[d,"\mathrm{id}_{\mathbf{M}_{\mathtt{l}\mathtt{k}}(\mathrm{H})}\circ_{\mathsf{h}}
\mu_{\mathtt{k}\mathtt{j}\mathtt{i}}^{\mathrm{G},\mathrm{F}}"]
\\
\ar[d,"\mu_{\mathtt{l}\mathtt{j}\mathtt{i}}^{\mathrm{H}\mathrm{G},\mathrm{F}}",swap]\mathbf{M}_{\mathtt{l}\mathtt{j}}(\mathrm{H}\mathrm{G})\mathbf{M}_{\mathtt{j}\mathtt{i}}(\mathrm{F}) \&
\mathbf{M}_{\mathtt{l}\mathtt{k}}(\mathrm{H})\mathbf{M}_{\mathtt{k}\mathtt{i}}(\mathrm{G}\mathrm{F})
\ar[d,"\mu_{\mathtt{l}\mathtt{k}\mathtt{i}}^{\mathrm{H},\mathrm{G}\mathrm{F}}"]
\\
\mathbf{M}_{\mathtt{l}\mathtt{i}}\big((\mathrm{H}\mathrm{G})\mathrm{F}\big)\ar[r,"\mathbf{M}_{\mathtt{l}\mathtt{i}}({\alpha}_{\mathrm{H},\mathrm{G},\mathrm{F}})",swap] \&
\mathbf{M}_{\mathtt{l}\mathtt{i}}\big(\mathrm{H}(\mathrm{G}\mathrm{F})\big)
\end{tikzcd}
\end{gather}
commute for all $\mathtt{i},\mathtt{j},\mathtt{k},\mathtt{l}\in\cC$ and
all $\mathrm{F}\in\cC(\mathtt{i},\mathtt{j}),\mathrm{G}\in
\cC(\mathtt{j},\mathtt{k}),\mathrm{H}\in\cC(\mathtt{k},\mathtt{l})$.

\begin{example}\label{example:fin-birepresentation}
If $\cC$ is a finitary $2$-category, then any finitary $2$-representation of $\cC$ is a finitary birepresentation (namely, one whose coherers are trivial).
\end{example}

\begin{example}\label{example:Yoneda-birep}
For any $\mathtt{i}\in\cC$, the \emph{principal birepresentation} (which is also called \emph{Yoneda birepresentation}) $\mathbf{P}_{\mathtt{i}}:=\cC(\mathtt{i},{}_{-})$ is a
finitary birepresentation of $\cC$. If $\cC$ is a finitary $2$-category, then the principal birepresentations are all finitary $2$-representations.
\end{example}

Let $\mathbf{M}$ and $\mathbf{N}$
be two finitary birepresentations of $\cC$.

\begin{definition}\label{definition:morphism-of-birepresentations}
A \emph{morphism of
finitary birepresentations}
$\Phi\colon\mathbf{M}\to\mathbf{N}$ is a
$\Bbbk$-linear strong transformation of pseudofunctors.
\end{definition}

Concretely, a morphism of birepresentations
$\Phi\colon\mathbf{M}\to\mathbf{N}$ associates
\begin{itemize}

\item a $\Bbbk$-linear functor $\Phi_{\mathtt{i}}\colon\mathbf{M}(\mathtt{i})\to\mathbf{N}(\mathtt{i})$ to each $\mathtt{i}\in\cC$;

\item a natural isomorphism
\begin{gather*}
\begin{tikzcd}[ampersand replacement=\&]
\cC(\mathtt{i},\mathtt{j})
\ar[d,"\mathbf{N}_{\mathtt{j}\mathtt{i}}",swap]
\ar[r,"\mathbf{M}_{\mathtt{j}\mathtt{i}}"]
\&
\mathfrak{A}^{f}_{\Bbbk}\big(\mathbf{M}(\mathtt{i}),\mathbf{M}(\mathtt{j})\big)
\ar[d,"\Phi_{\mathtt{j}}\circ_{\mathsf{h}}{}_{-}"]
\\
\mathfrak{A}^{f}_{\Bbbk}\big(\mathbf{N}(\mathtt{i}),\mathbf{N}(\mathtt{j})\big)
\ar[r,"{}_{-}\circ_{\mathsf{h}}\Phi_{\mathtt{i}}",swap]
\ar[ur, Rightarrow,shorten <= .5em,shorten >= .5em,"\phi_{\mathtt{j}\mathtt{i}}"]
\&
\mathfrak{A}^{f}_{\Bbbk}\big(\mathbf{M}(\mathtt{i}),\mathbf{N}(\mathtt{j})\big)
\end{tikzcd}
\end{gather*}
to every pair $\mathtt{i},\mathtt{j}\in\cC$.
\end{itemize}
These data are such that the diagrams
\begin{gather}\label{eq:morphism1}
\begin{tikzcd}[ampersand replacement=\&]
\ar[d,"\iota_{\mathtt{i}}^{\mathbf{N}}\mathrm{id}_{\Phi_{\mathtt{i}}}",swap]
\mathbf{N}_{\mathtt{i}\mathtt{i}}(\mathbbm{1}_{\mathtt{i}})\Phi_{\mathtt{i}}\ar[rr, "\phi_{\mathrm{i}\mathrm{i}}^{\mathbbm{1}_{\mathtt{i}}}"]
\& \&
\Phi_{\mathtt{i}}\mathbf{M}_{\mathtt{i}\mathtt{i}}(\mathbbm{1}_{\mathtt{i}})\ar[d,"\mathrm{id}_{\Phi_{\mathtt{i}}}\iota_{\mathtt{i}}^{\mathbf{M}}"]
\\
\mathrm{Id}_{\mathbf{N}(\mathtt{i})}\Phi_{\mathtt{i}}\ar[r, equal]
\&
\Phi_{\mathtt{i}}\ar[r,equal]
\&
\Phi_{\mathtt{i}}\mathrm{Id}_{\mathbf{M}(\mathtt{i})}
\end{tikzcd},
\\
\label{eq:morphism2}
\begin{tikzcd}[ampersand replacement=\&,column sep=3.8em]
\ar[d,"\nu_{\mathtt{k}\mathtt{j}\mathtt{i}}^{\mathrm{G},\mathrm{F}}\mathrm{id}_{\Phi_{\mathtt{i}}}",swap]\mathbf{N}_{\mathtt{k}\mathtt{j}}(\mathrm{G})
\mathbf{N}_{\mathtt{j}\mathtt{i}}(\mathrm{F})\Phi_{\mathtt{i}}
\ar[r,"\mathrm{id}_{\mathbf{N}_{\mathtt{k}\mathtt{j}}(\mathrm{G})}\phi_{\mathtt{j}\mathtt{i}}^{\mathrm{F}}"]
\&
\mathbf{N}_{\mathtt{k}\mathtt{j}}(\mathrm{G})\Phi_{\mathtt{j}}\mathbf{M}_{\mathtt{j}\mathtt{i}}(\mathrm{F})
\ar[r,"\phi_{\mathtt{k}\mathtt{j}}^{\mathrm{G}}\mathrm{id}_{\mathbf{M}_{\mathtt{j}\mathtt{i}}(\mathrm{F})}"]
\&
\Phi_{\mathtt{k}}\mathbf{M}_{\mathtt{k}\mathtt{j}}(\mathrm{G})
\mathbf{M}_{\mathtt{j}\mathtt{i}}(\mathrm{F})
\ar[d,"\mathrm{id}_{\Phi_{\mathtt{k}}}\mu_{\mathtt{k}\mathtt{j}\mathtt{i}}^{\mathrm{G},\mathrm{F}}"]
\\
\mathbf{N}_{\mathtt{k}\mathtt{i}}(\mathrm{G}\mathrm{F})\Phi_{\mathtt{i}}
\ar[rr,"\phi_{\mathtt{k}\mathtt{i}}^{\mathrm{G}\mathrm{F}}",swap]
\&\&
\Phi_{\mathtt{k}}\mathbf{M}_{\mathtt{k}\mathtt{i}}(\mathrm{G}\mathrm{F})
\end{tikzcd}
\end{gather}
commute
for all $\mathtt{i},\mathtt{j},\mathtt{k}\in\cC$ and
all $\mathrm{F}\in\cC(\mathtt{i},\mathtt{j}),\mathrm{G}\in
\cC(\mathtt{j},\mathtt{k})$.

Note that we will omit the symbol $\circ_{\mathsf{h}}$ for the horizontal composition of $2$-morphisms if it causes 
no confusion, e.g. \eqref{eq:morphism1}, \eqref{eq:morphism2} and so on.

Let $\Phi,\Psi\colon\mathbf{M}\to\mathbf{N}$ be two morphisms of birepresentations.

\begin{definition}\label{definition:modification}
A \emph{modification} $\sigma\colon\Phi\to\Psi$
is a modification between the strong $\Bbbk$-linear transformations.
\end{definition}

Concretely, a modification $\sigma\colon\Phi\to\Psi$ associates a natural transformation $\sigma_{\mathtt{i}}\colon\Phi_{\mathtt{i}}\to\Psi_{\mathtt{i}}$
to every $\mathtt{i}\in\cC$.
These data are such that the diagram
\begin{gather*}
\begin{tikzcd}[ampersand replacement=\&,column sep=4em]
\ar[d,"\phi_{\mathtt{j}\mathtt{i}}^{\mathrm{F}}",swap]\mathbf{N}_{\mathtt{j}\mathtt{i}}(\mathrm{F})\Phi_{\mathtt{i}}\ar[r,"\mathrm{id}_{\mathbf{N}_{\mathtt{j}\mathtt{i}}(\mathrm{F})}\sigma_{\mathtt{i}}"]
\& \mathbf{N}_{\mathtt{j}\mathtt{i}}(\mathrm{F})\Psi_{\mathtt{i}}
\ar[d,"\psi_{\mathtt{j}\mathtt{i}}^{\mathrm{F}}"]
\\
\Phi_{\mathtt{j}}\mathbf{M}_{\mathtt{j}\mathtt{i}}(\mathrm{F})\ar[r,"\sigma_{\mathtt{j}}
\mathrm{id}_{\mathbf{M}_{\mathtt{j}\mathtt{i}}(\mathrm{F})}",swap]
\& \Psi_{\mathtt{j}}\mathbf{M}_{\mathtt{j}\mathtt{i}}(\mathrm{F})
\end{tikzcd}
\end{gather*}
commutes for all $\mathtt{i},\mathtt{j}\in\cC$ and
all $\mathrm{F}\in\cC(\mathtt{i},\mathtt{j})$.

We say that two finitary birepresentations
$\mathbf{M},\mathbf{N}$ of $\cC$ are \emph{equivalent} if
there are morphisms of birepresentations
$\Phi\colon\mathbf{M}\to\mathbf{N}$ and $\Psi\colon\mathbf{N}\to\mathbf{M}$
and invertible modifications $\Psi\Phi\xrightarrow{\cong}\mathrm{Id}_{\mathbf{M}}$ and
$\Phi\Psi\xrightarrow{\cong}\mathrm{Id}_{\mathbf{N}}$.

\begin{definition}\label{definition:fincatcat-add}
As in Example \ref{example:functorbicat}, the fact that
$\mathfrak{A}^{f}_{\Bbbk}$ is a $\Bbbk$-linear additive $2$-category implies that
\begin{gather*}
\cC\text{-}\mathrm{afmod}:=
[\cC,\mathfrak{A}^{f}_{\Bbbk}]
\end{gather*}
is a $\Bbbk$-linear additive $2$-category, called the \emph{$2$-category of finitary birepresentations of $\cC$}.
\end{definition}

Recall the following terminology, where $\mathrm{add}$
denotes the \emph{additive closure}, meaning the closure under taking finite direct sums and summands.

\begin{definition}\label{definition:generator}
Let $\cC$ be a multifinitary bicategory and $\mathbf{M}$ a finitary birepresentation of $\cC$. Then
\begin{enumerate}[$($i$)$]

\item\label{definition:generator-1}
the birepresentation $\mathbf{M}$ is \emph{generated by} $X\in\mathbf{M}(\mathtt{i})$,
for some $\mathtt{i}\in\cC$, if the embedding
\begin{gather*}
\mathrm{add}\big\{\mathbf{M}_{\mathtt{j}\mathtt{i}}(\mathrm{F})X\mid\mathrm{F}\in\cC(\mathtt{i},\mathtt{j})\big\}
\hookrightarrow
\mathbf{M}(\mathtt{j})
\end{gather*}
is an equivalence for all $\mathtt{j}\in\cC$;

\item\label{definition:generator-2}
the birepresentation $\mathbf{M}$ is \emph{cyclic} if it is generated by some
$X\in\mathbf{M}(\mathtt{i})$ for some $\mathtt{i}\in\cC$;

\item\label{definition:generator-3}
the birepresentation $\mathbf{M}$ is \emph{transitive} if it is non-zero and is generated by any non-zero $X\in\mathbf{M}(\mathtt{i})$ for any $\mathtt{i}\in\cC$.
\end{enumerate}
\end{definition}

By definition, a \emph{$\cC$-stable ideal} $\mathbf{I}$ of a finitary birepresentation
$\mathbf{M}$ of $\cC$ is the assignment of an ideal
$\mathbf{I}(\mathtt{i})\subseteq\mathbf{M}(\mathtt{i})$ to each $\mathtt{i}\in\cC$, such that
$\mathbf{M}_{\mathtt{j}\mathtt{i}}(\mathrm{F})\big(\mathbf{I}(\mathtt{i})\big)
\subseteq\mathbf{I}(\mathtt{j})$, for all $\mathrm{F}\in\cC(\mathtt{i},\mathtt{j})$.
We say that $\mathbf{I}$ is \emph{proper} if $\{0\}\subsetneq\mathbf{I}(\mathtt{i})\subsetneq
\mathbf{M}(\mathtt{i})$ for some
$\mathtt{i}\in\cC$.

\begin{definition}\label{definition:simpletransitive}
A finitary birepresentation $\mathbf{M}$ is said to be \emph{simple transitive} if it has no proper $\cC$-stable ideals.
\end{definition}

It follows immediately from Definition \ref{definition:simpletransitive} that any simple transitive birepresentation is transitive. The converse is false in general, but every transitive birepresentation $\mathbf{M}$ of $\cC$ has a unique simple transitive quotient, by
the straightforward generalization of \cite[Lemma 4]{MM5} to bicategories.

\begin{definition}\label{def:cfmodstmod}
We use the following $1,2$-full $2$-subcategories of $\cC\text{-}\mathrm{afmod}$ (where $1,2$-full means $1$-full and $2$-full, by definition):
\begin{enumerate}[$($i$)$]

\item $\cC\text{-}\mathrm{cfmod}$ denotes the one consisting of all cyclic (finitary) birepresentations;

\item $\cC\text{-}\mathrm{tfmod}$ denotes the one consisting of all transitive (finitary) birepresentations;

\item $\cC\text{-}\mathrm{stmod}$ denotes the one consisting of all simple transitive (finitary) birepresentations.
\end{enumerate}
\end{definition}

Note that we have $\cC\text{-}\mathrm{stmod}\subseteq\cC\text{-}\mathrm{tfmod}\subseteq\cC\text{-}\mathrm{cfmod}\subseteq\cC\text{-}\mathrm{afmod}$.

Recall that we want to generalize some of the results in e.g. \cite{MMMT} and \cite{MMMZ}
to the weak setup.
Fortunately, most previous results carry over to this more general framework, due to two strictification theorems:
\begin{itemize}

\item every multifinitary bicategory $\cC$ is biequivalent to a multifinitary $2$-category, by the classical strictification results in this setting (see e.g. \cite[Section 1.4]{GPS} or \cite[Section 2.3]{Lei});

\item if $\cC$ is a multifinitary $2$-category, then its $2$-category of finitary birepresentations is biequivalent to its $2$-category of finitary $2$-representations, by \cite[Section 4.2]{Pow}.

\end{itemize}
A particular example of results that carry
over verbatim, and that we will need later, is the following \emph{weak Jordan--H{{\"o}}lder theorem},
cf. \cite[Section 4]{MM5}.

\begin{theorem}\label{thm:jh}
Let $\cC$ be a multifinitary bicategory.
For any finitary birepresentation $\mathbf{M}$ of $\cC$,
there is a finite filtration by subbirepresentations of $\cC$
\begin{gather*}
0=\mathbf{M}_{0}\subsetneq\mathbf{M}_{1}
\subsetneq\dots\subsetneq
\mathbf{M}_{m}=\mathbf{M},
\end{gather*}
where every $\mathbf{M}_{k}$ generates a $\cC$-stable ideal $\mathbf{I}_{k}$ in $\mathbf{M}_{k+1}$, such that
$\mathbf{M}_{k+1}/\mathbf{I}_{k}$ is transitive
and has a unique associated simple transitive quotient $\mathbf{L}_{k+1}$. Up to equivalence and ordering, the set $\{\mathbf{L}_{k}\mid 1\leq k\leq m\}$ is an invariant of $\mathbf{M}$.
\end{theorem}

For any finitary birepresentation $\mathbf{M}$,
we call those simple transitive birepresentations $\mathbf{L}_{k}$,
where $1\leq k\leq m$, defined as in Theorem \ref{thm:jh} the \emph{weak Jordan--H{\"o}lder constituents} of $\mathbf{M}$.

We also briefly need the following.

\begin{definition}\label{definition:fincatcat-abel}
We let $\mathfrak{R}_{\Bbbk}$ denote the $2$-category of abelian finitary
categories, i.e. objects are $\Bbbk$-linear additive categories which
are equivalent to categories of finitely generated modules over finite dimensional associative $\Bbbk$-algebras, $1$-morphisms are
$\Bbbk$-linear functors and $2$-morphisms are natural transformations.

An \emph{abelian finitary (left) birepresentation} of
$\cC$ is a $\Bbbk$-linear pseudofunctor
$\mathbf{M}\colon\cC\to\mathfrak{R}_{\Bbbk}$.

A finitary birepresentation $\mathbf{M}$ can be extended to an abelian birepresentation $\underline{\mathbf{M}}$, so that $\underline{\mathbf{M}}(\mathtt{i})
:=\underline{\mathbf{M}(\mathtt{i})}$ for any $\mathtt{i}\in\cC$. We can similarly abelianize $\mathbf{M}$ projectively to obtain $\overline{\mathbf{M}}$ with $\overline{\mathbf{M}}(\mathtt{i})
:=\overline{\mathbf{M}(\mathtt{i})}$ for any $\mathtt{i}\in\cC$.
\end{definition}

The notions we have seen above for additive finitary birepresentations carry over verbatim to abelian finitary birepresentations and we leave it to the reader to write out the details.

\begin{remark}\label{remark:fin-birepresentation-abel}
We only need abelian birepresentations very rarely in this paper and refer to \cite{MM2} for a comparison of finitary and abelian birepresentations.
\end{remark}

%%%%%%%%%%%%%%%%%%%%%%%%%%%%%%%%%%%%%%%%%

\subsection{The additive closure of a bicategory}\label{subsection:add-closure}

%%%%%%%%%%%%%%%%%%%%%%%%%%%%%%%%%%%%%%%%%

Let $\mathcal{C}_{j}$ be additive categories for $j=1,\dots,n$. We define $\bigoplus_{j=1}^n\mathcal{C}_{j}$ to be the additive category whose objects are formal direct sums $X_{1}\oplus\dots\oplus X_{s}$, where $X_{q}\in\mathcal{C}_j$ for some $j=\jmath(X_{q})\in\{1,\dots,n\}$.
Morphisms in $\mathrm{Hom}_{\bigoplus_{j=1}^n\mathcal{C}_j}(X_{1}\oplus\dots\oplus X_{s},Y_{1}\oplus\dots\oplus Y_{t})$ are matrices of morphisms

\begin{gather*}
(f_{pq})_{p=1,\dots,t;q=1,\dots,s},
\text{ where }
f_{pq}\in
\begin{cases}
\mathrm{Hom}_{\mathcal{C}_{j}}(X_{q},Y_{p})
& \text{if }\jmath(X_{q})=\jmath(Y_{p})=j,
\\
\{0\} & \text{otherwise}.
\end{cases}
\end{gather*}
Composition is given by matrix multiplication. The additive structure is given by concatenation, i.e. $(X_{1}\oplus\dots\oplus X_{s})\oplus(Y_{1}\oplus\dots\oplus Y_{t}):=X_{1}\oplus\dots\oplus X_{s}\oplus Y_{1}\oplus\dots\oplus Y_{t}$.

Let $\cC$ now be a finitary bicategory. Define $\cC^{\,\oplus}$ as follows.
It has one object $\bullet$ and $\cC^{\,\oplus}(\bullet,\bullet)=\bigoplus_{\mathtt{j},\mathtt{k}\in\ccC}\cC(\mathtt{j},\mathtt{k})$, as defined above. If
$\mathrm{F}_{q}\in\cC(\mathtt{j},\mathtt{k})$ for some $\mathtt{j},\mathtt{k}\in\cC$, we set $\mathtt{i}_{s(\mathrm{F}_q)}=\mathtt{j}$ and $\mathtt{i}_{t(\mathrm{F}_{q})}=\mathtt{k}$ for source and target, respectively.

Composition of $1$-morphisms is given by
\begin{gather*}
(\mathrm{F}_{1}\oplus\dots\oplus\mathrm{F}_{s})(\mathrm{G}_{1}\oplus\dots\oplus\mathrm{G}_{t})
\\
:=\mathrm{F}_{1}\mathrm{G}_{1}\oplus\dots\oplus\mathrm{F}_{1}\mathrm{G}_{t}\oplus\mathrm{F}_{2}\mathrm{G}_{1}\oplus\dots\oplus\mathrm{F}_{s}\mathrm{G}_{1}\oplus\dots\mathrm{F}_{s}\mathrm{G}_{t},
\end{gather*}
where we omit components that are not defined, which we interpret as being zero.
Vertical composition of $2$-morphisms is defined componentwise.

Given a matrix $f$ of morphisms in $\mathrm{Hom}_{\ccC^{\oplus}(\bullet,\bullet)}(\mathrm{F}_{1}\oplus\dots\oplus\mathrm{F}_{s},\mathrm{F}_{1}^{\prime}\oplus\dots\oplus\mathrm{F}_{s^{\prime}}^{\prime})$ and $g$ in $\mathrm{Hom}_{\ccC^{\oplus}(\bullet,\bullet)}(\mathrm{G}_{1}\oplus\dots\oplus\mathrm{G}_{t},\mathrm{G}_{1}^{\prime}\oplus\dots\oplus\mathrm{G}_{t^{\prime}}^{\prime})$, their horizontal composition $f\circ_{\mathsf{h}}g$ is a matrix whose
$(p^{\prime}q^{\prime},pq)$-component is given by
$f_{p^{\prime}p}\circ_{\mathsf{h}}g_{q^{\prime}q}$,
whenever this makes sense, and $0$ otherwise.

Taking into account that
\begin{gather*}
\big((\mathrm{F}_{1}\oplus\dots\oplus\mathrm{F}_{s})(\mathrm{G}_{1}\oplus\dots\oplus\mathrm{G}_{t})\big)(\mathrm{H}_{1}\oplus\dots\oplus\mathrm{H}_{u})
=\bigoplus_{p,q,r}(\mathrm{F}_{p}\mathrm{G}_{q})\mathrm{H}_{r},
\\
(\mathrm{F}_{1}\oplus\dots\oplus\mathrm{F}_{s})\big((\mathrm{G}_{1}\oplus\dots\oplus\mathrm{G}_{t})
(\mathrm{H}_{1}\oplus\dots\oplus\mathrm{H}_{u})\big)=\bigoplus_{p,q,r}\mathrm{F}_{p}(\mathrm{G}_{q}\mathrm{H}_{r}),
\end{gather*}
both with the ordering on the summands given by the reverse lexicographic ordering on the indices, the associator is given by the diagonal matrix of the respective associators.

To define the identity $1$-morphism in $\cC^{\,\oplus}$, fix an ordering $\mathtt{i}_1<\dots<\mathtt{i}_{m}$ on the objects of $\cC$. The identity $1$-morphism is then given by $\mathbbm{1}_{\mathtt{i}_{1}}\oplus\dots\oplus\mathbbm{1}_{\mathtt{i}_{m}}$. Note that reordering produces an isomorphic $1$-morphism. The right unitor is given by the component unitors in $\cC$. Indeed,
$(\mathrm{F}_{1}\oplus\dots\oplus\mathrm{F}_{s})(\mathbbm{1}_{\mathtt{i}_{1}}\oplus\dots\oplus\mathbbm{1}_{\mathtt{i}_{m}})$ will only have $s$ direct summands of the form
$\mathrm{F}_{p}\mathbbm{1}_{\mathtt{i}_{s(\mathrm{F}_{p})}}$, so the right unitor will be the diagonal $s\times s$-matrix of the corresponding unitors in $\cC$.
Similarly, the left unitor will be a permutation matrix with the unitors from $\cC$ as entries, since $(\mathbbm{1}_{\mathtt{i}_1}\oplus\dots\oplus\mathbbm{1}_{\mathtt{i}_{m}})(\mathrm{F}_{1}\oplus\dots\oplus\mathrm{F}_{s})$ has summands $\mathbbm{1}_{\mathtt{i}_{t(\mathrm{F}_{p})}}\mathrm{F}_{p}$, but ordered according to the ordering on the $\mathtt{i}_{t(\mathrm{F}_{p})}$.

\begin{lemma}\label{lemma:its-monoidal}
If $\cC$ is a finitary bicategory, then $\cC^{\,\oplus}$ is a multifinitary bicategory.
\end{lemma}

\begin{proof}
The pentagon axiom and the compatibility of left and right unitors follow immediately from the same axioms for $\cC$. The stated properties of $\cC^{\,\oplus}(\bullet,\bullet)$ are inherited from the same properties for $\cC$.
Observe that $\cC^{\,\oplus}$ is not finitary since the identity $1$-morphism is not indecomposable, but $\cC^{\,\oplus}$ is multifinitary.
\end{proof}

The following lemma is immediate.

\begin{lemma}\label{lemma:multifinitary}
If $\cC$ is (quasi) fiab, then $\cC^{\,\oplus}$ is (quasi) multifiab.
\end{lemma}

We now explain how to go back and forth between birepresentations of $\cC$ and $\cC^{\,\oplus}$.
Given a finitary birepresentation $\mathbf{M}$ of $\cC$, we can define a birepresentation $\mathbf{M}^{\oplus}$ by
\begin{itemize}

\item[--] $\mathbf{M}^{\oplus}(\bullet)=\bigoplus_{\mathtt{i}\in\ccC}\mathbf{M}(\mathtt{i})$;

\item[--] $\mathbf{M}^{\oplus}(\mathrm{F}_{1}\oplus\dots\oplus\mathrm{F}_{s})=\mathbf{M}(\mathrm{F}_{1})\oplus\dots\oplus\mathbf{M}(\mathrm{F}_{s})$ for a $1$-morphism $\mathrm{F}_{1}\oplus\dots\oplus\mathrm{F}_{s}$;

\item[--] $\mathbf{M}^{\oplus}\big((\beta_{pq})_{p=1,\dots,t;q=1,\dots,s}\big)=\big(\mathbf{M}(\beta)_{pq}\big)_{p=1,\dots,t;q=1,\dots,s}$ for a $2$-morphism
\begin{gather*}
\beta=(\beta_{pq})_{p=1,\dots,t;q=1,\dots,s}\colon\mathrm{F}_{1}\oplus\dots\oplus\mathrm{F}_{s}\to\mathrm{G}_{1}\oplus\dots\oplus\mathrm{G}_{t}.
\end{gather*}

\end{itemize}
Here we interpret the actions of direct sums of functors and their natural transformations on our chosen biproduct of additive categories in the evident way.

Conversely, given a birepresentation $\mathbf{N}$ of $\cC^{\,\oplus}$, we can associate a birepresentation $\mathbf{N}^{\prime}$ of $\cC$ by noting that projection onto $\mathbbm{1}_{\mathtt{i}}$ as a direct summand of the identity $1$-morphism in $\cC^{\,\oplus}$ defines an endomorphism of the identity functor on $\mathbf{N}(\bullet)$, and we thus have a  decomposition $\mathbf{N}(\bullet)=\bigoplus_{\mathtt{i}\in\ccC}\mathbf{N}(\bullet)_{\mathtt{i}}$. We can then define:

\begin{itemize}

\item[--] $\mathbf{N}^{\prime}(\mathtt{i})=\mathbf{N}(\bullet)_{\mathtt{i}}$
for any object $\mathtt{i}\in\cC$;

\item[--] $\mathbf{N}^{\prime}(\mathrm{F})=\mathbf{N}(\mathrm{F})$ for any $1$-morphism $\mathrm{F}$ in $\cC(\mathtt{i},\mathtt{j})$, where $\mathtt{i},\mathtt{j}\in\cC$;

\item[--] $\mathbf{N}^{\prime}(\beta)=\mathbf{N}(\beta)$ for any $2$-morphism $\beta\colon\mathrm{F}\to\mathrm{G}$, where $\mathrm{F},\mathrm{G}\in\cC(\mathtt{i},\mathtt{j})$ and $\mathtt{i},\mathtt{j}\in\cC$.
\end{itemize}

It is immediate that $(\mathbf{M}^{\,\oplus})^{\prime}$ is equivalent to $\mathbf{M}$ and $(\mathbf{N}^{\prime})^{\oplus}$ is equivalent to $\mathbf{N}$, which proves the following proposition.

\begin{proposition}\label{proposition:bieq-additive}
There is a biequivalence of $2$-categories $\cC^{\,\oplus}\text{-}\mathrm{afmod}\simeq\cC\text{-}\mathrm{afmod}$.
\end{proposition}

%%%%%%%%%%%%%%%%%%%%%%%%%%%%%%%%%%%%%%%%%

\subsection{Cell theory}\label{subsection:cells-bicats}

%%%%%%%%%%%%%%%%%%%%%%%%%%%%%%%%%%%%%%%%%

The theory of cells carries over verbatim from finitary $2$-categories to multifinitary bicategories.
Let us briefly recall its main features; details and references can be found in
\cite[Subsection 4.5]{MM1}, \cite[Subsection 3.2]{CM}, \cite[Section 3]{MM5} and
\cite[Subsection 4.2]{MMMZ}.

For each multifinitary bicategory $\cC$, one defines the \emph{left partial preorder}
$\geq_{L}$ on indecomposable $1$-morphisms by
\begin{gather*}
\mathrm{F}\geq_{L}\mathrm{G}
\Leftrightarrow
\text{there exists $\mathrm{H}$ such that $\mathrm{F}$ is isomorphic to a direct summand of $\mathrm{H}\mathrm{G}$}.
\end{gather*}
One then defines
\emph{left cells}, denoted by $\mathcal{L}$,
to be the equivalence classes with respect to $\geq_{L}$, on which $\geq_{L}$ naturally induces a partial order denoted by the same symbol.
Similarly, one defines  the \emph{right} and \emph{two-sided partial preorders} $\geq_{R}$ and $\geq_{J}$ and their corresponding \emph{right cells} and \emph{two-sided cells}, denoted by $\mathcal{R}$ and $\mathcal{J}$ respectively. Note that the source map $\mathtt{i}_{s({}_{-})}$ is constant on each left cell and the target map $\mathtt{i}_{t({}_{-})}$ is constant on each right cell.

\begin{example}\label{example:kl-cells}
\leavevmode	
\begin{enumerate}[$($i$)$]
\item A fusion category $\cC$ has only one (left, right and two-sided) cell, because, for any
$1$-morphism $\mathrm{F}\in\cC$, the decomposition of both $\mathrm{F}\mathrm{F}^{\star}$
and $\mathrm{F}^{\star}\mathrm{F}$ contains the identity on the unique object.
\item Recall that for any finite Coxeter group, the one-object bicategory of Soergel bimodules
$\cS$ is finitary when it is defined over the coinvariant algebra. The (left, right, or two-sided, respectively) cells and cell orders in $\cS$ correspond to the Kazhdan--Lusztig \cite{KL} (left, right, or two-sided, respectively) cells and
orders of $W$ by the Soergel--Elias--Williamson categorification theorem \cite[Theorem 1.1]{EW}.
This remains true when $\cS$ is defined over the polynomial algebra and/or the
Coxeter group is non-finite.
\end{enumerate}
\end{example}

For any left cell $\mathcal{L}$, one can define the so-called \emph{cell birepresentation $\mathbf{C}_{\mathcal{L}}$} as follows: Let $\mathtt{i}$ be the source of $\mathcal{L}$. Define
a subbirepresentation $\mathbf{M}^{\geq\mathcal{L}}$ of the principal birepresentation
$\mathbf{P}_{\mathtt{i}}$, using the induced action of $\cC$ on
\begin{gather*}
\mathrm{add}\big(
\{\mathrm{F}\mid\mathrm{F}\geq_{L}\mathcal{L}\}
\big).
\end{gather*}
Then $\mathbf{M}^{\geq\mathcal{L}}$ has a unique maximal ideal $\mathbf{I}$ and we define
\begin{gather*}
\mathbf{C}_{\mathcal{L}}:=\mathbf{M}^{\geq\mathcal{L}}/\mathbf{I},
\end{gather*}
which is always a simple transitive birepresentation.

\begin{example}
\leavevmode	
\begin{enumerate}[$($i$)$]
\item The (unique) cell birepresentation of a fusion category coincides with
its regular birepresentation, for which the action is defined by the monoidal product.
\item The cell birepresentations of $\cS$, for any Coxeter group $W$, categorify
the Kazhdan--Lusztig \cite{KL} cell representations of the Hecke algebra of $W$, by the Soergel--Elias--Williamson categorification theorem \cite[Theorem 1.1]{EW}.
\end{enumerate}
\end{example}

Let $\cC$ be a multifinitary bicategory. By the bicategorical analog of \cite[Subsection 3.2]{CM},
any transitive birepresentation $\mathbf{M}$ of $\cC$ has an associated invariant called \emph{apex}, which is the unique two-sided cell $\mathcal{J}$ of $\cC$ not annihilated by
$\mathbf{M}$ that is maximal with respect to the two-sided order $\geq_{J}$.

\begin{example} 
Suppose that $\cC$ is quasi multifiab. Let $\mathcal{L}$ be a left cell inside a two-sided cell $\mathcal{J}$ of
$\cC$. Then the apex of the cell birepresentation $\mathbf{C}_{\mathcal{L}}$
is equal to $\mathcal{J}$.
\end{example}

\begin{definition}\label{def:mod-apexJ}
Let $\cC$ be a multifinitary bicategory. Denote by $\cC\text{-}\mathrm{afmod}_{\mathcal{J}}$ the $1,2$-full $2$-subcategory of $\cC\text{-}\mathrm{afmod}$ consisting of the finitary birepresentations
whose weak Jordan--H{\"o}lder constituents all have apex $\mathcal{J}$. With respect to those $1,2$-full $2$-subcate\-gories of $\cC\text{-}\mathrm{afmod}$ in Definition \ref{def:cfmodstmod}, we denote by
\begin{enumerate}[$($i$)$]
\item $\cC\text{-}\mathrm{cfmod}_{\mathcal{J}}$ the one consisting of all cyclic (finitary) birepresentations
whose weak Jordan--H{\"o}lder constituents all have apex $\mathcal{J}$;
\item $\cC\text{-}\mathrm{tfmod}_{\mathcal{J}}$  the one consisting of all transitive (finitary) birepresentations with apex $\mathcal{J}$;
\item $\cC\text{-}\mathrm{stmod}_{\mathcal{J}}$ the one consisting of all simple transitive (finitary) birepresentations with apex $\mathcal{J}$.
\end{enumerate}
\end{definition}
Again, we have $\cC\text{-}\mathrm{stmod}_{\mathcal{J}}\subseteq\cC\text{-}\mathrm{tfmod}_{\mathcal{J}}\subseteq\cC\text{-}\mathrm{cfmod}_{\mathcal{J}}\subseteq\cC\text{-}\mathrm{afmod}_{\mathcal{J}}$.

Inside each two-sided cell, we define \emph{$\mathcal{H}$-cells} as the intersection of left and right cells. Note that $\mathtt{i}_{s(\mathcal{H})}$ need not be equal to
$\mathtt{i}_{t(\mathcal{H})}$ in general. Any two-sided cell $\mathcal{J}$ is the disjoint union of the $\mathcal{H}$-cells it contains. If $\cC$ is quasi multifiab, then ${}^{\star}$ exchanges the left and right cells inside each two-sided cell. For any left cell $\mathcal{L}$ inside a two-sided cell $\mathcal{J}$, the intersection
\begin{gather*}
\mathcal{H}(\mathcal{L}):=\mathcal{L}\cap\mathcal{L}^{\star}\subseteq\mathcal{J}
\end{gather*}
is called the $\mathcal{H}$-cell \emph{associated} to $\mathcal{L}$.
Note that all $1$-morphisms in an $\mathcal{H}$-cell $\mathcal{H}(\mathcal{L})$ are
$1$-endomorphisms of one fixed $\mathtt{i}\in\cC$, which we call the \emph{source of $\mathcal{H}$}. By the generalization of \cite[Proposition~17]{MM1} to quasi multifiab bicategories,
each left cell $\mathcal{L}$ contains a unique distinguished $1$-morphism $\mathrm{D}=\mathrm{D}(\mathcal{L})$, called \emph{Duflo involution}. If $\cC$ is multifiab, then every associated $\mathcal{H}$-cell $\mathcal{H}(\mathcal{L})$ is stable under ${}^{\star}$ and is called a \emph{diagonal} $\mathcal{H}$-cell. Since both $\mathrm{D}=\mathrm{D}(\mathcal{L})$ and
$\mathrm{D}^{\star}$ belong to $\mathcal{L}$ (c.f. \cite[Proposition~17]{MM1}),
we have $\mathrm{D}\in\mathcal{H}(\mathcal{L})$ in this case.

\begin{lemma}\label{lemma:supported}
Let $\cC$ be a multifinitary bicategory and let $\mathbf{M}\in\cC\text{-}\mathrm{tfmod}_{\mathcal{J}}$ and $\mathcal{H}$ be any $\mathcal{H}$-cell inside $\mathcal{J}$.
Then there exists some non-zero object $X\in\mathbf{M}(\mathtt{i}_{s(\mathcal{H})})$ which is not annihilated by $\mathcal{H}$.
\end{lemma}

\begin{proof}
Let $\mathbf{M}$ be a transitive birepresentation of $\cC$ with apex $\mathcal{J}$. Assume that the category $\mathbf{M}(\mathtt{i}_{s(\mathcal{H})})$ is annihilated by $\mathcal{H}$ and note that each $\mathbf{M}(\mathtt{j})$, where $\mathtt{j}\neq\mathtt{i}_{s(\mathcal{H})}$, is also annihilated by $\mathcal{H}$  by definition.  We deduce that $\mathbf{M}$ annihilates $\mathcal{H}$ and hence annihilates $\mathcal{J}$ as $\mathrm{add}(\mathcal{J})\subset\cC\mathrm{add}(\mathcal{H})\cC$, a contradiction.
\end{proof}

\begin{lemma}\label{lemma:cyclicbirep}
Let $\cC$ be a multifinitary bicategory. Any $\mathbf{M}$ in $\cC\text{-}\mathrm{afmod}_{\mathcal{J}}$ is cyclic, that is
\begin{gather*}
\cC\text{-}\mathrm{afmod}_{\mathcal{J}}=\cC\text{-}\mathrm{cfmod}_{\mathcal{J}}.
\end{gather*}
Moreover, for any $\mathcal{H}$-cell $\mathcal{H}$ inside $\mathcal{J}$, there exists a generator $X\in\mathbf{M}(\mathtt{i})$ of $\mathbf{M}$ such that, for any $\mathrm{F}\in\mathcal{H}$,
$\mathbf{M}_{\mathtt{j}\mathtt{i}}(\mathrm{F})X$ also generates $\mathbf{M}$,
where $\mathtt{i}:=\mathtt{i}_{s(\mathcal{H})}$ and $\mathtt{j}:=\mathtt{i}_{t(\mathcal{H})}$.
\end{lemma}

\begin{proof}
Let $\mathbf{M}\in\cC\text{-}\mathrm{afmod}_{\mathcal{J}}$ and recall the existence of weak Jordan--H{\"o}lder series from
Theorem \ref{thm:jh}.
Let $0\subset\mathbf{M}_{1}\subsetneq\dots\subsetneq\mathbf{M}_{m}=\mathbf{M}$ be a filtration of $\mathbf{M}$ by subbirepresentations such that each subquotient is transitive with apex $\mathcal{J}$. For each $q\in\{1,\dots,m\}$, by Lemma \ref{lemma:supported}, we can choose $X_q\in\mathbf{M}_q(\mathtt{i})$ such that $X_q\notin\mathbf{M}_{q-1}(\mathtt{i})$
and $\mathbf{M}_{\mathtt{j}\mathtt{i}}(\mathrm{F})\,X_{q}\notin\mathbf{M}_{q-1}(\mathtt{j})$ for any $\mathrm{F}\in\mathcal{H}$. Then, setting $\mathrm{X}=\mathrm{X}_{1}\oplus\dots\oplus\mathrm{X}_{q}$, both $X$ and $\mathbf{M}_{\mathtt{j}\mathtt{i}}(\mathrm{F})\,X$ generate $\mathbf{M}$. The statements follow.
\end{proof}

%%%%%%%%%%%%%%%%%%%%%%%%%%%%%%%%%%%%%%

\subsection{Quotients of bicategories via cell theory}\label{quotientbicat}

%%%%%%%%%%%%%%%%%%%%%%%%%%%%%%%%%%%%%%%
The main reason for introducing the various subbicategories of
$\cC\text{-}\mathrm{afmod}$ associated with a two-sided cell $\mathcal{J}$
in Definition \ref{def:mod-apexJ}, is the reduction of the Classification Problem of all simple transitive birepresentations of $\cC$
to that of the simple transitive birepresentations with apex $\mathcal{J}$ (where $\mathcal{J}$ is arbitrary but fixed). As explained in the introduction, we will show how to reduce
the Classification Problem even further by strong $\mathcal{H}$-reduction
in Theorems \ref{theorem:H-reduction1} and \ref{theorem:H-reduction2}, when $\cC$ is multifiab.
But before we can do that, we first have to prepare the ground. In this subsection we therefore show how the aforementioned subbicategories of $\cC\text{-}\mathrm{afmod}$, and certain generalizations of them, are related to certain (sub)quotients of $\cC$.
At the end of this subsection, we will indicate more precisely the relation with strong $\mathcal{H}$-reduction.

For now, let $\cC$ just be a multifinitary bicategory (i.e. not necessarily (quasi) multifiab)
and $\mathcal{J}$ a two-sided cell of $\cC$. We denote by $\mathcal{I}_{\not\leq\mathcal{J}}$ the biideal in $\cC$ generated by $\mathrm{id}_{\mathrm{F}}$  for all $\mathrm{F}\not\leq_{J}\mathcal{J}$. The quotient $\cC/\mathcal{I}_{\not\leq\mathcal{J}}$ is a multifinitary bicategory whose two-sided cells correspond exactly to the two-sided cells $\mathcal{J}^{\prime}$ of $\cC$ satisfying
$\mathcal{J}^{\prime}\leq_{J}\mathcal{J}$. In particular,
it has a unique maximal two-sided cell, corresponding to $\mathcal{J}$. If $\cC$ is (quasi) multifiab,
then $\cC/\mathcal{I}_{\not\leq\mathcal{J}}$ is (quasi) multifiab. By Theorem \ref{thm:jh}, it is easy to understand the relation between
the finitary birepresentations of $\cC$ and those of $\cC/\mathcal{I}_{\not\leq\mathcal{J}}$. Let $\cC\text{-}\mathrm{afmod}_{\leq\mathcal{J}}$ be the $1,2$-full $2$-subcategory of $\cC\text{-}\mathrm{afmod}$ consisting of birepresentations whose
weak Jordan--H{\"o}lder constituents all have apex $\leq_{J}\mathcal{J}$. Similarly, we have the $1,2$-full $2$-subcategory $\cC\text{-}\mathrm{cfmod}_{\leq\mathcal{J}}$. In the following, we will define various $2$-functors, some of which will be local equivalences. Here, we use the terminology that a pseudofunctor is a \emph{local equivalence} if it induces equivalences on the morphism categories
(but is not necessarily essentially surjective on objects).

\begin{theorem}\label{thm000}
Let $\cC$ be a multifinitary bicategory and $\mathcal{J}$ a two-sided cell in $\cC$.
The pullback via the $2$-full projection $\cC\to\cC/\mathcal{I}_{\not\leq\mathcal{J}}$ defines a $2$-functor
\begin{gather}\label{eq:pullback1}
\cC/\mathcal{I}_{\not\leq\mathcal{J}}\text{-}\mathrm{afmod}\to\cC\text{-}\mathrm{afmod}_{\leq\mathcal{J}},
\end{gather}
which is a local equivalence. It can be restricted to a local equivalence
\begin{gather}\label{eq:pullback1-1}
\cC/\mathcal{I}_{\not\leq\mathcal{J}}\text{-}\mathrm{cfmod}\to\cC\text{-}\mathrm{cfmod}_{\leq\mathcal{J}},
\end{gather}
and, for any two-sided cell $\mathcal{J}^{\prime}\leq_{J}\mathcal{J}$, to biequivalences
\begin{gather}\label{eq:pullback1-3}
\cC/\mathcal{I}_{\not\leq\mathcal{J}}\text{-}\mathrm{tfmod}_{\mathcal{J}^{\prime}}\xrightarrow{\simeq}
\cC\text{-}\mathrm{tfmod}_{\mathcal{J}^{\prime}},
\\
\label{eq:pullback1-4}
\cC/\mathcal{I}_{\not\leq\mathcal{J}}\text{-}\mathrm{stmod}_{\mathcal{J}^{\prime}}\xrightarrow{\simeq}
\cC\text{-}\mathrm{stmod}_{\mathcal{J}^{\prime}}.
\end{gather}
The local equivalences \eqref{eq:pullback1} and \eqref{eq:pullback1-1} preserve
weak Jordan--H{\"o}lder series and, for any two-sided cell $\mathcal{J}^{\prime}\leq_{J}\mathcal{J}$, they descend to a local equivalence
\begin{gather}\label{eq:pullback1-2}
\cC/\mathcal{I}_{\not\leq\mathcal{J}}\text{-}\mathrm{cfmod}_{\mathcal{J}^{\prime}}\to\cC\text{-}\mathrm{cfmod}_{\mathcal{J}^{\prime}}.
\end{gather}
If $\cC$ is quasi multifiab, then \eqref{eq:pullback1} is a biequivalence and hence so are \eqref{eq:pullback1-1} and \eqref{eq:pullback1-2}.
\end{theorem}

\begin{proof}
Note that for any $\mathbf{M}\in\cC/\mathcal{I}_{\not\leq\mathcal{J}}\text{-}\mathrm{afmod}$ we have $\mathcal{I}_{\not\leq\mathcal{J}}\subseteq\mathrm{ann}(\mathbf{M})$. Then $\mathcal{I}_{\not\leq\mathcal{J}}$ is annihilated by all weak Jordan--H{\"o}lder constituents of $\mathbf{M}$ which implies that the latter all have apex
$\leq_{J}\mathcal{J}$.
Thus the pullback \eqref{eq:pullback1} is well-defined and obviously a local equivalence which can be restricted to local equivalences \eqref{eq:pullback1-1}-\eqref{eq:pullback1-4}. For any birepresentation $\mathbf{M}$ in $\cC\text{-}\mathrm{tfmod}_{\mathcal{J}^{\prime}}$,
respectively $\cC\text{-}\mathrm{stmod}_{\mathcal{J}^{\prime}}$, we also have $\mathcal{I}_{\not\leq\mathcal{J}}\subseteq\mathrm{ann}(\mathbf{M})$ since $\mathrm{apex}(\mathbf{M})=\mathcal{J}^{\prime}\leq_{J}\mathcal{J}$.
Thus $\mathbf{M}$ belongs to $\cC/\mathcal{I}_{\not\leq\mathcal{J}}\text{-}\mathrm{tfmod}_{\mathcal{J}^{\prime}}$, respectively
$\cC/\mathcal{I}_{\not\leq\mathcal{J}}\text{-}\mathrm{stmod}_{\mathcal{J}^{\prime}}$. Therefore both \eqref{eq:pullback1-3} and \eqref{eq:pullback1-4} are biequivalences.
It follows from the definition and biequivalences \eqref{eq:pullback1-3}-\eqref{eq:pullback1-4} that the local equivalence \eqref{eq:pullback1}, respectively \eqref{eq:pullback1-1}, preserves weak Jordan--H{\"o}lder series and descends to a local equivalence
\eqref{eq:pullback1-2} for any $\mathcal{J}^{\prime} \leq_{J}\mathcal{J}$.

Now assume that $\cC$ is quasi multifiab. It suffices to prove that \eqref{eq:pullback1} is essentially surjective since essential surjectivity of its restrictions \eqref{eq:pullback1-1}
and \eqref{eq:pullback1-2} is straightforward. For any $\mathbf{M}\in\cC\text{-}\mathrm{afmod}_{\leq\mathcal{J}}$,
let $0=\mathbf{M}_{0}\subsetneq\mathbf{M}_{1}
\subsetneq\dots\subsetneq
\mathbf{M}_{m}=\mathbf{M}$ be a filtration by subbirepresentations as in Theorem \ref{thm:jh}.
Without loss of generality, we assume that $m=2$.
Since $\mathcal{I}_{\not\leq\mathcal{J}}\subseteq\mathrm{ann}(\mathbf{M}_{1})\cap\mathrm{ann}(\mathbf{M}_{2}/\mathbf{I}_{1})$,
where $\mathbf{I}_{1}$ is the $\cC$-stable ideal in $\mathbf{M}_{2}$ generated by $\mathbf{M}_{1}$, we obtain
$\mathcal{I}_{\not\leq\mathcal{J}}\circ_{\mathsf{h}}\mathcal{I}_{\not\leq\mathcal{J}}\subseteq\mathrm{ann}(\mathbf{M}_{2})=\mathrm{ann}(\mathbf{M})$.
Recall that each left cell $\mathcal{L}$ in a quasi multifiab bicategory $\cC$
contains the Duflo involution
$\mathrm{D}:=\mathrm{D}(\mathcal{L})$ and, in fact, by the generalization of \cite[Proposition 17]{MM1} to quasi multifiab bicategories, each $1$-morphism $\mathrm{FD}$ contains $\mathrm{F}$ as a direct summand for any $\mathrm{F}\in\mathcal{L}$. Hence, $\mathcal{I}_{\not\leq\mathcal{J}}\circ_{\mathsf{h}}\mathcal{I}_{\not\leq\mathcal{J}}$ contains $\mathrm{id}_{\mathrm{F}}$ for all $\mathrm{F}\not\leq_{J}\mathcal{J}$, that is to say, $\mathrm{id}_{\mathrm{F}}\in\mathrm{ann}(\mathbf{M})$ for all $\mathrm{F}\not\leq_{J}\mathcal{J}$. Finally, we have $\mathcal{I}_{\not\leq\mathcal{J}}\subseteq\mathrm{ann}(\mathbf{M})$, which completes the proof.
\end{proof}

By Lemma \ref{lemma:cyclicbirep}, for each $\mathcal{J}^{\prime}\leq_{J}\mathcal{J}$, the local equivalence \eqref{eq:pullback1-2} can be written as
\begin{gather}\label{eq:pullback1-5}
\cC/\mathcal{I}_{\not\leq\mathcal{J}}\text{-}\mathrm{afmod}_{\mathcal{J}^{\prime}}\to\cC\text{-}\mathrm{afmod}_{\mathcal{J}^{\prime}},
\end{gather}
which is a biequivalence provided that $\cC$ is quasi multifiab.

\begin{definition}\label{definition:j-simple}
Let $\mathcal{J}$ be a two-sided cell in a multifinitary bicategory $\cC$.
Then $\cC$ is called \emph{$\mathcal{J}$-simple} if
any non-zero biideal of $\cC$ contains
the identity $2$-morphisms of all $1$-morphisms in $\mathcal{J}$.
\end{definition}

By the analog of \cite[Theorem 15]{MM2} for multifinitary bicategories, examples of $\mathcal{J}$-simple multifinitary bicategories are not hard to find:
for any two-sided cell $\mathcal{J}$ of a multifinitary bicategory $\cC$, there is a unique quotient bicategory $\cC_{\leq\mathcal{J}}$ that is $\mathcal{J}$-simple  and whose two-sided
cells correspond exactly to those of $\cC/\mathcal{I}_{\not\leq\mathcal{J}}$. The bicategory $\cC_{\leq\mathcal{J}}$ is called the \emph{$\mathcal{J}$-simple quotient of $\cC$} and is unique up to biequivalence. If $\cC$ is (quasi) multifiab, then $\cC_{\leq\mathcal{J}}$ is also (quasi) multifiab.

\begin{remark}\label{remark-j-simple}
Note that $\cC/\mathcal{I}_{\not\leq\mathcal{J}}$ is, in general, not $\mathcal{J}$-simple. However, the $\mathcal{J}$-simple quotients of $\cC$ and $\cC/\mathcal{I}_{\not\leq\mathcal{J}}$ are biequivalent. By definition,
the two-sided cells of $\cC_{\leq\mathcal{J}}$ are the same as those of $\cC/\mathcal{I}_{\not\leq\mathcal{J}}$, but the $2$-morphism spaces of the $\mathcal{J}$-simple quotient
are smaller in general.
\end{remark}

\begin{example}\label{example:j-simple}
If $\cC$ is semisimple, $\cC/\mathcal{I}_{\not\leq\mathcal{J}}$ and $\cC_{\leq\mathcal{J}}$ coincide.
\end{example}

The above example is special, because in general $\cC/\mathcal{I}_{\not\leq\mathcal{J}}$ and $\cC_{\leq\mathcal{J}}$ do not coincide. To show why, let us give one simple example:

\begin{example}\label{example:soergel-quotients}
Let $D\cong\Bbbk[x]/(x^2)$ be the algebra of dual numbers and
$D\text{-}\mathrm{proj}$ the category of complex finite dimensional projective $D$-modules.
Then take $\cC$ to be the one-object finitary $2$-category of $\Bbbk$-linear endofunctors of $D\text{-}\mathrm{proj}$ that are isomorphic to direct sums of copies of the identity functor
$\mathrm{Id}$. By definition, the $2$-morphisms of $\cC$ are the natural transformations between those endofunctors.
Note that $\cC$ has only one two-sided cell $\mathcal{J}$: the one containing only the isomorphism class of
$\mathrm{Id}$. Therefore, $\cC/\mathcal{I}_{\not\leq\mathcal{J}}\cong\cC$ and
\begin{gather*}
\mathrm{End}_{\ccC/\mathcal{I}_{\not\leq\mathcal{J}}}(\mathrm{Id})\cong\mathrm{End}_{\ccC}(\mathrm{Id})\cong
D.
\end{gather*}
However,
\begin{gather*}
\mathrm{End}_{\ccC_{\leq\mathcal{J}}}(\mathrm{Id})\cong\Bbbk,
\end{gather*}
because $(x)$ is the unique maximal ideal of $D$ and $D/(x)\cong\Bbbk$. Thus, $\cC_{\leq\mathcal{J}}$ is a proper quotient of
$\cC/\mathcal{I}_{\not\leq\mathcal{J}}$. Note that $\cC_{\leq\mathcal{J}}$ is semisimple in this case, but that need not be true in general.
\end{example}

The pullback
\begin{gather}\label{eq:pullback2}
\cC_{\leq\mathcal{J}}\text{-}\mathrm{afmod}\to
\cC/\mathcal{I}_{\not\leq\mathcal{J}}\text{-}\mathrm{afmod}
\end{gather}
via the $2$-full projection
$\cC/\mathcal{I}_{\not\leq\mathcal{J}}\to\cC_{\leq\mathcal{J}}$ is a local equivalence.
It is not a biequivalence in general, because not every finitary birepresentation of $\cC/\mathcal{I}_{\not\leq\mathcal{J}}$ is equivalent to the pullback of a
birepresentation of $\cC_{\leq\mathcal{J}}$, e.g. the birepresentation defined by the
natural action of $\cC/\mathcal{I}_{\not\leq\mathcal{J}}$ on the additive closure of $\mathcal{J}$ inside $\cC/\mathcal{I}_{\not\leq\mathcal{J}}$. Restricting \eqref{eq:pullback2}
gives the pullback
\begin{gather}\label{eq:pullback2-1}
\cC_{\leq\mathcal{J}}\text{-}\mathrm{cfmod}\to\cC/\mathcal{I}_{\not\leq\mathcal{J}}\text{-}\mathrm{cfmod},
\end{gather}
which is also a local equivalence and descends to local equivalences
\begin{gather*}\label{eq:pullback2-3}
\cC_{\leq\mathcal{J}}\text{-}\mathrm{tfmod}_{\mathcal{J}^{\prime}}\to\cC/\mathcal{I}_{\not\leq\mathcal{J}}\text{-}\mathrm{tfmod}_{\mathcal{J}^{\prime}},
\end{gather*}
\begin{gather}\label{eq:pullback2-4}
\cC_{\leq\mathcal{J}}\text{-}\mathrm{stmod}_{\mathcal{J}^{\prime}}\to\cC/\mathcal{I}_{\not\leq\mathcal{J}}\text{-}\mathrm{stmod}_{\mathcal{J}^{\prime}}.
\end{gather}
for any two-sided cell $\mathcal{J}^{\prime}\leq_{J}\mathcal{J}$. The pullbacks \eqref{eq:pullback2}
and \eqref{eq:pullback2-1} both preserve weak Jordan--H{\"o}lder series and can be restricted
to a local equivalence
\begin{gather*}
\cC_{\leq\mathcal{J}}\text{-}\mathrm{afmod}_{\mathcal{J}^{\prime}}=\cC_{\leq\mathcal{J}}\text{-}\mathrm{cfmod}_{\mathcal{J}^{\prime}}\to\cC/\mathcal{I}_{\not\leq\mathcal{J}}\text{-}\mathrm{cfmod}_{\mathcal{J}^{\prime}}=\cC/\mathcal{I}_{\not\leq\mathcal{J}}\text{-}\mathrm{afmod}_{\mathcal{J}^{\prime}},\label{eq:pullback2-2}
\end{gather*}
for any two-sided cell $\mathcal{J}^{\prime}\leq_{J}\mathcal{J}$,
where the two equalities hold by Lemma \ref{lemma:cyclicbirep}.
Moreover, if $\cC$ is quasi multifiab, the local equivalence \eqref{eq:pullback2-4} for $\mathcal{J}^{\prime}=\mathcal{J}$ is a biequivalence, see the proof of Proposition \ref{prop:J-simple-descend-stmod}.

By precomposing \eqref{eq:pullback1} with \eqref{eq:pullback2}, we obtain
the pullback
\begin{gather}\label{eq:pullback3}
\cC_{\leq\mathcal{J}}\text{-}\mathrm{afmod}\to\cC\text{-}\mathrm{afmod}_{\leq\mathcal{J}},
\end{gather}
which is a local equivalence. Similarly, we have the local equivalence
\begin{gather}\label{eq:pullback3-1}
\cC_{\leq\mathcal{J}}\text{-}\mathrm{cfmod}\to\cC\text{-}\mathrm{cfmod}_{\leq\mathcal{J}}.
\end{gather}
Both \eqref{eq:pullback3} and \eqref{eq:pullback3-1} can be restricted to a series of local equivalences
\begin{gather*}
\label{eq:pullback3-3}
\cC_{\leq\mathcal{J}}\text{-}\mathrm{tfmod}_{\mathcal{J}^{\prime}}\to\cC\text{-}\mathrm{tfmod}_{\mathcal{J}^{\prime}},
\end{gather*}
\begin{gather}
\label{eq:pullback3-4}
\cC_{\leq\mathcal{J}}\text{-}\mathrm{stmod}_{\mathcal{J}^{\prime}}\to\cC\text{-}\mathrm{stmod}_{\mathcal{J}^{\prime}},
\end{gather}
\begin{gather*}
\label{eq:pullback3-2}
\cC_{\leq\mathcal{J}}\text{-}\mathrm{afmod}_{\mathcal{J}^{\prime}}=\cC_{\leq\mathcal{J}}\text{-}\mathrm{cfmod}_{\mathcal{J}^{\prime}}\to\cC\text{-}\mathrm{cfmod}_{\mathcal{J}^{\prime}}=\cC\text{-}\mathrm{afmod}_{\mathcal{J}^{\prime}},
\end{gather*}
for any two-sided cell $\mathcal{J}^{\prime}\leq_{J}\mathcal{J}$, where the two equalities appearing in the last local equivalence hold by Lemma \ref{lemma:cyclicbirep}.
Moreover, if $\cC$ is quasi multifiab, the local equivalence \eqref{eq:pullback3-4} for $\mathcal{J}^{\prime}=\mathcal{J}$ is a biequivalence, cf. Proposition \ref{prop:J-simple-descend-stmod}.

Denote by $\cC_{(\mathcal{J})}$ be the $2$-full subbicategory of $\cC/\mathcal{I}_{\not\leq\mathcal{J}}$ whose objects are all $\mathtt{i}_{s(\mathrm{F})},\mathtt{i}_{t(\mathrm{F})}$ for $\mathrm{F}\in\mathcal{J}$, and whose morphism categories are given by
\begin{gather*}
\bigoplus_{\mathtt{i},\mathtt{j}\in\ccC_{(\mathcal{J})}}\cC_{(\mathcal{J})}(\mathtt{i},\mathtt{j}):=\mathrm{add}\{\mathrm{F},\mathbbm{1}_{\mathtt{i}}\mid\mathrm{F}\in\mathcal{J},\mathtt{i}\in\cC_{(\mathcal{J})}\}.
\end{gather*}
Define the $2$-full subbicategory $\cC_{\mathcal{J}}$ of $\cC_{\leq\mathcal{J}}$ similarly.
By definition, $\mathcal{J}$ is the only two-sided cell of $\cC_{(\mathcal{J})}$ and $\cC_{\mathcal{J}}$ not necessarily consisting only of identity $1$-morphisms.
If $\cC$ is (quasi) multifiab, then $\cC_{(\mathcal{J})}$ and $\cC_{\mathcal{J}}$ are also (quasi) multifiab.

\begin{lemma}\label{lem:J-simple1}
Suppose that $\cC$ is multifinitary. Then $\cC_{\mathcal{J}}$ is $\mathcal{J}$-simple.
Moreover, $\cC_{\mathcal{J}}$ is the $\mathcal{J}$-simple quotient of $\cC_{(\mathcal{J})}$.
\end{lemma}

\begin{proof}
This follows from $\mathcal{J}$-simplicity of $\cC_{\leq\mathcal{J}}$ and the fact that the unique maximal biideal of $\cC_{(\mathcal{J})}$ not containing identities on $1$-morphisms in $\mathcal{J}$ is the restriction of the analogous biideal of $\cC/\mathcal{I}_{\nleq\mathcal{J}}$.
\end{proof}

Pulling back via the $2$-fully faithful embedding $\cC_{\mathcal{J}}\to\cC_{\leq\mathcal{J}}$ yields a $2$-functor
\begin{gather}\label{eq:pullback4}
\cC_{\leq\mathcal{J}}\text{-}\mathrm{afmod}\to\cC_{\mathcal{J}}\text{-}\mathrm{afmod},
\end{gather}
which can be restricted to $2$-functors
\begin{gather}\label{eq:pullback4-1}
\cC_{\leq\mathcal{J}}\text{-}\mathrm{tfmod}_{\mathcal{J}}\to\cC_{\mathcal{J}}\text{-}\mathrm{tfmod}_{\mathcal{J}},
\\
\label{eq:pullback4-2}
\cC_{\leq\mathcal{J}}\text{-}\mathrm{stmod}_{\mathcal{J}}\to\cC_{\mathcal{J}}\text{-}\mathrm{stmod}_{\mathcal{J}}.
\end{gather}
Indeed, for a birepresentation $\mathbf{M}\in\cC_{\leq\mathcal{J}}\text{-}\mathrm{tfmod}_{\mathcal{J}}$, its underlying category is equal to $\mathrm{add}\{\mathbf{M}(\mathrm{F})X\vert\,\mathrm{F}\in\mathcal{J}\}$ for any non-zero object $X$ in $\mathbf{M}(\mathtt{i})$ for some $\mathtt{i}$. Thus the $2$-functor \eqref{eq:pullback4-1} is well-defined. Since $\mathcal{J}$ is the unique maximal two-sided cell in both $\cC_{\leq\mathcal{J}}$ and $\cC_{\mathcal{J}}$,
any proper $\cC_{\mathcal{J}}$-stable ideal of $\mathbf{M}\in\cC_{\leq\mathcal{J}}\text{-}\mathrm{tmod}_{\mathcal{J}}$ is $\cC_{\leq\mathcal{J}}$-stable as well, which implies that \eqref{eq:pullback4-2} is also well-defined.
Since the $2$-functor \eqref{eq:pullback4} preserves weak Jordan--H{\"o}lder series, it restricts to a $2$-functor
\begin{gather}\label{eq:pullback4-3}
\cC_{\leq\mathcal{J}}\text{-}\mathrm{afmod}_{\mathcal{J}}\to\cC_{\mathcal{J}}\text{-}\mathrm{afmod}_{\mathcal{J}}.
\end{gather}
In Theorem \ref{cor:MT2} (cf. also Remark \ref{remark0.0}), provided that $\cC$ is quasi multifiab,  we show that \eqref{eq:pullback4-2} is a biequivalence, which can be viewed as the restriction of \eqref{eq:pullback4-3}. The latter is a local equivalence by Theorem \ref{theorem:00}. If $\cC$ is quasi multifiab, composing the biequivalences in \eqref{eq:pullback3-4} for $\mathcal{J}^{\prime}=\mathcal{J}$ and \eqref{eq:pullback4-2} yields a biequivalence
\begin{gather*}
\cC_{\mathcal{J}}\text{-}\mathrm{stmod}_{\mathcal{J}}\to\cC\text{-}\mathrm{stmod}_{\mathcal{J}},
\end{gather*}
see Theorem \ref{cor:MT2} for details.
Recall that a diagonal $\mathcal{H}$-cell of a multifiab bicategory
is the intersection of a left cell and its dual. Of crucial importance for the birepresentation theory of $\cC$, cf. Theorems \ref{theorem:H-reduction1} and \ref{theorem:H-reduction2}, is the following.

\begin{definition}\label{definition:h-qoutient}
Suppose that $\cC$ is multifiab
and let $\mathcal{J}$ be a two-sided cell of $\cC$ and $\mathcal{H}\subseteq\mathcal{J}$ a diagonal $\mathcal{H}$-cell with source
$\mathtt{i}\in\cC$. Define the $2$-full subbicategory $\cC_{(\mathcal{H})}$ of $\cC_{(\mathcal{J})}$ with single object $\mathtt{i}$ and
\begin{gather*}
\cC_{(\mathcal{H})}(\mathtt{i},\mathtt{i}):=\mathrm{add}\{\mathrm{F},\mathbbm{1}_{\mathtt{i}}\mid\mathrm{F}\in\mathcal{H}\}.
\end{gather*}
Define the $2$-full subbicategory $\cC_{\mathcal{H}}$ of $\cC_{\mathcal{J}}$ similarly.
\end{definition}

If $\mathbbm{1}_{\mathtt{i}}$ does not belong to $\mathcal{H}$, then $\cC_{(\mathcal{H})}$ and $\cC_{\mathcal{H}}$ have precisely two cells, which are both left, right and two-sided: the trivial cell $\{\mathbbm{1}_{\mathtt{i}}\}$ and the non-trivial cell
$\mathcal{H}$. Note that both $\cC_{(\mathcal{H})}$ and $\cC_{\mathcal{H}}$ are
multifiab, because $\mathcal{H}$ is preserved by ${}^{\star}$ when $\cC$ is multifiab.

\begin{lemma}\label{lemma:h-qoutient}
The bicategory $\cC_{\mathcal{H}}$ is $\mathcal{H}$-simple. Moreover, it is the $\mathcal{H}$-simple quotient of $\cC_{(\mathcal{H})}$.
\end{lemma}

\begin{proof}
Consider the cell birepresentation $\mathbf{C}_{\mathcal{L}}$ of $\cC_{\mathcal{J}}$,
where $\mathcal{H}=\mathcal{H}(\mathcal{L})$, and note that it is $2$-faithful by
$\mathcal{J}$-simplicity of $\cC_{\mathcal{J}}$.
By the generalization of \cite[Theorem 2]{KMMZ} to bicategories, the action of each $\mathrm{F}\in\mathcal{H}$ is represented
via $\mathbf{C}_{\mathcal{L}}$ by a projective bimodule over the underlying algebra of
$\mathbf{C}_{\mathcal{L}}$. Let $\mathrm{D}=\mathrm{D}(\mathcal{L})$ be the Duflo involution in $\mathcal{L}$, which also belongs to $\mathcal{H}$
(see Subsection \ref{subsection:cells-bicats}). Then, by the generalization of \cite[Lemma 12]{MM1} to bicategories,
$\mathrm{D}$ does not annihilate any simples indexed by
elements of $\mathcal{H}$ in the (projective) abelianization of $\mathbf{C}_{\mathcal{L}}$. Therefore, given
$\mathrm{F},\mathrm{G}\in\mathcal{H}$ and a non-zero $\alpha\colon\mathrm{F}\to\mathrm{G}$ in $\cC_{\mathcal{H}}$,
the $2$-morphism
\begin{gather*}
(\mathrm{id}_{\mathrm{D}}\circ_{\mathsf{h}}\alpha)\circ_{\mathsf{h}}\mathrm{id}_{\mathrm{D}}\colon
(\mathrm{D}\mathrm{F})\mathrm{D}\to(\mathrm{D}\mathrm{G})\mathrm{D}
\end{gather*}
is non-zero. As $\mathbf{C}_{\mathcal{L}}(\mathrm{D})$ is a projective bimodule,
the morphism $\mathbf{C}_{\mathcal{L}}\big(\mathrm{id}_{\mathrm{D}}
\circ_{\mathsf{h}}\alpha\circ_{\mathsf{h}}\mathrm{id}_{\mathrm{D}}\big)$
is not a radical morphism in the category of bimodules. Thus,
$(\mathrm{id}_{\mathrm{D}}\circ_{\mathsf{h}}\alpha)\circ_{\mathsf{h}}\mathrm{id}_{\mathrm{D}}$
contains, as a direct summand, an isomorphism from some
non-zero summand of $(\mathrm{D}\mathrm{F})\mathrm{D}$
to some summand of $(\mathrm{D}\mathrm{G})\mathrm{D}$. Therefore, $\cC_{\mathcal{H}}$ is $\mathcal{H}$-simple and the second claim follows by definition.
\end{proof}

In Theorem \ref{theorem:strongH}, we will show that there is a biequivalence between $\cC\text{-}\mathrm{stmod}_{\mathcal{J}}$
and $\cC_{\mathcal{H}}\text{-}\mathrm{stmod}_{\mathcal{H}}$.

%%%%%%%%%%%%%%%%%%%%%%%%%%%%%%%%%%%%%%%%%

\section{Coalgebras and comodules in bicategories}\label{section:coalgebras}

%%%%%%%%%%%%%%%%%%%%%%%%%%%%%%%%%%%%%%%%%

In this section, let $\cC$ be a multifinitary bicategory.

%%%%%%%%%%%%%%%%%%%%%%%%%%%%%%%%%%%%%%%%%

\subsection{Coalgebras and comodules}\label{subsection:coalgebras}

%%%%%%%%%%%%%%%%%%%%%%%%%%%%%%%%%%%%%%%%%

The following definitions are analogs of those in \cite[Section 7.8]{EGNO}.

\begin{definition}\label{definition:coalgebra}
A \emph{coalgebra} $\mathrm{C}=(\mathrm{C},\delta_{\mathrm{C}},\epsilon_{\mathrm{C}})$ in $\cC$ consists of a $1$-morphism
$\mathrm{C}\in\cC(\mathtt{i},\mathtt{i})$, for some object $\mathtt{i}\in\cC$, a comultiplication $2$-morphism $\delta_{\mathrm{C}}\colon\mathrm{C}\to\mathrm{C}\mathrm{C}$ and a counit $2$-morphism $\epsilon_{\mathrm{C}}
\colon\mathrm{C}\to\mathbbm{1}_{\mathtt{i}}$.
These should satisfy the usual
coassociativity and counitality axioms of a
coalgebra, i.e. the diagrams
\begin{gather*}
\begin{tikzcd}[ampersand replacement=\&]
\arrow[d,"\delta_{\mathrm{C}}\circ_{\mathsf{h}}\mathrm{id}_{\mathrm{C}}",swap]\mathrm{C}\mathrm{C}
\&
\arrow[l,"\delta_{\mathrm{C}}",swap]
\mathrm{C}\arrow[r,"\delta_{\mathrm{C}}"]
\&
\mathrm{C}\mathrm{C}\arrow[d,"\mathrm{id}_{\mathrm{C}}\circ_{\mathsf{h}}\delta_{\mathrm{C}}"]
\\
(\mathrm{C}\mathrm{C})\mathrm{C}\arrow[rr,"\alpha_{\mathrm{C},\mathrm{C},\mathrm{C}}",swap]
\&\&
\mathrm{C}(\mathrm{C}\mathrm{C})
\end{tikzcd},
\\[.5ex]
\begin{tikzcd}[ampersand replacement=\&]
\arrow[d,"\delta_{\mathrm{C}}",swap]\mathrm{C}\ar[r,equal]
\& \mathrm{C}
\\
\mathrm{C}\mathrm{C}
\arrow[r,"\epsilon_{\mathrm{C}}\circ_{\mathsf{h}}\mathrm{id}_{\mathrm{C}}",swap]
\& \mathbbm{1}_{\mathtt{i}}\mathrm{C}\arrow[u,"\lunit_{\mathrm{C}}",swap]
\end{tikzcd}
,\quad
\begin{tikzcd}[ampersand replacement=\&]
\arrow[d,"\delta_{\mathrm{C}}",swap]\mathrm{C}\ar[r,equal]
\& \mathrm{C}
\\
\mathrm{C}\mathrm{C}\arrow[r,"\mathrm{id}_{\mathrm{C}}\circ_{\mathsf{h}}\epsilon_{\mathrm{C}}",swap]
\&
\mathrm{C}\mathbbm{1}_{\mathtt{i}}\arrow[u,"\runit_{\mathrm{C}}",swap]
\end{tikzcd}
\end{gather*}
should commute.
\end{definition}

\begin{example}\label{example:ADE}
\leavevmode
\begin{enumerate}[$($i$)$]

\item The identity $1$-morphism $\mathbbm{1}_\mathtt{i}$, for any $\mathtt{i}\in\cC$,
is naturally a coalgebra.

\item In finite dihedral type, there is an explicit construction of coalgebras in $\cS$
corresponding to $ADE$ Dynkin diagrams \cite[Section 7]{MMMT}.
\end{enumerate}
\end{example}

\begin{definition}\label{definition:coalg-homo}
Let $\mathrm{C},\mathrm{D}\in\cC(\mathtt{i},\mathtt{i})$ be coalgebras in $\cC$.
A \emph{homomorphism of coalgebras} in $\cC$ is a $2$-morphism
$\phi\colon\mathrm{C}\to\mathrm{D}$ such that the
diagrams
\begin{gather*}
\begin{tikzcd}[ampersand replacement=\&]
\arrow[d,"\delta_{\mathrm{C}}",swap]\mathrm{C}\arrow[r,"\phi"]
\&
\mathrm{D}\arrow[d,"\delta_{\mathrm{D}}"]
\\
\mathrm{C}\mathrm{C}\arrow[r,"\phi\circ_{\mathsf{h}}\phi",swap]
\&
\mathrm{D}\mathrm{D}
\end{tikzcd}
,\quad
\begin{tikzcd}[ampersand replacement=\&]
\arrow[dr,"\epsilon_{\mathrm{C}}",swap]
\mathrm{C}
\arrow[rr,"\phi"]\&
\&
\mathrm{D}
\arrow[dl,"\epsilon_{\mathrm{D}}"]
\\
\& \mathbbm{1}_{\mathtt{i}}
\&
\end{tikzcd}
\end{gather*}
commute.

The coalgebras $\mathrm{C}$ and $\mathrm{D}$ are \emph{isomorphic} if there exists an invertible homomorphism between them.
\end{definition}

Next, let us recall the definitions of
left, right and bicomodules in $\cC$.

\begin{definition}\label{definition:right-comodules}
Let $\mathrm{C}\in\cC(\mathtt{i},\mathtt{i})$ be a coalgebra in $\cC$.
A \emph{left $\mathrm{C}$-comodule}
$\mathrm{M}=(\mathrm{M},
\delta_{\mathrm{C},\mathrm{M}})$ in $\cC$ consists of a $1$-morphism
$\mathrm{M}\in\cC(\mathtt{j},\mathtt{i})$, for some object $\mathtt{j}\in\cC$, and a $2$-morphism $\delta_{\mathrm{C},\mathrm{M}}\colon\mathrm{M}\to\mathrm{C}\mathrm{M}$, called the \emph{left coaction}, such that the diagrams
\begin{gather*}
\begin{tikzcd}[ampersand replacement=\&]
\arrow[d,"\delta_{\mathrm{C}}\circ_{\mathsf{h}}\mathrm{id}_{\mathrm{M}}",swap]\mathrm{C}\mathrm{M}
\&
\arrow[l,"\delta_{\mathrm{C},\mathrm{M}}",swap]
\mathrm{M}\arrow[r,"\delta_{\mathrm{C},\mathrm{M}}"]
\&
\mathrm{C}\mathrm{M}\arrow[d,"\mathrm{id}_{\mathrm{C}}\circ_{\mathsf{h}}\delta_{\mathrm{C},\mathrm{M}}"]
\\
(\mathrm{C}\mathrm{C})\mathrm{M}\arrow[rr,"\alpha_{\mathrm{C},\mathrm{C},\mathrm{M}}",swap]
\&\&
\mathrm{C}(\mathrm{C}\mathrm{M})
\end{tikzcd},\quad
\begin{tikzcd}[ampersand replacement=\&,column sep=4em]
\arrow[d,"\delta_{\mathrm{C},\mathrm{M}}",swap]\mathrm{M}\ar[r,equal]
\& \mathrm{M}
\\
\mathrm{C}\mathrm{M}\arrow[r,"\epsilon_{\mathrm{C}}\circ_{\mathsf{h}}\mathrm{id}_{\mathrm{M}}",swap]
\& \mathbbm{1}_{\mathtt{i}}\mathrm{M}\arrow[u,"\lunit_{\mathrm{M}}",swap]
\end{tikzcd}
\end{gather*}
commute. The definition of
a \emph{right $\mathrm{C}$-comodule} $\mathrm{M}=(\mathrm{M},\delta_{\mathrm{M},\mathrm{C}})$ in $\cC$ is similar.
\end{definition}

\begin{definition}\label{definition:bicomodules}
For coalgebras $\mathrm{C}\in\cC(\mathtt{i},\mathtt{i})$ and
$\mathrm{D}\in\cC(\mathtt{j},\mathtt{j})$ in $\cC$, a \emph{$\mathrm{C}\text{-}\mathrm{D}$-bicomodule} $\mathrm{M}=(\mathrm{M},\delta_{\mathrm{C},\mathrm{M}},\delta_{\mathrm{M},\mathrm{D}})$ in $\cC$
is a left $\mathrm{C}$- and a right $\mathrm{D}$-comodule in $\cC(\mathtt{j},\mathtt{i})$
such that the diagram
\begin{gather*}
\begin{tikzcd}[ampersand replacement=\&]
\arrow[d,"\delta_{\mathrm{C},\mathrm{M}}\circ_{\mathsf{h}}\mathrm{id}_{\mathrm{D}}",swap]
\mathrm{M}\mathrm{D}
\&
\arrow[l,"\delta_{\mathrm{M},\mathrm{D}}",swap]
\mathrm{M}\arrow[r,"\delta_{\mathrm{C},\mathrm{M}}"]
\&
\mathrm{C}\mathrm{M}\arrow[d,"\mathrm{id}_{\mathrm{C}}\circ_{\mathsf{h}}\delta_{\mathrm{M},\mathrm{D}}"]
\\
(\mathrm{C}\mathrm{M})\mathrm{D}\arrow[rr,"\alpha_{\mathrm{C},\mathrm{M},\mathrm{D}}",swap]
\&\&
\mathrm{C}(\mathrm{M}\mathrm{D})
\end{tikzcd}
\end{gather*}
commutes.
\end{definition}

The definitions of \emph{homomorphisms} of left, right and bicomodules
should now be clear and are omitted for brevity.

\begin{remark}\label{remark:algebras}
There are, of course, also the dual notions of \emph{algebras} and \emph{modules} in
$\cC$. Their definition can be obtained from the above by inverting all arrows.

Coalgebras, comodules, bicomodules and the respective homomorphisms in the injective abelianization $\underline{\cC}$ or in the projective abelianization $\overline{\cC}$ are defined just as in $\cC$.
\end{remark}

We say that a coalgebra $\mathrm{D}$
is a \emph{subcoalgebra} of another coalgebra $\mathrm{C}$
if there is a monic $2$-morphism
$\phi\colon\mathrm{D}\to\mathrm{C}$ that is a homomorphism of coalgebras.
A subcoalgebra of $\mathrm{C}$ is called \emph{proper} if it is neither
zero nor isomorphic to $\mathrm{C}$.
A coalgebra $\mathrm{C}$ is \emph{cosimple}
if it has no proper subcoalgebras.

\begin{example} A cosimple coalgebra in $\cV\mathrm{ect}$ is a cosimple coalgebra in the
usual sense.
\end{example}

\begin{example}\label{example:cosimple}
Let $\mathbbm{1}$ and $\mathrm{s}$ denote the two simple
$1$-morphisms in $\cV\mathrm{ect}(\mathbb{Z}/2\mathbb{Z})$.
Then $\mathrm{C}=\mathbbm{1}\oplus\mathrm{s}$ has an essentially unique
structure of a cosimple coalgebra in $\cV\mathrm{ect}(\mathbb{Z}/2\mathbb{Z})$. The forgetful functor $\cV\mathrm{ect}(\mathbb{Z}/2\mathbb{Z})\to\cV\mathrm{ect}$ is monoidal, so $\mathrm{C}$ is also a coalgebra in $\cV\mathrm{ect}$. However, it is not cosimple in $\cV\mathrm{ect}$, as $\mathbbm{1}$ and $\mathrm{s}$ are mapped to isomorphic
$1$-morphisms by the forgetful functor.
\end{example}

%%%%%%%%%%%%%%%%%%%%%%%%%%%%%%%%%%%%%%%%%

\subsection{Cotensor product of bicomodules}\label{subsection:cotensor}

%%%%%%%%%%%%%%%%%%%%%%%%%%%%%%%%%%%%%%%%%

Let us briefly review the cotensor product of bicomodules over a coalgebra
in $\underline{\cC}$
(or in any bicategory having equalizers), see also e.g.
\cite[Subsection 3.3]{MMMZ}.

Let $\mathrm{M}$ be a right $\mathrm{C}$-comodule
and $\mathrm{N}$ a left $\mathrm{C}$-comodule in $\underline{\cC}$. The
\emph{cotensor product} of $\mathrm{M}$ and $\mathrm{N}$ over $\mathrm{C}$, denoted by
$\mathrm{M}\s\mathrm{N}$, is by definition
the equalizer of the diagram
\begin{gather*}
\begin{tikzcd}[ampersand replacement=\&,column sep=4em]
\mathrm{M}\mathrm{N}
\arrow[rr,"\mathrm{id}_{\mathrm{M}}
\delta_{\mathrm{C},\mathrm{N}}",yshift=.1cm]
\arrow[rr,"(\alpha_{\mathrm{M},\mathrm{C},\mathrm{N}})\circ_{\mathsf{v}}(\delta_{\mathrm{M},\mathrm{C}}\mathrm{id}_{\mathrm{N}})",swap,yshift=-.1cm]
\& \&
\mathrm{M}(\mathrm{C}\mathrm{N})
\end{tikzcd}
\leftrightsquigarrow
\begin{tikzcd}[ampersand replacement=\&]
\& \arrow[dl, "\delta_{\mathrm{M},\mathrm{C}}\mathrm{id}_{\mathrm{N}}",swap]\mathrm{M}\mathrm{N}
\arrow[dr,"\mathrm{id}_{\mathrm{M}}
\delta_{\mathrm{C},\mathrm{N}}"]
\&
\\
(\mathrm{M}\mathrm{C})\mathrm{N}\arrow[rr,"\alpha_{\mathrm{M},\mathrm{C},\mathrm{N}}",swap]
\& \& \mathrm{M}(\mathrm{C}\mathrm{N})
\end{tikzcd}.
\end{gather*}
Due to coassociativity of the right coaction, $\delta_{\mathrm{M},\mathrm{C}}$ induces a left comodule isomorphism $\delta_{\mathrm{M},\mathrm{C}}\colon\mathrm{M}\xrightarrow{\cong}
\mathrm{M}\s\mathrm{C}$, see \cite[Lemma~5]{MMMZ} for this statement in the strict setting. Similarly, $\delta_{\mathrm{C},\mathrm{N}}$ induces a right comodule isomorphism
$\delta_{\mathrm{C},\mathrm{N}}\colon\mathrm{N}\xrightarrow{\cong}\mathrm{C}\s\mathrm{N}$. Furthermore, the associator in $\cC$ induces an associator for the cotensor product.

\begin{lemma}\label{lemma:associator-cotensor-product}
Suppose that $\mathrm{K}$ is a
right $\mathrm{C}$-comodule, $\mathrm{M}$
a $\mathrm{C}\text{-}\mathrm{D}$-bicomodule and $\mathrm{N}$ a left $\mathrm{D}$-comodule,
all in $\underline{\cC}$.
Then $\alpha_{\mathrm{K},\mathrm{M},\mathrm{N}}$ induces a natural $2$-isomorphism, for which we use the same
notation,
\begin{gather}\label{eq:associativity-cotensor}
\alpha_{\mathrm{K},\mathrm{M},\mathrm{N}}\colon(\mathrm{K}\s\mathrm{M})\square_{\mathrm{D}}\mathrm{N}\xrightarrow{\cong}\mathrm{K}\s(\mathrm{M}\square_{\mathrm{D}}\mathrm{N}).
\end{gather}
\end{lemma}

\begin{proof}
First we claim that $\alpha_{\mathrm{K},\mathrm{M},\mathrm{N}}$ induces two
intermediate natural $2$-isomorphisms, for which we also use the same notation,
\begin{gather}\label{eq:associativity-cotensor2}
\alpha_{\mathrm{K},\mathrm{M},\mathrm{N}}\colon(\mathrm{K}\s\mathrm{M})\mathrm{N}\xrightarrow{\cong}\mathrm{K}\s(\mathrm{M}\mathrm{N})
,\quad
\alpha_{\mathrm{K},\mathrm{M},\mathrm{N}}\colon(\mathrm{K}\mathrm{M})\square_{\mathrm{D}}\mathrm{N}\xrightarrow{\cong}\mathrm{K}(\mathrm{M}\square_{\mathrm{D}}\mathrm{N}).
\end{gather}
We only prove the existence of the first one, which is the one we need below. The existence of the second one can be proved analogously. Consider the following diagram.
\begin{gather*}
\begin{tikzcd}[ampersand replacement=\&,column sep=4em]
\arrow[ddd,"\alpha_{\mathrm{K},\mathrm{M},\mathrm{N}}",swap] (\mathrm{K}\mathrm{M})\mathrm{N}\arrow[rr, "(\delta_{\mathrm{K},\mathrm{C}}\mathrm{id}_{\mathrm{M}})\mathrm{id}_{\mathrm{N}}"]\arrow[dr, "(\mathrm{id}_{\mathrm{K}}\delta_{\mathrm{C},\mathrm{M}})\mathrm{id}_{\mathrm{N}}",near end,swap]
\& \&
\big((\mathrm{K}\mathrm{C})\mathrm{M}\big)\mathrm{N}\arrow[ddd,"\alpha_{\mathrm{K}\mathrm{C},\mathrm{M},\mathrm{N}}"]\arrow[dl, "\alpha_{\mathrm{K},\mathrm{C},\mathrm{M}}\mathrm{id}_{\mathrm{N}}"]
\\
\arrow[rd,phantom,"\phantom{aaa}\circled{1}\text{\tiny(front and back)}"]
\&
\big(\mathrm{K}(\mathrm{C}\mathrm{M})\big)\mathrm{N}
\arrow[d,"\alpha_{\mathrm{K},\mathrm{C}\mathrm{M},\mathrm{N}}"]
\&
\arrow[ld,phantom,"\circled{2}"]
\\
\&
\mathrm{K}\big((\mathrm{C}\mathrm{M})\mathrm{N}\big)
\&
\\
\mathrm{K}(\mathrm{M}\mathrm{N})\arrow[rr, "\delta_{\mathrm{K},\mathrm{C}}\mathrm{id}_{\mathrm{M}\mathrm{N}}",near start]\arrow[dr,"\mathrm{id}_{\mathrm{K}}\delta_{\mathrm{C},\mathrm{MN}}",swap]\arrow[ur,"\mathrm{id}_{\mathrm{K}}(\delta_{\mathrm{C},\mathrm{M}}\mathrm{id}_{\mathrm{N}})" near end]\arrow[dr,phantom,"\circled{3}",bend left=20]
\&
\& (\mathrm{K}\mathrm{C})(\mathrm{M}\mathrm{N})\arrow[dl,"\alpha_{\mathrm{K},\mathrm{C},\mathrm{M}\mathrm{N}}"]
\\
\& \mathrm{K}\big(\mathrm{C}(\mathrm{M}\mathrm{N})\big)
\arrow[from=uu,crossing over,"\mathrm{id}_{\mathrm{K}}\alpha_{\mathrm{C},\mathrm{M},\mathrm{N}}",very near start]
\&
\end{tikzcd}
.
\end{gather*}
The vertical faces commute:
the faces labeled $1$ by
naturality of the associator, the one labeled
$2$ by the pentagon coherence condition for the
associator, and the triangle labeled $3$ by definition of $\delta_{\mathrm{C},\mathrm{MN}}$. Since all the vertical maps
are isomorphisms, this implies that $\alpha_{\mathrm{K},\mathrm{M},\mathrm{N}}$ induces
a $2$-isomorphism between the equalizer of the top triangle, which is $ (\mathrm{K}\s\mathrm{M})\mathrm{N}$ since right composition with $\mathrm{N}$ is left exact, and the equalizer of the bottom triangle,
which is $\mathrm{K}\s(\mathrm{M}\mathrm{N})$.
This proves the existence of the first natural $2$-isomorphism in
\eqref{eq:associativity-cotensor2}. Next, consider
\begin{gather*}
\begin{tikzcd}[ampersand replacement=\&,column sep=4em]
\ar[ddd,"\alpha_{\mathrm{K},\mathrm{M},\mathrm{N}}",swap] (\mathrm{K}\s\mathrm{M})\mathrm{N}
\ar[rr,"\delta_{\mathrm{K}\s\mathrm{M},\mathrm{D}}\mathrm{id}_{\mathrm{N}}"]
\ar[dr,"\mathrm{id}_{\mathrm{K}\s\mathrm{M}}\delta_{\mathrm{D},\mathrm{N}}",near end]
\&
\&
\big((\mathrm{K}\s\mathrm{M})\mathrm{D}\big)\mathrm{N}\ar[dl,"\alpha_{\mathrm{K}\s\mathrm{M},\mathrm{D},\mathrm{N}}"]\ar[dd,"\alpha_{\mathrm{K},\mathrm{M},\mathrm{D}}\mathrm{id}_{\mathrm{N}}"]
\\
\&
(\mathrm{K}\s\mathrm{M})(\mathrm{D}\mathrm{N})
\&
\\
\arrow[rd,phantom,"\phantom{aaa}\circled{1}\text{\tiny(front and back)}",near start]
\&
\phantom{.}
\& \big(\mathrm{K}\s(\mathrm{M}\mathrm{D})\big)\mathrm{N}\ar[d,"\alpha_{\mathrm{K},\mathrm{M}\mathrm{D},\mathrm{N}}"]
\ar[from=uull,bend right=20,crossing over,"(\mathrm{id}_{\mathrm{K}}
\square\delta_{\mathrm{M},\mathrm{D}})\mathrm{id}_{\mathrm{N}}",swap]
\arrow[lu,phantom,"\circled{3}"]
\\
\mathrm{K}\s(\mathrm{M}\mathrm{N})
\ar[rr,"\mathrm{id}_{\mathrm{K}}\square (\delta_{\mathrm{M},\mathrm{D}}\mathrm{id}_{\mathrm{N}})",near start]
\ar[dr, "\mathrm{id}_{\mathrm{K}} (\mathrm{id}_{\mathrm{M}}
\delta_{\mathrm{D},\mathrm{N}})",swap]
\&
\phantom{.}
\&
\mathrm{K}\s\big((\mathrm{M}\mathrm{D})\mathrm{N}\big)
\ar[dl,"\mathrm{id}_{\mathrm{K}}\square\alpha_{\mathrm{M},\mathrm{D},\mathrm{N}}"]
\arrow[lu,phantom,"\circled{2}"]
\\
\& \mathrm{K}\s\big(\mathrm{M}(\mathrm{D}\mathrm{N})\big)\ar[from=uuu, crossing over,"\alpha_{\mathrm{K},\mathrm{M},\mathrm{DN}}",very near start]
\&
\end{tikzcd}.
\end{gather*}
Again, all vertical faces commute:
the left and back quadrilaterals labeled $1$ by
naturality of the induced associator,
the right pentagon labeled $2$ by the pentagon coherence
condition for the induced associator, and the
triangle labeled $3$ by definition of $\delta_{\mathrm{K}\s\mathrm{M},\mathrm{D}}$. Recall that equalizers are unique up to isomorphisms,
which implies that the induced associators in \eqref{eq:associativity-cotensor2} are also
natural and satisfy the pentagon coherence condition. Again, the vertical maps
are $2$-isomorphisms, so $\alpha_{\mathrm{K},\mathrm{M},\mathrm{N}}$ induces a $2$-isomorphism between the equalizer of the top triangle, which is
$(\mathrm{K}\s\mathrm{M})\square_{\mathrm{D}}\mathrm{N}$, and the equalizer of the bottom triangle,
which is $\mathrm{K}\s
(\mathrm{M}\square_{\mathrm{D}}\mathrm{N})$.
\end{proof}

\begin{corollary}\label{corollary:coalgebras-cats}
Coalgebras, bicomodules and bicomodule homomorphisms in $\underline{\cC}$ form a bicategory, in which horizontal composition is given by the cotensor product.
\end{corollary}

We will denote the bicategory of coalgebras, bicomodules and bicomodule homomorphisms in $\underline{\cC}$ by $\B_{\underline{\ccC}}$.

%%%%%%%%%%%%%%%%%%%%%%%%%%%%%%%%%%%%%%%%%

\subsection{Coalgebras and bicomodules under pseudofunctors}\label{subsection:coalg-pseudofunctors}

%%%%%%%%%%%%%%%%%%%%%%%%%%%%%%%%%%%%%%%%%

Let $\Phi\colon\cC\to\cD$ be a $\Bbbk$-linear pseudofunctor between two multifinitary bicategories with structural $2$-isomor\-phisms
\begin{gather*}
\phi_{\mathrm{F},\mathrm{G}}\colon\Phi(\mathrm{F}\mathrm{G})
\xrightarrow{\cong}\Phi(\mathrm{F})\Phi(\mathrm{G}),\quad\phi_{\mathtt{i}}\colon\Phi(\mathbbm{1}_{\mathtt{i}})\xrightarrow{\cong}\mathbbm{1}_{\Phi(\mathtt{i})}.
\end{gather*}
Denote by $\underline{\Phi}\colon\underline{\cC}\to\underline{\cD}$
its extension to the abelianizations, which is left exact by definition.

\begin{lemma}\label{lemma:coalgebras-transported}
The $\Bbbk$-linear pseudofunctor $\underline{\Phi}$ induces a $\Bbbk$-linear pseudofunctor,
for which we use the same notation, $\underline{\Phi}\colon\B_{\underline{\ccC}}\to\B_{\underline{\ccD}}$.
\end{lemma}

\begin{proof}
The proof consists of five parts, the first four of which are straightforward:
\begin{enumerate}[$($i$)$]

\item
If $\mathrm{C}=(\mathrm{C},\delta_{\mathrm{C}},\epsilon_{\mathrm{C}})$
is a coalgebra in $\underline{\cC}(\mathtt{i},\mathtt{i})$,
then the $1$-morphism $\underline{\Phi}(\mathrm{C})$ is a
coalgebra in $\underline{\cD}(\underline{\Phi}(\mathtt{i}),\underline{\Phi}(\mathtt{i}))$, with comultiplication and counit
\begin{gather*}
\delta_{\underline{\Phi}(\mathrm{C})}:=
\big[
\underline{\Phi}(\mathrm{C})\xrightarrow{\underline{\Phi}(\delta_{\mathrm{C}})}\underline{\Phi}(\mathrm{C}\mathrm{C})\xrightarrow{\phi_{\mathrm{C},\mathrm{C}}}\underline{\Phi}(\mathrm{C})\underline{\Phi}(\mathrm{C})\big]
,
\\
\epsilon_{\underline{\Phi}(\mathrm{C})}:=
\big[
\underline{\Phi}(\mathrm{C})\xrightarrow{\underline{\Phi}(\epsilon_{\mathrm{C}})} \underline{\Phi}(\mathbbm{1}_{\mathtt{i}})\xrightarrow{\phi_{\mathtt{i}}} \mathbbm{1}_{\underline{\Phi}(\mathtt{i})}
\big].
\end{gather*}

\item
If $\mathrm{M}=(\mathrm{M},\delta_{\mathrm{C},\mathrm{M}})$ is a left $\mathrm{C}$-comodule in $\underline{\cC}$, then $\underline{\Phi}(\mathrm{M})$ is a
left $\underline{\Phi}(\mathrm{C})$-comodule in $\underline{\cD}$ with left coaction
\begin{gather*}
\delta_{\underline{\Phi}(\mathrm{C}),\underline{\Phi}(\mathrm{M})}:=
[
\underline{\Phi}(\mathrm{M})\xrightarrow{\underline{\Phi}(\delta_{\mathrm{C},\mathrm{M}})} \underline{\Phi}(\mathrm{CM})
\xrightarrow{\phi_{\mathrm{C},\mathrm{M}}} \underline{\Phi}(\mathrm{C})\underline{\Phi}(\mathrm{M})
].
\end{gather*}

\item
If $\mathrm{M}=(\mathrm{M},\delta_{\mathrm{M},\mathrm{C}})$ is a right $\mathrm{C}$-comodule
in $\underline{\cC}$, then $\underline{\Phi}(\mathrm{M})$ is a right
$\underline{\Phi}(\mathrm{C})$-comodule in $\underline{\cD}$ with right coaction
\begin{gather*}
\delta_{\underline{\Phi}(\mathrm{M}),\underline{\Phi}(\mathrm{C})}:=
\big[
\underline{\Phi}(\mathrm{M})\xrightarrow{\underline{\Phi}(\delta_{\mathrm{M},\mathrm{C}})} \underline{\Phi}(\mathrm{MC})
\xrightarrow{\phi_{\mathrm{M},\mathrm{C}}} \underline{\Phi}(\mathrm{M})\underline{\Phi}(\mathrm{C})
\big].
\end{gather*}

\item
If $\mathrm{C}$ and $\mathrm{D}$ are two coalgebras and $\mathrm{M}$
is a $\mathrm{C}\text{-}\mathrm{D}$-bicomodule in $\underline{\cC}$, then
$\underline{\Phi}(\mathrm{M})$ is a $\underline{\Phi}(\mathrm{C})\text{-}\underline{\Phi}(\mathrm{D})$ bicomodule in $\underline{\cD}$ with bicoactions defined by the previous two points.

\item
Let $\mathrm{C}$ be a coalgebra in $\underline{\cC}$. If $\mathrm{M}$ is a right $\mathrm{C}$-comodule and $\mathrm{N}$ is a left $\mathrm{C}$-comodule
in $\underline{\cC}$, then there is a $2$-isomorphism
\begin{gather*}
\underline{\Phi}(\mathrm{M}\s\mathrm{N})\xrightarrow{\cong}
\underline{\Phi}(\mathrm{M})\square_{\underline{\Phi}(\mathrm{C})}\underline{\Phi}(\mathrm{N})
\end{gather*}
in $\underline{\cD}$. To prove this claim, consider the following diagram,
where we distinguish the associators in the abelianizations of
$\cC$ and $\cD$ by superscripts:
\begin{gather*}
\begin{tikzcd}[ampersand replacement=\&,column sep=4em]
\ar[dd,"\phi_{\mathrm{M},\mathrm{N}}"] \underline{\Phi}(\mathrm{M}\mathrm{N})
\ar[rr,"\underline{\Phi}(\delta_{\mathrm{M},\mathrm{C}}\mathrm{id}_{\mathrm{N}})"]
\ar[dr,"\underline{\Phi}(\mathrm{id}_{\mathrm{M}}\delta_{\mathrm{C},\mathrm{N}})",swap]
\&\&
\underline{\Phi}\big((\mathrm{MC})\mathrm{N}\big)
\ar[dd,"\phi_{\mathrm{MC},\mathrm{N}}"]
\ar[dl,"\underline{\Phi}(\alpha_{\mathrm{M},\mathrm{C},\mathrm{N}}^{\cccC})"]
\\
\&
\underline{\Phi}\big(\mathrm{M}(\mathrm{CN})\big)
\&
\\
\ar[dd,equal]\underline{\Phi}(\mathrm{M})\underline{\Phi}(\mathrm{N})
\ar[rr,"\underline{\Phi}(\delta_{\mathrm{M},\mathrm{C}})\mathrm{id}_{\underline{\Phi}(\mathrm{N})}" near start]
\ar[dr,"\mathrm{id}_{\underline{\Phi}(\mathrm{M})}\underline{\Phi}(\delta_{\mathrm{C},\mathrm{N}})",swap]
\& \&
\underline{\Phi}(\mathrm{MC})\underline{\Phi}(\mathrm{N})
\ar[dd,"\phi_{\mathrm{M},\mathrm{C}}\mathrm{id}_{\underline{\Phi}(\mathrm{N})}"]
\\
\&
\underline{\Phi}(\mathrm{M})\underline{\Phi}(\mathrm{CN})
\ar[from=uu, crossing over, "\phi_{\mathrm{M},\mathrm{CN}}" very near start]
\&
\phantom{.}
\\
\underline{\Phi}(\mathrm{M})\underline{\Phi}(\mathrm{N})
\ar[rr,"\delta_{\underline{\Phi}(\mathrm{M}),\underline{\Phi}(\mathrm{C})}\mathrm{id}_{\underline{\Phi}(\mathrm{N})}" near start]
\ar[dr,"\mathrm{id}_{\underline{\Phi}(\mathrm{M})}\delta_{\underline{\Phi}(\mathrm{C}),\underline{\Phi}(\mathrm{N})}",swap]
\& \&
\big(\underline{\Phi}(\mathrm{M})\underline{\Phi}(\mathrm{C})\big)\underline{\Phi}(\mathrm{N})
\ar[dl,"\alpha_{\underline{\Phi}(\mathrm{M}),\underline{\Phi}(\mathrm{C}),\underline{\Phi}(\mathrm{N})}^{\cccD}"]
\\
\&
\underline{\Phi}(\mathrm{M})\big(\underline{\Phi}(\mathrm{C})\underline{\Phi}(\mathrm{N})\big)
\ar[from=uu, crossing over, "\mathrm{id}_{\underline{\Phi}(\mathrm{M})}\phi_{\mathrm{C},\mathrm{N}}" very near start] \&
\end{tikzcd}
\end{gather*}
As before, all lateral faces commute due to
naturality and the coherence condition
of $\phi_{{}_{-},{}_{-}}$, as well as the definition of $\delta_{\underline{\Phi}(\mathrm{M}),\underline{\Phi}(\mathrm{C})}$ and $\delta_{\underline{\Phi}(\mathrm{C}),\underline{\Phi}(\mathrm{N})}$. Since all vertical arrows are $2$-isomorphisms, this implies that
the equalizer of the top triangle, which is $\underline{\Phi}(\mathrm{M}\s\mathrm{N})$ by left exactness of $\underline{\Phi}$,
is isomorphic to the equalizer
of the bottom triangle, which is
$\underline{\Phi}(\mathrm{M})\square_{\underline{\Phi}(\mathrm{C})}
\underline{\Phi}(\mathrm{N})$.\qedhere
\end{enumerate}
\end{proof}

%%%%%%%%%%%%%%%%%%%%%%%%%%%%%%%%%%%%%%%%%%%

\section{Coalgebras and birepresentations}\label{section:coalgebras-and-birepresentations}

As before, throughout this section $\cC$ is assumed to be a multifinitary bicategory. We will recall how finitary birepresentations of $\cC$ give rise to coalgebras in $\underline{\cC}$ and vice versa.
This is a bicategorical version of \cite{MMMT}, which in turn was inspired by \cite{Os}.

%%%%%%%%%%%%%%%%%%%%%%%%%%%%%%%%%%%%%%%%%%%%%%

\subsection{The finitary birepresentation associated to a coalgebra}

Let $\mathrm{C}\in\underline{\cC}(\mathtt{i},\mathtt{i})$ be a coalgebra.
For every $\mathtt{j}\in\cC$, take $\mathrm{comod}_{\underline{\ccC}}(\mathrm{C})_{\mathtt{j}}$
to be the category of right $\mathrm{C}$-comodules and comodule homomorphisms in
$\underline{\cC}(\mathtt{i},\mathtt{j})$. Then there is a birepresentation of $\underline{\cC}$ which assigns
\begin{itemize}

\item the category
$\mathrm{comod}_{\underline{\ccC}}(\mathrm{C})_{\mathtt{j}}$
to an object $\mathtt{j}\in\underline{\cC}$;

\item the functor
\begin{gather*}
\mathrm{X}_{\mathtt{k}\mathtt{j}}\circ_{\mathsf{h}}{}_{-}\colon \mathrm{comod}_{\underline{\ccC}}
(\mathrm{C})_{\mathtt{j}}\to \mathrm{comod}_{\underline{\ccC}}(\mathrm{C})_{\mathtt{k}}
\end{gather*}
to a $1$-morphism $\mathrm{X}_{\mathtt{k}\mathtt{j}}$ in $\underline{\cC}(\mathtt{j},
\mathtt{k})$, for $\mathtt{j},\mathtt{k}\in\underline{\cC}$;

\item
the natural transformation
\begin{gather*}
\beta_{\mathtt{k}\mathtt{j}}\circ_{\mathsf{h}}{}_{-}\colon
\mathrm{X}_{\mathtt{k}\mathtt{j}}\circ_{\mathsf{h}}{}_{-}\to
\mathrm{Y}_{\mathtt{k}\mathtt{j}}\circ_{\mathsf{h}}{}_{-}
\end{gather*}
to a $2$-morphism $\beta_{\mathtt{k}\mathtt{j}}\colon
\mathrm{X}_{\mathtt{k}\mathtt{j}}\to
\mathrm{Y}_{\mathtt{k}\mathtt{j}}$ in $\underline{\cC}(\mathtt{j},
\mathtt{k})$, for $\mathtt{j},\mathtt{k}\in\underline{\cC}$;

\item the natural isomorphism
\begin{gather*}
\begin{tikzcd}[ampersand replacement=\&]
\mathrm{comod}_{\underline{\ccC}}(\mathrm{C})_{\mathtt{j}}
\ar[rr,bend left=15,"\mathbbm{1}_{\mathtt{j}}\circ_{\mathsf{h}}{}_{-}" {name=U}]
\ar[rr, bend right=15,"\mathrm{Id}_{\mathrm{comod}_{\underline{\cccC}}(\mathrm{C})_{\mathtt{j}}}" {name=D}, swap]
\ar[d, Rightarrow, shorten <= .1em, shorten >= .1em, from=U, to=D, "\iota_{\mathtt{j}}"]
\& [2em]\&
\mathrm{comod}_{\underline{\ccC}}(\mathrm{C})_{\mathtt{j}}
\end{tikzcd}
\end{gather*}
to each object $\mathtt{j}\in\underline{\cC}$, where
$\iota_{\mathtt{j}}$ is the natural transformation induced by the left unitor $\lunit_{-}$ in $\cC$;

\item the natural isomorphism
\begin{gather*}
\begin{tikzcd}[ampersand replacement=\&,column sep=4em]
\underline{\cC}(\mathtt{k},\mathtt{l})\boxtimes \underline{\cC}(\mathtt{j},\mathtt{k})
\boxtimes \mathrm{comod}_{\underline{\ccC}}(\mathrm{C})_{\mathtt{j}}
\ar[d,"\mathrm{Id}\,\boxtimes\,\circ_{\mathsf{h}}",swap]
\ar[r,"\circ_{\mathsf{h}}\,\boxtimes\,\mathrm{Id}"]
\&
\underline{\cC}(\mathtt{j},\mathtt{l})
\boxtimes \mathrm{comod}_{\underline{\ccC}}(\mathrm{C})_{\mathtt{j}}
\ar[d,"\circ_{\mathsf{h}}"]
\\[3ex]
\underline{\cC}(\mathtt{k},\mathtt{l})
\boxtimes \mathrm{comod}_{\underline{\ccC}}(\mathrm{C})_{\mathtt{k}}
\ar[r,"\circ_{\mathsf{h}}",swap]
\ar[ur, Rightarrow, shorten <=2em, shorten >=2em, "\mu_{\mathtt{l}\mathtt{k}\mathtt{j}}:=
\alpha^{\mone}_{\mathtt{l}\mathtt{k}\mathtt{j}\mathtt{i}}"]
\&
\mathrm{comod}_{\underline{\ccC}}(\mathrm{C})_{\mathtt{l}}
\end{tikzcd}
\end{gather*}
to each triple of objects $\mathtt{j},\mathtt{k},\mathtt{l}\in\underline{\cC}$.
\end{itemize}

The \emph{underlying category} of this birepresentation is defined as
\begin{gather*}
\mathrm{comod}_{\underline{\ccC}}(\mathrm{C}):=\bigoplus_{\mathtt{j}\in\ccC}
\mathrm{comod}_{\underline{\ccC}}(\mathrm{C})_{\mathtt{j}}.
\end{gather*}
This birepresentation of $\underline{\cC}$ restricts to a finitary birepresentation of $\cC$ sending each $\mathtt{j}\in\cC$ to
$\mathrm{inj}_{\underline{\ccC}}(\mathrm{C})_{\mathtt{j}}$, the full subcategory of injective objects in
$\mathrm{comod}_{\underline{\ccC}}(\mathrm{C})_{\mathtt{j}}$. The underlying category of this
restriction is defined as
\begin{gather*}
\mathrm{inj}_{\underline{\ccC}}(\mathrm{C}):=
\bigoplus_{\mathtt{j}\in\ccC}
\mathrm{inj}_{\underline{\ccC}}(\mathrm{C})_{\mathtt{j}}.
\end{gather*}

We will use the notation $\boldsymbol{\mathrm{comod}}_{\underline{\ccC}}(\mathrm{C})$ and $\boldsymbol{\mathrm{inj}}_{\underline{\ccC}}(\mathrm{C})$ for these two
birepresentations, respectively.

We record a useful fact describing objects in finitary birepresentations:

\begin{lemma}\label{lemma:addGX}
If $\cC$ is quasi multifiab, then, for a coalgebra $\mathrm{C}$ in $\underline{\cC}(\mathtt{i},\mathtt{i})$, the category $\mathrm{inj}_{\underline{\ccC}}(\mathrm{C})_{\mathtt{j}}$
is the additive closure of $\{\mathrm{GC}\mid\mathrm{G}\in\cC(\mathtt{i},\mathtt{j})\}$
inside $\mathrm{comod}_{\underline{\ccC}}(\mathrm{C})_{\mathtt{j}}$.
\end{lemma}

\begin{proof}
First note that, since $\mathrm{C}$ is an injective $\mathrm{C}$-comodule and $\cC$ is multifiab, $\mathrm{G}\mathrm{C}$ is an injective $\mathrm{C}$-comodule for any $1$-morphism $\mathrm{G}$ in $\cC$.

Any $\mathrm{X}\in\mathrm{comod}_{\underline{\ccC}}(\mathrm{C})_{\mathtt{j}}$ embeds into $\mathrm{X}\mathrm{C}$, due to counitality and the coaction being a comodule morphism (note that
we are not claiming that this embedding is split in $\mathrm{comod}_{\underline{\ccC}}(\mathrm{C})_{\mathtt{j}}$). Suppose that $\mathrm{X}\in\mathrm{inj}_{\underline{\ccC}}(\mathrm{C})_{\mathtt{j}}$ and that $\mathrm{X}_0\to\mathrm{X}_1$ is an injective presentation of $\mathrm{X}$ in $\underline{\cC}$, where $\mathrm{X}_0,\mathrm{X}_1$ are $1$-morphisms in $\cC$.
Then $\mathrm{XC}$ embeds further into $\mathrm{X}_0\mathrm{C}$, which is injective. The claim follows.
\end{proof}

We call a morphism of finitary birepresentations \emph{exact} if it extends to an exact morphism of the corresponding abelianized birepresentations, i.e. its component functors extend to exact functors between the (injective) abelianizations of the component categories. We say that a functor between additive categories is an \emph{injective functor} if it is injective in the category of functors between the injective abelianizations. We call a morphism of finitary birepresentations \emph{injective} if its extension to the corresponding abelianized birepresentations is given by injective functors.

\begin{lemma}\label{lemma:biinjective-vs-exact}
Assume that $\cC$ is quasi multifiab, and let $\mathrm{C}$ and $\mathrm{D}$ be two coalgebras and $\mathrm{M}$ a biinjective
$\mathrm{C}\text{-}\mathrm{D}$-bicomodule in $\underline{\cC}$.
Cotensoring defines an exact morphism of birepresentations of $\cC$
\begin{gather*}
{}_{-}\s\mathrm{M}\colon
\boldsymbol{\mathrm{comod}}_{\underline{\ccC}}(\mathrm{C})\to \boldsymbol{\mathrm{comod}}_{\underline{\ccC}}(\mathrm{D})
\end{gather*}
which sends injective objects in the underlying categories to injective objects. In particular, it restricts to a morphism of
birepresentations of $\cC$
\begin{gather*}
{}_{-}\s\mathrm{M}\colon
\boldsymbol{\mathrm{inj}}_{\underline{\ccC}}(\mathrm{C})\to \boldsymbol{\mathrm{inj}}_{\underline{\ccC}}(\mathrm{D}).
\end{gather*}
\end{lemma}

\begin{proof}
Since $\mathrm{M}$ is injective as a left $\mathrm{C}$-comodule, it is a direct summand of $\mathrm{C}\mathrm{F}$ for some $1$-morphism $\mathrm{F}$ in $\cC$, in view of Lemma \ref{lemma:addGX}. The cotensor functor is therefore a direct summand of right multiplication by $\mathrm{F}$, which is exact, so exactness of ${}_{-}\s\mathrm{M}$ follows.

Similarly, if $\mathrm{N}$ is an injective right $\mathrm{C}$-comodule, it is a direct summand of $\mathrm{G}\mathrm{C}$ for some $1$-morphism $\mathrm{G}$ in $\cC$. Moreover, $\mathrm{M}$ being injective as a right $\mathrm{D}$-comodule, it is a direct summand of $\mathrm{H}\mathrm{D}$ for some $1$-morphism $\mathrm{H}$ in $\cC$. Thus $\mathrm{N}\s\mathrm{M}$ is a direct summand of $\mathrm{G}\mathrm{C}
\s\mathrm{H}\mathrm{D}\cong
\mathrm{G}(\mathrm{H}\mathrm{D})\cong(\mathrm{G}\mathrm{H})\mathrm{D}$ which is an injective right $\mathrm{D}$-comodule, so $\mathrm{N}\s\mathrm{M}$ is itself injective as a right $\mathrm{D}$-comodule. This completes the proof.
\end{proof}

Finally, note that if $f\colon\mathrm{M}\to
\mathrm{N}$ is
a homomorphism between two biinjective $\mathrm{C}\text{-}\mathrm{D}$-bicomodules $\mathrm{M},\mathrm{N}$ in $\underline{\cC}$, then
\begin{gather*}
{}_{-}\s f\colon
{}_{-}\s\mathrm{M}\to
{}_{-}\s\mathrm{N}
\end{gather*}
defines a modification.

%%%%%%%%%%%%%%%%%%%%%%%%%%%%%%%%%%%%%%%%%

\subsection{Morita--Takeuchi theory in bicategories}\label{subsection:mt}

We start by discussing the notion of Morita--Takeuchi equivalence in finitary bicategories (MT equivalence for short).

\begin{definition}\label{definition:mt-equi}
We say that two coalgebras $\mathrm{C}$ and $\mathrm{D}$
in $\underline{\cC}$ are \emph{MT equivalent} if
\begin{gather*}
\boldsymbol{\mathrm{inj}}_{\underline{\ccC}}(\mathrm{C})
\simeq
\boldsymbol{\mathrm{inj}}_{\underline{\ccC}}(\mathrm{D})
\end{gather*}
as birepresentations of ${\cC}$.
\end{definition}

The following theorem is a straightforward generalization of \cite[Theorem 5.1]{MMMT} in the context of bicategories, so we omit the proof. It resembles the classical Morita--Takeuchi equivalence for coalgebras over a field.

\begin{theorem}\label{theorem:MT}
Two coalgebras $\mathrm{C}$ and
$\mathrm{D}$ in $\underline{\cC}$ are MT equivalent if and only if
there exist a $\mathrm{C}\text{-}\mathrm{D}$-bicomodule
$\mathrm{M}$ and a $\mathrm{D}\text{-}\mathrm{C}$-bicomodule
$\mathrm{N}$, and
bicomodule isomorphisms
\begin{gather*}
f\colon \mathrm{C}\xrightarrow{\cong}\mathrm{M}\square_{\mathrm{D}}\mathrm{N}
,\quad
g\colon\mathrm{D}\xrightarrow{\cong} \mathrm{N}\s\mathrm{M}
\end{gather*}
in $\underline{\cC}$ such that we have commuting diagrams
\begin{gather*}
\begin{tikzcd}[ampersand replacement=\&]
\arrow[d,"f\s \mathrm{id}_{\mathrm{M}}",swap]
\mathrm{C}\s\mathrm{M}
\&
\arrow[l,"\delta_{\mathrm{C},\mathrm{M}}",swap]\mathrm{M}
\arrow[r,"\delta_{\mathrm{M},\mathrm{D}}"]
\&
\mathrm{M}\square_{\mathrm{D}}\mathrm{D}
\arrow[d,"\mathrm{id}_{\mathrm{M}} \square_{\mathrm{D}} g"]
\\
\big(\mathrm{M}\square_{\mathrm{D}}\mathrm{N}\big)\s\mathrm{M} \arrow[rr,"\alpha_{\mathrm{M},\mathrm{N},\mathrm{M}}",swap]
\&\&
\mathrm{M}\square_{\mathrm{D}}\big(\mathrm{N}\s\mathrm{M}\big)
\end{tikzcd}
,
\\
\begin{tikzcd}[ampersand replacement=\&]
\arrow[d,"g\square_{\mathrm{D}} \mathrm{id}_{\mathrm{N}}",swap]
\mathrm{D}\square_{\mathrm{D}}\mathrm{N}
\&
\arrow[l, "\delta_{\mathrm{D},\mathrm{N}}",swap]
\mathrm{N}
\arrow[r,"\delta_{\mathrm{N},\mathrm{C}}"]
\&
\mathrm{N}
\s\mathrm{C}
\arrow[d,"\mathrm{id}_{\mathrm{N}} \s f"]
\\
\big(\mathrm{N}\s\mathrm{M}\big)\square_{\mathrm{D}}\mathrm{N} \arrow[rr,"\alpha_{\mathrm{N},\mathrm{M},\mathrm{N}}",swap]
\&\&
\mathrm{N}\s\big(\mathrm{M}\square_{\mathrm{D}}\mathrm{N}\big)
\end{tikzcd}
.
\end{gather*}
\end{theorem}

\begin{remark}\label{remark:biinjective}
Note that $\mathrm{M}$ and $\mathrm{N}$ are automatically biinjective if they satisfy
the conditions in Theorem \ref{theorem:MT}.
\end{remark}

Suppose that $\Phi\colon \cC\to \cD$ is a $\Bbbk$-linear pseudofunctor between two multifinitary bicategories. The following corollary follows immediately
from Lemma \ref{lemma:coalgebras-transported} and Theorem \ref{theorem:MT}:

\begin{corollary}\label{cor:cotensor-under-pseudofunctor}
The extension $\underline{\Phi}\colon\underline{\cC}\to \underline{\cD}$ sends MT equivalent coalgebras in $\underline{\cC}$ to MT equivalent coalgebras in $\underline{\cD}$.
\end{corollary}

%%%%%%%%%%%%%%%%%%%%%%%%%%%%%%%%%%%%%%%%%

\subsection{The internal cohom construction}\label{subsection:int-hom}

%%%%%%%%%%%%%%%%%%%%%%%%%%%%%%%%%%%%%%%%%

Let $\cC$ be a multifinitary bicategory, $\mathbf{M}$ a finitary birepresentation of $\cC$,
and let $\underline{\mathbf{M}}$ denote the corresponding abelian
birepresentation of $\underline{\cC}$ (see Definition \ref{definition:fincatcat-abel}).

For all $X\in\mathbf{M}(\mathtt{j})$, $Y\in\mathbf{M}(\mathtt{i})$
the left exact functor
\begin{gather*}
\Gamma_{X,Y}\colon \underline{\cC}(\mathtt{i},\mathtt{j})\to \cV\mathrm{ect},
\quad
\mathrm{F}\mapsto\mathrm{Hom}_{\underline{\mathbf{M}}(\mathtt{j})}
\big(X,\underline{\mathbf{M}}_{\mathtt{j}\mathtt{i}}(\mathrm{F})\,Y\big)
\end{gather*}
is representable by (the dual of) the Eilenberg--Watts theorem, see e.g. 
\cite[Theorem 2.3 on page 58]{}.
This means that there exist a $1$-morphism $[Y,X]\in\underline{\cC}(\mathtt{i},\mathtt{j})$, called the \emph{internal cohom} from $Y$ to $X$, and a natural isomorphism
\begin{gather*}
\gamma_{Y,X}\colon \mathrm{Hom}_{\underline{\ccC}(\mathtt{i},\mathtt{j})}
\big([Y,X],\mathrm{F}\big) \xrightarrow{\cong}\mathrm{Hom}_{\underline{\mathbf{M}}(\mathtt{j})}\big(X,\underline{\mathbf{M}}_{\mathtt{j}\mathtt{i}}(\mathrm{F})\,Y\big),
\text{ for all } \mathrm{F}\in\underline{\cC}(\mathtt{i},\mathtt{j}).
\end{gather*}
By Yoneda's lemma, the pair $\big([Y,X],\gamma_{Y,X}\big)$ is unique up to a unique natural isomorphism in the following sense. If $\big([Y,X],\gamma_{Y,X}\big)$ and $\big([Y,X]^{\prime},\gamma_{Y,X}^{\prime}\big)$
are both internal cohoms from
$Y$ to $X$, then there exists a unique $2$-isomorphism $\phi\colon [Y,X]\to[Y,X]^{\prime}$ such that
\begin{gather*}
\gamma_{Y,X}^{\mone}\big(\gamma_{Y,X}^{\prime}({}_{-})\big)={}_{-}\circ_{\mathsf{v}}\phi
\end{gather*}
as natural isomorphisms
\begin{gather*}
\mathrm{Hom}_{\underline{\ccC}(\mathtt{i},\mathtt{j})}([Y,X]^{\prime},
\mathrm{F})\xrightarrow{\cong}\mathrm{Hom}_{\underline{\ccC}(\mathtt{i},\mathtt{j})}([Y,X], \mathrm{F}),\text{ for all }\mathrm{F}\in\underline{\cC}(\mathtt{i},\mathtt{j}).
\end{gather*}
The coevaluation $\mathrm{coev}_{Y,X}\colon
X\to\underline{\mathbf{M}}_{\mathtt{j}\mathtt{i}}\big([Y,X]\big)Y$ in $\underline{\mathbf{M}}(\mathtt{j})$
is defined as the image of $\mathrm{id}_{[Y,X]}$ under the isomorphism
\begin{gather}\label{eq:coev}
\gamma_{Y,X}\colon\mathrm{Hom}_{\underline{\ccC}(\mathtt{i},\mathtt{j})}\big([Y,X],[Y,X]\big)
\xrightarrow{\cong}
\mathrm{Hom}_{\underline{\mathbf{M}}(\mathtt{j})}\big(X,
\underline{\mathbf{M}}_{\mathtt{j}\mathtt{i}}\big([Y,X]\big)Y\big).
\end{gather}
Using the coevaluation morphisms, we can define a canonical $2$-morphism
\begin{gather*}
\delta_{Z,Y,X}\colon[Z,X]\to[Y,X][Z,Y],
\end{gather*}
for all $X\in\mathbf{M}(\mathtt{i})$, $Y\in\mathbf{M}(\mathtt{j})$, $Z\in\mathbf{M}(\mathtt{k})$ as follows.
Consider the morphism $\tau$ defined by
\begin{gather*}
\begin{tikzcd}[ampersand replacement=\&,column sep=6em]
X
\arrow[d,xshift=-2cm,phantom, ""{coordinate, name=Z}]
\arrow[r,"\mathrm{coev}_{Y,X}"]
\&
\underline{\mathbf{M}}_{\mathtt{i}\mathtt{j}}
\big([Y,X]\big)\,Y
\arrow[ld,"{\underline{\mathbf{M}}_{\mathtt{i}\mathtt{j}}
([Y,X])\mathrm{coev}_{Z,Y}}" description,rounded corners,to path={ --([xshift=2ex]\tikztostart.east)|- (Z)[near end]\tikztonodes-| ([xshift=-2ex]\tikztotarget.west)--(\tikztotarget)}]
\\
\underline{\mathbf{M}}_{\mathtt{i}\mathtt{j}}\big([Y,X]\big)
\underline{\mathbf{M}}_{\mathtt{j}\mathtt{k}}\big([Z,Y]\big)Z
\arrow[r,"(\mu_{\mathtt{i}\mathtt{j}\mathtt{k}}^{[Y,X],[Z,Y]})_{Z}",swap]
\&
\underline{\mathbf{M}}_{\mathtt{i}\mathtt{k}}\big([Y,X][Z,Y]\big)Z
\end{tikzcd}
.
\end{gather*}
The $2$-morphism
$\delta_{Z,Y,X}$ is defined as the image
$\gamma_{Z,X}^{\mone}(\tau)$ under the isomorphism
\begin{gather*}
\gamma_{Z,X}\colon
\mathrm{Hom}_{\underline{\ccC}(\mathtt{k},\mathtt{i})}
\big([Z,X],[Y,X][Z,Y]\big)
\xrightarrow{\cong}
\mathrm{Hom}_{\underline{\mathbf{M}}(\mathtt{i})}\big(X,
\underline{\mathbf{M}}_{\mathtt{i}\mathtt{k}}\big([Y,X][Z,Y]\big)Z\big).
\end{gather*}

\begin{lemma}\label{lemma:coass-gen-internal-hom}
For all $X\in\mathbf{M}(\mathtt{l})$, $Y\in\mathbf{M}(\mathtt{k})$,
$Z\in\mathbf{M}(\mathtt{j})$, $W\in\mathbf{M}(\mathtt{i})$, there is a commutative diagram
\begin{gather}\label{eq:coass-gen-internal-hom1}
\begin{tikzcd}[ampersand replacement=\&]
\arrow[d,"\delta_{Z,Y,X}\mathrm{id}_{[W,Z]}",swap]
[Z,X][W,Z]
\&
\arrow[l,"\delta_{W,Z,X}",swap]
\text{$[W,X]$}
\arrow[r,"\delta_{W,Y,X}"]
\&
\text{$[Y,X] [W,Y]$}
\arrow[d,"\mathrm{id}_{[Y,X]}\delta_{W,Z,Y}"]
\\
\text{$\big([Y,X][Z,Y]\big)[W,Z]$}
\arrow[rr,"\alpha_{[Y,X],[Z,Y],[W,Z]}",swap]
\&
\&
\text{$[Y,X]\big([Z,Y][W,Z]\big)$}
\end{tikzcd}
.
\end{gather}
\end{lemma}

\begin{proof}
By the isomorphisms in \eqref{eq:coev}, the
commutativity of the diagram in
\eqref{eq:coass-gen-internal-hom1} is
equivalent to the commutativity of the boundary
of the diagram
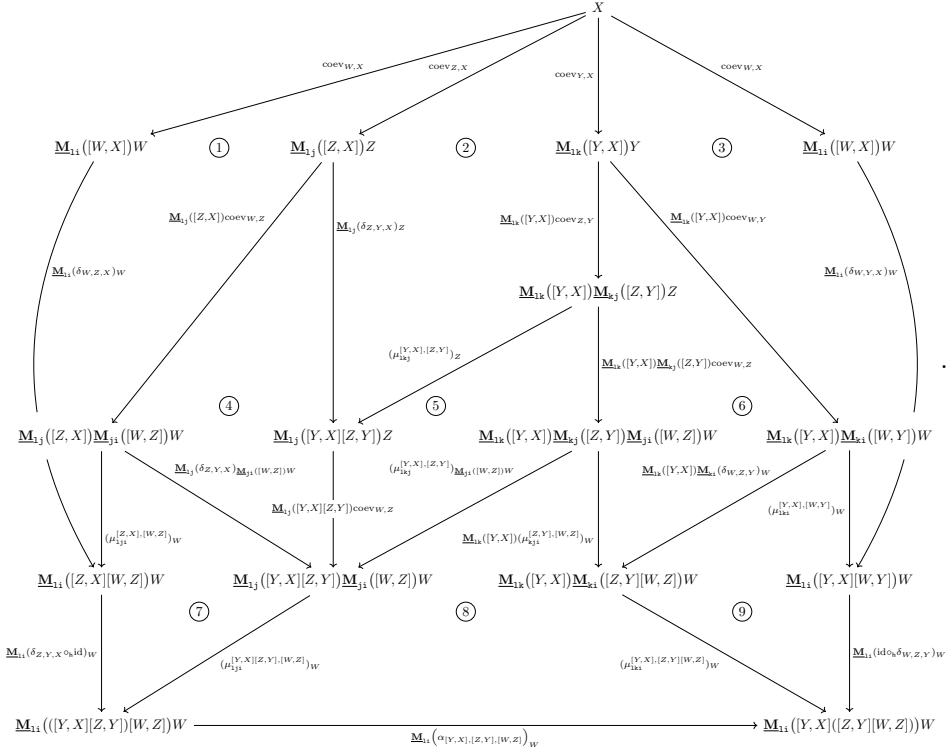
\begin{gather*}
\adjustbox{scale=.52,center}{%
\begin{tikzcd}[ampersand replacement=\&]
\&
\&
X
\ar[dll,"\mathrm{coev}_{W,X}", swap]
\ar[dl,"\mathrm{coev}_{Z,X}", swap]
\ar[d,"\mathrm{coev}_{Y,X}", swap]
\ar[dr,"\mathrm{coev}_{W,X}"]
\&
\\[15ex]
\underline{\mathbf{M}}_{\mathtt{l}\mathtt{i}}\big([W,X]\big)W
\ar[ddd, bend right, dash pattern=on 190pt off 35pt, "\underline{\mathbf{M}}_{\mathtt{l}\mathtt{i}}(\delta_{W,Z,X})_W" near start]
\arrow[r, phantom, "\circled{1}"]
\&
\underline{\mathbf{M}}_{\mathtt{l}\mathtt{j}}\big([Z,X]\big)Z
\ar[ddl,"{\underline{\mathbf{M}}_{\mathtt{l}\mathtt{j}}([Z,X])\mathrm{coev}_{W,Z}}"near start, swap]
\ar[dd,"\underline{\mathbf{M}}_{\mathtt{l}\mathtt{j}}(\delta_{Z,Y,X})_Z" near start]
\arrow[r, phantom, "\circled{2}"]
\&
\underline{\mathbf{M}}_{\mathtt{l}\mathtt{k}}\big([Y,X]\big)Y
\ar[d,"{\underline{\mathbf{M}}_{\mathtt{l}\mathtt{k}}([Y,X])\mathrm{coev}_{Z,Y}}", swap]
\ar[ddr,"{\underline{\mathbf{M}}_{\mathtt{l}\mathtt{k}}([Y,X])\mathrm{coev}_{W,Y}}" near start]
\arrow[r, phantom, "\circled{3}"]
\&
\underline{\mathbf{M}}_{\mathtt{l}\mathtt{i}}\big([W,X]\big)W
\ar[ddd, bend left, dash pattern=on 190pt off 35pt,  "\underline{\mathbf{M}}_{\mathtt{l}\mathtt{i}}(\delta_{W,Y,X})_W" near start, swap]
\\[15ex] 		
\&
\&
\underline{\mathbf{M}}_{\mathtt{l}\mathtt{k}}\big([Y,X]\big)
\underline{\mathbf{M}}_{\mathtt{k}\mathtt{j}}\big([Z,Y]\big)Z
\ar[dl, "(\mu_{\mathtt{l}\mathtt{k}\mathtt{j}}^{[Y,X],
[Z,Y]})_Z", swap]
\ar[d,"{\underline{\mathbf{M}}_{\mathtt{l}\mathtt{k}}([Y,X])
\underline{\mathbf{M}}_{\mathtt{k}\mathtt{j}}([Z,Y])\mathrm{coev}_{W,Z}}"]
\&
\\[15ex]
\underline{\mathbf{M}}_{\mathtt{l}\mathtt{j}}\big([Z,X]\big)
\underline{\mathbf{M}}_{\mathtt{j}\mathtt{i}}\big([W,Z]\big)W
\ar[d,"(\mu_{\mathtt{l}\mathtt{j}\mathtt{i}}^{[Z,X],
[W,Z]})_W" near end]
\ar[dr,  "\underline{\mathbf{M}}_{\mathtt{l}\mathtt{j}}(\delta_{Z,Y,X})_{\underline{\mathbf{M}}_{\mathtt{j}\mathtt{i}}([W,Z])W}" near start]
\arrow[r, phantom, yshift=5ex, "\circled{4}"]
\&
\underline{\mathbf{M}}_{\mathtt{l}\mathtt{j}}\big([Y,X]
[Z,Y]\big)Z
\ar[d,"{\underline{\mathbf{M}}_{\mathtt{l}\mathtt{j}}([Y,X]
[Z,Y])\mathrm{coev}_{W,Z}}"description]
\arrow[r, phantom, yshift=5ex,"\circled{5}"]
\&
\underline{\mathbf{M}}_{\mathtt{l}\mathtt{k}}\big([Y,X]\big)
\underline{\mathbf{M}}_{\mathtt{k}\mathtt{j}}\big([Z,Y]\big)
\underline{\mathbf{M}}_{\mathtt{j}\mathtt{i}}\big([W,Z]\big)W
\ar[dl, "(\mu_{\mathtt{l}\mathtt{k}\mathtt{j}}^{[Y,X],
[Z,Y]})_{\underline{\mathbf{M}}_{\mathtt{j}\mathtt{i}}([W,Z])W}"
near start, swap]
\ar[d,"{\underline{\mathbf{M}}_{\mathtt{l}\mathtt{k}}([Y,X])(\mu_{\mathtt{k}\mathtt{j}
\mathtt{i}}^{[Z,Y],[W,Z]})_{W}}" near end, swap]
\arrow[r, phantom, yshift=5ex,"\circled{6}"]
\&
\underline{\mathbf{M}}_{\mathtt{l}\mathtt{k}}\big([Y,X]\big)
\underline{\mathbf{M}}_{\mathtt{k}\mathtt{i}}\big([W,Y]\big)W
\ar[dl, "{\underline{\mathbf{M}}_{\mathtt{l}\mathtt{k}}([Y,X]) \underline{\mathbf{M}}_{\mathtt{k}\mathtt{i}}
(\delta_{W,Z,Y})_W}" near start, swap]
\ar[d,"(\mu_{\mathtt{l}\mathtt{k}\mathtt{i}}^{[Y,X],
[W,Y]})_{W}", swap]
\\[15ex]
\underline{\mathbf{M}}_{\mathtt{l}\mathtt{i}}\big([Z,X][W,Z]\big)W
\ar[d,"\underline{\mathbf{M}}_{\mathtt{l}\mathtt{i}}
(\delta_{Z,Y,X}\circ_{\mathsf{h}}\mathrm{id})_W", swap]
\arrow[r, phantom, yshift=-5ex, "\circled{7}"]
\&
\underline{\mathbf{M}}_{\mathtt{l}\mathtt{j}}\big([Y,X]
[Z,Y]\big)\underline{\mathbf{M}}_{\mathtt{j}\mathtt{i}}\big([W,Z]\big)W
\ar[dl, "(\mu_{\mathtt{l}\mathtt{j}\mathtt{i}}^{[Y,X][Z,Y], [W,Z]})_W"]
\arrow[r, phantom, yshift=-5ex,"\circled{8}"]
\&
\underline{\mathbf{M}}_{\mathtt{l}\mathtt{k}}\big([Y,X]\big)
\underline{\mathbf{M}}_{\mathtt{k}\mathtt{i}}\big([Z,Y][W,Z]\big)W
\ar[dr, "(\mu_{\mathtt{l}\mathtt{k}\mathtt{i}}^{[Y,X],
[Z,Y] [W,Z]})_W",swap]
\arrow[r,phantom,yshift=-5ex,"\circled{9}"]
\&
\underline{\mathbf{M}}_{\mathtt{l}\mathtt{i}}\big([Y,X][W,Y]\big)W
\ar[d,"\underline{\mathbf{M}}_{\mathtt{l}\mathtt{i}}
(\mathrm{id}\circ_{\mathsf{h}}\delta_{W,Z,Y})_W"]
\\[15ex]	
\underline{\mathbf{M}}_{\mathtt{l}\mathtt{i}}\big(([Y,X]
[Z,Y])[W,Z]\big)W
\ar[rrr,"\underline{\mathbf{M}}_{\mathtt{l}\mathtt{i}}\big(
\alpha_{[Y,X],[Z,Y],[W,Z]}\big)_W",swap]
\&
\&
\&
\underline{\mathbf{M}}_{\mathtt{l}\mathtt{i}}\big([Y,X]
([Z,Y][W,Z])\big)W
\end{tikzcd}
}
\hspace{-.25cm}.
\end{gather*}
We note that
\begin{itemize}

\item the facets labeled $1$, $2$, $3$ and $6$ commute
by definition of $\delta_{W,Z,X}$,
$\delta_{Z,Y,X}$, $\delta_{W,Y,X}$ and $\delta_{W,Z,Y}$, respectively;

\item the facets labeled $4$, $5$, $7$ and $9$ commute
by naturality of $\underline{\mathbf{M}}_{\mathtt{lj}}(\delta_{Z,Y,X})$,
$\mu_{\mathtt{l}\mathtt{k}\mathtt{j}}$,
$\mu_{\mathtt{l}\mathtt{j}\mathtt{i}}$ and $\mu_{\mathtt{l}\mathtt{k}\mathtt{i}}$, respectively;

\item the facet labeled $8$ commutes due to the coherence condition for $\mu$ in \eqref{eq:birepresentation3} and the fact that $\underline{\mathbf{M}}_{\mathtt{l}\mathtt{k}}([Y,X])(\mu_{\mathtt{k}\mathtt{j}
\mathtt{i}}^{[Z,Y],[W,Z]})_{W}=\big(\mathrm{id}_{\underline{\mathbf{M}}_{\mathtt{l}\mathtt{k}}([Y,X])}\circ_{\mathsf{h}}\mu_{\mathtt{k}\mathtt{j}
\mathtt{i}}^{[Z,Y],[W,Z]}\big)_W$.

\end{itemize}
Commutativity of these facets implies that all paths in
the above diagram from $X$ at the top to
$\underline{\mathbf{M}}_{\mathtt{l}\mathtt{i}}\big([Y,X]
\big([Z,Y][W,Z]\big)\big)W$ at the bottom are equal, in particular,
the two paths around the boundary, which is exactly what we had to show.
\end{proof}

For every $X\in\mathbf{M}(\mathtt{i})$ there is also
a canonical $2$-morphism
\begin{gather*}
\epsilon_X\colon[X,X]\to\mathbbm{1}_{\mathtt{i}},
\end{gather*}
defined as the image of $\big(\iota_{\mathtt{i}}^{\mone}\big)_X$ under the isomorphism
\begin{gather*}
\mathrm{Hom}_{\underline{\mathbf{M}}(\mathtt{i})}\big(X,\underline{\mathbf{M}}_{\mathtt{i}\mathtt{i}}(\mathbbm{1}_{\mathtt{i}})X\big)
\xrightarrow{\cong}
\mathrm{Hom}_{\underline{\ccC}(\mathtt{i},\mathtt{i})}\big([X,X],\mathbbm{1}_{\mathtt{i}}\big),
\end{gather*}
where $\iota_{\mathtt{j}}$ was defined in Definition \ref{definition:fin-birepresentation}.

\begin{lemma}\label{lemma:counitality-gen-intern-hom}
For every $X\in\mathbf{M}(\mathtt{i})$, $Y\in\mathbf{M}(\mathtt{j})$, we have commutative diagrams
\begin{gather}\label{eq:right-counit-gen-intern-hom}
\begin{tikzcd}[ampersand replacement=\&,column sep=4em]
[X,Y]
\ar[r,"\delta_{X,X,Y}"]
\ar[dr,"(\runit_{[X,Y]})^{\mone}",swap]
\&
\text{$[X,Y][X,X]$}
\ar[d,"\mathrm{id}_{[X,Y]}\epsilon_X"]
\\
\&
\text{$[X,Y]$}
\mathbbm{1}_{\mathtt{i}}
\end{tikzcd}
,\quad
\begin{tikzcd}[ampersand replacement=\&,column sep=4em]
[X,Y]
\ar[r,"\delta_{X,Y,Y}"]
\ar[dr,"(\lunit_{[X,Y]})^{\mone}", swap]
\&
\text{$[Y,Y][X,Y]$}
\ar[d,"\epsilon_Y\mathrm{id}_{[X,Y]}"]
\\
\&
\mathbbm{1}_{\mathtt{j}}
\text{$[X,Y]$}
\end{tikzcd}
.
\end{gather}
\end{lemma}

\begin{proof}
We only prove commutativity of the left diagram in \eqref{eq:right-counit-gen-intern-hom}. Commutativity of the right diagram can be proved by similar arguments.

First, consider the commutative diagram
\begin{gather}\label{eq:right-counit-gen-intern-hom1}
\begin{tikzcd}[ampersand replacement=\&]
\mathrm{Hom}_{\underline{\ccC}(\mathtt{i},\mathtt{i})}\big([X,X],[X,X]\big)
\ar[d,"\epsilon_X\circ_{\mathsf{v}}{}_{-}",swap]
\ar[r,"\cong"]
\&
\mathrm{Hom}_{\underline{\mathbf{M}}(\mathtt{i})}
\big(X,\underline{\mathbf{M}}_{\mathtt{i}\mathtt{i}}\big([X,X]\big) X\big)
\ar[d,"\underline{\mathbf{M}}_{\mathtt{i}\mathtt{i}}(\epsilon_X)_X\circ_{\mathsf{v}}{}_{-}"]
\\
\mathrm{Hom}_{\underline{\ccC}(\mathtt{i},\mathtt{i})}\big([X,X],\mathbbm{1}_{\mathtt{i}} \big)
\ar[r,"\cong",swap]
\&
\mathrm{Hom}_{\underline{\mathbf{M}}(\mathtt{i})}
\big(X, \underline{\mathbf{M}}_{\mathtt{i}\mathtt{i}}(\mathbbm{1}_{\mathtt{i}}) X\big)
\end{tikzcd}.
\end{gather}
By comparing the image of
$\mathrm{id}_{[X,X]}\in\mathrm{Hom}_{\underline{\ccC}(\mathtt{i},\mathtt{i})}\big([X,X],[X,X] \big)$ under the two maps to $\mathrm{Hom}_{\underline{\mathbf{M}}(\mathtt{i})}
\big(X, \underline{\mathbf{M}}_{\mathtt{i}\mathtt{i}}(\mathbbm{1}_{\mathtt{i}})X\big)$ corresponding to the two paths in \eqref{eq:right-counit-gen-intern-hom1}, we see that
\begin{gather}\label{eq:right-counit-gen-intern-hom2}
\underline{\mathbf{M}}_{\mathtt{i}\mathtt{i}}(\epsilon_X)_X \circ_{\mathsf{v}}
\mathrm{coev}_{X,X}=(\iota_{\mathtt{i}})_X^{\mone}.
\end{gather}

The next observation is that commutativity of the left diagram in
\eqref{eq:right-counit-gen-intern-hom} is equivalent to commutativity of the boundary of
\begin{gather*}
\adjustbox{scale=.75,center}{%
\begin{tikzcd}[ampersand replacement=\&,column sep=3.5em]
Y
\ar[rrrr,"\mathrm{coev}_{X,Y}"]
\ar[d, "\mathrm{coev}_{X,Y}",swap]
\& [2em] \&
\& [2em]\&
\underline{\mathbf{M}}_{\mathtt{j}\mathtt{i}}\big([X,Y]\big)X
\ar[d,"\underline{\mathbf{M}}_{\mathtt{j}\mathtt{i}}(\delta_{X,X,Y})_X"]
\\[5ex]
\underline{\mathbf{M}}_{\mathtt{j}\mathtt{i}}\big([X,Y]\big)X
\ar[rr,"{\underline{\mathbf{M}}_{\mathtt{j}\mathtt{i}}([X,Y])\mathrm{coev}_{X,X}}"]
\ar[d,equal]
\ar[rrrr, phantom, yshift=7ex, "\circled{1}"]
\ar[drr, phantom, "\circled{2}"]
\& [2em] \&
\underline{\mathbf{M}}_{\mathtt{j}\mathtt{i}}\big([X,Y]\big)
\underline{\mathbf{M}}_{\mathtt{i}\mathtt{i}}\big([X,X]\big)X
\ar[rr,"(\mu_{\mathtt{j}\mathtt{i}\mathtt{i}}^{[X,Y],[X,X]})_X"]
\ar[d,"{\underline{\mathbf{M}}_{\mathtt{j}\mathtt{i}}([X,Y])
\underline{\mathbf{M}}_{\mathtt{i}\mathtt{i}}(\epsilon_X)_X}"]
\ar[drr, phantom, "\circled{3}"]
\& [2em]\&
\underline{\mathbf{M}}_{\mathtt{j}\mathtt{i}}\big([X,Y][X,X]\big)X
\ar[d,"\underline{\mathbf{M}}_{\mathtt{j}\mathtt{i}}(\mathrm{id}
\circ_{\mathsf{h}}\epsilon_X)_X"]
\\[5ex]
\underline{\mathbf{M}}_{\mathtt{j}\mathtt{i}}\big([X,Y]\big)X
\ar[rr, "{\underline{\mathbf{M}}_{\mathtt{j}\mathtt{i}}([X,Y])(\iota_{\mathtt{i}})_X^{\mone}}",swap]
\ar[rrrr, bend right,"\underline{\mathbf{M}}_{\mathtt{j}\mathtt{i}}\big(
(\runit_{[X,Y]})^{\mone}\big)_X",swap]
\ar[rrrr, phantom, yshift=-7ex,"\circled{4}"]
\&[2em] \&
\underline{\mathbf{M}}_{\mathtt{j}\mathtt{i}}\big([X,Y]\big)
\underline{\mathbf{M}}_{\mathtt{i}\mathtt{i}}(\mathbbm{1}_{\mathtt{i}} )X
\ar[rr, "(\mu_{\mathtt{j}\mathtt{i}\mathtt{i}}^{[X,Y] ,
\mathbbm{1}_{\mathtt{i}}} )_X",swap]
\&[2em] \&
\underline{\mathbf{M}}_{\mathtt{j}\mathtt{i}}\big([X,Y]\mathbbm{1}_{\mathtt{i}}\big)X
\end{tikzcd}
}\hspace*{-.25cm}.
\end{gather*}
Commutativity of the boundary of this diagram follows from commutativity of the internal facets. The latter commute due to
\begin{itemize}

\item the definition of $\delta_{X,X,Y}$ for the facet labeled $1$;

\item \eqref{eq:right-counit-gen-intern-hom2} for the facets labeled $2$;

\item naturality of $\mu_{\mathtt{j}\mathtt{i}\mathtt{i}}$ for the facet labeled $3$;

\item the left coherence condition in \eqref{eq:birepresentation2}
for the facet labeled $4$.
\end{itemize}
This completes the proof.
\end{proof}

The following proposition is an immediate consequence of Lemmas
\ref{lemma:coass-gen-internal-hom} and \ref{lemma:counitality-gen-intern-hom}.

\begin{proposition}\label{proposition:coalg-comod-via-internal-hom}
Let $\mathbf{M}$ be a finitary birepresentation of $\cC$.
For any $X\in\underline{\mathbf{M}}(\mathtt{i})$
and any $Y\in\underline{\mathbf{M}}(\mathtt{j})$,

\begin{enumerate}[$($i$)$]

\item\label{proposition:coalg-comod-via-internal-hom-1}
the triple $\big([X,X],\delta_{X,X,X},\epsilon_X\big)$ is a coalgebra in
$\underline{\cC}$;

\item\label{proposition:coalg-comod-via-internal-hom-2}
the triple $\big([X,Y], \delta_{X,Y,Y},\delta_{X,X,Y}\big)$ is
a $[Y,Y]\text{-}[X,X]$-bicomodule in $\underline{\cC}$.

\end{enumerate}
\end{proposition}

As in \cite{MMMT}, we will often use the notation $\mathrm{C}^{X}$ for the coalgebra $[X,X]$.
The following theorem is the analog of \cite[Theorem 4.7]{MMMT} for quasi fiab bicategories and finitary birepresentations. The proof is entirely analogous and therefore omitted.

\begin{theorem}\label{theorem:generator}
Assume that $\cC$ is quasi multifiab and $\mathbf{M}$ is a finitary birepresentation of $\cC$.
Let $X\in\mathbf{M}(\mathtt{i})$ be a generator of $\mathbf{M}$.
Then there is an equivalence of finitary
birepresentations
\begin{gather*}
\mathbf{M}\simeq\boldsymbol{\mathrm{inj}}_{\underline{\ccC}}(\mathrm{C}^{X})
\end{gather*} such that
$Y\mapsto [X,Y]$
for all $Y\in\mathbf{M}(\mathtt{j})$.
\end{theorem}

\begin{remark}
The existence of a single generator can be an obstacle to applying Theorem \ref{theorem:generator} in the setup of quasi multifiab bicategories, since such a generator might not exist in a single $\mathbf{M}(\mathtt{i})$. However, we can always pass to the birepresentation $\mathbf{M}^{\oplus}$ of $\cC^{\,\oplus}$, which will have a generator. We can thus always associate a coalgebra in $\cC^{\,\oplus}$ to $\mathbf{M}$.
\end{remark}

\begin{corollary}\label{corollary:MT}
If, in the setup of Theorem \ref{theorem:generator}, $X,Y$ are two generators of $\mathbf{M}$, then $\mathrm{C}^{X}$ and $\mathrm{C}^{Y}$ are
MT equivalent coalgebras in $\underline{\cC}$.
\end{corollary}

%%%%%%%%%%%%%%%%%%%%%%%%%%%%%%%%%%%%%%%%%

\subsection{Framing coalgebras}\label{subsection:framed coalgebras}

%%%%%%%%%%%%%%%%%%%%%%%%%%%%%%%%%%%%%%%%%

Let $\cC$ be a quasi multifiab bicategory.
Recall that, for all $1$-morphisms $\mathrm{F}\in\cC$, the tuple
$(\mathrm{F},\mathrm{F}^{\star})$ forms an adjoint pair in $\cC$.

\begin{lemma}\label{lem0.1}
If $\mathrm{C}\in\underline{\cC}(\mathtt{i},\mathtt{i})$ is a coalgebra in $\underline{\cC}$ such that $0\neq(\mathrm{FC})\mathrm{F}^{\star}$ for some $1$-morphism $\mathrm{F}\in\cC(\mathtt{i},\mathtt{j})$,
then the $1$-morphism $(\mathrm{FC})\mathrm{F}^{\star}\in\underline{\cC}(\mathtt{j},\mathtt{j})$ has a coalgebra structure in $\underline{\cC}$ with comultiplication
\begin{gather*}
\begin{aligned}
\delta_{(\mathrm{FC})\mathrm{F}^{\star}}:=&
\big(\mathrm{id}_{(\mathrm{FC})\mathrm{F}^{\star}}\circ_{\mathsf{h}}\alpha_{\mathrm{F},\mathrm{C},\mathrm{F}^{\star}}^{\mone}\big)
\circ_{\mathsf{v}}
\alpha_{(\mathrm{FC})\mathrm{F}^{\star},\mathrm{F},\mathrm{C}\mathrm{F}^{\star}}
\circ_{\mathsf{v}}
(\alpha_{\mathrm{FC},\mathrm{F}^{\star},\mathrm{F}}^{\mone}\circ_{\mathsf{h}}\mathrm{id}_{\mathrm{CF}^{\star}})\\
&\circ_{\mathsf{v}}
\big((\mathrm{id}_{\mathrm{FC}}\circ_{\mathsf{h}}\mathrm{coev}_{\mathrm{F}})\circ_{\mathsf{h}}\mathrm{id}_{\mathrm{CF}^{\star}}\big)
\circ_{\mathsf{v}}
\big((\runit_{\mathrm{FC}})^{\mone}\circ_{\mathsf{h}}\mathrm{id}_{\mathrm{CF}^{\star}}\big)
\circ_{\mathsf{v}}
\alpha_{\mathrm{FC},\mathrm{C},\mathrm{F}^{\star}}\\
&\circ_{\mathsf{v}}
\big(\alpha_{\mathrm{F},\mathrm{C},\mathrm{C}}^{\mone}\circ_{\mathsf{h}}\mathrm{id}_{\mathrm{F}^{\star}}\big)
\circ_{\mathsf{v}}
\big((\mathrm{id}_{\mathrm{F}}\circ_{\mathsf{h}}\delta_{\mathrm{C}})\circ_{\mathsf{h}}\mathrm{id}_{\mathrm{F}^{\star}}\big)
\end{aligned}
\end{gather*}
and counit
\begin{gather*}
\epsilon_{(\mathrm{FC})\mathrm{F}^{\star}}:=\mathrm{ev}_{\mathrm{F}}\circ_{\mathsf{v}}(\runit_{\mathrm{F}}\circ_{\mathsf{h}}\mathrm{id}_{\mathrm{F}^{\star}})
\circ_{\mathsf{v}}
\big((\mathrm{id}_{\mathrm{F}}\circ_{\mathsf{h}}\epsilon_{\mathrm{C}})\circ_{\mathsf{h}}\mathrm{id}_{\mathrm{F}^{\star}}\big).
\end{gather*}
\end{lemma}

\begin{proof}
Associativity and counitality for $(\mathrm{FC})\mathrm{F}^{\star}$ follow
from those for $\mathrm{C}$, the coherence conditions
for the associator and the unitors of $\underline{\cC}$ and the adjunction conditions for $\mathrm{F}$.
\end{proof}

\begin{remark}
The $1$-morphism $\mathrm{F}(\mathrm{C}\mathrm{F}^{\star})$, if non-zero, acquires a coalgebra structure
via the isomorphism
$\alpha_{\mathrm{F},\mathrm{C},\mathrm{F}^{\star}}\colon(\mathrm{FC})\mathrm{F}^{\star}\xrightarrow{\cong}\mathrm{F}(\mathrm{C}\mathrm{F}^{\star})$.
\end{remark}

\begin{remark}
If $\cC$ is a fiat $2$-category, we can picture the coalgebra structure of
$\mathrm{F}\mathrm{C}\mathrm{F}^{\star}=
(\mathrm{F}\mathrm{C})\mathrm{F}^{\star}=
\mathrm{F}(\mathrm{C}\mathrm{F}^{\star})$
from Lemma \ref{lem0.1} in the form of string diagrams.
Using solid black strands for $\mathrm{C}$ and
dotted blue strands for $\mathrm{F}$ and $\mathrm{F}^{\star}$, we denote
the $2$-morphisms $\delta_{\mathrm{C}}, \epsilon_{\mathrm{C}}, \mathrm{ev}_{\mathrm{F}},
\mathrm{coev}_{\mathrm{F}}$ by
\begin{gather*}
\delta_{\mathrm{C}}
=
\begin{tikzpicture}[anchorbase,scale=.4,smallnodes]
\draw[cstrand] (0,0) node[below,black]{$\mathrm{C}$} to (0,1);
\draw[cstrand] (0,1) to (-1,2) node[above,black,yshift=-2pt]{$\mathrm{C}$};
\draw[cstrand] (0,1) to (1,2) node[above,black,yshift=-2pt]{$\mathrm{C}$};
\end{tikzpicture},
\quad\quad
\epsilon_{\mathrm{C}}
=
\begin{tikzpicture}[anchorbase,scale=.4,smallnodes]
\draw[white] (0,1) to (0,2) node[above,black,yshift=-2pt]{$\mathbbm{1}$};
\draw[cstrand,marked=1] (0,0) node[below,black]{$\mathrm{C}$} to (0,1);
\end{tikzpicture},
\quad\quad
\mathrm{ev}_{\mathrm{F}}=
\begin{tikzpicture}[anchorbase,scale=.4,smallnodes]
\draw[white] (.5,.5) to (.5,1) node[above,black,yshift=-2pt]{$\mathbbm{1}$};
\draw[dstrand] (0,0) node[below,black]{$\mathrm{F}$} to[out=90, in=180] (.5,.5) to[out=0, in=90] (1,0) node[below,black]{$\mathrm{F}^{\star}$};
\draw[dstrand,directed=1] (.6,.5) to (.61,.5);
\end{tikzpicture},
\quad\quad
\mathrm{coev}_{\mathrm{F}}=
\begin{tikzpicture}[anchorbase,scale=.4,smallnodes]
\draw[white] (.5,-.5) to (.5,-1) node[below,black]{$\mathbbm{1}$};
\draw[dstrand] (0,0) node[above,black,yshift=-2pt]{$\mathrm{F}^{\star}$}
to[out=270, in=180] (.5,-.5) to[out=0, in=270] (1,0) node[above,black,yshift=-2pt]{$\mathrm{F}$};
\draw[dstrand,directed=1] (.6,-.5) to (.61,-.5);
\end{tikzpicture}
.
\end{gather*}
In this diagrammatic notation, the comultiplication and counit of $\mathrm{F}\mathrm{C}\mathrm{F}^{\star}$ become
\begin{gather*}
\delta_{\mathrm{FCF}^{\star}}=
\begin{tikzpicture}[anchorbase,scale=.4,smallnodes]
\draw[dstrand] (-.5,2) to[out=270, in=180] (0,1.5) to[out=0, in=270] (.5,2);
\draw[dstrand,directed=1] (-.5,0) to (-.5,1) to (-1.5,2);
\draw[dstrand,directed=1] (1.5,2) to (.5,1) to (.5,0);
\draw[cstrand] (0,0) to (0,1);
\draw[cstrand] (0,1) to (-1,2);
\draw[cstrand] (0,1) to (1,2);
\draw[dstrand,directed=1] (.1,1.5) to (.11,1.5);
\end{tikzpicture}
,\quad\quad
\epsilon_{\mathrm{FCF}^{\star}}=
\begin{tikzpicture}[anchorbase,scale=.4,smallnodes]
\draw[dstrand,directed=1] (.5,2) to[out=90, in=180] (1,3.5) to[out=0, in=90] (1.5,2);
\draw[cstrand,marked=1] (1,2) to (1,3);
\end{tikzpicture}
.
\end{gather*}
This explains our choice of the term \emph{framing}. Using these diagrams, the proof of Lemma \ref{lem0.1} becomes an easy exercise in planar topology and many of the statements below also have natural topological interpretations.

The idea to use duals for the construction of (co)algebras is not new,
see e.g. \cite[Section 3]{Mu} for framings of the identity object in a strict tensor category,
although we do not know of any reference for the general content of Lemma \ref{lem0.1} (either in the framework of $2$-categories or bicategories).
\end{remark}

Note that for any $\mathrm{F}\in\cC(\mathtt{i},\mathtt{j})$  the adjoint pair $(\mathrm{F},\mathrm{F}^{\star})$ gives rise to the natural isomorphism
\begin{gather*}
\mathrm{Hom}_{\underline{\ccC}(\mathtt{k},\mathtt{j})}(\mathrm{F}\mathrm{H},\mathrm{G})
\xrightarrow{\cong} \mathrm{Hom}_{\underline{\ccC}(\mathtt{k},\mathtt{i})}(\mathrm{H},\mathrm{F}^{\star}\mathrm{G}),
\end{gather*}
where $\mathrm{G}\in\underline{\cC}(\mathtt{k},\mathtt{j}),\mathrm{H}\in\underline{\cC}(\mathtt{k},\mathtt{i})$,
given by sending $\beta\in\mathrm{Hom}_{\underline{\ccC}(\mathtt{k},\mathtt{j})}(\mathrm{F}\mathrm{H},\mathrm{G})$ to the element
\begin{gather*}
(\mathrm{id}_{\mathrm{F}^{\star}}\circ_{\mathsf{h}}\beta)\circ_{\mathsf{v}}\alpha_{\mathrm{F}^{\star},\mathrm{F},\mathrm{H}}\circ_{\mathsf{v}}(\mathrm{coev}_{\mathrm{F}}\circ_{\mathsf{h}}\mathrm{id}_{\mathrm{H}})\circ_{\mathsf{v}}(\lunit_{\mathrm{H}})^{\mone}
\in\mathrm{Hom}_{\underline{\ccC}(\mathtt{k},\mathtt{i})}(\mathrm{H},\mathrm{F}^{\star}\mathrm{G})
\end{gather*}
with inverse given by sending $\gamma\in\mathrm{Hom}_{\underline{\ccC}(\mathtt{k},\mathtt{i})}(\mathrm{H},\mathrm{F}^{\star}\mathrm{G})$
to
\begin{gather*}
\lunit_{\mathrm{G}}\circ_{\mathsf{v}}(\mathrm{ev}_{\mathrm{F}}\circ_{\mathsf{h}}\mathrm{id}_{\mathrm{G}})\circ_{\mathsf{v}}\alpha_{\mathrm{F},\mathrm{F}^{\star},\mathrm{G}}^{\mone}\circ_{\mathsf{v}}(\mathrm{id}_{\mathrm{F}}\circ_{\mathsf{h}}\gamma)
\in\mathrm{Hom}_{\underline{\ccC}(\mathtt{k},\mathtt{j})}(\mathrm{F}\mathrm{H},\mathrm{G}).
\end{gather*}
We also have the natural isomorphism
\begin{gather*}
\mathrm{Hom}_{\underline{\ccC}(\mathtt{j},\mathtt{k})}(\mathrm{H}\mathrm{F}^{\star},\mathrm{G})\xrightarrow{\cong} \mathrm{Hom}_{\underline{\ccC}(\mathtt{i},\mathtt{k})}(\mathrm{H},\mathrm{G}\mathrm{F}),
\end{gather*}
where $\mathrm{G}\in\underline{\cC}(\mathtt{j},\mathtt{k}),\mathrm{H}\in\underline{\cC}(\mathtt{i},\mathtt{k})$,
given by sending $\beta\in\mathrm{Hom}_{\underline{\ccC}(\mathtt{j},\mathtt{k})}(\mathrm{H}\mathrm{F}^{\star},\mathrm{G})$ to the element
\begin{gather*}
(\beta\circ_{\mathsf{h}}\mathrm{id}_{\mathrm{F}})\circ_{\mathsf{v}}\alpha_{\mathrm{H},\mathrm{F}^{\star},\mathrm{F}}^{\mone}\circ_{\mathsf{v}}(\mathrm{id}_{\mathrm{H}}\circ_{\mathsf{h}}\mathrm{coev}_{\mathrm{F}})\circ_{\mathsf{v}}(\runit_{\mathrm{H}})^{\mone}
\in\mathrm{Hom}_{\underline{\ccC}(\mathtt{i},\mathtt{k})}(\mathrm{H},\mathrm{G}\mathrm{F})
\end{gather*}
with inverse given by sending $\gamma\in\mathrm{Hom}_{\underline{\ccC}(\mathtt{i},\mathtt{k})}(\mathrm{H},\mathrm{G}\mathrm{F})$
to
\begin{gather*}
\runit_{\mathrm{G}}\circ_{\mathsf{v}}(\mathrm{id}_{\mathrm{G}}\circ_{\mathsf{h}}\mathrm{ev}_{\mathrm{F}})\circ_{\mathsf{v}}\alpha_{\mathrm{G},\mathrm{F},\mathrm{F}^{\star}}\circ_{\mathsf{v}}
(\gamma\circ_{\mathsf{h}}\mathrm{id}_{\mathrm{F}^{\star}})
\in\mathrm{Hom}_{\underline{\ccC}(\mathtt{j},\mathtt{k})}(\mathrm{H}\mathrm{F}^{\star},\mathrm{G}).
\end{gather*}

\begin{theorem}\label{prop0.4}
Suppose that, additionally to the hypotheses of Theorem \ref{theorem:generator},
$\mathrm{F}$ is a $1$-morphism in $\cC(\mathtt{i},\mathtt{j})$ such that $\mathbf{M}_{\mathtt{ji}}(\mathrm{F})\,X$ generates $\mathbf{M}$. Then the $1$-morphism $(\mathrm{FC}^{X})\mathrm{F}^{\star}\in\underline{\cC}$ with coalgebra structure defined in Lemma \ref{lem0.1}
is, up to isomorphism, the coalgebra associated with the object $\mathbf{M}_{\mathtt{ji}}(\mathrm{F})\,X$.
\end{theorem}

\begin{proof}
By adjunction and the natural isomorphism $\gamma_{X,X}$, we have natural isomorphisms
\begin{gather}\label{eq:0.4}
\begin{aligned}
\mathrm{Hom}_{\underline{\mathbf{M}}(\mathtt{j})}\big(\underline{\mathbf{M}}_{\mathtt{ji}}(\mathrm{F})\,X,
\underline{\mathbf{M}}_{\mathtt{jj}}(\mathrm{G})\underline{\mathbf{M}}_{\mathtt{ji}}(\mathrm{F})\,X\big)
&\cong
\mathrm{Hom}_{\underline{\mathbf{M}}(\mathtt{j})}\big(\underline{\mathbf{M}}_{\mathtt{ji}}(\mathrm{F})\,X,\underline{\mathbf{M}}_{\mathtt{ji}}(\mathrm{GF})\,X\big)
\\
&\cong
\mathrm{Hom}_{\underline{\mathbf{M}}(\mathtt{i})}\big(X,\underline{\mathbf{M}}_{\mathtt{ij}}(\mathrm{F}^{\star})\underline{\mathbf{M}}_{\mathtt{ji}}(\mathrm{GF})\,X\big)
\\
&\cong
\mathrm{Hom}_{\underline{\mathbf{M}}(\mathtt{i})}\big(X,\underline{\mathbf{M}}_{\mathtt{ii}}\big(\mathrm{F}^{\star}(\mathrm{GF})\big)\,X\big)
\\
&\cong
\mathrm{Hom}_{\underline{\ccC}(\mathtt{i},\mathtt{i})}\big(\mathrm{C}^{X},\mathrm{F}^{\star}(\mathrm{GF})\big)
\\
&\cong
\mathrm{Hom}_{\underline{\ccC}(\mathtt{i},\mathtt{j})}(\mathrm{F}\mathrm{C}^{X},\mathrm{GF})
\\
&\cong
\mathrm{Hom}_{\underline{\ccC}(\mathtt{j},\mathtt{j})}\big((\mathrm{FC}^{X})\mathrm{F}^{\star},\mathrm{G}\big),
\end{aligned}
\end{gather}
for all $1$-morphisms $\mathrm{G}\in\underline{\cC}(\mathtt{j},\mathtt{j})$. Below we will
give and use these isomorphisms explicitly, e.g. the first and third isomorphisms are given by $(\mu_{\mathtt{j}\mathtt{j}\mathtt{i}}^{\mathrm{G},\mathrm{F}})_X\circ_{\mathsf{v}}{}_{-}$
and $(\mu_{\mathtt{i}\mathtt{j}\mathtt{i}}^{\mathrm{F}^{\star},\mathrm{GF}})_X\circ_{\mathsf{v}}{}_{-}$, respectively.

Considering $\mathrm{G}=\mathbbm{1}_{\mathtt{j}}$, we now prove that
$\epsilon_{(\mathrm{FC}^{X})\mathrm{F}^{\star}}$ (recall the notation $\epsilon_{X}:=\epsilon_{\mathrm{C}^{X}}$) is the image of $(\iota_{\mathtt{j}}^{\mone})_{\underline{\mathbf{M}}_{\mathtt{ji}}(\mathrm{F}) X}$ under the isomorphisms in \eqref{eq:0.4}.
It suffices to show that the image of $\epsilon_{(\mathrm{FC}^{X})\mathrm{F}^{\star}}$ and the image of $(\iota_{\mathtt{j}}^{\mone})_{\underline{\mathbf{M}}_{\mathtt{ji}}(\mathrm{F})\,X}$
under those isomorphisms coincide in any of the morphism spaces appearing in \eqref{eq:0.4}.
On one hand, the first isomorphism
\begin{gather*}
\mathrm{Hom}_{\underline{\mathbf{M}}(\mathtt{j})}
\big(\underline{\mathbf{M}}_{\mathtt{ji}}(\mathrm{F})\,X,
\underline{\mathbf{M}}_{\mathtt{jj}}(\mathbbm{1}_{\mathtt{j}})
\underline{\mathbf{M}}_{\mathtt{ji}}(\mathrm{F})\,X\big)
\xrightarrow{\cong}
\mathrm{Hom}_{\underline{\mathbf{M}}(\mathtt{j})}
\big(\underline{\mathbf{M}}_{\mathtt{ji}}(\mathrm{F})\,X,
\underline{\mathbf{M}}_{\mathtt{ji}}(\mathbbm{1}_{\mathtt{j}}\mathrm{F})\,X\big)
\end{gather*}
sends $(\iota_{\mathtt{j}}^{\mone})_{\underline{\mathbf{M}}_{\mathtt{ji}}(\mathrm{F})\,X}$ to $\underline{\mathbf{M}}_{\mathtt{ji}}\big((\lunit_{\mathrm{F}})^{\mone}\big)_X$,
by \eqref{eq:birepresentation2}. The second isomorphism
\begin{gather*}
\mathrm{Hom}_{\underline{\mathbf{M}}(\mathtt{j})}
\big(\underline{\mathbf{M}}_{\mathtt{ji}}(\mathrm{F})\,X,
\underline{\mathbf{M}}_{\mathtt{ji}}(\mathbbm{1}_{\mathtt{j}}\mathrm{F})\,X\big)
\xrightarrow{\cong}
\mathrm{Hom}_{\underline{\mathbf{M}}(\mathtt{i})}\big(X,\underline{\mathbf{M}}_{\mathtt{ij}}(\mathrm{F}^{\star})
\underline{\mathbf{M}}_{\mathtt{ji}}(\mathbbm{1}_{\mathtt{j}}\mathrm{F})\,X\big)
\end{gather*}
sends $\underline{\mathbf{M}}_{\mathtt{ji}}\big((\lunit_{\mathrm{F}})^{\mone}\big)_X$ to
\begin{gather*}
\begin{aligned}
&\Big(\mathrm{id}_{\underline{\mathbf{M}}_{\mathtt{ij}}(\mathrm{F}^{\star})}\circ_{\mathsf{h}}
\underline{\mathbf{M}}_{\mathtt{ji}}\big((\lunit_{\mathrm{F}})^{\mone}\big)_X\Big)\circ_{\mathsf{v}}
(\mu_{\mathtt{i}\mathtt{j}\mathtt{i}}^{\mathrm{F}^{\star},\mathrm{F}})_X^{\mone}\circ_{\mathsf{v}}
\underline{\mathbf{M}}_{\mathtt{ii}}(\mathrm{coev}_{\mathrm{F}})_{X}\circ_{\mathsf{v}}
(\iota_{\mathtt{i}}^{\mone})_{X}\\
&=(\mu_{\mathtt{i}\mathtt{j}\mathtt{i}}^{\mathrm{F}^{\star},\mathbbm{1}_{\mathtt{j}}\mathrm{F}})_X^{\mone}\circ_{\mathsf{v}}
\underline{\mathbf{M}}_{\mathtt{ii}}\big(\mathrm{id}_{\mathrm{F}^{\star}}\circ_{\mathsf{h}}(\lunit_{\mathrm{F}})^{\mone}\big)_X\circ_{\mathsf{v}}
\underline{\mathbf{M}}_{\mathtt{ii}}(\mathrm{coev}_{\mathrm{F}})_{X}\circ_{\mathsf{v}}
(\iota_{\mathtt{i}}^{\mone})_{X},
\end{aligned}
\end{gather*}
where the equality holds by the naturality of $\mu_{\mathtt{i}\mathtt{j}\mathtt{i}}$. Further, the latter is sent to $\underline{\mathbf{M}}_{\mathtt{ii}}\big(\mathrm{id}_{\mathrm{F}^{\star}}\circ_{\mathsf{h}}(\lunit_{\mathrm{F}})^{\mone}\big)_X\circ_{\mathsf{v}}
\underline{\mathbf{M}}_{\mathtt{ii}}(\mathrm{coev}_{\mathrm{F}})_{X}\circ_{\mathsf{v}}
(\iota_{\mathtt{i}}^{\mone})_{X}$ under the third isomorphism in \eqref{eq:0.4} for $\mathrm{G}=\mathbbm{1}_{\mathtt{j}}$.
On the other hand, consider the commutative diagram
\begin{gather}\label{diag:01}
\adjustbox{scale=.79,center}{%
\begin{tikzcd}[ampersand replacement=\&, column sep=3.5em, row sep=2.5em]
\mathrm{F}\mathrm{C}^{X}
\ar[rr,"(\runit_{\mathrm{FC}^{X}})^{\mone}"]
\ar[rr, phantom, yshift=-7ex, "\circled{1}"]
\ar[d, "{\mathrm{id}_{\mathrm{F}}\circ_{\mathsf{h}}\epsilon_{X}}"description]
\&\& (\mathrm{F}\mathrm{C}^{X})\mathbbm{1}_{\mathtt{i}}
\ar[rr,"\mathrm{id}_{\mathrm{FC}^{X}}\circ_{\mathsf{h}}\mathrm{coev}_{\mathrm{F}}"]
\ar[rr, phantom, yshift=-7ex, "\circled{2}"]
\ar[d, "{(\mathrm{id}_{\mathrm{F}}\circ_{\mathsf{h}}\epsilon_{X})\circ_{\mathsf{h}}\mathrm{id}_{\mathbbm{1}_{\mathtt{i}}}}"description]
\&\&(\mathrm{F}\mathrm{C}^{X})(\mathrm{F}^{\star}\mathrm{F})
\ar[rr,"\alpha_{\mathrm{F}\mathrm{C}^{X},\mathrm{F}^{\star},\mathrm{F}}^{\mone}"]
\ar[rr, phantom, yshift=-7ex, "\circled{3}"]
\ar[d, "{(\mathrm{id}_{\mathrm{F}}\circ_{\mathsf{h}}\epsilon_{X})\circ_{\mathsf{h}}\mathrm{id}_{\mathrm{F}^{\star}\mathrm{F}}}"description]
\&\&\big((\mathrm{F}\mathrm{C}^{X})\mathrm{F}^{\star}\big)\mathrm{F}
\ar[d, "{\big((\mathrm{id}_{\mathrm{F}}\circ_{\mathsf{h}}\epsilon_{X})
\circ_{\mathsf{h}}\mathrm{id}_{\mathrm{F}^{\star}}\big)\circ_{\mathsf{h}}\mathrm{id}_{\mathrm{F}}}"description]
\\[5ex]
\mathrm{F}\mathbbm{1}_{\mathtt{i}}
\ar[rr,"(\runit_{\mathrm{F}\mathbbm{1}_{\mathtt{i}}})^{\mone}"]
\ar[rr, phantom,xshift=7ex, yshift=-10ex, "\circled{5}"]
\ar[rr, phantom,xshift=-3ex, yshift=-4ex, "\circled{4}"]
\ar[d, "{\mathrm{id}_{\mathrm{F}}\circ_{\mathsf{h}}(\runit_{\mathbbm{1}_{\mathtt{i}}})^{\mone}}"description]
\&\& (\mathrm{F}\mathbbm{1}_{\mathtt{i}})\mathbbm{1}_{\mathtt{i}}
\ar[dll,"{\alpha_{\mathrm{F},\mathbbm{1}_{\mathtt{i}},\mathbbm{1}_{\mathtt{i}}}}"description]
\ar[rr,"\mathrm{id}_{\mathrm{F}\mathbbm{1}_{\mathtt{i}}}\circ_{\mathsf{h}}\mathrm{coev}_{\mathrm{F}}"]
\ar[rr, phantom, yshift=-7ex, "\circled{6}"]
\ar[d, "{\runit_{\mathrm{F}}\circ_{\mathsf{h}}\mathrm{id}_{\mathbbm{1}_{\mathtt{i}}}}"description]
\&\&
(\mathrm{F}\mathbbm{1}_{\mathtt{i}})(\mathrm{F}^{\star}\mathrm{F})
\ar[rr,"\alpha_{\mathrm{F}\mathbbm{1}_{\mathtt{i}},\mathrm{F}^{\star},\mathrm{F}}^{\mone}"]
\ar[rr, phantom, yshift=-7ex, "\circled{7}"]
\ar[d, "{\runit_{\mathrm{F}}\circ_{\mathsf{h}}\mathrm{id}_{\mathrm{F}^{\star}\mathrm{F}}}"description]
\&\& \big((\mathrm{F}\mathbbm{1}_{\mathtt{i}})\mathrm{F}^{\star}\big)\mathrm{F}
\ar[d, "{(\runit_{\mathrm{F}}\circ_{\mathsf{h}}\mathrm{id}_{\mathrm{F}^{\star}})\circ_{\mathsf{h}}\mathrm{id}_{\mathrm{F}}}"description]
\\[5ex]
\mathrm{F}(\mathbbm{1}_{\mathtt{i}}\mathbbm{1}_{\mathtt{i}})
\ar[rr,"\mathrm{id}_{\mathrm{F}}\circ_{\mathsf{h}}\lunit_{\mathbbm{1}_{\mathtt{i}}}",swap]
\&\&\mathrm{F}\mathbbm{1}_{\mathtt{i}}
\ar[rr,"\mathrm{id}_{\mathrm{F}}\circ_{\mathsf{h}}\mathrm{coev}_{\mathrm{F}}",swap]
\&\&\mathrm{F}(\mathrm{F}^{\star}\mathrm{F})
\ar[rr,"\alpha_{\mathrm{F},\mathrm{F}^{\star},\mathrm{F}}^{\mone}",swap]
\&\&(\mathrm{F}\mathrm{F}^{\star})\mathrm{F}
\ar[d,"\mathrm{ev}_{\mathrm{F}}\circ_{\mathsf{h}}\mathrm{id}_{\mathrm{F}}",swap]
\\[5ex]
\&\&\&\&\&\&\mathbbm{1}_{\mathtt{j}}\mathrm{F}
\end{tikzcd}
}
\end{gather}
Its commutativity follows from
\begin{itemize}

\item naturality of $(\runit)^{\mone}$ for the facet labeled $1$ (noting that $(\runit)^{\mone}_{\mathrm{H}}=(\runit_{\mathrm{H}})^{\mone}$ for any $1$-morphism $\mathrm{H}\in\cC$);

\item the interchange law for the facets labeled $2$ and $6$;

\item naturality of $\alpha^{\mone}$ for the facets labeled $3$ and $7$;

\item the right diagram in \eqref{eq:0.00} for the facet labeled $4$;

\item the triangle coherence condition of the unitors for the facet labeled $5$.
\end{itemize}

The inverse
$\mathrm{Hom}_{\underline{\ccC}}\big((\mathrm{FC}^{X})\mathrm{F}^{\star},\mathbbm{1}_{\mathtt{j}}\big)\xrightarrow{\cong}
\mathrm{Hom}_{\underline{\ccC}}\big(\mathrm{FC}^{X},\mathbbm{1}_{\mathtt{j}}\mathrm{F}\big)$ of the last isomorphism in \eqref{eq:0.4}
sends $\epsilon_{(\mathrm{FC}^{X})\mathrm{F}^{\star}}$ to
the composite of the paths going right and then down to the bottom along the boundary of the diagram in \eqref{diag:01}. This composite,
due to the commutativity of the diagram, is the same as the composite of the path going first down and then right, and down as the final step. Since $\lunit_{\mathbbm{1}_{\mathtt{i}}}=\runit_{\mathbbm{1}_{\mathtt{i}}}$, the latter equals
\begin{gather}\label{eq:0.5}
(\mathrm{ev}_{\mathrm{F}}\circ_{\mathsf{h}}\mathrm{id}_{\mathrm{F}})\circ_{\mathsf{v}}\alpha_{\mathrm{F},\mathrm{F}^{\star},\mathrm{F}}^{\mone}
\circ_{\mathsf{v}}(\mathrm{id}_{\mathrm{F}}\circ_{\mathsf{h}}\mathrm{coev}_{\mathrm{F}})
\circ_{\mathsf{v}}(\mathrm{id}_{\mathrm{F}}\circ_{\mathsf{h}}\epsilon_{X}).
\end{gather}

Consider now the diagram
\begin{gather}\label{diag:02}
\adjustbox{scale=.79,center}{%
\begin{tikzcd}[ampersand replacement=\&, column sep=3.5em, row sep=3em]
\mathrm{C}^{X}
\ar[rr,"(\lunit_{\mathrm{C}^{X}})^{\mone}"]
\ar[rr, phantom, yshift=-7ex, "\circled{1}"]
\ar[d, "\epsilon_{X}", swap]
\&\&\mathbbm{1}_{\mathtt{i}} \mathrm{C}^{X}
\ar[rr,"\mathrm{coev}_{\mathrm{F}}\circ_{\mathsf{h}}\mathrm{id}_{\mathrm{C}^{X}}"]
\ar[rr, phantom, yshift=-7ex, "\circled{2}"]
\ar[d, "{\mathrm{id}_{\mathbbm{1}_{\mathtt{i}}}\circ_{\mathsf{h}}\epsilon_{X}}"description]
\&\&(\mathrm{F}^{\star}\mathrm{F})\mathrm{C}^{X}
\ar[rr,"\alpha_{\mathrm{F}^{\star},\mathrm{F},\mathrm{C}^{X}}"]
\ar[rr, phantom, yshift=-7ex, "\circled{3}"]
\ar[d, "{\mathrm{id}_{\mathrm{F}^{\star}\mathrm{F}}\circ_{\mathsf{h}}\epsilon_{X}}"description]
\&\&\mathrm{F}^{\star}(\mathrm{F}\mathrm{C}^{X})
\ar[d, "{\mathrm{id}_{\mathrm{F}^{\star}}\circ_{\mathsf{h}}(\mathrm{id}_{\mathrm{F}}\circ_{\mathsf{h}}\epsilon_{X})}"description]
\\[5ex]
\mathbbm{1}_{\mathtt{i}}
\ar[rr,"(\lunit_{\mathbbm{1}_{\mathtt{i}}})^{\mone}"]
\ar[rr, phantom,yshift=-7ex, "\circled{4}"]
\ar[d, "\mathrm{coev}_{\mathrm{F}}",swap]
\&\& \mathbbm{1}_{\mathtt{i}}\mathbbm{1}_{\mathtt{i}}
\ar[rr,"\mathrm{coev}_{\mathrm{F}}\circ_{\mathsf{h}}\mathrm{id}_{\mathbbm{1}_{\mathtt{i}}}"]
\ar[rr, phantom, yshift=-7ex, "\circled{5}"]
\ar[d, "{\mathrm{id}_{\mathbbm{1}_{\mathtt{i}}}\circ_{\mathsf{h}}\mathrm{coev}_{\mathrm{F}}}"description]
\&\&
(\mathrm{F}^{\star}\mathrm{F})\mathbbm{1}_{\mathtt{i}}
\ar[rr,"\alpha_{\mathrm{F}^{\star},\mathrm{F},\mathbbm{1}_{\mathtt{i}}}"]
\ar[rr, phantom, yshift=-7ex, "\circled{6}"]
\ar[d, "{\mathrm{id}_{\mathrm{F}^{\star}\mathrm{F}}\circ_{\mathsf{h}}\mathrm{coev}_{\mathrm{F}}}"description]
\&\& \mathrm{F}^{\star}(\mathrm{F}\mathbbm{1}_{\mathtt{i}})
\ar[d, "{\mathrm{id}_{\mathrm{F}^{\star}}\circ_{\mathsf{h}}(\mathrm{id}_{\mathrm{F}}\circ_{\mathsf{h}}\mathrm{coev}_{\mathrm{F}})}"description]
\\[5ex]
\mathrm{F}^{\star}\mathrm{F}
\ar[ddd,equal]
\ar[rr,"(\lunit_{\mathrm{F}^{\star}\mathrm{F}})^{\mone}"]
\ar[drr, "{(\lunit_{\mathrm{F}^{\star}})^{\mone}\circ_{\mathsf{h}}\mathrm{id}_{\mathrm{F}}}"description]
\ar[rr, phantom, yshift=-7ex, xshift=8ex, "\circled{7}"]
\&\&\mathbbm{1}_{\mathtt{i}}(\mathrm{F}^{\star}\mathrm{F})
\ar[rr,"\mathrm{coev}_{\mathrm{F}}\circ_{\mathsf{h}}\mathrm{id}_{\mathrm{F}^{\star}\mathrm{F}}"]
\ar[d,"\alpha_{\mathbbm{1}_{\mathtt{i}},\mathrm{F}^{\star},\mathrm{F}}^{\mone}"]
\ar[rr, phantom, yshift=-7ex, "\circled{8}"]
\&\&(\mathrm{F}^{\star}\mathrm{F})(\mathrm{F}^{\star}\mathrm{F})
\ar[rr, "\alpha_{\mathrm{F}^{\star},\mathrm{F},\mathrm{F}^{\star}\mathrm{F}}"]
\ar[d,"{\alpha_{\mathrm{F}^{\star}\mathrm{F},\mathrm{F}^{\star},\mathrm{F}}^{\mone}}"description]
\ar[rr, phantom, yshift=-16ex, "\circled{9}"]
\&\&\mathrm{F}^{\star}\big(\mathrm{F}(\mathrm{F}^{\star}\mathrm{F})\big)
\ar[dd, "{\mathrm{id}_{\mathrm{F}^{\star}}\circ_{\mathsf{h}}\alpha_{\mathrm{F},\mathrm{F}^{\star},\mathrm{F}}^{\mone}}"description]
\\[5ex]
\&\&(\mathbbm{1}_{\mathtt{i}}\mathrm{F}^{\star})\mathrm{F}
\ar[rr,"(\mathrm{coev}_{\mathrm{F}}\circ_{\mathsf{h}}\mathrm{id}_{\mathrm{F}^{\star}})\circ_{\mathsf{h}}\mathrm{id}_{\mathrm{F}}",swap]
\ar[rr, phantom, yshift=-16ex, xshift=-12ex, "\circled{10}"]
\&\&\big((\mathrm{F}^{\star}\mathrm{F})\mathrm{F}^{\star}\big)\mathrm{F}
\ar[d, "{\alpha_{\mathrm{F}^{\star},\mathrm{F},\mathrm{F}^{\star}}\circ_{\mathsf{h}}\mathrm{id}_{\mathrm{F}}}"description]
\&\&
\\[5ex]
\&\&
\&\&\big(\mathrm{F}^{\star}(\mathrm{F}\mathrm{F}^{\star})\big)\mathrm{F}
\ar[rr, "\alpha_{\mathrm{F}^{\star},\mathrm{F}\mathrm{F}^{\star},\mathrm{F}}"]
\ar[d, "{(\mathrm{id}_{\mathrm{F}^{\star}}\circ_{\mathsf{h}}\mathrm{ev}_{\mathrm{F}})\circ_{\mathsf{h}}\mathrm{id}_{\mathrm{F}}}"description]
\ar[rr,phantom,yshift=-7ex,"\circled{11}"]
\&\&\mathrm{F}^{\star}\big((\mathrm{F}\mathrm{F}^{\star})\mathrm{F}\big)
\ar[d, "{\mathrm{id}_{\mathrm{F}^{\star}}\circ_{\mathsf{h}}(\mathrm{ev}_{\mathrm{F}}\circ_{\mathsf{h}}\mathrm{id}_{\mathrm{F}})}"description]
\\[5ex]
\mathrm{F}^{\star}\mathrm{F}
\ar[rrrr,"(\runit_{\mathrm{F}^{\star}})^{\mone}\circ_{\mathsf{h}}\mathrm{id}_{\mathrm{F}}"]
\ar[rrrrrr, bend right,"\mathrm{id}_{\mathrm{F}^{\star}}\circ_{\mathsf{h}}(\lunit_{\mathrm{F}})^{\mone}",swap]
\ar[rrrrrr, phantom, yshift=-7ex,"\circled{12}"]
\&\&
\&\&(\mathrm{F}^{\star}\mathbbm{1}_{\mathtt{j}})\mathrm{F}
\ar[rr, "\alpha_{\mathrm{F}^{\star},\mathbbm{1}_{\mathtt{j}},\mathrm{F}}"]
\&\&\mathrm{F}^{\star}(\mathbbm{1}_{\mathtt{j}}\mathrm{F})
\end{tikzcd}
}
\end{gather}
which commutes due to
\begin{itemize}

\item naturality of $(\lunit_{})^{\mone}$ for the facets labeled $1$ and $4$;

\item the interchange law for the facets labeled $2$ and $5$;

\item naturality of $\alpha$ and $\alpha^{\mone}$ for the facets labeled $3$, $6$, $8$ and $11$, respectively;

\item the left diagram in \eqref{eq:0.00} for the facet labeled $7$;

\item the pentagon coherence condition of the associator for the facet labeled $9$;

\item the adjunction condition of the adjoint pair $(\mathrm{F},\mathrm{F}^{\star})$ for the facet labeled $10$;
\item the triangle coherence condition of the unitors for the facet labeled $12$.
\end{itemize}
Then the inverse $\mathrm{Hom}_{\underline{\ccC}}(\mathrm{F}\mathrm{C}^{X},\mathbbm{1}_{\mathtt{j}}\mathrm{F})
\xrightarrow{\cong}\mathrm{Hom}_{\underline{\ccC}}\big(\mathrm{C}^{X},\mathrm{F}^{\star}(\mathbbm{1}_{\mathtt{j}}\mathrm{F})\big)$
of the fifth isomorphism in \eqref{eq:0.4} sends \eqref{eq:0.5} to
the composite of the paths going right and then down to the bottom along the boundary of the diagram in \eqref{diag:02}. As before, this is
the same as the composite of the paths going down and then right,
i.e.,
\begin{gather}\label{eq:0.6}
\big(\mathrm{id}_{\mathrm{F}^{\star}}\circ_{\mathsf{h}}(\lunit_{\mathrm{F}})^{\mone}\big)\circ_{\mathsf{v}}\mathrm{coev}_{\mathrm{F}}\circ_{\mathsf{v}}\epsilon_{X}.
\end{gather}
The inverse  $\mathrm{Hom}_{\underline{\ccC}}(\mathrm{C}^{X},\mathrm{F}^{\star}(\mathbbm{1}_{\mathtt{j}}\mathrm{F}))
\xrightarrow{\cong}\mathrm{Hom}_{\underline{\mathbf{M}}(\mathtt{i})}(X,\underline{\mathbf{M}}_{\mathtt{ii}}\big(\mathrm{F}^{\star}(\mathbbm{1}_{\mathtt{j}}\mathrm{F})\big)\,X)$ of the fourth isomorphism in \eqref{eq:0.4}
sends \eqref{eq:0.6} to
\begin{gather*}
\underline{\mathbf{M}}_{\mathtt{i}\mathtt{i}}(\mathrm{id}_{\mathrm{F}^{\star}}\circ_{\mathsf{h}}(\lunit_{\mathrm{F}})^{\mone})_{X}\circ_{\mathsf{v}}
\underline{\mathbf{M}}_{\mathtt{i}\mathtt{i}}(\mathrm{coev}_{\mathrm{F}})_{X}\circ_{\mathsf{v}}
\underline{\mathbf{M}}_{\mathtt{i}\mathtt{i}}(\epsilon_{X})_{X}
\circ_{\mathsf{v}}\mathrm{coev}_{X,X},
\end{gather*}
which equals $\underline{\mathbf{M}}_{\mathtt{ii}}\big(\mathrm{id}_{\mathrm{F}^{\star}}\circ_{\mathsf{h}}(\lunit_{\mathrm{F}})^{\mone}\big)_X\circ_{\mathsf{v}}
\underline{\mathbf{M}}_{\mathtt{ii}}(\mathrm{coev}_{\mathrm{F}})_{X}\circ_{\mathsf{v}}
(\iota_{\mathtt{i}}^{\mone})_{X}$, by \eqref{eq:right-counit-gen-intern-hom2}.
Hence, we obtain that $\epsilon_{(\mathrm{FC}^{X})\mathrm{F}^{\star}}\in\mathrm{Hom}_{\underline{\ccC}}\big((\mathrm{FC}^{X})\mathrm{F}^{\star},\mathbbm{1}_{\mathtt{j}}\big)$
is equal to the image of $(\iota_{\mathtt{j}}^{\mone})_{\underline{\mathbf{M}}_{\mathtt{ji}}(\mathrm{F}) X}$ under the isomorphisms in \eqref{eq:0.4}.

Next, taking $\mathrm{G}=(\mathrm{FC}^{X})\mathrm{F}^{\star}$
in \eqref{eq:0.4}, we have to determine the image of $\mathrm{id}_{(\mathrm{FC}^{X})\mathrm{F}^{\star}}$ in $\mathrm{Hom}_{\underline{\mathbf{M}}(\mathtt{j})}(\underline{\mathbf{M}}_{\mathtt{ji}}(\mathrm{F})\,X,
\underline{\mathbf{M}}_{\mathtt{jj}}\big((\mathrm{FC}^{X})\mathrm{F}^{\star}\big)\underline{\mathbf{M}}_{\mathtt{ji}}(\mathrm{F}) X)$
under the (inverses of the) isomorphisms in \eqref{eq:0.4}, which is equal to
$\mathrm{coev}_{\underline{\mathbf{M}}_{\mathtt{ji}}(\mathrm{F}) X,\underline{\mathbf{M}}_{\mathtt{ji}}(\mathrm{F}) X}$.
In detail, the inverse of the last isomorphism sends $\mathrm{id}_{(\mathrm{FC}^{X})\mathrm{F}^{\star}}$ to
\begin{gather}\label{eq:0.7}
\alpha_{\mathrm{F}\mathrm{C}^{X},\mathrm{F}^{\star},\mathrm{F}}^{\mone}\circ_{\mathsf{v}}
(\mathrm{id}_{\mathrm{F}\mathrm{C}^{X}}\circ_{\mathsf{h}}\mathrm{coev}_{\mathrm{F}})\circ_{\mathsf{v}}(\runit_{\mathrm{F}\mathrm{C}^{X}})^{\mone}
\end{gather}
in $\mathrm{Hom}_{\underline{\ccC}}\Big(\mathrm{F}\mathrm{C}^{X},\big((\mathrm{FC}^{X})\mathrm{F}^{\star}\big)\mathrm{F}\Big)$.

Further, the inverse
$\mathrm{Hom}_{\underline{\ccC}}\big(\mathrm{F}\mathrm{C}^{X},\big((\mathrm{FC}^{X})\mathrm{F}^{\star}\big)\mathrm{F}\big)
\xrightarrow{\cong}
\mathrm{Hom}_{\underline{\ccC}}\Big(\mathrm{C}^{X},\mathrm{F}^{\star}\big(\big((\mathrm{FC}^{X})\mathrm{F}^{\star}\big)\mathrm{F}\big)\Big)$
of the fifth isomorphism
sends \eqref{eq:0.7} to
\begin{gather*}
\begin{aligned}
&(\mathrm{id}_{\mathrm{F}^{\star}}\circ_{\mathsf{h}}\alpha_{\mathrm{F}\mathrm{C}^{X},\mathrm{F}^{\star},\mathrm{F}}^{\mone})\circ_{\mathsf{v}}
\big(\mathrm{id}_{\mathrm{F}^{\star}}\circ_{\mathsf{h}}(\mathrm{id}_{\mathrm{F}\mathrm{C}^{X}}\circ_{\mathsf{h}}\mathrm{coev}_{\mathrm{F}})\big)
\circ_{\mathsf{v}}\big(\mathrm{id}_{\mathrm{F}^{\star}}\circ_{\mathsf{h}}(\runit_{\mathrm{F}\mathrm{C}^{X}})^{\mone}\big)
\\
&\circ_{\mathsf{v}}\alpha_{\mathrm{F}^{\star},\mathrm{F},\mathrm{C}^{X}}\circ_{\mathsf{v}}(\mathrm{coev}_{\mathrm{F}}\circ_{\mathsf{h}}\mathrm{id}_{\mathrm{C}^{X}})\circ_{\mathsf{v}}
(\lunit_{\mathrm{C}^{X}})^{\mone}
.
\end{aligned}
\end{gather*}
This $2$-morphism equals
\begin{gather}\label{eq:0.8}
\begin{aligned}
&(\mathrm{id}_{\mathrm{F}^{\star}}\circ_{\mathsf{h}}\alpha_{\mathrm{F}\mathrm{C}^{X},\mathrm{F}^{\star},\mathrm{F}}^{\mone})\circ_{\mathsf{v}}
\big(\mathrm{id}_{\mathrm{F}^{\star}}\circ_{\mathsf{h}}(\mathrm{id}_{\mathrm{F}\mathrm{C}^{X}}\circ_{\mathsf{h}}\mathrm{coev}_{\mathrm{F}})\big)
\circ_{\mathsf{v}}(\mathrm{id}_{\mathrm{F}^{\star}}\circ_{\mathsf{h}}\alpha_{\mathrm{F},\mathrm{C}^{X},\mathbbm{1}_{\mathtt{i}}}^{\mone})\\
&\circ_{\mathsf{v}}\alpha_{\mathrm{F}^{\star},\mathrm{F},\mathrm{C}^{X}\mathbbm{1}_{\mathtt{i}}}\circ_{\mathsf{v}}
(\mathrm{coev}_{\mathrm{F}}\circ_{\mathsf{h}}\mathrm{id}_{\mathrm{C}^{X}\mathbbm{1}_{\mathtt{i}}})\circ_{\mathsf{v}}
(\lunit_{\mathrm{C}^{X}\mathbbm{1}_{\mathtt{i}}})^{\mone}\circ_{\mathsf{v}}
(\runit_{\mathrm{C}^{X}})^{\mone}
\end{aligned}
\end{gather}
due to commutativity of
\begin{gather*}
\scalebox{0.7}{
$\begin{tikzcd}[ampersand replacement=\&, column sep=3.5em, row sep=3em]
\mathrm{C}^{X}
\ar[rr,"(\lunit_{\mathrm{C}^{X}})^{\mone}"]
\ar[rr, phantom, yshift=-7ex, "\circled{1}"]
\ar[d, "{(\runit_{\mathrm{C}^{X}})^{\mone}}"description]
\&\&\mathbbm{1}_{\mathtt{i}}\mathrm{C}^{X}
\ar[rr,"\mathrm{coev}_{\mathrm{F}}\circ_{\mathsf{h}}\mathrm{id}_{\mathrm{C}^{X}}"]
\ar[rr, phantom, yshift=-7ex, "\circled{2}"]
\ar[d, "{\mathrm{id}_{\mathbbm{1}_{\mathtt{i}}}\circ_{\mathsf{h}}(\runit_{\mathrm{C}^{X}})^{\mone}}"description]
\&\&(\mathrm{F}^{\star}\mathrm{F})\mathrm{C}^{X}
\ar[rr,"\alpha_{\mathrm{F}^{\star},\mathrm{F},\mathrm{C}^{X}}"]
\ar[rr, phantom, yshift=-7ex, "\circled{3}"]
\ar[rr, phantom, yshift=-14ex, xshift=10ex,"\circled{4}"]
\ar[d, "{\mathrm{id}_{\mathrm{F}^{\star}\mathrm{F}}\circ_{\mathsf{h}}(\runit_{\mathrm{C}^{X}})^{\mone}}"description]
\&\&\mathrm{F}^{\star}(\mathrm{F}\mathrm{C}^{X})
\ar[dd, "\mathrm{id}_{\mathrm{F}^{\star}}\circ_{\mathsf{h}}(\runit_{\mathrm{F}\mathrm{C}^{X}})^{\mone}"]
\ar[ddll, "\mathrm{id}_{\mathrm{F}^{\star}}\circ_{\mathsf{h}}\big(\mathrm{id}_{\mathrm{F}}\circ_{\mathsf{h}}(\runit_{\mathrm{C}^{X}})^{\mone}\big)", sloped]
\\[5ex]
\mathrm{C}^{X}\mathbbm{1}_{\mathtt{i}}
\ar[rr,"(\lunit_{\mathrm{C}^{X}\mathbbm{1}_{\mathtt{i}}})^{\mone}",swap]
\&\& \mathbbm{1}_{\mathtt{i}} (\mathrm{C}^{X}\mathbbm{1}_{\mathtt{i}})
\ar[rr,"\mathrm{coev}_{\mathrm{F}}\circ_{\mathsf{h}}\mathrm{id}_{\mathrm{C}^{X}\mathbbm{1}_{\mathtt{i}}}",swap]
\&\&
(\mathrm{F}^{\star}\mathrm{F})(\mathrm{C}^{X}\mathbbm{1}_{\mathtt{i}})
\ar[d,"\alpha_{\mathrm{F}^{\star},\mathrm{F},\mathrm{C}^{X}\mathbbm{1}_{\mathtt{i}}}", swap]
\&\&
\\[5ex]
\&\&
\&\&\mathrm{F}^{\star}\big(\mathrm{F}(\mathrm{C}^{X}\mathbbm{1}_{\mathtt{i}})\big)
\ar[rr,"\mathrm{id}_{\mathrm{F}^{\star}}\circ_{\mathsf{h}}\alpha_{\mathrm{F},\mathrm{C}^{X},\mathbbm{1}_{\mathtt{i}}}^{\mone}",swap]
\&\&\mathrm{F}^{\star}\big((\mathrm{F}\mathrm{C}^{X})\mathbbm{1}_{\mathtt{i}}\big)
\end{tikzcd}$}
.
\end{gather*}
To see that this diagram commutes we note that
\begin{itemize}

\item the facets labeled $1$ and $3$ commute by naturality of $(\lunit_{})^{\mone}$ and $\alpha$, respectively;

\item the facet labeled $2$ commutes by the interchange law;

\item the facet labeled $4$ commutes by the right diagram in \eqref{eq:0.00}.
\end{itemize}

The map $\scalebox{0.9}{$\mathrm{Hom}_{\underline{\ccC}}\Big(\mathrm{C}^{X},\mathrm{F}^{\star}\big(\big((\mathrm{FC}^{X})\mathrm{F}^{\star}\big)\mathrm{F}\big)\Big)
{\xrightarrow{\cong}}
\mathrm{Hom}_{\underline{\mathbf{M}}(\mathtt{i})}\Big(X,\underline{\mathbf{M}}_{\mathtt{ii}}
\big(\mathrm{F}^{\star}\big(\big((\mathrm{FC}^{X})\mathrm{F}^{\star}\big)\mathrm{F}\big)\big)\,X\Big)$}$
sends \eqref{eq:0.8} to
\[
\scalebox{0.91}{$\begin{aligned}
&\underline{\mathbf{M}}_{\mathtt{ii}}\big(\mathrm{id}_{\mathrm{F}^{\star}}\circ_{\mathsf{h}}\alpha_{\mathrm{F}\mathrm{C}^{X},\mathrm{F}^{\star},\mathrm{F}}^{\mone}\big)_X
\circ_{\mathsf{v}}
\underline{\mathbf{M}}_{\mathtt{ii}}\big(\mathrm{id}_{\mathrm{F}^{\star}}\circ_{\mathsf{h}}(\mathrm{id}_{\mathrm{F}\mathrm{C}^{X}}\circ_{\mathsf{h}}\mathrm{coev}_{\mathrm{F}})\big)_X
\circ_{\mathsf{v}}
\underline{\mathbf{M}}_{\mathtt{ii}}\big(\mathrm{id}_{\mathrm{F}^{\star}}\circ_{\mathsf{h}}\alpha_{\mathrm{F},\mathrm{C}^{X},\mathbbm{1}_{\mathtt{i}}}^{\mone}\big)_X\\
&\circ_{\mathsf{v}}
\underline{\mathbf{M}}_{\mathtt{ii}}\big(\alpha_{\mathrm{F}^{\star},\mathrm{F},\mathrm{C}^{X}\mathbbm{1}_{\mathtt{i}}}\big)_X
\circ_{\mathsf{v}}
\underline{\mathbf{M}}_{\mathtt{ii}}\big(\mathrm{coev}_{\mathrm{F}}\circ_{\mathsf{h}}\mathrm{id}_{\mathrm{C}^{X}\mathbbm{1}_{\mathtt{i}}}\big)_X
\circ_{\mathsf{v}}
\underline{\mathbf{M}}_{\mathtt{ii}}\big((\lunit_{\mathrm{C}^{X}\mathbbm{1}_{\mathtt{i}}})^{\mone}\big)_X\\
&
\circ_{\mathsf{v}}
\underline{\mathbf{M}}_{\mathtt{ii}}\big((\runit_{\mathrm{C}^{X}})^{\mone}\big)_X
\circ_{\mathsf{v}}
\mathrm{coev}_{X,X}.
\end{aligned}$}
\]
The latter, under the composite of the inverses of the third and second isomorphisms, is sent to
\begin{gather}\label{eq:0.9}\scalebox{0.86}{$
\begin{aligned}
&\big(\iota_{\mathtt{j}}\big)_{\underline{\mathbf{M}}_{\mathtt{ji}}\Big(\big((\mathrm{FC}^{X})\mathrm{F}^{\star}\big)\mathrm{F}\Big)X}
\circ_{\mathsf{v}}
\underline{\mathbf{M}}_{\mathtt{jj}}\big(\mathrm{ev}_{\mathrm{F}}\big)_{\underline{\mathbf{M}}_{\mathtt{ji}}\Big(\big((\mathrm{FC}^{X})\mathrm{F}^{\star}\big)\mathrm{F}\Big)X}
\circ_{\mathsf{v}}
\big(\mu_{\mathtt{jij}}^{\mathrm{F},\mathrm{F}^{\star}}\big)_{\underline{\mathbf{M}}_{\mathtt{ji}}\Big(\big((\mathrm{FC}^{X})\mathrm{F}^{\star}\big)\mathrm{F}\Big)X}\\
&\circ_{\mathsf{v}}
\big(\underline{\mathbf{M}}_{\mathtt{ji}}(\mathrm{F})\big(\mu_{\mathtt{iji}}^{\mathrm{F}^{\star},\big((\mathrm{FC}^{X})\mathrm{F}^{\star}\big)\mathrm{F}}\big)^{\mone}_X\big)
\circ_{\mathsf{v}}
\Big(\underline{\mathbf{M}}_{\mathtt{ji}}(\mathrm{F})\underline{\mathbf{M}}_{\mathtt{ii}}\big(\mathrm{id}_{\mathrm{F}^{\star}}\circ_{\mathsf{h}}\alpha_{\mathrm{F}\mathrm{C}^{X},\mathrm{F}^{\star},\mathrm{F}}^{\mone}\big)_X\Big)\\
&\circ_{\mathsf{v}}
\Big({\underline{\mathbf{M}}_{\mathtt{ji}}(\mathrm{F})}\underline{\mathbf{M}}_{\mathtt{ii}}\big(\mathrm{id}_{\mathrm{F}^{\star}}\circ_{\mathsf{h}}(\mathrm{id}_{\mathrm{F}\mathrm{C}^{X}}\circ_{\mathsf{h}}\mathrm{coev}_{\mathrm{F}})\big)_X\Big)
\circ_{\mathsf{v}}
\Big({\underline{\mathbf{M}}_{\mathtt{ji}}(\mathrm{F})}\underline{\mathbf{M}}_{\mathtt{ii}}\big(\mathrm{id}_{\mathrm{F}^{\star}}\circ_{\mathsf{h}}\alpha_{\mathrm{F},\mathrm{C}^{X},\mathbbm{1}_{\mathtt{i}}}^{\mone}\big)_X\Big)\\
&\circ_{\mathsf{v}}
\Big({\underline{\mathbf{M}}_{\mathtt{ji}}(\mathrm{F})}\underline{\mathbf{M}}_{\mathtt{ii}}\big(\alpha_{\mathrm{F}^{\star},\mathrm{F},\mathrm{C}^{X}\mathbbm{1}_{\mathtt{i}}}\big)_X\Big)
\circ_{\mathsf{v}}
\Big({\underline{\mathbf{M}}_{\mathtt{ji}}(\mathrm{F})}\underline{\mathbf{M}}_{\mathtt{ii}}\big(\mathrm{coev}_{\mathrm{F}}\circ_{\mathsf{h}}\mathrm{id}_{\mathrm{C}^{X}\mathbbm{1}_{\mathtt{i}}}\big)_X\Big)
\\&
\circ_{\mathsf{v}}
\Big({\underline{\mathbf{M}}_{\mathtt{ji}}(\mathrm{F})}\underline{\mathbf{M}}_{\mathtt{ii}}\big((\lunit_{\mathrm{C}^{X}\mathbbm{1}_{\mathtt{i}}})^{\mone}\big)_X\Big)
\circ_{\mathsf{v}}
\Big({\underline{\mathbf{M}}_{\mathtt{ji}}(\mathrm{F})}\underline{\mathbf{M}}_{\mathtt{ii}}\big((\runit_{\mathrm{C}^{X}})^{\mone}\big)_X\Big)
\\&
\circ_{\mathsf{v}}
({\underline{\mathbf{M}}_{\mathtt{ji}}(\mathrm{F})}\mathrm{coev}_{X,X}),
\end{aligned}$}
\end{gather}
which is an element in $\mathrm{Hom}_{\underline{\mathbf{M}}(\mathtt{j})}\Big(\underline{\mathbf{M}}_{\mathtt{ji}}(\mathrm{F})\,X,\underline{\mathbf{M}}_{\mathtt{ji}}\big(\big((\mathrm{FC}^{X})\mathrm{F}^{\star}\big)\mathrm{F}\big)X\Big)$.

Consider now the diagram
\begin{gather*}
\adjustbox{scale=.52,center}{%
\begin{tikzcd}[ampersand replacement=\&, column sep=4em, row sep=3em]
\underline{\mathbf{M}}_{\mathtt{ji}}(\mathrm{F})\underline{\mathbf{M}}_{\mathtt{ii}}(\mathrm{C}^{X}\mathbbm{1}_{\mathtt{i}})X
\ar[rrr,"\big(\mu_{\mathtt{jii}}^{\mathrm{F},\mathrm{C}^{X}\mathbbm{1}_{\mathtt{i}}}\big)_X"]
\ar[rrr, phantom,yshift=-7ex, "\circled{1}"]
\ar[d, "{{\underline{\mathbf{M}}_{\mathtt{ji}}(\mathrm{F})}\underline{\mathbf{M}}_{\mathtt{ii}}\big((\lunit_{\mathrm{C}^{X}\mathbbm{1}_{\mathtt{i}}})^{\mone}\big)_X}"description]
\&\&\&
\underline{\mathbf{M}}_{\mathtt{ji}}\big(\mathrm{F}(\mathrm{C}^{X}\mathbbm{1}_{\mathtt{i}})\big)X
\ar[d, "{\underline{\mathbf{M}}_{\mathtt{ji}}\big(\mathrm{id}_{\mathrm{F}}\circ_{\mathsf{h}}(\lunit_{\mathrm{C}^{X}\mathbbm{1}_{\mathtt{i}}})^{\mone}\big)_X}"description]
\ar[rr, "\underline{\mathbf{M}}_{\mathtt{ji}}\big((\runit_{\mathrm{F}})^{\mone}\circ_{\mathsf{h}}\mathrm{id}_{\mathrm{C}^{X}\mathbbm{1}_{\mathtt{i}}}\big)_X"]
\ar[rr, phantom,xshift=-6ex, yshift=-4ex, "\circled{2}"]
\&\&
\underline{\mathbf{M}}_{\mathtt{ji}}\big((\mathrm{F}\mathbbm{1}_{\mathtt{i}})(\mathrm{C}^{X}\mathbbm{1}_{\mathtt{i}})\big)X
\ar[d, "{\underline{\mathbf{M}}_{\mathtt{ji}}\big((\mathrm{id}_{\mathrm{F}}\circ_{\mathsf{h}}\mathrm{coev}_{\mathrm{F}})\circ_{\mathsf{h}}\mathrm{id}_{\mathrm{C}^{X}\mathbbm{1}_{\mathtt{i}}}\big)_X}"description]
\\[5ex]
\underline{\mathbf{M}}_{\mathtt{ji}}(\mathrm{F})\underline{\mathbf{M}}_{\mathtt{ii}}\big(\mathbbm{1}_{\mathtt{i}}(\mathrm{C}^{X}\mathbbm{1}_{\mathtt{i}})\big)X
\ar[d, "{{\underline{\mathbf{M}}_{\mathtt{ji}}(\mathrm{F})}\underline{\mathbf{M}}_{\mathtt{ii}}(\mathrm{coev}_{\mathrm{F}}\circ_{\mathsf{h}}\mathrm{id}_{\mathrm{C}^{X}\mathbbm{1}_{\mathtt{i}}})_X}"description]
\ar[rrr,"\big(\mu_{\mathtt{jii}}^{\mathrm{F},\mathbbm{1}_{\mathtt{i}}(\mathrm{C}^{X}\mathbbm{1}_{\mathtt{i}})}\big)_X"]
\ar[rrr, phantom, yshift=-7ex, "\circled{3}"]
\&\&\&
\underline{\mathbf{M}}_{\mathtt{ji}}\Big(\mathrm{F}\big(\mathbbm{1}_{\mathtt{i}}(\mathrm{C}^{X}\mathbbm{1}_{\mathtt{i}})\big)\Big)X
\ar[urr,"\underline{\mathbf{M}}_{\mathtt{ji}}\big(\alpha_{\mathrm{F},\mathbbm{1}_{\mathtt{i}},\mathrm{C}^{X}\mathbbm{1}_{\mathtt{i}}}^{\mone}\big)_X", sloped]
\ar[d, "{\underline{\mathbf{M}}_{\mathtt{ji}}\big(\mathrm{id}_{\mathrm{F}}\circ_{\mathsf{h}}(\mathrm{coev}_{\mathrm{F}}\circ_{\mathsf{h}}\mathrm{id}_{\mathrm{C}^{X}\mathbbm{1}_{\mathtt{i}}})\big)_X}"description]
\ar[urr, phantom, yshift=-7ex, "\circled{4}"]
\&\&
\underline{\mathbf{M}}_{\mathtt{ji}}\Big(\big(\mathrm{F}(\mathrm{F}^{\star}\mathrm{F})\big)(\mathrm{C}^{X}\mathbbm{1}_{\mathtt{i}})\Big)X
\ar[d,"{\underline{\mathbf{M}}_{\mathtt{ji}}(\alpha_{\mathrm{F},\mathrm{F}^{\star},\mathrm{F}}^{\mone}\circ_{\mathsf{h}}\mathrm{id}_{\mathrm{C}^{X}\mathbbm{1}_{\mathtt{i}}})_X}"description]
\\[5ex]
\underline{\mathbf{M}}_{\mathtt{ji}}(\mathrm{F})\underline{\mathbf{M}}_{\mathtt{ii}}\big((\mathrm{F}^{\star}\mathrm{F})(\mathrm{C}^{X}\mathbbm{1}_{\mathtt{i}})\big)X
\ar[d,"{{\underline{\mathbf{M}}_{\mathtt{ji}}(\mathrm{F})}\underline{\mathbf{M}}_{\mathtt{ii}}(\alpha_{\mathrm{F}^{\star},\mathrm{F},\mathrm{C}^{X}\mathbbm{1}_{\mathtt{i}}})_X}"description]
\ar[rrr,"\big(\mu_{\mathtt{jii}}^{\mathrm{F},(\mathrm{F}^{\star}\mathrm{F})(\mathrm{C}^{X}\mathbbm{1}_{\mathtt{i}})}\big)_X"]
\ar[rrr, phantom,yshift=-7ex, "\circled{5}"]
\&\&\&
\underline{\mathbf{M}}_{\mathtt{ji}}\Big(\mathrm{F}\big((\mathrm{F}^{\star}\mathrm{F})(\mathrm{C}^{X}\mathbbm{1}_{\mathtt{i}})\big)\Big)X
\ar[d,"{\underline{\mathbf{M}}_{\mathtt{ji}}(\mathrm{id}_{\mathrm{F}}\circ_{\mathsf{h}}\alpha_{\mathrm{F}^{\star},\mathrm{F},\mathrm{C}^{X}\mathbbm{1}_{\mathtt{i}}})_X}"description]
\ar[urr,"\underline{\mathbf{M}}_{\mathtt{ji}}\big(\alpha_{\mathrm{F},\mathrm{F}^{\star}\mathrm{F},\mathrm{C}^{X}\mathbbm{1}_{\mathtt{i}}}^{\mone}\big)_X", sloped]
\ar[urr, phantom, yshift=-12ex, "\circled{6}"]
\&\&
\underline{\mathbf{M}}_{\mathtt{ji}}\Big(\big((\mathrm{F}\mathrm{F}^{\star})\mathrm{F}\big)(\mathrm{C}^{X}\mathbbm{1}_{\mathtt{i}})\Big)X
\ar[d,"{\underline{\mathbf{M}}_{\mathtt{ji}}\big(\alpha_{\mathrm{F}\mathrm{F}^{\star},\mathrm{F},\mathrm{C}^{X}\mathbbm{1}_{\mathtt{i}}}\big)_X}"description]
\\[5ex]
\underline{\mathbf{M}}_{\mathtt{ji}}(\mathrm{F})\underline{\mathbf{M}}_{\mathtt{ii}}\Big(\mathrm{F}^{\star}\big(\mathrm{F}(\mathrm{C}^{X}\mathbbm{1}_{\mathtt{i}})\big)\Big)X
\ar[d,"{{\underline{\mathbf{M}}_{\mathtt{ji}}(\mathrm{F})}\underline{\mathbf{M}}_{\mathtt{ii}}(\mathrm{id}_{\mathrm{F}^{\star}}\circ_{\mathsf{h}}\alpha_{\mathrm{F},\mathrm{C}^{X},\mathbbm{1}_{\mathtt{i}}}^{\mone})_X}"description]
\ar[rrr,"\big(\mu_{\mathtt{jii}}^{\mathrm{F},\mathrm{F}^{\star}\left(\mathrm{F}(\mathrm{C}^{X}\mathbbm{1}_{\mathtt{i}})\right)}\big)_X"]
\ar[rrr, phantom,yshift=-7ex, "\circled{7}"]
\&\&\&
\underline{\mathbf{M}}_{\mathtt{ji}}\bigg(\mathrm{F}\Big(\mathrm{F}^{\star}\big(\mathrm{F}(\mathrm{C}^{X}\mathbbm{1}_{\mathtt{i}})\big)\Big)\bigg)X
\ar[d,"{\underline{\mathbf{M}}_{\mathtt{ji}}\big(\mathrm{id}_{\mathrm{F}}\circ_{\mathsf{h}}(\mathrm{id}_{\mathrm{F}^{\star}}\circ_{\mathsf{h}}\alpha_{\mathrm{F},\mathrm{C}^{X},\mathbbm{1}_{\mathtt{i}}}^{\mone})\big)_X}"description]
\ar[rr,"\underline{\mathbf{M}}_{\mathtt{ji}}(\alpha_{\mathrm{F},\mathrm{F}^{\star},\mathrm{F}(\mathrm{C}^{X}\mathbbm{1}_{\mathtt{i}})}^{\mone})_X"]
\ar[rr, phantom,yshift=-7ex, "\circled{8}"]
\&\&
\underline{\mathbf{M}}_{\mathtt{ji}}\Big((\mathrm{F}\mathrm{F}^{\star})\big(\mathrm{F}(\mathrm{C}^{X}\mathbbm{1}_{\mathtt{i}})\big)\Big)X
\ar[d, "{\underline{\mathbf{M}}_{\mathtt{ji}}(\mathrm{id}_{\mathrm{FF}^{\star}}\circ_{\mathsf{h}}\alpha_{\mathrm{F},\mathrm{C}^{X},\mathbbm{1}_{\mathtt{i}}}^{\mone})_X}"description]
\\[5ex]
\underline{\mathbf{M}}_{\mathtt{ji}}(\mathrm{F})\underline{\mathbf{M}}_{\mathtt{ii}}\Big(\mathrm{F}^{\star}\big((\mathrm{F}\mathrm{C}^{X})\mathbbm{1}_{\mathtt{i}}\big)\Big)X
\ar[d, "{{\underline{\mathbf{M}}_{\mathtt{ji}}(\mathrm{F})}\underline{\mathbf{M}}_{\mathtt{ii}}\big(\mathrm{id}_{\mathrm{F}^{\star}}\circ_{\mathsf{h}}(\mathrm{id}_{\mathrm{F}\mathrm{C}^{X}}\circ_{\mathsf{h}}\mathrm{coev}_{\mathrm{F}})\big)_X}"description]
\ar[rrr,"\big(\mu_{\mathtt{jii}}^{\mathrm{F},\mathrm{F}^{\star}\left((\mathrm{F}\mathrm{C}^{X})\mathbbm{1}_{\mathtt{i}}\right)}\big)_X"]
\ar[rrr, phantom, yshift=-7ex,"\circled{9}"]
\&\&\&
\underline{\mathbf{M}}_{\mathtt{ji}}\bigg(\mathrm{F}\Big(\mathrm{F}^{\star}\big((\mathrm{F}\mathrm{C}^{X})\mathbbm{1}_{\mathtt{i}}\big)\Big)\bigg)X
\ar[d, "{\underline{\mathbf{M}}_{\mathtt{ji}}\Big(\mathrm{id}_{\mathrm{F}}\circ_{\mathsf{h}}\big(\mathrm{id}_{\mathrm{F}^{\star}}\circ_{\mathsf{h}}(\mathrm{id}_{\mathrm{F}\mathrm{C}^{X}}\circ_{\mathsf{h}}\mathrm{coev}_{\mathrm{F}})\big)\Big)_X}"description]
\ar[rr,"\underline{\mathbf{M}}_{\mathtt{ji}}\big(\alpha_{\mathrm{F},\mathrm{F}^{\star},(\mathrm{F}\mathrm{C}^{X})\mathbbm{1}_{\mathtt{i}}}^{\mone}\big)_X"]
\ar[rr, phantom,yshift=-7ex, "\circled{10}"]
\&\&
\underline{\mathbf{M}}_{\mathtt{ji}}\Big((\mathrm{F}\mathrm{F}^{\star})\big((\mathrm{F}\mathrm{C}^{X})\mathbbm{1}_{\mathtt{i}}\big)\Big)X
\ar[d, "{\underline{\mathbf{M}}_{\mathtt{ji}}\big(\mathrm{id}_{\mathrm{FF}^{\star}}\circ_{\mathsf{h}}(\mathrm{id}_{\mathrm{FC}^{X}}\circ_{\mathsf{h}}
\mathrm{coev}_{\mathrm{F}})\big)_X}"description]
\\[5ex]
\underline{\mathbf{M}}_{\mathtt{ji}}(\mathrm{F})\underline{\mathbf{M}}_{\mathtt{ii}}\Big(\mathrm{F}^{\star}\big((\mathrm{F}\mathrm{C}^{X})(\mathrm{F}^{\star}\mathrm{F})\big)\Big)X
\ar[d, "{{\underline{\mathbf{M}}_{\mathtt{ji}}(\mathrm{F})}\underline{\mathbf{M}}_{\mathtt{ii}}(\mathrm{id}_{\mathrm{F}^{\star}}\circ_{\mathsf{h}}\alpha_{\mathrm{F}\mathrm{C}^{X},\mathrm{F}^{\star},\mathrm{F}}^{\mone})_X}"description]
\ar[rrr,"\big(\mu_{\mathtt{jii}}^{\mathrm{F},\mathrm{F}^{\star}\left((\mathrm{F}\mathrm{C}^{X})\left(\mathrm{F}^{\star}\mathrm{F}\right)\right)}\big)_X"]
\ar[rrr, phantom,yshift=-7ex, "\circled{11}"]
\&\&\&
\underline{\mathbf{M}}_{\mathtt{ji}}\bigg(\mathrm{F}\Big(\mathrm{F}^{\star}\big((\mathrm{F}\mathrm{C}^{X})(\mathrm{F}^{\star}\mathrm{F})\big)\Big)\bigg)X
\ar[d, "{\underline{\mathbf{M}}_{\mathtt{ji}}\big(\mathrm{id}_{\mathrm{F}}\circ_{\mathsf{h}}(\mathrm{id}_{\mathrm{F}^{\star}}\circ_{\mathsf{h}}\alpha_{\mathrm{F}\mathrm{C}^{X},\mathrm{F}^{\star},\mathrm{F}}^{\mone})\big)_X}"description]
\ar[rr,"\underline{\mathbf{M}}_{\mathtt{ji}}\big(\alpha_{\mathrm{F},\mathrm{F}^{\star},(\mathrm{F}\mathrm{C}^{X})(\mathrm{F}^{\star}\mathrm{F})}^{\mone}\big)_X"]
\ar[rr, phantom,yshift=-7ex, "\circled{12}"]
\&\&
\underline{\mathbf{M}}_{\mathtt{ji}}\bigg((\mathrm{F}\mathrm{F}^{\star})\big((\mathrm{F}\mathrm{C}^{X})(\mathrm{F}^{\star}\mathrm{F})\big)\bigg)X
\ar[d, "{\underline{\mathbf{M}}_{\mathtt{ji}}\big(\mathrm{id}_{\mathrm{F}\mathrm{F}^{\star}}\circ_{\mathsf{h}}\alpha_{\mathrm{F}\mathrm{C}^{X},\mathrm{F}^{\star},\mathrm{F}}^{\mone})\big)_X}"description]
\\[5ex]
\underline{\mathbf{M}}_{\mathtt{ji}}(\mathrm{F})\underline{\mathbf{M}}_{\mathtt{ii}}\bigg(\mathrm{F}^{\star}\Big(\big((\mathrm{FC}^{X})\mathrm{F}^{\star}\big)\mathrm{F}\Big)\bigg)X
\ar[d, "{{\underline{\mathbf{M}}_{\mathtt{ji}}(\mathrm{F})}\big(\mu_{\mathtt{iji}}^{\mathrm{F}^{\star},((\mathrm{F}\mathrm{C}^{X})\mathrm{F}^{\star})\mathrm{F}}\big)^{\mone}_X}"description]
\ar[rrr,"\big(\mu_{\mathtt{jii}}^{\mathrm{F},\mathrm{F}^{\star}\big(\left((\mathrm{FC}^{X})\mathrm{F}^{\star}\right)\mathrm{F}\big)}\big)_X"]
\ar[rrr, phantom, yshift=-7ex, "\circled{13}"]
\&\&\&
\underline{\mathbf{M}}_{\mathtt{ji}}\Bigg(\mathrm{F}\bigg(\mathrm{F}^{\star}\Big(\big((\mathrm{FC}^{X})\mathrm{F}^{\star}\big)\mathrm{F}\Big)\bigg)\Bigg)X
\ar[rr,"\underline{\mathbf{M}}_{\mathtt{ji}}\big(\alpha_{\mathrm{F},\mathrm{F}^{\star},\left(\left(\mathrm{FC}^{X}\right)\mathrm{F}^{\star}\right)\mathrm{F}}^{\mone}\big)_X"]
\&\&
\underline{\mathbf{M}}_{\mathtt{ji}}\bigg((\mathrm{F}\mathrm{F}^{\star})\Big(\big((\mathrm{FC}^{X})\mathrm{F}^{\star}\big)\mathrm{F}\Big)\bigg)X
\ar[d, "{\underline{\mathbf{M}}_{\mathtt{ji}}(\mathrm{ev}_{\mathrm{F}}\circ_{\mathsf{h}}\mathrm{id}_{\left(\left(\mathrm{FC}^{X}\right)\mathrm{F}^{\star}\right)\mathrm{F}})_X}"description]
\\[5ex]
\underline{\mathbf{M}}_{\mathtt{ji}}(\mathrm{F})\underline{\mathbf{M}}_{\mathtt{ij}}(\mathrm{F}^{\star})
\underline{\mathbf{M}}_{\mathtt{ji}}\Big(\big((\mathrm{FC}^{X})\mathrm{F}^{\star}\big)\mathrm{F}\Big)X
\ar[rrr, "\big(\mu_{\mathtt{jij}}^{\mathrm{F},\mathrm{F}^{\star}}\big)_{\underline{\mathbf{M}}_{\mathtt{ji}}\big(\left((\mathrm{FC}^{X})\mathrm{F}^{\star}\right)\mathrm{F}\big)X}"]
\&\&\&
\underline{\mathbf{M}}_{\mathtt{jj}}(\mathrm{F}\mathrm{F}^{\star})
\underline{\mathbf{M}}_{\mathtt{ji}}\Big(\big((\mathrm{FC}^{X})\mathrm{F}^{\star}\big)\mathrm{F}\Big)X
\ar[d, "{\underline{\mathbf{M}}_{\mathtt{jj}}\big(\mathrm{ev}_{\mathrm{F}}\big)_{\underline{\mathbf{M}}_{\mathtt{ji}}\big(\left((\mathrm{FC}^{X})\mathrm{F}^{\star}\right)\mathrm{F}\big)X}}"description]
\ar[urr,"\big(\mu_{\mathtt{jji}}^{\mathrm{F}\mathrm{F}^{\star},\left((\mathrm{FC}^{X})\mathrm{F}^{\star}\right)\mathrm{F}}\big)_X",sloped]
\ar[urr, phantom,yshift=-7ex, "\circled{14}"]
\&\&
\underline{\mathbf{M}}_{\mathtt{ji}}\bigg(\mathbbm{1}_{\mathtt{j}}\Big(\big((\mathrm{FC}^{X})\mathrm{F}^{\star}\big)\mathrm{F}\Big)\bigg)X
\ar[d,"{\underline{\mathbf{M}}_{\mathtt{ji}}\big(\lunit_{\left((\mathrm{FC}^{X})\mathrm{F}^{\star}\right)\mathrm{F}}\big)}"description]
\\[5ex]
\&\&\&
\underline{\mathbf{M}}_{\mathtt{jj}}(\mathbbm{1}_{\mathtt{j}})
\underline{\mathbf{M}}_{\mathtt{ji}}\Big(\big((\mathrm{FC}^{X})\mathrm{F}^{\star}\big)\mathrm{F}\Big)X
\ar[rr, "\big(\iota_{\mathtt{j}}\big)_{\underline{\mathbf{M}}_{\mathtt{ji}}\big(\left((\mathrm{FC}^{X})\mathrm{F}^{\star}\right)\mathrm{F}\big)X}", swap]
\ar[urr,"\big(\mu_{\mathtt{jii}}^{\mathbbm{1}_{\mathtt{j}},\left((\mathrm{FC}^{X})\mathrm{F}^{\star}\right)\mathrm{F}}\big)_X",sloped]
\ar[rr, phantom, xshift=7ex, yshift=6ex, "\circled{15}"]
\&\&
\underline{\mathbf{M}}_{\mathtt{ji}}\Big(\big((\mathrm{FC}^{X})\mathrm{F}^{\star}\big)\mathrm{F}\Big)X
\end{tikzcd}
}
\end{gather*}
in which the path going down the left side and then right along the bottom includes all but the first two morphisms in \eqref{eq:0.9}. The diagram commutes by
\begin{itemize}

\item naturality of $\mu_{\mathtt{jii}}$ for the facets labeled $1$, $3$, $5$, $7$, $9$ and $11$;

\item the triangle coherence condition of the unitors for the facet labeled $2$;

\item naturality of $\alpha$ for the facets labeled $4$, $8$, $10$ and $12$;

\item the pentagon coherence condition of the associator for the facet labeled $6$;

\item the diagram in \eqref{eq:birepresentation3} for the facet labeled $13$;

\item naturality of $\mu_{\mathtt{jji}}$ for the facet labeled $14$;

\item the right diagram in \eqref{eq:birepresentation2} for the facet labeled $15$.
\end{itemize}
Let us also consider the diagram
\begin{gather*}
\adjustbox{scale=.62,center}{%
\begin{tikzcd}[ampersand replacement=\&, column sep=3.9em, row sep=3em]
\underline{\mathbf{M}}_{\mathtt{ji}}\Big(\big(\mathrm{F}(\mathrm{F}^{\star}\mathrm{F})\big)(\mathrm{C}^{X}\mathbbm{1}_{\mathtt{i}})\Big)X
\ar[d,"{\underline{\mathbf{M}}_{\mathtt{ji}}(\alpha_{\mathrm{F},\mathrm{F}^{\star},\mathrm{F}}^{\mone}\circ_{\mathsf{h}}\mathrm{id}_{\mathrm{C}^{X}\mathbbm{1}_{\mathtt{i}}})_X}"description]
\&\&\&
\underline{\mathbf{M}}_{\mathtt{ji}}\big((\mathrm{F}\mathbbm{1}_{\mathtt{i}})(\mathrm{C}^{X}\mathbbm{1}_{\mathtt{i}})\big)X
\ar[lll, "\underline{\mathbf{M}}_{\mathtt{ji}}\big((\mathrm{id}_{\mathrm{F}}\circ_{\mathsf{h}}\mathrm{coev}_{\mathrm{F}})\circ_{\mathsf{h}}\mathrm{id}_{\mathrm{C}^{X}\mathbbm{1}_{\mathtt{i}}}\big)_X", swap]
\ar[lll, phantom,yshift=-9ex, xshift=16ex,"\circled{1}"]
\&\&
\underline{\mathbf{M}}_{\mathtt{ji}}\big(\mathrm{F}(\mathrm{C}^{X}\mathbbm{1}_{\mathtt{i}})\big)X
\ar[ll, "\underline{\mathbf{M}}_{\mathtt{ji}}\big((\runit_{\mathrm{F}})^{\mone}\circ_{\mathsf{h}}\mathrm{id}_{\mathrm{C}^{X}\mathbbm{1}_{\mathtt{i}}}\big)_X", swap]
\ar[d,equal]
\\[5ex]
\underline{\mathbf{M}}_{\mathtt{ji}}\Big(\big((\mathrm{F}\mathrm{F}^{\star})\mathrm{F}\big)(\mathrm{C}^{X}\mathbbm{1}_{\mathtt{i}})\Big)X
\ar[d,"{\underline{\mathbf{M}}_{\mathtt{ji}}(\alpha_{\mathrm{F}\mathrm{F}^{\star},\mathrm{F},\mathrm{C}^{X}\mathbbm{1}_{\mathtt{i}}})_X}"description]
\ar[rrr,"\underline{\mathbf{M}}_{\mathtt{ji}}\big((\mathrm{ev}_{\mathrm{F}}\circ_{\mathsf{h}}\mathrm{id}_{\mathrm{F}})\circ_{\mathsf{h}}\mathrm{id}_{\mathrm{C}^{X}\mathbbm{1}_{\mathtt{i}}}\big)_X"]
\ar[rrr, phantom,yshift=-7ex, "\circled{2}"]
\&\&\&
\underline{\mathbf{M}}_{\mathtt{ji}}\Big(\big(\mathbbm{1}_{\mathtt{j}}\mathrm{F}\big)(\mathrm{C}^{X}\mathbbm{1}_{\mathtt{i}})\Big)X
\ar[d,"{\underline{\mathbf{M}}_{\mathtt{ji}}(\alpha_{\mathbbm{1}_{\mathtt{j}},\mathrm{F},\mathrm{C}^{X}\mathbbm{1}_{\mathtt{i}}})_X}"description]
\ar[rr,"\underline{\mathbf{M}}_{\mathtt{ji}}(\lunit_{\mathrm{F}}\circ_{\mathsf{h}}\mathrm{id}_{\mathrm{C}^{X}\mathbbm{1}_{\mathtt{i}}})_X"]
\ar[rr, phantom,yshift=-7ex, "\circled{3}"]
\&\&
\underline{\mathbf{M}}_{\mathtt{ji}}\big(\mathrm{F}(\mathrm{C}^{X}\mathbbm{1}_{\mathtt{i}})\big)X
\ar[d,equal]
\\[5ex]
\underline{\mathbf{M}}_{\mathtt{ji}}\Big((\mathrm{F}\mathrm{F}^{\star})\big(\mathrm{F}(\mathrm{C}^{X}\mathbbm{1}_{\mathtt{i}})\big)\Big)X
\ar[d, "{\underline{\mathbf{M}}_{\mathtt{ji}}(\mathrm{id}_{\mathrm{FF}^{\star}}\circ_{\mathsf{h}}\alpha_{\mathrm{F},\mathrm{C}^{X},\mathbbm{1}_{\mathtt{i}}}^{\mone})_X}"description]
\ar[rrr,"\underline{\mathbf{M}}_{\mathtt{ji}}\big(\mathrm{ev}_{\mathrm{F}}\circ_{\mathsf{h}}\mathrm{id}_{\mathrm{F}(\mathrm{C}^{X}\mathbbm{1}_{\mathtt{i}})}\big)_X"]
\ar[rrr, phantom,yshift=-7ex, "\circled{4}"]
\&\&\&
\underline{\mathbf{M}}_{\mathtt{ji}}\Big(\mathbbm{1}_{\mathtt{j}}\big(\mathrm{F}(\mathrm{C}^{X}\mathbbm{1}_{\mathtt{i}})\big)\Big)X
\ar[rr,"\underline{\mathbf{M}}_{\mathtt{ji}}\big(\lunit_{\mathrm{F}(\mathrm{C}^{X}\mathbbm{1}_{\mathtt{i}})}\big)_X"]
\ar[d, "{\underline{\mathbf{M}}_{\mathtt{ji}}(\mathrm{id}_{\mathbbm{1}_{\mathtt{j}}}\circ_{\mathsf{h}}\alpha_{\mathrm{F},\mathrm{C}^{X},\mathbbm{1}_{\mathtt{i}}}^{\mone})_X}"description]
\ar[rr, phantom,yshift=-7ex, "\circled{5}"]
\&\&
\underline{\mathbf{M}}_{\mathtt{ji}}\big(\mathrm{F}(\mathrm{C}^{X}\mathbbm{1}_{\mathtt{i}})\big)X
\ar[d, "{\underline{\mathbf{M}}_{\mathtt{ji}}(\alpha_{\mathrm{F},\mathrm{C}^{X},\mathbbm{1}_{\mathtt{i}}}^{\mone})_X}"description]
\\[5ex]
\underline{\mathbf{M}}_{\mathtt{ji}}\Big((\mathrm{F}\mathrm{F}^{\star})\big((\mathrm{F}\mathrm{C}^{X})\mathbbm{1}_{\mathtt{i}}\big)\Big)X
\ar[d, "{\underline{\mathbf{M}}_{\mathtt{ji}}\big(\mathrm{id}_{\mathrm{FF}^{\star}}\circ_{\mathsf{h}}(\mathrm{id}_{\mathrm{FC}^{X}}\circ_{\mathsf{h}}
\mathrm{coev}_{\mathrm{F}})\big)_X}"description]
\ar[rrr,"\underline{\mathbf{M}}_{\mathtt{ji}}\big(\mathrm{ev}_{\mathrm{F}}\circ_{\mathsf{h}}\mathrm{id}_{(\mathrm{F}\mathrm{C}^{X})\mathbbm{1}_{\mathtt{i}}}\big)_X"]
\ar[rrr, phantom,yshift=-7ex, "\circled{6}"]
\&\&\&
\underline{\mathbf{M}}_{\mathtt{ji}}\Big(\mathbbm{1}_{\mathtt{j}}\big((\mathrm{F}\mathrm{C}^{X})\mathbbm{1}_{\mathtt{i}}\big)\Big)X
\ar[d, "{\underline{\mathbf{M}}_{\mathtt{ji}}\big(\mathrm{id}_{\mathbbm{1}_{\mathtt{j}}}\circ_{\mathsf{h}}(\mathrm{id}_{\mathrm{FC}^{X}}\circ_{\mathsf{h}}
\mathrm{coev}_{\mathrm{F}})\big)_X}"description]
\ar[rr,"\underline{\mathbf{M}}_{\mathtt{ji}}\big(\lunit_{(\mathrm{F}\mathrm{C}^{X})\mathbbm{1}_{\mathtt{i}}}\big)_X"]
\ar[rr, phantom,yshift=-7ex, "\circled{7}"]
\&\&
\underline{\mathbf{M}}_{\mathtt{ji}}\big((\mathrm{F}\mathrm{C}^{X})\mathbbm{1}_{\mathtt{i}}\big)X
\ar[d, "{\underline{\mathbf{M}}_{\mathtt{ji}}(\mathrm{id}_{\mathrm{FC}^{X}}\circ_{\mathsf{h}}
\mathrm{coev}_{\mathrm{F}})_X}"description]
\\[5ex]
\underline{\mathbf{M}}_{\mathtt{ji}}\Big((\mathrm{F}\mathrm{F}^{\star})\big((\mathrm{F}\mathrm{C}^{X})(\mathrm{F}^{\star}\mathrm{F})\big)\Big)X
\ar[d, "{\underline{\mathbf{M}}_{\mathtt{ji}}(\mathrm{id}_{\mathrm{F}\mathrm{F}^{\star}}\circ_{\mathsf{h}}\alpha_{\mathrm{F}\mathrm{C}^{X},\mathrm{F}^{\star},\mathrm{F}}^{\mone})_X}"description]
\ar[rrr, "\underline{\mathbf{M}}_{\mathtt{ji}}(\mathrm{ev}_{\mathrm{F}}\circ_{\mathsf{h}}\mathrm{id}_{(\mathrm{F}\mathrm{C}^{X})(\mathrm{F}^{\star}\mathrm{F})})_X"]
\ar[rrr, phantom,yshift=-7ex, "\circled{8}"]
\&\&\&
\underline{\mathbf{M}}_{\mathtt{ji}}\Big(\mathbbm{1}_{\mathtt{j}}\big((\mathrm{F}\mathrm{C}^{X})(\mathrm{F}^{\star}\mathrm{F})\big)\Big)X
\ar[d, "{\underline{\mathbf{M}}_{\mathtt{ji}}(\mathrm{id}_{\mathbbm{1}_{\mathtt{j}}}\circ_{\mathsf{h}}\alpha_{\mathrm{F}\mathrm{C}^{X},\mathrm{F}^{\star},\mathrm{F}}^{\mone})_X}"description]
\ar[rr,"\underline{\mathbf{M}}_{\mathtt{ji}}\big(\lunit_{\left((\mathrm{F}\mathrm{C}^{X})(\mathrm{F}^{\star}\mathrm{F})\right)}\big)_X"]
\ar[rr, phantom,yshift=-7ex, "\circled{9}"]
\&\&
\underline{\mathbf{M}}_{\mathtt{ji}}\big((\mathrm{F}\mathrm{C}^{X})(\mathrm{F}^{\star}\mathrm{F})\big)X
\ar[d, "{\underline{\mathbf{M}}_{\mathtt{ji}}(\alpha_{\mathrm{F}\mathrm{C}^{X},\mathrm{F}^{\star},\mathrm{F}}^{\mone})_X}"description]
\\[5ex]
\underline{\mathbf{M}}_{\mathtt{ji}}\bigg((\mathrm{F}\mathrm{F}^{\star})\Big(\big((\mathrm{FC}^{X})\mathrm{F}^{\star}\big)\mathrm{F}\Big)\bigg)X
\ar[rrr, "\underline{\mathbf{M}}_{\mathtt{ji}}\big(\mathrm{ev}_{\mathrm{F}}\circ_{\mathsf{h}}\mathrm{id}_{\left((\mathrm{FC}^{X})\mathrm{F}^{\star}\right)\mathrm{F}}\big)_X"]
\&\&\&
\underline{\mathbf{M}}_{\mathtt{ji}}\bigg(\mathbbm{1}_{\mathtt{j}}\Big(\big((\mathrm{FC}^{X})\mathrm{F}^{\star}\big)\mathrm{F}\Big)\bigg)X
\ar[rr,"\underline{\mathbf{M}}_{\mathtt{ji}}\big(\lunit_{\left((\mathrm{FC}^{X})\mathrm{F}^{\star}\right)\mathrm{F}}\big)_X"]
\&\&
\underline{\mathbf{M}}_{\mathtt{ji}}\Big(\big((\mathrm{FC}^{X})\mathrm{F}^{\star}\big)\mathrm{F}\Big)X
\end{tikzcd}}
\end{gather*}
which commutes by
\begin{itemize}
\item the adjunction condition of the adjoint pair $(\mathrm{F},\mathrm{F}^{\star})$ for the facet labeled $1$;
\item naturality of $\alpha$ for the facet labeled $2$;
\item the first condition in \eqref{eq:0.00} for the facet labeled $3$;
\item the interchange law for the facets labeled $4$, $6$ and $8$;

\item naturality of $\lunit_{}$ for the facets labeled $5$, $7$ and $9$.
\end{itemize}
In this diagram, going from the top right corner to the bottom right corner by first going all the way left, then down and then right again corresponds to going right and then down in the previous diagram, starting from the second entry in the first row.

Hence, we obtain that \eqref{eq:0.9} equals
\begin{gather}\label{eq:0.100}
\begin{aligned}
&\underline{\mathbf{M}}_{\mathtt{ji}}(\alpha_{\mathrm{F}\mathrm{C}^{X},\mathrm{F}^{\star},\mathrm{F}}^{\mone})_X
\circ_{\mathsf{v}}
\underline{\mathbf{M}}_{\mathtt{ji}}(\mathrm{id}_{\mathrm{FC}^{X}}\circ_{\mathsf{h}}\mathrm{coev}_{\mathrm{F}})_X
\circ_{\mathsf{v}}
\underline{\mathbf{M}}_{\mathtt{ji}}(\alpha_{\mathrm{F},\mathrm{C}^{X},\mathbbm{1}_{\mathtt{i}}}^{\mone})_X
\\&
\circ_{\mathsf{v}}
\big(\mu_{\mathtt{jii}}^{\mathrm{F},\mathrm{C}^{X}\mathbbm{1}_{\mathtt{i}}}\big)_X
\circ_{\mathsf{v}}
\Big({\underline{\mathbf{M}}_{\mathtt{ji}}(\mathrm{F})}\underline{\mathbf{M}}_{\mathtt{ii}}\big((\runit_{\mathrm{C}^{X}})^{\mone}\big)_X\Big)
\circ_{\mathsf{v}}
({\underline{\mathbf{M}}_{\mathtt{ji}}(\mathrm{F})}\mathrm{coev}_{X,X})
\\&
=\underline{\mathbf{M}}_{\mathtt{ji}}(\alpha_{\mathrm{F}\mathrm{C}^{X},\mathrm{F}^{\star},\mathrm{F}}^{\mone})_X
\circ_{\mathsf{v}}
\underline{\mathbf{M}}_{\mathtt{ji}}(\mathrm{id}_{\mathrm{FC}^{X}}\circ_{\mathsf{h}}\mathrm{coev}_{\mathrm{F}})_X
\circ_{\mathsf{v}}
\underline{\mathbf{M}}_{\mathtt{ji}}\big((\runit_{\mathrm{FC}^{X}})^{\mone}\big)_X
\\&
\circ_{\mathsf{v}}
\big(\mu_{\mathtt{jii}}^{\mathrm{F},\mathrm{C}^{X}}\big)_X
\circ_{\mathsf{v}}
({\underline{\mathbf{M}}_{\mathtt{ji}}(\mathrm{F})}\mathrm{coev}_{X,X}),
\end{aligned}
\end{gather}
due to commutativity of the diagram
\begin{gather*}
\scalebox{0.75}{
$\begin{tikzcd}[ampersand replacement=\&, column sep=3.5em, row sep=2.5em]
\underline{\mathbf{M}}_{\mathtt{ji}}(\mathrm{F})\underline{\mathbf{M}}_{\mathtt{ii}}(\mathrm{C}^{X})X
\ar[rr,"\big(\mu_{\mathtt{jii}}^{\mathrm{F},\mathrm{C}^{X}}\big)_X"]
\ar[d, "{{\underline{\mathbf{M}}_{\mathtt{ji}}(\mathrm{F})}\underline{\mathbf{M}}_{\mathtt{ii}}\big((\runit_{\mathrm{C}^{X}})^{\mone}\big)_X}"description]
\&\&
\underline{\mathbf{M}}_{\mathtt{ji}}(\mathrm{F}\mathrm{C}^{X})X
\ar[d, "{\underline{\mathbf{M}}_{\mathtt{ji}}\big(\mathrm{id}_{\mathrm{F}}\circ_{\mathsf{h}}(\runit_{\mathrm{C}^{X}})^{\mone}\big)_X}"description]
\ar[drr, "\underline{\mathbf{M}}_{\mathtt{ji}}\big((\runit_{\mathrm{FC}^{X}})^{\mone}\big)_X"]
\&\&
\\[5ex]
\underline{\mathbf{M}}_{\mathtt{ji}}(\mathrm{F})\underline{\mathbf{M}}_{\mathtt{ii}}(\mathrm{C}^{X}\mathbbm{1}_{\mathtt{i}})X
\ar[rr,"\big(\mu_{\mathtt{jii}}^{\mathrm{F},\mathrm{C}^{X}\mathbbm{1}_{\mathtt{i}}}\big)_X",swap]
\&\&
\underline{\mathbf{M}}_{\mathtt{ji}}\big(\mathrm{F}(\mathrm{C}^{X}\mathbbm{1}_{\mathtt{i}})\big)X
\ar[rr,"\underline{\mathbf{M}}_{\mathtt{ji}}\big(\alpha_{\mathrm{F},\mathrm{C}^{X},\mathbbm{1}_{\mathtt{i}}}^{\mone}\big)_X",swap]
\&\&
\underline{\mathbf{M}}_{\mathtt{ji}}\big((\mathrm{F}\mathrm{C}^{X})\mathbbm{1}_{\mathtt{i}}\big)X
\end{tikzcd}$
}
,
\end{gather*}
where the left square and the right triangle commute by naturality of $\mu_{\mathtt{jii}}$ and the right diagram in \eqref{eq:0.00}, respectively.
This shows that $\mathrm{coev}_{\underline{\mathbf{M}}_{\mathtt{ji}}(\mathrm{F}) X,\underline{\mathbf{M}}_{\mathtt{ji}}(\mathrm{F})\,X}$, i.e. the image of \eqref{eq:0.100} under the inverse of the first isomorphism in \eqref{eq:0.4}, is equal to
\begin{gather*}\scalebox{0.95}{$
\begin{aligned}
&(\mu_{\mathtt{j}\mathtt{j}\mathtt{i}}^{(\mathrm{F}\mathrm{C}^{X})\mathrm{F}^{\star},\mathrm{F}})_X^{\mone}
\circ_{\mathsf{v}}
\underline{\mathbf{M}}_{\mathtt{ji}}(\alpha_{\mathrm{F}\mathrm{C}^{X},\mathrm{F}^{\star},\mathrm{F}}^{\mone})_X
\circ_{\mathsf{v}}
\underline{\mathbf{M}}_{\mathtt{ji}}(\mathrm{id}_{\mathrm{FC}^{X}}\circ_{\mathsf{h}}\mathrm{coev}_{\mathrm{F}})_X
\circ_{\mathsf{v}}
\underline{\mathbf{M}}_{\mathtt{ii}}\big((\runit_{\mathrm{FC}^{X}})^{\mone}\big)_X
\\&
\circ_{\mathsf{v}}
\big(\mu_{\mathtt{jii}}^{\mathrm{F},\mathrm{C}^{X}}\big)_X
\circ_{\mathsf{v}}
({\underline{\mathbf{M}}_{\mathtt{ji}}(\mathrm{F})}\mathrm{coev}_{X,X}).
\end{aligned}$}
\end{gather*}

By \eqref{eq:0.4}, we have $\mathrm{C}^{\underline{\mathbf{M}}_{\mathtt{ji}}(\mathrm{F})X}\cong (\mathrm{F}\mathrm{C}^{X}) \mathrm{F}^{\star}$ as $1$-morphisms and the comultiplication $\delta_{\underline{\mathbf{M}}_{\mathtt{ji}}(\mathrm{F})X}$ corresponds to the element
\begin{gather*}
\begin{aligned}
f:=&\big(\mu_{\mathtt{jjj}}^{(\mathrm{FC}^{X})\mathrm{F}^{\star},(\mathrm{FC}^{X})\mathrm{F}^{\star}}\big)_{\underline{\mathbf{M}}_{\mathtt{ji}}(\mathrm{F}) X}
\circ_{\mathsf{v}}
\big({\underline{\mathbf{M}}_{\mathtt{jj}}\big((\mathrm{FC}^{X})\mathrm{F}^{\star}\big)}\mathrm{coev}_{\underline{\mathbf{M}}_{\mathtt{ji}}(\mathrm{F}) X,\underline{\mathbf{M}}_{\mathtt{ji}}(\mathrm{F}) X}\big)\\
&\circ_{\mathsf{v}}
\mathrm{coev}_{\underline{\mathbf{M}}_{\mathtt{ji}}(\mathrm{F}) X,\underline{\mathbf{M}}_{\mathtt{ji}}(\mathrm{F}) X}
\end{aligned}
\end{gather*}
in $\mathrm{Hom}_{\underline{\mathbf{M}}(\mathtt{j})}\Big(\underline{\mathbf{M}}_{\mathtt{ji}}(\mathrm{F})X,
\underline{\mathbf{M}}_{\mathtt{jj}}\Big(\big((\mathrm{FC}^{X})\mathrm{F}^{\star}\big)\big((\mathrm{FC}^{X})\mathrm{F}^{\star}\big)\Big)
\underline{\mathbf{M}}_{\mathtt{ji}}(\mathrm{F})X\Big)$. It remains to show that
$f$ corresponds to the comultiplication of $(\mathrm{FC}^{X})\mathrm{F}^{\star}$ via the isomorphisms in \eqref{eq:0.4}, for $\mathrm{G}=\big((\mathrm{FC}^{X})\mathrm{F}^{\star}\big)\big((\mathrm{FC}^{X})\mathrm{F}^{\star}\big)$. First consider the diagram
\begin{gather*}
\adjustbox{scale=.51,center}{%
\begin{tikzcd}[ampersand replacement=\&, column sep=2.8em, row sep=3.5em]
\underline{\mathbf{M}}_{\mathtt{ji}}(\mathrm{F})\underline{\mathbf{M}}_{\mathtt{ii}}(\mathrm{C}^{X})X
\ar[d,"{\big(\mu_{\mathtt{jii}}^{\mathrm{F},\mathrm{C}^{X}}\big)_X}"description]
\ar[rrrr, "{\underline{\mathbf{M}}_{\mathtt{ji}}(\mathrm{F})\underline{\mathbf{M}}_{\mathtt{ii}}(\mathrm{C}^{X})}
\mathrm{coev}_{X,X}"]
\ar[rrrr, phantom,yshift=-7ex, "\circled{1}"]
\&\&\&\&
\underline{\mathbf{M}}_{\mathtt{ji}}(\mathrm{F})\underline{\mathbf{M}}_{\mathtt{ii}}(\mathrm{C}^{X})\underline{\mathbf{M}}_{\mathtt{ii}}(\mathrm{C}^{X})X
\ar[d, "\big(\mu_{\mathtt{jii}}^{\mathrm{F},\mathrm{C}^{X}}\big)_{\underline{\mathbf{M}}_{\mathtt{ii}}(\mathrm{C}^{X})X}"]
\&\&\&\&\&
\\[5ex]
\underline{\mathbf{M}}_{\mathtt{ji}}(\mathrm{F}\mathrm{C}^{X})X
\ar[d,"{\underline{\mathbf{M}}_{\mathtt{ji}}\big((\runit_{\mathrm{FC}^{X}})^{\mone}\big)_X}"description]
\ar[rrrr,"{\underline{\mathbf{M}}_{\mathtt{ji}}(\mathrm{F}\mathrm{C}^{X})}
\mathrm{coev}_{X,X}"]
\ar[rrrr, phantom,yshift=-7ex, "\circled{2}"]
\&\&\&\&
\underline{\mathbf{M}}_{\mathtt{ji}}(\mathrm{F}\mathrm{C}^{X})\underline{\mathbf{M}}_{\mathtt{ii}}(\mathrm{C}^{X})X
\ar[d,"{\underline{\mathbf{M}}_{\mathtt{ji}}\big((\runit_{\mathrm{FC}^{X}})^{\mone}\big)_{\underline{\mathbf{M}}_{\mathtt{ii}}(\mathrm{C}^{X})X}}"description]
\ar[rrrrr,"{\underline{\mathbf{M}}_{\mathtt{ji}}(\mathrm{F}\mathrm{C}^{X})}
\underline{\mathbf{M}}_{\mathtt{ii}}\big((\runit_{\mathrm{C}^{X}})^{\mone}\big)_X"]
\ar[rrrrr, phantom,yshift=-7ex, "\circled{3}"]
\&\&\&\&\&
\underline{\mathbf{M}}_{\mathtt{ji}}(\mathrm{F}\mathrm{C}^{X})\underline{\mathbf{M}}_{\mathtt{ii}}(\mathrm{C}^{X}\mathbbm{1}_{\mathtt{i}})X
\ar[d,"{\underline{\mathbf{M}}_{\mathtt{ii}}\big((\runit_{\mathrm{FC}^{X}})^{\mone}\big)_{\underline{\mathbf{M}}_{\mathtt{ii}}(\mathrm{C}^{X}\mathbbm{1}_{\mathtt{i}})X}}"description]
\\[5ex]
\underline{\mathbf{M}}_{\mathtt{ji}}\big((\mathrm{F}\mathrm{C}^{X})\mathbbm{1}_{\mathtt{i}}\big)X
\ar[d,"{\underline{\mathbf{M}}_{\mathtt{ji}}\big(\mathrm{id}_{\mathrm{F}\mathrm{C}^{X}}\circ_{\mathsf{h}}\mathrm{coev}_{\mathrm{F}}\big)_X}"description]
\ar[rrrr,"{\underline{\mathbf{M}}_{\mathtt{ji}}\left(\left(\mathrm{F}\mathrm{C}^{X}\right)\mathbbm{1}_{\mathtt{i}}\right)}
\mathrm{coev}_{X,X}"]
\ar[rrrr, phantom,yshift=-7ex, "\circled{4}"]
\&\&\&\&
\underline{\mathbf{M}}_{\mathtt{ji}}\big((\mathrm{F}\mathrm{C}^{X})\mathbbm{1}_{\mathtt{i}}\big)\underline{\mathbf{M}}_{\mathtt{ii}}(\mathrm{C}^{X})X
\ar[d,"{\underline{\mathbf{M}}_{\mathtt{ji}}\big(\mathrm{id}_{\mathrm{F}\mathrm{C}^{X}}\circ_{\mathsf{h}}\mathrm{coev}_{\mathrm{F}}\big)_{\underline{\mathbf{M}}_{\mathtt{ii}}(\mathrm{C}^{X})X}}"description]
\ar[rrrrr,"{\underline{\mathbf{M}}_{\mathtt{ji}}\left(\left(\mathrm{F}\mathrm{C}^{X}\right)\mathbbm{1}_{\mathtt{i}}\right)}
\underline{\mathbf{M}}_{\mathtt{ii}}\big((\runit_{\mathrm{C}^{X}})^{\mone}\big)_X"]
\ar[rrrrr, phantom,yshift=-7ex, "\circled{5}"]
\&\&\&\&\&
\underline{\mathbf{M}}_{\mathtt{ji}}\big((\mathrm{F}\mathrm{C}^{X})\mathbbm{1}_{\mathtt{i}}\big)
\underline{\mathbf{M}}_{\mathtt{ii}}(\mathrm{C}^{X}\mathbbm{1}_{\mathtt{i}})X
\ar[d,"{\underline{\mathbf{M}}_{\mathtt{ji}}\big(\mathrm{id}_{\mathrm{F}\mathrm{C}^{X}}\circ_{\mathsf{h}}\mathrm{coev}_{\mathrm{F}}\big)_{\underline{\mathbf{M}}_{\mathtt{ii}}(\mathrm{C}^{X}\mathbbm{1}_{\mathtt{i}})X}}"description]
\\[5ex]
\underline{\mathbf{M}}_{\mathtt{ji}}\big((\mathrm{F}\mathrm{C}^{X})(\mathrm{F}^{\star}\mathrm{F})\big)X
\ar[d, "{\underline{\mathbf{M}}_{\mathtt{ji}}\big(\alpha_{\mathrm{F}\mathrm{C}^{X},\mathrm{F}^{\star},\mathrm{F}}^{\mone}\big)_X}"description]
\ar[rrrr,"{\underline{\mathbf{M}}_{\mathtt{ji}}\left(\left(\mathrm{F}\mathrm{C}^{X}\right)\left(\mathrm{F}^{\star}\mathrm{F}\right)\right)}
\mathrm{coev}_{X,X}"]
\ar[rrrr, phantom,yshift=-7ex, "\circled{6}"]
\&\&\&\&
\underline{\mathbf{M}}_{\mathtt{ji}}\big((\mathrm{F}\mathrm{C}^{X})(\mathrm{F}^{\star}\mathrm{F})\big)\underline{\mathbf{M}}_{\mathtt{ii}}(\mathrm{C}^{X})X
\ar[d, "{\underline{\mathbf{M}}_{\mathtt{ji}}\big(\alpha_{\mathrm{F}\mathrm{C}^{X},\mathrm{F}^{\star},\mathrm{F}}^{\mone}\big)_{\underline{\mathbf{M}}_{\mathtt{ii}}(\mathrm{C}^{X})X}}"description]
\ar[rrrrr,"{\underline{\mathbf{M}}_{\mathtt{ji}}\left(\left(\mathrm{F}\mathrm{C}^{X}\right)\left(\mathrm{F}^{\star}\mathrm{F}\right)\right)}
\underline{\mathbf{M}}_{\mathtt{ii}}\big((\runit_{\mathrm{C}^{X}})^{\mone}\big)_X"]
\ar[rrrrr, phantom,yshift=-7ex, "\circled{7}"]
\&\&\&\&\&
\underline{\mathbf{M}}_{\mathtt{ji}}\big((\mathrm{F}\mathrm{C}^{X})(\mathrm{F}^{\star}\mathrm{F})\big)\underline{\mathbf{M}}_{\mathtt{ii}}(\mathrm{C}^{X}\mathbbm{1}_{\mathtt{i}})X
\ar[d, "{\underline{\mathbf{M}}_{\mathtt{ji}}\big(\alpha_{\mathrm{F}\mathrm{C}^{X},\mathrm{F}^{\star},\mathrm{F}}^{\mone}\big)_{\underline{\mathbf{M}}_{\mathtt{ii}}(\mathrm{C}^{X}\mathbbm{1}_{\mathtt{i}})X}}"description]
\\[5ex]
\underline{\mathbf{M}}_{\mathtt{ji}}\Big(\big((\mathrm{F}\mathrm{C}^{X})\mathrm{F}^{\star}\big)\mathrm{F}\Big)X
\ar[d, "{\big(\mu_{\mathtt{j}\mathtt{j}\mathtt{i}}^{(\mathrm{F}\mathrm{C}^{X})\mathrm{F}^{\star},\mathrm{F}}\big)_X^{\mone}}"description]
\ar[rrrr,"{\underline{\mathbf{M}}_{\mathtt{ji}}\left(\left(\left(\mathrm{F}\mathrm{C}^{X}\right)\mathrm{F}^{\star}\right)\mathrm{F}\right)}
\mathrm{coev}_{X,X}"]
\ar[rrrr, phantom,yshift=-7ex, "\circled{8}"]
\&\&\&\&
\underline{\mathbf{M}}_{\mathtt{ji}}\Big(\big((\mathrm{F}\mathrm{C}^{X})\mathrm{F}^{\star}\big)\mathrm{F}\Big)\underline{\mathbf{M}}_{\mathtt{ii}}(\mathrm{C}^{X})X
\ar[d, "{\big(\mu_{\mathtt{j}\mathtt{j}\mathtt{i}}^{(\mathrm{F}\mathrm{C}^{X})\mathrm{F}^{\star},\mathrm{F}}\big)^{\mone}_{\underline{\mathbf{M}}_{\mathtt{ii}}(\mathrm{C}^{X})X}}"description]
\ar[rrrrr,"{\underline{\mathbf{M}}_{\mathtt{ji}}\left(\left(\left(\mathrm{F}\mathrm{C}^{X}\right)\mathrm{F}^{\star}\right)\mathrm{F}\right)}
\underline{\mathbf{M}}_{\mathtt{ii}}\big((\runit_{\mathrm{C}^{X}})^{\mone}\big)_X"]
\ar[rrrrr, phantom,yshift=-7ex, "\circled{9}"]
\&\&\&\&\&
\underline{\mathbf{M}}_{\mathtt{ji}}\Big(\big((\mathrm{F}\mathrm{C}^{X})\mathrm{F}^{\star}\big)\mathrm{F}\Big)\underline{\mathbf{M}}_{\mathtt{ii}}(\mathrm{C}^{X}\mathbbm{1}_{\mathtt{i}})X
\ar[d, "{\big(\mu_{\mathtt{j}\mathtt{j}\mathtt{i}}^{(\mathrm{F}\mathrm{C}^{X})\mathrm{F}^{\star},\mathrm{F}}\big)^{\mone}_{\underline{\mathbf{M}}_{\mathtt{ii}}(\mathrm{C}^{X}\mathbbm{1}_{\mathtt{i}})X}}"description]
\\[5ex]
\underline{\mathbf{M}}_{\mathtt{jj}}\big((\mathrm{F}\mathrm{C}^{X})\mathrm{F}^{\star}\big)\underline{\mathbf{M}}_{\mathtt{ji}}(\mathrm{F})X
\ar[rrrr,"{\underline{\mathbf{M}}_{\mathtt{jj}}\left(\left(\mathrm{F}\mathrm{C}^{X}\right)\mathrm{F}^{\star}\right)
\underline{\mathbf{M}}_{\mathtt{ji}}(\mathrm{F})}\mathrm{coev}_{X,X}"]
\&\&\&\&
\underline{\mathbf{M}}_{\mathtt{jj}}\big((\mathrm{F}\mathrm{C}^{X})\mathrm{F}^{\star}\big)\underline{\mathbf{M}}_{\mathtt{ji}}(\mathrm{F})\underline{\mathbf{M}}_{\mathtt{ii}}(\mathrm{C}^{X})X
\ar[rrrrr,"{\underline{\mathbf{M}}_{\mathtt{jj}}\left(\left(\mathrm{F}\mathrm{C}^{X}\right)\mathrm{F}^{\star}\right)\underline{\mathbf{M}}_{\mathtt{ji}}(\mathrm{F})}
\underline{\mathbf{M}}_{\mathtt{ii}}\big((\runit_{\mathrm{C}^{X}})^{\mone}\big)_X"]
\ar[rrrrr, phantom,yshift=-7ex, "\circled{10}"]
\ar[d,"{\underline{\mathbf{M}}_{\mathtt{jj}}\left(\left(\mathrm{F}\mathrm{C}^{X}\right)\mathrm{F}^{\star}\right)\big(\mu_{\mathtt{jii}}^{\mathrm{F},\mathrm{C}^{X}}\big)_X}"description]
\&\&\&\&\&
\underline{\mathbf{M}}_{\mathtt{jj}}\big((\mathrm{F}\mathrm{C}^{X})\mathrm{F}^{\star}\big)\underline{\mathbf{M}}_{\mathtt{ji}}(\mathrm{F})\underline{\mathbf{M}}_{\mathtt{ii}}(\mathrm{C}^{X}\mathbbm{1}_{\mathtt{i}})X
\ar[d,"{\underline{\mathbf{M}}_{\mathtt{jj}}\left(\left(\mathrm{F}\mathrm{C}^{X}\right)\mathrm{F}^{\star}\right)\big(\mu_{\mathtt{jii}}^{\mathrm{F},\mathrm{C}^{X}\mathbbm{1}_{\mathtt{i}}}\big)_X}"description]
\\[5ex]
\&\&\&\&
\underline{\mathbf{M}}_{\mathtt{jj}}\big((\mathrm{F}\mathrm{C}^{X})\mathrm{F}^{\star}\big)\underline{\mathbf{M}}_{\mathtt{ji}}(\mathrm{F}\mathrm{C}^{X})X
\ar[rrrrr,"\underline{\mathbf{M}}_{\mathtt{jj}}\left(\left(\mathrm{F}\mathrm{C}^{X}\right)\mathrm{F}^{\star}\right)\underline{\mathbf{M}}_{\mathtt{ji}}(\mathrm{F})\underline{\mathbf{M}}_{\mathtt{ii}}\big(\mathrm{id}_{\mathrm{F}}\circ_{\mathsf{h}}(\runit_{\mathrm{C}^{X}})^{\mone}\big)_X"]
\ar[drrrrr,"\underline{\mathbf{M}}_{\mathtt{jj}}\left(\left(\mathrm{F}\mathrm{C}^{X}\right)\mathrm{F}^{\star}\right)\underline{\mathbf{M}}_{\mathtt{ji}}\big((\runit_{\mathrm{FC}^{X}})^{\mone}\big)_X", swap]
\ar[rrrrr, phantom,yshift=-6ex, xshift=7ex, "\circled{11}"]
\&\&\&\&\&
\underline{\mathbf{M}}_{\mathtt{jj}}\big((\mathrm{F}\mathrm{C}^{X})\mathrm{F}^{\star}\big)\underline{\mathbf{M}}_{\mathtt{ji}}\big(\mathrm{F}(\mathrm{C}^{X}\mathbbm{1}_{\mathtt{i}})\big)X
\ar[d,"{\underline{\mathbf{M}}_{\mathtt{jj}}\left(\left(\mathrm{F}\mathrm{C}^{X}\right)\mathrm{F}^{\star}\right)\underline{\mathbf{M}}_{\mathtt{ji}}\big(\alpha_{\mathrm{F},\mathrm{C}^{X},\mathbbm{1}_{\mathtt{i}}}^{\mone}\big)_X}"description]
\\[5ex]
\&\&\&\&
\&\&\&\&\&
\underline{\mathbf{M}}_{\mathtt{jj}}\big((\mathrm{F}\mathrm{C}^{X})\mathrm{F}^{\star}\big)\underline{\mathbf{M}}_{\mathtt{ji}}\big((\mathrm{F}\mathrm{C}^{X})\mathbbm{1}_{\mathtt{i}}\big)X
\end{tikzcd}}
\end{gather*}
which commutes by
\begin{itemize}
\item naturality of $\mu_{\mathtt{jii}}^{\mathrm{F},\mathrm{C}^{X}}$ for the facet labeled $1$;

\item naturality of $\underline{\mathbf{M}}_{\mathtt{ii}}\big((\runit_{\mathrm{FC}^{X}})^{\mone}\big)$ for the facets labeled $2$ and $3$;

\item naturality of $\underline{\mathbf{M}}_{\mathtt{ji}}\big(\mathrm{id}_{\mathrm{F}\mathrm{C}^{X}}\circ_{\mathsf{h}}\mathrm{coev}_{\mathrm{F}}\big)$
for the facets labeled $4$ and $5$;

\item naturality of $\underline{\mathbf{M}}_{\mathtt{ji}}\big(\alpha_{\mathrm{F}\mathrm{C}^{X},\mathrm{F}^{\star},\mathrm{F}}^{\mone}\big)$ for the facets labeled $6$ and $7$;

\item naturality of $\mu_{\mathtt{j}\mathtt{j}\mathtt{i}}^{(\mathrm{F}\mathrm{C}^{X})\mathrm{F}^{\star},\mathrm{F}}$ for the facets labeled $8$ and $9$;

\item the diagram below \eqref{eq:0.100} for the facets labeled $10$ and $11$.
\end{itemize}
To simplify notation, we set $\mathrm{H}_1:=\mathrm{F}^{\star}\mathrm{F}$ and
$\mathrm{H}_2:=(\mathrm{F}\mathrm{C}^{X})\mathrm{F}^{\star}$.
Then $\mathrm{coev}_{\mathrm{F}}$ is a $2$-morphism from $\mathbbm{1}_{\mathtt{i}}$ to $\mathrm{H}_1$
and $\alpha_{\mathrm{FC}^{X},\mathrm{F}^{\star},\mathrm{F}}^{\mone}$ is a $2$-morphism from $(\mathrm{F}\mathrm{C}^{X})\mathrm{H}_1$ to
$\mathrm{H}_2\mathrm{F}$. Consider the diagram
\begin{gather*}
\scalebox{0.47}{
$\begin{tikzcd}[ampersand replacement=\&, column sep=2.5em, row sep=3.5em]
\underline{\mathbf{M}}_{\mathtt{ji}}(\mathrm{F}\mathrm{C}^{X})\underline{\mathbf{M}}_{\mathtt{ii}}(\mathrm{C}^{X}\mathbbm{1}_{\mathtt{i}})X
\ar[d,"{\underline{\mathbf{M}}_{\mathtt{ii}}\big((\runit_{\mathrm{FC}^{X}})^{\mone}\big)_{\underline{\mathbf{M}}_{\mathtt{ii}}(\mathrm{C}^{X}\mathbbm{1}_{\mathtt{i}})X}}"description]
\ar[rrrrrr,"{\underline{\mathbf{M}}_{\mathtt{ji}}(\mathrm{F}\mathrm{C}^{X})}\underline{\mathbf{M}}_{\mathtt{ii}}\big(\mathrm{id}_{\mathrm{C}^{X}}\circ_{\mathsf{h}}
\mathrm{coev}_{\mathrm{F}}\big)_X"]
\ar[rrrrrr, phantom,yshift=-7ex, "\circled{1}"]
\&\&\&\&\&\&
\underline{\mathbf{M}}_{\mathtt{ji}}(\mathrm{F}\mathrm{C}^{X})\underline{\mathbf{M}}_{\mathtt{ii}}(\mathrm{C}^{X}\mathrm{H}_1)X
\ar[rrrrr, "{\underline{\mathbf{M}}_{\mathtt{ji}}(\mathrm{F}\mathrm{C}^{X})}
\underline{\mathbf{M}}_{\mathtt{ii}}\big(\alpha_{\mathrm{C}^{X},\mathrm{F}^{\star},\mathrm{F}}^{\mone}\big)_X"]
\ar[d,"{\underline{\mathbf{M}}_{\mathtt{ii}}\big((\runit_{\mathrm{FC}^{X}})^{\mone}\big)_{\underline{\mathbf{M}}_{\mathtt{ii}}(\mathrm{C}^{X}\mathrm{H}_1)X}}"description]
\ar[rrrrr, phantom,yshift=-7ex, "\circled{2}"]
\&\&\&\&\&
\underline{\mathbf{M}}_{\mathtt{ji}}(\mathrm{F}\mathrm{C}^{X})\underline{\mathbf{M}}_{\mathtt{ii}}\big((\mathrm{C}^{X}\mathrm{F}^{\star})\mathrm{F}\big)X
\ar[d,"{\underline{\mathbf{M}}_{\mathtt{ii}}\big((\runit_{\mathrm{FC}^{X}})^{\mone}\big)_{\underline{\mathbf{M}}_{\mathtt{ii}}\left(\left(\mathrm{C}^{X}\mathrm{F}^{\star}\right)\mathrm{F}\right)X}}"description]
\\[5ex]
\underline{\mathbf{M}}_{\mathtt{ji}}\big((\mathrm{F}\mathrm{C}^{X})\mathbbm{1}_{\mathtt{i}}\big)
\underline{\mathbf{M}}_{\mathtt{ii}}(\mathrm{C}^{X}\mathbbm{1}_{\mathtt{i}})X
\ar[d,"{\underline{\mathbf{M}}_{\mathtt{ji}}\big(\mathrm{id}_{\mathrm{F}\mathrm{C}^{X}}\circ_{\mathsf{h}}\mathrm{coev}_{\mathrm{F}}\big)_{\underline{\mathbf{M}}_{\mathtt{ii}}(\mathrm{C}^{X}\mathbbm{1}_{\mathtt{i}})X}}"description]
\ar[rrrrrr,"{\underline{\mathbf{M}}_{\mathtt{ji}}\left((\mathrm{F}\mathrm{C}^{X})\mathbbm{1}_{\mathtt{i}}\right)}\underline{\mathbf{M}}_{\mathtt{ii}}\big(\mathrm{id}_{\mathrm{C}^{X}}\circ_{\mathsf{h}}
\mathrm{coev}_{\mathrm{F}}\big)_X"]
\ar[rrrrrr, phantom,yshift=-7ex, "\circled{3}"]
\&\&\&\&\&\&
\underline{\mathbf{M}}_{\mathtt{ji}}\big((\mathrm{F}\mathrm{C}^{X})\mathbbm{1}_{\mathtt{i}}\big)\underline{\mathbf{M}}_{\mathtt{ii}}(\mathrm{C}^{X}\mathrm{H}_1)X
\ar[d, "{\underline{\mathbf{M}}_{\mathtt{ji}}(\mathrm{id}_{\mathrm{F}\mathrm{C}^{X}}\circ_{\mathsf{h}}\mathrm{coev}_{\mathrm{F}})_{\underline{\mathbf{M}}_{\mathtt{ii}}(\mathrm{C}^{X}\mathrm{H}_1)X}}"description]
\ar[rrrrr, "{\underline{\mathbf{M}}_{\mathtt{ji}}\left(\left(\mathrm{F}\mathrm{C}^{X}\right)\mathbbm{1}_{\mathtt{i}}\right)}
\underline{\mathbf{M}}_{\mathtt{ii}}\big(\alpha_{\mathrm{C}^{X},\mathrm{F}^{\star},\mathrm{F}}^{\mone}\big)_X"]
\ar[rrrrr, phantom,yshift=-7ex, "\circled{4}"]
\&\&\&\&\&
\underline{\mathbf{M}}_{\mathtt{ji}}\big((\mathrm{F}\mathrm{C}^{X})\mathbbm{1}_{\mathtt{i}}\big)\underline{\mathbf{M}}_{\mathtt{ii}}\big((\mathrm{C}^{X}\mathrm{F}^{\star})\mathrm{F}\big)X
\ar[d, "{\underline{\mathbf{M}}_{\mathtt{ji}}(\mathrm{id}_{\mathrm{F}\mathrm{C}^{X}}\circ_{\mathsf{h}}\mathrm{coev}_{\mathrm{F}})_{\underline{\mathbf{M}}_{\mathtt{ii}}\left(\left(\mathrm{C}^{X}\mathrm{F}^{\star}\right)\mathrm{F}\right)X}}"description]
\\[5ex]
\underline{\mathbf{M}}_{\mathtt{ji}}\big((\mathrm{F}\mathrm{C}^{X})\mathrm{H}_1\big)\underline{\mathbf{M}}_{\mathtt{ii}}(\mathrm{C}^{X}\mathbbm{1}_{\mathtt{i}})X
\ar[d, "{\underline{\mathbf{M}}_{\mathtt{ji}}\big(\alpha_{\mathrm{F}\mathrm{C}^{X},\mathrm{F}^{\star},\mathrm{F}}^{\mone}\big)_{\underline{\mathbf{M}}_{\mathtt{ii}}(\mathrm{C}^{X}\mathbbm{1}_{\mathtt{i}})X}}"description]
\ar[rrrrrr,"{\underline{\mathbf{M}}_{\mathtt{ji}}\left((\mathrm{F}\mathrm{C}^{X})\mathrm{H}_1\right)}\underline{\mathbf{M}}_{\mathtt{ii}}\big(\mathrm{id}_{\mathrm{C}^{X}}\circ_{\mathsf{h}}
\mathrm{coev}_{\mathrm{F}}\big)_X"]
\ar[rrrrrr, phantom,yshift=-7ex, "\circled{5}"]
\&\&\&\&\&\&
\underline{\mathbf{M}}_{\mathtt{ji}}\big((\mathrm{F}\mathrm{C}^{X})\mathrm{H}_1\big)\underline{\mathbf{M}}_{\mathtt{ii}}(\mathrm{C}^{X}\mathrm{H}_1)X
\ar[rrrrr, "{\underline{\mathbf{M}}_{\mathtt{ji}}\left(\left(\mathrm{F}\mathrm{C}^{X}\right)\mathrm{H}_1\right)}
\underline{\mathbf{M}}_{\mathtt{ii}}\big(\alpha_{\mathrm{C}^{X},\mathrm{F}^{\star},\mathrm{F}}^{\mone}\big)_X"]
\ar[d, "{\underline{\mathbf{M}}_{\mathtt{ji}}\big(\alpha_{\mathrm{F}\mathrm{C}^{X},\mathrm{F}^{\star},\mathrm{F}}^{\mone}\big)_{\underline{\mathbf{M}}_{\mathtt{ii}}(\mathrm{C}^{X}\mathrm{H}_1)X}}"description]
\ar[rrrrr, phantom,yshift=-7ex, "\circled{6}"]
\&\&\&\&\&
\underline{\mathbf{M}}_{\mathtt{ji}}\big((\mathrm{F}\mathrm{C}^{X})\mathrm{H}_1\big)
\underline{\mathbf{M}}_{\mathtt{ii}}\big((\mathrm{C}^{X}\mathrm{F}^{\star})\mathrm{F}\big)X
\ar[d, "{\underline{\mathbf{M}}_{\mathtt{ji}}\big(\alpha_{\mathrm{F}\mathrm{C}^{X},\mathrm{F}^{\star},\mathrm{F}}^{\mone}\big)_{\underline{\mathbf{M}}_{\mathtt{ii}}\left(\left(\mathrm{C}^{X}\mathrm{F}^{\star}\right)\mathrm{F}\right)X}}"description]
\\[5ex]
\underline{\mathbf{M}}_{\mathtt{ji}}(\mathrm{H}_2\mathrm{F})\underline{\mathbf{M}}_{\mathtt{ii}}(\mathrm{C}^{X}\mathbbm{1}_{\mathtt{i}})X
\ar[d, "{\big(\mu_{\mathtt{j}\mathtt{j}\mathtt{i}}^{\mathrm{H}_2,\mathrm{F}}\big)^{\mone}_{\underline{\mathbf{M}}_{\mathtt{ii}}(\mathrm{C}^{X}\mathbbm{1}_{\mathtt{i}})X}}"description]
\ar[rrrrrr,"{\underline{\mathbf{M}}_{\mathtt{ji}}(\mathrm{H}_2\mathrm{F})}
\underline{\mathbf{M}}_{\mathtt{ii}}\big(\mathrm{id}_{\mathrm{C}^{X}}\circ_{\mathsf{h}}\mathrm{coev}_{\mathrm{F}}\big)_X"]
\ar[rrrrrr, phantom,yshift=-7ex, "\circled{7}"]
\&\&\&\&\&\&
\underline{\mathbf{M}}_{\mathtt{ji}}(\mathrm{H}_2\mathrm{F})\underline{\mathbf{M}}_{\mathtt{ii}}(\mathrm{C}^{X}\mathrm{H}_1)X
\ar[d, "{\big(\mu_{\mathtt{j}\mathtt{j}\mathtt{i}}^{\mathrm{H}_2,\mathrm{F}}\big)^{\mone}_{\underline{\mathbf{M}}_{\mathtt{ii}}(\mathrm{C}^{X}\mathrm{H}_1)X}}"description]
\ar[rrrrr, "{\underline{\mathbf{M}}_{\mathtt{ji}}(\mathrm{H}_2\mathrm{F})}
\underline{\mathbf{M}}_{\mathtt{ii}}\big(\alpha_{\mathrm{C}^{X},\mathrm{F}^{\star},\mathrm{F}}^{\mone}\big)_X"]
\ar[rrrrr, phantom,yshift=-7ex, "\circled{8}"]
\&\&\&\&\&
\underline{\mathbf{M}}_{\mathtt{ji}}(\mathrm{H}_2\mathrm{F})\underline{\mathbf{M}}_{\mathtt{ii}}\big((\mathrm{C}^{X}\mathrm{F}^{\star})\mathrm{F}\big)X
\ar[d, "{\big(\mu_{\mathtt{j}\mathtt{j}\mathtt{i}}^{\mathrm{H}_2,\mathrm{F}}\big)^{\mone}_{\underline{\mathbf{M}}_{\mathtt{ii}}\left(\left(\mathrm{C}^{X}\mathrm{F}^{\star}\right)\mathrm{F}\right)X}}"description]
\\[5ex]
\underline{\mathbf{M}}_{\mathtt{jj}}(\mathrm{H}_2)\underline{\mathbf{M}}_{\mathtt{ji}}(\mathrm{F})\underline{\mathbf{M}}_{\mathtt{ii}}(\mathrm{C}^{X}\mathbbm{1}_{\mathtt{i}})X
\ar[rrrrrr,"{\underline{\mathbf{M}}_{\mathtt{jj}}(\mathrm{H}_2)\underline{\mathbf{M}}_{\mathtt{ji}}(\mathrm{F})}
\underline{\mathbf{M}}_{\mathtt{ii}}\big(\mathrm{id}_{\mathrm{C}^{X}}\circ_{\mathsf{h}}\mathrm{coev}_{\mathrm{F}}\big)_X"]
\ar[d, "{{\underline{\mathbf{M}}_{\mathtt{jj}}(\mathrm{H}_2)}
\big(\mu_{\mathtt{j}\mathtt{i}\mathtt{i}}^{\mathrm{F},\mathrm{C}^{X}\mathbbm{1}_{\mathtt{i}}}\big)_X}"description]
\ar[rrrrrr, phantom,yshift=-7ex, "\circled{9}"]
\&\&\&\&\&\&
\underline{\mathbf{M}}_{\mathtt{jj}}(\mathrm{H}_2)\underline{\mathbf{M}}_{\mathtt{ji}}(\mathrm{F})\underline{\mathbf{M}}_{\mathtt{ii}}(\mathrm{C}^{X}\mathrm{H}_1)X
\ar[d, "{{\underline{\mathbf{M}}_{\mathtt{jj}}(\mathrm{H}_2)}
\big(\mu_{\mathtt{j}\mathtt{i}\mathtt{i}}^{\mathrm{F},\mathrm{C}^{X}\mathrm{H}_1}\big)_X}"description]
\ar[rrrrr,"{\underline{\mathbf{M}}_{\mathtt{jj}}(\mathrm{H}_2)\underline{\mathbf{M}}_{\mathtt{ji}}(\mathrm{F})}
\underline{\mathbf{M}}_{\mathtt{ji}}\big(\alpha_{\mathrm{C}^{X},\mathrm{F}^{\star},\mathrm{F}}^{\mone}\big)_X"]
\ar[rrrrr, phantom,yshift=-7ex, "\circled{10}"]
\&\&\&\&\&
\underline{\mathbf{M}}_{\mathtt{jj}}(\mathrm{H}_2)\underline{\mathbf{M}}_{\mathtt{ji}}(\mathrm{F})\underline{\mathbf{M}}_{\mathtt{ii}}\big((\mathrm{C}^{X}\mathrm{F}^{\star})\mathrm{F}\big)X
\ar[d, "{{\underline{\mathbf{M}}_{\mathtt{jj}}(\mathrm{H}_2)}
\big(\mu_{\mathtt{j}\mathtt{i}\mathtt{i}}^{\mathrm{F},(\mathrm{C}^{X}\mathrm{F}^{\star})\mathrm{F}}\big)_X}"description]
\\[5ex]
\underline{\mathbf{M}}_{\mathtt{jj}}(\mathrm{H}_2)\underline{\mathbf{M}}_{\mathtt{ji}}\big(\mathrm{F}(\mathrm{C}^{X}\mathbbm{1}_{\mathtt{i}})\big)X
\ar[rrrrrr,"{\underline{\mathbf{M}}_{\mathtt{jj}}(\mathrm{H}_2)}
\underline{\mathbf{M}}_{\mathtt{ii}}\big(\mathrm{id}_{\mathrm{F}}\circ_{\mathsf{h}}(\mathrm{id}_{\mathrm{C}^{X}}\circ_{\mathsf{h}}\mathrm{coev}_{\mathrm{F}})\big)_X"]
\ar[d, "{{\underline{\mathbf{M}}_{\mathtt{jj}}(\mathrm{H}_2)}\underline{\mathbf{M}}_{\mathtt{ji}}(\alpha_{\mathrm{F},\mathrm{C}^{X},\mathbbm{1}_{\mathtt{i}}}^{\mone})_X}"description]
\ar[rrrrrr, phantom,yshift=-7ex, "\circled{11}"]
\&\&\&\&\&\&
\underline{\mathbf{M}}_{\mathtt{jj}}(\mathrm{H}_2)\underline{\mathbf{M}}_{\mathtt{ji}}\big(\mathrm{F}(\mathrm{C}^{X}\mathrm{H}_1)\big)X
\ar[d, "{{\underline{\mathbf{M}}_{\mathtt{jj}}(\mathrm{H}_2)}\underline{\mathbf{M}}_{\mathtt{ji}}(\alpha_{\mathrm{F},\mathrm{C}^{X},\mathrm{H}_1}^{\mone})_X}"description]
\ar[rrrrr,"{\underline{\mathbf{M}}_{\mathtt{jj}}(\mathrm{H}_2)}
\underline{\mathbf{M}}_{\mathtt{ji}}\big(\mathrm{id}_{\mathrm{F}}\circ_{\mathsf{h}}\alpha_{\mathrm{C}^{X},\mathrm{F}^{\star},\mathrm{F}}^{\mone}\big)_X"]
\ar[rrrrr, phantom,yshift=-14ex, "\circled{12}"]
\&\&\&\&\&
\underline{\mathbf{M}}_{\mathtt{jj}}(\mathrm{H}_2)\underline{\mathbf{M}}_{\mathtt{ji}}\Big(\mathrm{F}\big((\mathrm{C}^{X}\mathrm{F}^{\star})\mathrm{F}\big)\Big)X
\ar[d,"{{\underline{\mathbf{M}}_{\mathtt{jj}}(\mathrm{H}_2)}
\underline{\mathbf{M}}_{\mathtt{ji}}\big(\alpha_{\mathrm{F},\mathrm{C}^{X}\mathrm{F}^{\star},\mathrm{F}}^{\mone}\big)_X}"description]
\\[5ex]
\underline{\mathbf{M}}_{\mathtt{jj}}(\mathrm{H}_2)\underline{\mathbf{M}}_{\mathtt{ji}}\big((\mathrm{F}\mathrm{C}^{X})\mathbbm{1}_{\mathtt{i}}\big)X
\ar[rrrrrr,"{\underline{\mathbf{M}}_{\mathtt{jj}}(\mathrm{H}_2)}
\underline{\mathbf{M}}_{\mathtt{ii}}\big(\mathrm{id}_{\mathrm{F}\mathrm{C}^{X}}\circ_{\mathsf{h}}\mathrm{coev}_{\mathrm{F}}\big)_X",swap]
\&\&\&\&\&\&
\underline{\mathbf{M}}_{\mathtt{jj}}(\mathrm{H}_2)\underline{\mathbf{M}}_{\mathtt{ji}}\big((\mathrm{F}\mathrm{C}^{X})\mathrm{H}_1)\big)X
\ar[drrrrr,"{\underline{\mathbf{M}}_{\mathtt{jj}}(\mathrm{H}_2)}
\underline{\mathbf{M}}_{\mathtt{ji}}(\alpha_{\mathrm{F}\mathrm{C}^{X},\mathrm{F}^{\star},\mathrm{F}}^{\mone})_X", swap]
\&\&\&\&\&
\underline{\mathbf{M}}_{\mathtt{jj}}(\mathrm{H}_2)\underline{\mathbf{M}}_{\mathtt{ji}}\Big(\big(\mathrm{F}(\mathrm{C}^{X}\mathrm{F}^{\star})\big)\mathrm{F}\Big)X
\ar[d,"{{\underline{\mathbf{M}}_{\mathtt{jj}}(\mathrm{H}_2)}\underline{\mathbf{M}}_{\mathtt{ji}}\big(\alpha_{\mathrm{F},\mathrm{C}^{X},\mathrm{F}^{\star}}^{\mone}\circ_{\mathsf{h}}\mathrm{id}_{\mathrm{F}}\big)_X}"description]
\\[5ex]
\&\&\&\&\&\&
\&\&\&\&\&
\underline{\mathbf{M}}_{\mathtt{jj}}(\mathrm{H}_2)\underline{\mathbf{M}}_{\mathtt{ji}}(\mathrm{H}_2\mathrm{F})X
\end{tikzcd}$}
.
\end{gather*}
This diagram commutes by
\begin{itemize}
\item naturality of $\underline{\mathbf{M}}_{\mathtt{ii}}\big((\runit_{\mathrm{FC}^{X}})^{\mone}\big)$ for the facets labeled $1$ and $2$;

\item naturality of $\underline{\mathbf{M}}_{\mathtt{ji}}\big(\mathrm{id}_{\mathrm{F}\mathrm{C}^{X}}\circ_{\mathsf{h}}\mathrm{coev}_{\mathrm{F}}\big)$
for the facets labeled $3$ and $4$;

\item naturality of $\underline{\mathbf{M}}_{\mathtt{ji}}\big(\alpha_{\mathrm{F}\mathrm{C}^{X},\mathrm{F}^{\star},\mathrm{F}}^{\mone}\big)$ for the facets labeled $5$ and $6$;

\item naturality of $\big(\mu_{\mathtt{jji}}^{\mathrm{H}_2,\mathrm{F}}\big)^{\mone}$ for the facets labeled $7$ and $8$;

\item naturality of $\mu_{\mathtt{j}\mathtt{i}\mathtt{i}}$ for the facets labeled $9$ and $10$;

\item naturality of $\alpha^{\mone}$ for the facet labeled $11$;

\item the pentagon coherence condition of the associator for the facet labeled $12$.
\end{itemize}
Further, we consider
\begin{gather*}
\scalebox{0.47}{
$\begin{tikzcd}[ampersand replacement=\&, column sep=2.5em, row sep=3.5em]
\underline{\mathbf{M}}_{\mathtt{ji}}(\mathrm{F}\mathrm{C}^{X})\underline{\mathbf{M}}_{\mathtt{ii}}\big((\mathrm{C}^{X}\mathrm{F}^{\star})\mathrm{F}\big)X
\ar[d,"{\underline{\mathbf{M}}_{\mathtt{ji}}\big((\runit_{\mathrm{FC}^{X}})^{\mone}\big)_{\underline{\mathbf{M}}_{\mathtt{ii}}\left(\left(\mathrm{C}^{X}\mathrm{F}^{\star}\right)\mathrm{F}\right)X}}"description]
\ar[rrrrr,"{\underline{\mathbf{M}}_{\mathtt{ji}}(\mathrm{F}\mathrm{C}^{X})}\big(\mu_{\mathtt{i}\mathtt{j}\mathtt{i}}^{\mathrm{C}^{X}\mathrm{F}^{\star},\mathrm{F}}\big)_X^{\mone}"]
\ar[rrrrr, phantom,yshift=-8ex, "\circled{1}"]
\&\&\&\&\&
\underline{\mathbf{M}}_{\mathtt{ji}}(\mathrm{F}\mathrm{C}^{X})\underline{\mathbf{M}}_{\mathtt{ij}}(\mathrm{C}^{X}\mathrm{F}^{\star})\underline{\mathbf{M}}_{\mathtt{ji}}(\mathrm{F})X
\ar[d,"{\underline{\mathbf{M}}_{\mathtt{ji}}\big((\runit_{\mathrm{FC}^{X}})^{\mone}\big)_{\underline{\mathbf{M}}_{\mathtt{ij}}(\mathrm{C}^{X}\mathrm{F}^{\star})\underline{\mathbf{M}}_{\mathtt{ji}}(\mathrm{F})X}}"description]
\ar[rrrrr,"\big(\mu_{\mathtt{jij}}^{\mathrm{F}\mathrm{C}^{X},\mathrm{C}^{X}\mathrm{F}^{\star}}\big)_{\underline{\mathbf{M}}_{\mathtt{ji}}(\mathrm{F})X}"]
\ar[rrrrr, phantom,yshift=-8ex, "\circled{2}"]
\&\&\&\&\&
\underline{\mathbf{M}}_{\mathtt{jj}}\big((\mathrm{F}\mathrm{C}^{X})(\mathrm{C}^{X}\mathrm{F}^{\star})\big)\underline{\mathbf{M}}_{\mathtt{ji}}(\mathrm{F})X
\ar[d,"{\underline{\mathbf{M}}_{\mathtt{jj}}\big((\runit_{\mathrm{FC}^{X}})^{\mone}\circ_{\mathsf{h}}\mathrm{id}_{\mathrm{C}^{X}\mathrm{F}^{\star}}\big)_{\underline{\mathbf{M}}_{\mathtt{ji}}(\mathrm{F})X}}"description]
\\[5ex]
\underline{\mathbf{M}}_{\mathtt{ji}}\big((\mathrm{F}\mathrm{C}^{X})\mathbbm{1}_{\mathtt{i}}\big)\underline{\mathbf{M}}_{\mathtt{ii}}\big((\mathrm{C}^{X}\mathrm{F}^{\star})\mathrm{F}\big)X
\ar[d, "{\underline{\mathbf{M}}_{\mathtt{ji}}\big(\mathrm{id}_{\mathrm{F}\mathrm{C}^{X}}\circ_{\mathsf{h}}\mathrm{coev}_{\mathrm{F}}\big)_{\underline{\mathbf{M}}_{\mathtt{ii}}\left(\left(\mathrm{C}^{X}\mathrm{F}^{\star}\right)\mathrm{F}\right)X}}"description]
\ar[rrrrr,"{\underline{\mathbf{M}}_{\mathtt{ji}}\left(\left(\mathrm{F}\mathrm{C}^{X}\right)\mathbbm{1}_{\mathtt{i}}\right)}\big(\mu_{\mathtt{i}\mathtt{j}\mathtt{i}}^{\mathrm{C}^{X}\mathrm{F}^{\star},\mathrm{F}}\big)_X^{\mone}"]
\ar[rrrrr, phantom,yshift=-8ex, "\circled{3}"]
\&\&\&\&\&
\underline{\mathbf{M}}_{\mathtt{ji}}\big((\mathrm{F}\mathrm{C}^{X})\mathbbm{1}_{\mathtt{i}}\big)\underline{\mathbf{M}}_{\mathtt{ij}}(\mathrm{C}^{X}\mathrm{F}^{\star})\underline{\mathbf{M}}_{\mathtt{ji}}(\mathrm{F})X
\ar[d, "{\underline{\mathbf{M}}_{\mathtt{jj}}(\mathrm{id}_{\mathrm{F}\mathrm{C}^{X}}\circ_{\mathsf{h}}\mathrm{coev}_{\mathrm{F}})_{\underline{\mathbf{M}}_{\mathtt{ij}}(\mathrm{C}^{X}\mathrm{F}^{\star})\underline{\mathbf{M}}_{\mathtt{ji}}(\mathrm{F})X}}"description]
\ar[rrrrr,"\big(\mu_{\mathtt{jij}}^{(\mathrm{F}\mathrm{C}^{X})\mathbbm{1}_{\mathtt{i}},\mathrm{C}^{X}\mathrm{F}^{\star}}\big)_{\underline{\mathbf{M}}_{\mathtt{ji}}(\mathrm{F})X}"]
\ar[rrrrr, phantom,yshift=-8ex, "\circled{4}"]
\&\&\&\&\&
\underline{\mathbf{M}}_{\mathtt{jj}}\Big(\big((\mathrm{F}\mathrm{C}^{X})\mathbbm{1}_{\mathtt{i}}\big)(\mathrm{C}^{X}\mathrm{F}^{\star})\Big)\underline{\mathbf{M}}_{\mathtt{ji}}(\mathrm{F})X
\ar[d, "{\underline{\mathbf{M}}_{\mathtt{jj}}\big((\mathrm{id}_{\mathrm{F}\mathrm{C}^{X}}\circ_{\mathsf{h}}\mathrm{coev}_{\mathrm{F}})_{\mathrm{C}^{X}\mathrm{F}^{\star}}\big)_{\underline{\mathbf{M}}_{\mathtt{ji}}(\mathrm{F})X}}"description]
\\[5ex]
\underline{\mathbf{M}}_{\mathtt{ji}}\big((\mathrm{F}\mathrm{C}^{X})\mathrm{H}_1\big)
\underline{\mathbf{M}}_{\mathtt{ii}}\big((\mathrm{C}^{X}\mathrm{F}^{\star})\mathrm{F}\big)X
\ar[d, "{\underline{\mathbf{M}}_{\mathtt{ji}}\big(\alpha_{\mathrm{F}\mathrm{C}^{X},\mathrm{F}^{\star},\mathrm{F}}^{\mone}\big)_{\underline{\mathbf{M}}_{\mathtt{ii}}\left(\left(\mathrm{C}^{X}\mathrm{F}^{\star}\right)\mathrm{F}\right)X}}"description]
\ar[rrrrr,"{\underline{\mathbf{M}}_{\mathtt{ji}}\left(\left(\mathrm{F}\mathrm{C}^{X}\right)\mathrm{H}_1\right)}\big(\mu_{\mathtt{i}\mathtt{j}\mathtt{i}}^{\mathrm{C}^{X}\mathrm{F}^{\star},\mathrm{F}}\big)_X^{\mone}"]
\ar[rrrrr, phantom,yshift=-8ex, "\circled{5}"]
\&\&\&\&\&
\underline{\mathbf{M}}_{\mathtt{ji}}\big((\mathrm{F}\mathrm{C}^{X})\mathrm{H}_1\big)
\underline{\mathbf{M}}_{\mathtt{ij}}(\mathrm{C}^{X}\mathrm{F}^{\star})\underline{\mathbf{M}}_{\mathtt{ji}}(\mathrm{F})X
\ar[d, "{\underline{\mathbf{M}}_{\mathtt{ji}}\big(\alpha_{\mathrm{F}\mathrm{C}^{X},\mathrm{F}^{\star},\mathrm{F}}^{\mone}\big)_{\underline{\mathbf{M}}_{\mathtt{ij}}(\mathrm{C}^{X}\mathrm{F}^{\star})\underline{\mathbf{M}}_{\mathtt{ji}}(\mathrm{F})X}}"description]
\ar[rrrrr,"\big(\mu_{\mathtt{jij}}^{(\mathrm{F}\mathrm{C}^{X})\mathrm{H}_1,\mathrm{C}^{X}\mathrm{F}^{\star}}\big)_{\underline{\mathbf{M}}_{\mathtt{ji}}(\mathrm{F})X}"]
\ar[rrrrr, phantom,yshift=-8ex, "\circled{6}"]
\&\&\&\&\&
\underline{\mathbf{M}}_{\mathtt{jj}}\Big(\big((\mathrm{F}\mathrm{C}^{X})\mathrm{H}_1\big)(\mathrm{C}^{X}\mathrm{F}^{\star})\Big)\underline{\mathbf{M}}_{\mathtt{ji}}(\mathrm{F})X
\ar[d, "{\underline{\mathbf{M}}_{\mathtt{jj}}\big(\alpha_{\mathrm{F}\mathrm{C}^{X},\mathrm{F}^{\star},\mathrm{F}}^{\mone}\circ_{\mathsf{h}}
\mathrm{id}_{\mathrm{C}^{X}\mathrm{F}^{\star}}\big)_{\underline{\mathbf{M}}_{\mathtt{ji}}(\mathrm{F})X}}"description]
\\[5ex]
\underline{\mathbf{M}}_{\mathtt{ji}}(\mathrm{H}_2\mathrm{F})\underline{\mathbf{M}}_{\mathtt{ii}}\big((\mathrm{C}^{X}\mathrm{F}^{\star})\mathrm{F}\big)X
\ar[d, "{\big(\mu_{\mathtt{j}\mathtt{j}\mathtt{i}}^{\mathrm{H}_2,\mathrm{F}}\big)^{\mone}_{\underline{\mathbf{M}}_{\mathtt{ii}}\left(\left(\mathrm{C}^{X}\mathrm{F}^{\star}\right)\mathrm{F}\right)X}}"description]
\ar[rrrrr,"{\underline{\mathbf{M}}_{\mathtt{jj}}(\mathrm{H}_2\mathrm{F})}\big(\mu_{\mathtt{i}\mathtt{j}\mathtt{i}}^{\mathrm{C}^{X}\mathrm{F}^{\star},\mathrm{F}}\big)_X^{\mone}"]
\ar[rrrrr, phantom,yshift=-8ex, "\circled{7}"]
\&\&\&\&\&
\underline{\mathbf{M}}_{\mathtt{ji}}(\mathrm{H}_2\mathrm{F})\underline{\mathbf{M}}_{\mathtt{ij}}(\mathrm{C}^{X}\mathrm{F}^{\star})\underline{\mathbf{M}}_{\mathtt{ji}}(\mathrm{F})X
\ar[d, "{\big(\mu_{\mathtt{j}\mathtt{j}\mathtt{i}}^{\mathrm{H}_2,\mathrm{F}}\big)^{\mone}_{\underline{\mathbf{M}}_{\mathtt{ij}}(\mathrm{C}^{X}\mathrm{F}^{\star})\underline{\mathbf{M}}_{\mathtt{ji}}(\mathrm{F})X}}"description]
\ar[rrrrr,"\big(\mu_{\mathtt{jij}}^{\mathrm{H}_2\mathrm{F},\mathrm{C}^{X}\mathrm{F}^{\star}}\big)_{\underline{\mathbf{M}}_{\mathtt{ji}}(\mathrm{F})X}"]
\ar[rrrrr, phantom,yshift=-27ex, "\circled{9}"]
\&\&\&\&\&
\underline{\mathbf{M}}_{\mathtt{jj}}\big((\mathrm{H}_2\mathrm{F})(\mathrm{C}^{X}\mathrm{F}^{\star})\big)\underline{\mathbf{M}}_{\mathtt{ji}}(\mathrm{F})X
\ar[ddd, "{\underline{\mathbf{M}}_{\mathtt{jj}}\big(\alpha_{\mathrm{H}_2,\mathrm{F},\mathrm{C}^{X}\mathrm{F}^{\star}}\big)_{\underline{\mathbf{M}}_{\mathtt{ji}}(\mathrm{F})X}}"description]
\\[5ex]
\underline{\mathbf{M}}_{\mathtt{jj}}(\mathrm{H}_2)\underline{\mathbf{M}}_{\mathtt{ji}}(\mathrm{F})\underline{\mathbf{M}}_{\mathtt{ii}}\big((\mathrm{C}^{X}\mathrm{F}^{\star})\mathrm{F}\big)X
\ar[d, "{{\underline{\mathbf{M}}_{\mathtt{jj}}(\mathrm{H}_2)}
\big(\mu_{\mathtt{j}\mathtt{i}\mathtt{i}}^{\mathrm{F},(\mathrm{C}^{X}\mathrm{F}^{\star})\mathrm{F}}\big)_X}"description]
\ar[rrrrr,"{\underline{\mathbf{M}}_{\mathtt{jj}}(\mathrm{H}_2)\underline{\mathbf{M}}_{\mathtt{ji}}(\mathrm{F})}\big(\mu_{\mathtt{i}\mathtt{j}\mathtt{i}}^{\mathrm{C}^{X}\mathrm{F}^{\star},\mathrm{F}}\big)_X^{\mone}"]
\ar[rrrrr, phantom,yshift=-18ex, "\circled{8}"]
\&\&\&\&\&
\underline{\mathbf{M}}_{\mathtt{jj}}(\mathrm{H}_2)\underline{\mathbf{M}}_{\mathtt{ji}}(\mathrm{F})\underline{\mathbf{M}}_{\mathtt{ij}}(\mathrm{C}^{X}\mathrm{F}^{\star})\underline{\mathbf{M}}_{\mathtt{ji}}(\mathrm{F})X
\ar[dd, "{{\underline{\mathbf{M}}_{\mathtt{jj}}(\mathrm{H}_2)}
\big(\mu_{\mathtt{j}\mathtt{i}\mathtt{j}}^{\mathrm{F},\mathrm{C}^{X}\mathrm{F}^{\star}}\big)_{\underline{\mathbf{M}}_{\mathtt{ji}}(\mathrm{F})X}}"description]
\&\&\&\&\&
\\[5ex]
\underline{\mathbf{M}}_{\mathtt{jj}}(\mathrm{H}_2)\underline{\mathbf{M}}_{\mathtt{ji}}\Big(\mathrm{F}\big((\mathrm{C}^{X}\mathrm{F}^{\star})\mathrm{F}\big)\Big)X
\ar[d,"{{\underline{\mathbf{M}}_{\mathtt{jj}}(\mathrm{H}_2)}
\underline{\mathbf{M}}_{\mathtt{ji}}\big(\alpha_{\mathrm{F},\mathrm{C}^{X}\mathrm{F}^{\star},\mathrm{F}}^{\mone}\big)_X}"description]
\&\&\&\&\&
\&\&\&\&\&
\\[5ex]
\underline{\mathbf{M}}_{\mathtt{jj}}(\mathrm{H}_2)\underline{\mathbf{M}}_{\mathtt{ji}}\Big(\big(\mathrm{F}(\mathrm{C}^{X}\mathrm{F}^{\star})\big)\mathrm{F}\Big)X
\ar[d,"{{\underline{\mathbf{M}}_{\mathtt{jj}}(\mathrm{H}_2)}
\underline{\mathbf{M}}_{\mathtt{ji}}\big(\alpha_{\mathrm{F},\mathrm{C}^{X},\mathrm{F}^{\star}}^{\mone}\circ_{\mathsf{h}}\mathrm{id}_{\mathrm{F}}\big)_X}"description]
\ar[rrrrr,"{\underline{\mathbf{M}}_{\mathtt{jj}}(\mathrm{H}_2)}\big(\mu_{\mathtt{j}\mathtt{j}\mathtt{i}}^{\mathrm{F}(\mathrm{C}^{X}\mathrm{F}^{\star}),\mathrm{F}}\big)_X^{\mone}"]
\ar[rrrrr, phantom,yshift=-8ex, "\circled{10}"]
\&\&\&\&\&
\underline{\mathbf{M}}_{\mathtt{jj}}(\mathrm{H}_2)\underline{\mathbf{M}}_{\mathtt{jj}}\big(\mathrm{F}(\mathrm{C}^{X}\mathrm{F}^{\star})\big)\underline{\mathbf{M}}_{\mathtt{ji}}(\mathrm{F})X
\ar[d,"{{\underline{\mathbf{M}}_{\mathtt{jj}}(\mathrm{H}_2)}
\underline{\mathbf{M}}_{\mathtt{ji}}\big(\alpha_{\mathrm{F},\mathrm{C}^{X},\mathrm{F}^{\star}}^{\mone}\big)_{\underline{\mathbf{M}}_{\mathtt{ji}}(\mathrm{F})X}}"description]
\ar[rrrrr,"\big(\mu_{\mathtt{jjj}}^{\mathrm{H}_2,\mathrm{F}(\mathrm{C}^{X}\mathrm{F}^{\star})}\big)_{\underline{\mathbf{M}}_{\mathtt{ji}}(\mathrm{F})X}"]
\ar[rrrrr, phantom,yshift=-8ex, "\circled{11}"]
\&\&\&\&\&
\underline{\mathbf{M}}_{\mathtt{jj}}\Big(\mathrm{H}_2\big(\mathrm{F}(\mathrm{C}^{X}\mathrm{F}^{\star})\big)\Big)\underline{\mathbf{M}}_{\mathtt{ji}}(\mathrm{F})X
\ar[d,"{\underline{\mathbf{M}}_{\mathtt{jj}}\big(\mathrm{id}_{\mathrm{H}_2}\circ_{\mathsf{h}}\alpha_{\mathrm{F},\mathrm{C}^{X},\mathrm{F}^{\star}}^{\mone}\big)_{\underline{\mathbf{M}}_{\mathtt{ji}}(\mathrm{F})X}}"description]
\\[5ex]
\underline{\mathbf{M}}_{\mathtt{jj}}(\mathrm{H}_2)\underline{\mathbf{M}}_{\mathtt{ji}}(\mathrm{H}_2\mathrm{F})X
\ar[rrrrr,"{\underline{\mathbf{M}}_{\mathtt{jj}}(\mathrm{H}_2)}\big(\mu_{\mathtt{j}\mathtt{j}\mathtt{i}}^{\mathrm{H}_2,\mathrm{F}}\big)_X^{\mone}",swap]
\&\&\&\&\&
\underline{\mathbf{M}}_{\mathtt{jj}}(\mathrm{H}_2)\underline{\mathbf{M}}_{\mathtt{jj}}(\mathrm{H}_2)\underline{\mathbf{M}}_{\mathtt{ji}}(\mathrm{F})X
\ar[rrrrr,"\big(\mu_{\mathtt{jjj}}^{\mathrm{H}_2,\mathrm{H}_2}\big)_{\underline{\mathbf{M}}_{\mathtt{ji}}(\mathrm{F})X}",swap]
\&\&\&\&\&
\underline{\mathbf{M}}_{\mathtt{jj}}(\mathrm{H}_2\mathrm{H}_2)\underline{\mathbf{M}}_{\mathtt{ji}}(\mathrm{F})X
\end{tikzcd}$}
.
\end{gather*}
This diagram is commutative by
\begin{itemize}
\item naturality of $\underline{\mathbf{M}}_{\mathtt{ii}}\big((\runit_{\mathrm{FC}^{X}})^{\mone}\big)$ for the facet labeled $1$;

\item naturality of $\mu_{\mathtt{jij}}$ for the facets labeled $2$, $4$ and $6$;

\item naturality of $\underline{\mathbf{M}}_{\mathtt{ji}}\big(\mathrm{id}_{\mathrm{F}\mathrm{C}^{X}}\circ_{\mathsf{h}}\mathrm{coev}_{\mathrm{F}}\big)$
for the facet labeled $3$;

\item naturality of $\underline{\mathbf{M}}_{\mathtt{ji}}\big(\alpha_{\mathrm{F}\mathrm{C}^{X},\mathrm{F}^{\star},\mathrm{F}}^{\mone}\big)$ for the facet labeled $5$;

\item naturality of $\big(\mu_{\mathtt{jji}}^{\mathrm{H}_2,\mathrm{F}}\big)^{\mone}$ for the facet labeled $7$;
\item the diagram in \eqref{eq:birepresentation3} for the facets labeled $8$ and $9$;

\item naturality of $\mu_{\mathtt{j}\mathtt{j}\mathtt{i}}$ and $\mu_{\mathtt{j}\mathtt{j}\mathtt{j}}$ for the facets labeled $10$ and $11$ respectively.
\end{itemize}
In the last three diagrams, the last column of the previous diagram coincides with the first
column of the next, so we can glue these three diagrams from left to right into one big diagram. By commutativity of this big diagram, the two paths along its boundary
from northwest to southeast represent the same $2$-morphism.
The path which first goes down and then to the right, after precomposing
the composite $2$-morphism with $\underline{\mathbf{M}}_{\mathtt{ji}}(\mathrm{F}) \mathrm{coev}_{X,X}$, corresponds to $f$. By considering the other path along the
boundary of the big diagram, we obtain
\begin{gather*}\scalebox{0.88}{$
\begin{aligned}
f=&\underline{\mathbf{M}}_{\mathtt{ji}}\big(\mathrm{id}_{\mathrm{H}_2}\circ_{\mathsf{h}}\alpha_{\mathrm{F},\mathrm{C}^{X},\mathrm{F}^{\star}}^{\mone}\big)
_{\underline{\mathbf{M}}_{\mathtt{ji}}(\mathrm{F})X}
\circ_{\mathsf{v}}
\underline{\mathbf{M}}_{\mathtt{jj}}\big(\alpha_{\mathrm{H}_2,\mathrm{F},\mathrm{C}^{X}\mathrm{F}^{\star}}\big)_{\underline{\mathbf{M}}_{\mathtt{ji}}(\mathrm{F})X}
\\&
\circ_{\mathsf{v}}
\underline{\mathbf{M}}_{\mathtt{jj}}\big(\alpha_{\mathrm{F}\mathrm{C}^{X},\mathrm{F}^{\star},\mathrm{F}}^{\mone}\circ_{\mathsf{h}}
\mathrm{id}_{\mathrm{C}^{X}\mathrm{F}^{\star}}\big)_{\underline{\mathbf{M}}_{\mathtt{ji}}(\mathrm{F})X}
\circ_{\mathsf{v}}
\underline{\mathbf{M}}_{\mathtt{jj}}\big((\mathrm{id}_{\mathrm{F}\mathrm{C}^{X}}\circ_{\mathsf{h}}\mathrm{coev}_{\mathrm{F}})\circ_{\mathsf{h}}\mathrm{id}_{\mathrm{C}^{X}\mathrm{F}^{\star}}\big)_{\underline{\mathbf{M}}_{\mathtt{ji}}(\mathrm{F})X}
\\&
\circ_{\mathsf{v}}
\underline{\mathbf{M}}_{\mathtt{jj}}\big((\runit_{\mathrm{FC}^{X}})^{\mone}\circ_{\mathsf{h}}\mathrm{id}_{\mathrm{C}^{X}\mathrm{F}^{\star}}\big)_{\underline{\mathbf{M}}_{\mathtt{ji}}(\mathrm{F})X}
\circ_{\mathsf{v}}
\big(\mu_{\mathtt{jij}}^{\mathrm{F}\mathrm{C}^{X},\mathrm{C}^{X}\mathrm{F}^{\star}}\big)_{\underline{\mathbf{M}}_{\mathtt{ji}}(\mathrm{F})X}
\\&
\circ_{\mathsf{v}}
\Big({\underline{\mathbf{M}}_{\mathtt{ji}}(\mathrm{F}\mathrm{C}^{X})}\big(\mu_{\mathtt{j}\mathtt{j}\mathtt{i}}^{\mathrm{C}^{X}\mathrm{F}^{\star},\mathrm{F}}\big)_X^{\mone}\Big)
\circ_{\mathsf{v}}
\big({\underline{\mathbf{M}}_{\mathtt{ji}}(\mathrm{F}\mathrm{C}^{X})}
\underline{\mathbf{M}}_{\mathtt{ii}}(\alpha_{\mathrm{C}^{X},\mathrm{F}^{\star},\mathrm{F}}^{\mone})_X\big)
\\&
\circ_{\mathsf{v}}
\Big({\underline{\mathbf{M}}_{\mathtt{ji}}(\mathrm{F}\mathrm{C}^{X})}\underline{\mathbf{M}}_{\mathtt{ii}}\big(\mathrm{id}_{\mathrm{C}^{X}}\circ_{\mathsf{h}}
\mathrm{coev}_{\mathrm{F}}\big)_X\Big)
\circ_{\mathsf{v}}
\Big({\underline{\mathbf{M}}_{\mathtt{ji}}(\mathrm{F}\mathrm{C}^{X})}
\underline{\mathbf{M}}_{\mathtt{ii}}\big((\runit_{\mathrm{C}^{X}})^{\mone}\big)_X\Big)
\\&
\circ_{\mathsf{v}}
\big(\mu_{\mathtt{jii}}^{\mathrm{F},\mathrm{C}^{X}}\big)_{\underline{\mathbf{M}}_{\mathtt{ii}}(\mathrm{C}^{X})X}
\circ_{\mathsf{v}}
\big({\underline{\mathbf{M}}_{\mathtt{ji}}(\mathrm{F})\underline{\mathbf{M}}_{\mathtt{ii}}(\mathrm{C}^{X})}
\mathrm{coev}_{X,X}\big)
\\&
\circ_{\mathsf{v}}
\big({\underline{\mathbf{M}}_{\mathtt{ji}}(\mathrm{F})}\mathrm{coev}_{X,X}\big)
\end{aligned}$}
\end{gather*}
Similarly, applying naturality of $\mu$ and using the diagram in \eqref{eq:birepresentation3}, we obtain that
the composite $\big(\mu_{\mathtt{jji}}^{\mathrm{H}_2\mathrm{H}_2,\mathrm{F}}\big)_X
\circ_{\mathsf{v}} f$ equals
\begin{gather*}
\begin{aligned}
&\underline{\mathbf{M}}_{\mathtt{ji}}\big((\mathrm{id}_{\mathrm{H}_2}\circ_{\mathsf{h}}\alpha_{\mathrm{F},\mathrm{C}^{X},\mathrm{F}^{\star}}^{\mone})
\circ_{\mathsf{h}}\mathrm{id}_{\mathrm{F}}\big)_X
\circ_{\mathsf{v}}
\underline{\mathbf{M}}_{\mathtt{ji}}(\alpha_{\mathrm{H}_2,\mathrm{F},\mathrm{C}^{X}\mathrm{F}^{\star}}\circ_{\mathsf{h}}\mathrm{id}_{\mathrm{F}})_X
\\&
\circ_{\mathsf{v}}
\underline{\mathbf{M}}_{\mathtt{ji}}\big((\alpha_{\mathrm{F}\mathrm{C}^{X},\mathrm{F}^{\star},\mathrm{F}}^{\mone}\circ_{\mathsf{h}}
\mathrm{id}_{\mathrm{C}^{X}\mathrm{F}^{\star}})\circ_{\mathsf{h}}\mathrm{id}_{\mathrm{F}}\big)_X
\\&
\circ_{\mathsf{v}}
\underline{\mathbf{M}}_{\mathtt{ji}}\Big(\big((\mathrm{id}_{\mathrm{F}\mathrm{C}^{X}}\circ_{\mathsf{h}}\mathrm{coev}_{\mathrm{F}})\circ_{\mathsf{h}}\mathrm{id}_{\mathrm{C}^{X}\mathrm{F}^{\star}}\big)\circ_{\mathsf{h}}\mathrm{id}_{\mathrm{F}}\Big)_X
\\&
\circ_{\mathsf{v}}
\underline{\mathbf{M}}_{\mathtt{ji}}\Big(\big((\runit_{\mathrm{FC}^{X}})^{\mone}\circ_{\mathsf{h}}\mathrm{id}_{\mathrm{C}^{X}\mathrm{F}^{\star}}\big)\circ_{\mathsf{h}}
\mathrm{id}_{\mathrm{F}}\Big)_X
\circ_{\mathsf{v}}
\underline{\mathbf{M}}_{\mathtt{ji}}(\alpha_{\mathrm{F}\mathrm{C}^{X}, \mathrm{C}^{X}\mathrm{F}^{\star},\mathrm{F}}^{\mone})_X
\\&
\circ_{\mathsf{v}}
\underline{\mathbf{M}}_{\mathtt{ji}}\big(\mathrm{id}_{\mathrm{F}\mathrm{C}^{X}}\circ_{\mathsf{h}}\alpha_{\mathrm{C}^{X},\mathrm{F}^{\star},\mathrm{F}}^{\mone}\big)_X
\circ_{\mathsf{v}}
\underline{\mathbf{M}}_{\mathtt{ji}}\big(\mathrm{id}_{\mathrm{F}\mathrm{C}^{X}}\circ_{\mathsf{h}}(\mathrm{id}_{\mathrm{C}^{X}}\circ_{\mathsf{h}}
\mathrm{coev}_{\mathrm{F}})\big)_X
\\&
\circ_{\mathsf{v}}
\underline{\mathbf{M}}_{\mathtt{ji}}\big(\mathrm{id}_{\mathrm{F}\mathrm{C}^{X}}\circ_{\mathsf{h}}(\runit_{\mathrm{C}^{X}})^{\mone}\big)_X
\circ_{\mathsf{v}}
\underline{\mathbf{M}}_{\mathtt{ji}}(\alpha_{\mathrm{F},\mathrm{C}^{X},\mathrm{C}^{X}}^{\mone})_X
\circ_{\mathsf{v}}
\big(\mu_{\mathtt{jii}}^{\mathrm{F},\mathrm{C}^{X}\mathrm{C}^{X}}\big)_X
\\&
\circ_{\mathsf{v}}
\Big({\underline{\mathbf{M}}_{\mathtt{ji}}(\mathrm{F})}\big(\mu_{\mathtt{iii}}^{\mathrm{C}^{X},\mathrm{C}^{X}}\big)_X\Big)
\circ_{\mathsf{v}}
\Big({\underline{\mathbf{M}}_{\mathtt{ji}}(\mathrm{F})\underline{\mathbf{M}}_{\mathtt{ii}}(\mathrm{C}^{X})}
\mathrm{coev}_{X,X}\Big)
\\&
\circ_{\mathsf{v}}
({\underline{\mathbf{M}}_{\mathtt{ji}}(\mathrm{F})}\mathrm{coev}_{X,X})\\
&=\underline{\mathbf{M}}_{\mathtt{ji}}\big((\mathrm{id}_{\mathrm{H}_2}\circ_{\mathsf{h}}\alpha_{\mathrm{F},\mathrm{C}^{X},\mathrm{F}^{\star}}^{\mone})
\circ_{\mathsf{h}}\mathrm{id}_{\mathrm{F}}\big)_X
\circ_{\mathsf{v}}
\underline{\mathbf{M}}_{\mathtt{ji}}(\alpha_{\mathrm{H}_2,\mathrm{F},\mathrm{C}^{X}\mathrm{F}^{\star}}\circ_{\mathsf{h}}\mathrm{id}_{\mathrm{F}})_X
\\&
\circ_{\mathsf{v}}
\underline{\mathbf{M}}_{\mathtt{ji}}\big((\alpha_{\mathrm{F}\mathrm{C}^{X},\mathrm{F}^{\star},\mathrm{F}}^{\mone}\circ_{\mathsf{h}}
\mathrm{id}_{\mathrm{C}^{X}\mathrm{F}^{\star}})\circ_{\mathsf{h}}\mathrm{id}_{\mathrm{F}}\big)_X
\\&
\circ_{\mathsf{v}}
\underline{\mathbf{M}}_{\mathtt{ji}}\Big(\big((\mathrm{id}_{\mathrm{F}\mathrm{C}^{X}}\circ_{\mathsf{h}}\mathrm{coev}_{\mathrm{F}})\circ_{\mathsf{h}}\mathrm{id}_{\mathrm{C}^{X}\mathrm{F}^{\star}}\big)\circ_{\mathsf{h}}\mathrm{id}_{\mathrm{F}}\Big)_X
\\&
\circ_{\mathsf{v}}
\underline{\mathbf{M}}_{\mathtt{ji}}\Big(\big((\runit_{\mathrm{FC}^{X}})^{\mone}\circ_{\mathsf{h}}\mathrm{id}_{\mathrm{C}^{X}\mathrm{F}^{\star}}\big)\circ_{\mathsf{h}}
\mathrm{id}_{\mathrm{F}}\Big)_X
\circ_{\mathsf{v}}
\underline{\mathbf{M}}_{\mathtt{ji}}(\alpha_{\mathrm{F}\mathrm{C}^{X}, \mathrm{C}^{X}\mathrm{F}^{\star},\mathrm{F}}^{\mone})_X
\\&
\circ_{\mathsf{v}}
\underline{\mathbf{M}}_{\mathtt{ji}}\big(\mathrm{id}_{\mathrm{F}\mathrm{C}^{X}}\circ_{\mathsf{h}}\alpha_{\mathrm{C}^{X},\mathrm{F}^{\star},\mathrm{F}}^{\mone}\big)_X
\circ_{\mathsf{v}}
\underline{\mathbf{M}}_{\mathtt{ji}}\big(\mathrm{id}_{\mathrm{F}\mathrm{C}^{X}}\circ_{\mathsf{h}}(\mathrm{id}_{\mathrm{C}^{X}}\circ_{\mathsf{h}}
\mathrm{coev}_{\mathrm{F}})\big)_X
\\&
\circ_{\mathsf{v}}
\underline{\mathbf{M}}_{\mathtt{ji}}\big(\mathrm{id}_{\mathrm{F}\mathrm{C}^{X}}\circ_{\mathsf{h}}(\runit_{\mathrm{C}^{X}})^{\mone}\big)_X
\circ_{\mathsf{v}}
\underline{\mathbf{M}}_{\mathtt{ji}}(\alpha_{\mathrm{F},\mathrm{C}^{X},\mathrm{C}^{X}}^{\mone})_X
\circ_{\mathsf{v}}
\big(\mu_{\mathtt{jii}}^{\mathrm{F},\mathrm{C}^{X}\mathrm{C}^{X}}\big)_X
\\&
\circ_{\mathsf{v}}
\Big({\underline{\mathbf{M}}_{\mathtt{ji}}(\mathrm{F})}\underline{\mathbf{M}}_{\mathtt{ii}}\big(\delta_{X,X,X}\big)_X\Big)
\circ_{\mathsf{v}}
\big({\underline{\mathbf{M}}_{\mathtt{ji}}(\mathrm{F})}\mathrm{coev}_{X,X}\big),
\end{aligned}
\end{gather*}
where the equality follows from the definition of $\delta_{X,X,X}=\delta_{\mathrm{C}^{X}}$.

Consider the diagram
\begin{gather*}
\adjustbox{scale=.44,center}{%
\begin{tikzcd}[ampersand replacement=\&, column sep=3em, row sep=3.5em]
\underline{\mathbf{M}}_{\mathtt{ji}}(\mathrm{F})\underline{\mathbf{M}}_{\mathtt{ii}}(\mathrm{C}^{X})X
\ar[d,"{{\underline{\mathbf{M}}_{\mathtt{ji}}(\mathrm{F})}\underline{\mathbf{M}}_{\mathtt{ii}}\big(\delta_{X,X,X}\big)_X}"description]
\ar[rr, "\big(\mu_{\mathtt{jii}}^{\mathrm{F},\mathrm{C}^{X}}\big)_X"]
\ar[rr, phantom,yshift=-9ex, xshift=-4ex, "\circled{1}"]
\&\&
\underline{\mathbf{M}}_{\mathtt{ji}}(\mathrm{F}\mathrm{C}^{X})X
\ar[ddll,bend left=10,"\underline{\mathbf{M}}_{\mathtt{ji}}\big(\mathrm{id}_{\mathrm{F}}\circ_{\mathsf{h}}\delta_{X,X,X}\big)_X",sloped,near start]
\ar[rr,"\underline{\mathbf{M}}_{\mathtt{ji}}\Big(\big(\runit_{\mathrm{FC}^{X}}\big)^{\mone}\Big)_X"]
\ar[rr, phantom,yshift=-9ex, xshift=-3ex, "\circled{2}"]
\&\&
\underline{\mathbf{M}}_{\mathtt{ji}}\big((\mathrm{F}\mathrm{C}^{X})\mathbbm{1}_{\mathtt{i}}\big)X
\ar[d,"{\underline{\mathbf{M}}_{\mathtt{ji}}\big((\mathrm{id}_{\mathrm{F}}\circ_{\mathsf{h}}\delta_{X,X,X})\circ_{\mathsf{h}}\mathrm{id}_{\mathbbm{1}_{\mathtt{i}}}\big)_X}"description]
\ar[rrr,"\underline{\mathbf{M}}_{\mathtt{ji}}\big(\mathrm{id}_{\mathrm{FC}^{X}}\circ_{\mathsf{h}}\mathrm{coev}_{\mathrm{F}}\big)_X"]
\ar[rrr, phantom,yshift=-12ex, xshift=-4ex, "\circled{3}"]
\&\&\&
\underline{\mathbf{M}}_{\mathtt{ji}}\big((\mathrm{F}\mathrm{C}^{X})\mathrm{H}_1\big)X
\ar[ddl,"\underline{\mathbf{M}}_{\mathtt{ji}}\big((\mathrm{id}_{\mathrm{F}}\circ_{\mathsf{h}}\delta_{X,X,X})\circ_{\mathsf{h}}\mathrm{id}_{\mathrm{H}_1}\big)_X",
near start, swap]
\ar[dddd,"{\underline{\mathbf{M}}_{\mathtt{ji}}\big(\alpha_{\mathrm{FC}^{X},\mathrm{F}^{\star},\mathrm{F}}^{\mone}\big)_X}"description]
\ar[dddd, phantom,yshift=-8ex, xshift=-14ex, "\circled{9}"]
\\[5ex]
\underline{\mathbf{M}}_{\mathtt{ji}}(\mathrm{F})\underline{\mathbf{M}}_{\mathtt{ii}}(\mathrm{C}^{X}\mathrm{C}^{X})X
\ar[d,"{\big(\mu_{\mathtt{jii}}^{\mathrm{F},\mathrm{C}^{X}\mathrm{C}^{X}}\big)_X}"description]
\&\&
\&\&
\underline{\mathbf{M}}_{\mathtt{ji}}\Big(\big(\mathrm{F}(\mathrm{C}^{X}\mathrm{C}^{X})\big)\mathbbm{1}_{\mathtt{i}}\Big)X
\ar[d,"{\underline{\mathbf{M}}_{\mathtt{ji}}\big(\alpha_{\mathrm{F},\mathrm{C}^{X}\mathrm{C}^{X},\mathbbm{1}_{\mathtt{i}}}^{\mone}\big)_X}"description]
\ar[ddd,bend left=60,"{\underline{\mathbf{M}}_{\mathtt{ji}}\big(\alpha_{\mathrm{F},\mathrm{C}^{X},\mathrm{C}^{X}}^{\mone}\circ_{\mathsf{h}}\mathrm{id}_{\mathbbm{1}_{\mathtt{i}}}\big)_X}"description]
\ar[drr,"\underline{\mathbf{M}}_{\mathtt{ji}}\big(\mathrm{id}_{\mathrm{F}(\mathrm{C}^{X}\mathrm{C}^{X})}\circ_{\mathsf{h}}\mathrm{coev}_{\mathrm{F}}\big)_X"]
\ar[drr, phantom,yshift=-30ex, xshift=5ex, "\circled{8}"]
\&\&\&
\\[5ex]
\underline{\mathbf{M}}_{\mathtt{ji}}\big(\mathrm{F}(\mathrm{C}^{X}\mathrm{C}^{X})\big)X
\ar[d,"{\underline{\mathbf{M}}_{\mathtt{ji}}\big(\alpha_{\mathrm{F},\mathrm{C}^{X},\mathrm{C}^{X}}^{\mone}\big)_X}"description]
\ar[urrrr,"\underline{\mathbf{M}}_{\mathtt{ji}}\Big(\big(\runit_{\mathrm{F}(\mathrm{C}^{X}\mathrm{C}^{X})}\big)^{\mone}\Big)_X", near end]
\ar[drr,"\underline{\mathbf{M}}_{\mathtt{ji}}\bigg(\mathrm{id}_{\mathrm{F}}\circ_{\mathsf{h}}\Big(\mathrm{id}_{\mathrm{C}^{X}}\circ_{\mathsf{h}}\big(\runit_{\mathrm{C}^{X}}\big)^{\mone}\Big)\bigg)_X"]
\ar[rrrr,"\underline{\mathbf{M}}_{\mathtt{ji}}\Big(\mathrm{id}_{\mathrm{F}}\circ_{\mathsf{h}}\big(\runit_{\mathrm{C}^{X}\mathrm{C}^{X}}\big)^{\mone}\Big)_X",near end]
\ar[rrrr, phantom,yshift=8ex, xshift=18ex, "\circled{4}"]
\ar[rrrr, phantom,yshift=-12ex, "\circled{5}"]
\ar[drr, phantom,yshift=-8ex, "\circled{6}"]
\&\&\&\&
\ar[dll,"\underline{\mathbf{M}}_{\mathtt{ji}}\big(\mathrm{id}_{\mathrm{F}}\circ_{\mathsf{h}}\alpha_{\mathrm{C}^{X},\mathrm{C}^{X},\mathbbm{1}_{\mathtt{i}}}\big)_X"]
\underline{\mathbf{M}}_{\mathtt{ji}}\Big(\mathrm{F}\big((\mathrm{C}^{X}\mathrm{C}^{X})\mathbbm{1}_{\mathtt{i}}\big)\Big)X
\ar[dll, phantom,yshift=-15ex, xshift=3ex, "\circled{7}"]
\&\&
\underline{\mathbf{M}}_{\mathtt{ji}}\Big(\big(\mathrm{F}(\mathrm{C}^{X}\mathrm{C}^{X})\big)\mathrm{H}_1\Big)X
\ar[dddll,bend left=25, "\underline{\mathbf{M}}_{\mathtt{ji}}\big(\alpha_{\mathrm{F},\mathrm{C}^{X},\mathrm{C}^{X}}^{\mone}\circ_{\mathsf{h}}\mathrm{id}_{\mathrm{H}_1}\big)_X",near end]
\ar[dddd, "\underline{\mathbf{M}}_{\mathtt{ji}}\big(\alpha_{\mathrm{F}(\mathrm{C}^{X}\mathrm{C}^{X}),\mathrm{F}^{\star},\mathrm{F}}^{\mone}\big)_X",near start]
\&
\\[5ex]
\underline{\mathbf{M}}_{\mathtt{ji}}\big((\mathrm{F}\mathrm{C}^{X})\mathrm{C}^{X}\big)X
\ar[d, "{\underline{\mathbf{M}}_{\mathtt{ji}}\big(\mathrm{id}_{\mathrm{F}\mathrm{C}^{X}}\circ_{\mathsf{h}}(\runit_{\mathrm{C}^{X}})^{\mone}\big)_X}"description]
\&\&
\underline{\mathbf{M}}_{\mathtt{ji}}\Big(\mathrm{F}\big(\mathrm{C}^{X}(\mathrm{C}^{X}\mathbbm{1}_{\mathtt{i}})\big)\Big)X
\ar[dll, "\underline{\mathbf{M}}_{\mathtt{ji}}\big(\alpha_{\mathrm{F},\mathrm{C}^{X},\mathrm{C}^{X}\mathbbm{1}_{\mathtt{i}}}^{\mone}\big)_X"]
\&\&
\&\&\&
\\[5ex]
\underline{\mathbf{M}}_{\mathtt{ji}}\big((\mathrm{F}\mathrm{C}^{X})(\mathrm{C}^{X}\mathbbm{1}_{\mathtt{i}})\big)X
\ar[d, "{\underline{\mathbf{M}}_{\mathtt{ji}}\big(\mathrm{id}_{\mathrm{F}\mathrm{C}^{X}}\circ_{\mathsf{h}}(\mathrm{id}_{\mathrm{C}^{X}}\circ_{\mathsf{h}}\mathrm{coev}_{\mathrm{F}})\big)_X}"description]
\&\&\&\&
\underline{\mathbf{M}}_{\mathtt{ji}}\Big(\big((\mathrm{F}\mathrm{C}^{X})\mathrm{C}^{X}\big)\mathbbm{1}_{\mathtt{i}}\Big)X
\ar[d, "{\underline{\mathbf{M}}_{\mathtt{ji}}\big((\mathrm{id}_{\mathrm{F}\mathrm{C}^{X}}\circ_{\mathsf{h}}\mathrm{id}_{\mathrm{C}^{X}})\circ_{\mathsf{h}}\mathrm{coev}_{\mathrm{F}}\big)_X}"description]
\ar[llll,"\underline{\mathbf{M}}_{\mathtt{ji}}\big(\alpha_{\mathrm{F}\mathrm{C}^{X},\mathrm{C}^{X},\mathbbm{1}_{\mathtt{i}}}\big)_X",near start,swap]
\ar[llll, phantom,yshift=-8ex, "\circled{10}"]
\&\&\&
\underline{\mathbf{M}}_{\mathtt{ji}}(\mathrm{H}_2\mathrm{F})X
\ar[ddl,"{\underline{\mathbf{M}}_{\mathtt{ji}}\Big(\big((\mathrm{id}_{\mathrm{F}}\circ_{\mathsf{h}}\delta_{X,X,X})\circ_{\mathsf{h}}\mathrm{id}_{\mathrm{F}^{\star}}\big)\circ_{\mathsf{h}}\mathrm{id}_{\mathrm{F}}\Big)_X}"description]
\\[5ex]
\underline{\mathbf{M}}_{\mathtt{ji}}\big((\mathrm{F}\mathrm{C}^{X})(\mathrm{C}^{X}\mathrm{H}_1)\big)X
\ar[d,"{\underline{\mathbf{M}}_{\mathtt{ji}}\big(\mathrm{id}_{\mathrm{FC}^{X}}\circ_{\mathsf{h}}\alpha_{\mathrm{C}^{X},\mathrm{F}^{\star},\mathrm{F}}^{\mone}\big)_X}"description]
\&\&\&\&
\underline{\mathbf{M}}_{\mathtt{ji}}\Big(\big((\mathrm{F}\mathrm{C}^{X})\mathrm{C}^{X}\big)\mathrm{H}_1\Big)X
\ar[llll,"\underline{\mathbf{M}}_{\mathtt{ji}}\big(\alpha_{\mathrm{F}\mathrm{C}^{X},\mathrm{C}^{X},\mathrm{H}_1}\big)_X",swap]
\ar[dd,"\underline{\mathbf{M}}_{\mathtt{ji}}\big(\alpha_{(\mathrm{FC}^{X})\mathrm{C}^{X},\mathrm{F}^{\star},\mathrm{F}}^{\mone}\big)_X"]
\ar[llll, phantom,yshift=-18ex, "\circled{11}"]
\&\&\&
\\[5ex]
\underline{\mathbf{M}}_{\mathtt{ji}}\Big((\mathrm{F}\mathrm{C}^{X})\big((\mathrm{C}^{X}\mathrm{F}^{\star})\mathrm{F}\big)\Big)X
\ar[d,"{\underline{\mathbf{M}}_{\mathtt{ji}}\big(\alpha_{\mathrm{FC}^{X},\mathrm{C}^{X}\mathrm{F}^{\star},\mathrm{F}}^{\mone}\big)_X}"description]
\&\&\&\&
\&\&
\underline{\mathbf{M}}_{\mathtt{ji}}\bigg(\Big(\big(\mathrm{F}(\mathrm{C}^{X}\mathrm{C}^{X})\big)\mathrm{F}^{\star}\Big)\mathrm{F}\bigg)X
\ar[dll,"\underline{\mathbf{M}}_{\mathtt{ji}}\Big(\big(\alpha_{\mathrm{F},\mathrm{C}^{X},\mathrm{C}^{X}}^{\mone}\circ_{\mathsf{h}}\mathrm{id}_{\mathrm{F}^{\star}}\big)\circ_{\mathsf{h}}\mathrm{id}_{\mathrm{F}}\Big)_X"]
\ar[dll, phantom,yshift=18ex, "\circled{12}"]
\&
\\[5ex]
\underline{\mathbf{M}}_{\mathtt{ji}}\Big(\big((\mathrm{F}\mathrm{C}^{X})(\mathrm{C}^{X}\mathrm{F}^{\star})\big)\mathrm{F}\Big)X
\&\&\&\&
\underline{\mathbf{M}}_{\mathtt{ji}}\bigg(\Big(\big((\mathrm{F}\mathrm{C}^{X})\mathrm{C}^{X}\big)\mathrm{F}^{\star}\Big)\mathrm{F}\bigg)X
\ar[llll,"\underline{\mathbf{M}}_{\mathtt{ji}}\big(\alpha_{\mathrm{F}\mathrm{C}^{X},\mathrm{C}^{X},\mathrm{F}^{\star}}\circ_{\mathsf{h}}\mathrm{id}_{\mathrm{F}}\big)_X"]
\&\&
\&
\end{tikzcd}}
\end{gather*}
which commutes due to
\begin{itemize}
\item naturality of $\mu_{\mathtt{jii}}$ for the facet labeled $1$;

\item naturality of $(\runit_{})^{\mone}$ for the facet labeled $2$;

\item the interchange law for the facets labeled $3$ and $8$;

\item the right diagram in \eqref{eq:0.00} for the facets labeled $4$ and $5$;

\item naturality of $\alpha^{\mone}$(respectively $\alpha$) for the facets labeled $6$, $9$  and $12$ (respectively $10$);

\item the pentagon coherence condition of the associator for the facets labeled $7$ and $11$.
\end{itemize}
By commutativity, the $2$-morphisms corresponding to the paths along the boundary of this diagram from northwest to southwest are equal. The path going straight down corresponds to
the composite of seven consecutive factors of the expression for $\big(\mu_{\mathtt{jji}}^{\mathrm{H}_2\mathrm{H}_2,\mathrm{F}}\big)_X
\circ_{\mathsf{v}} f$ above (reading from left to right, these are the factors six until twelve). Replacing those factors by the ones corresponding to the other path along the boundary in the diagram above, yields the equation

\begin{gather}\label{eq:0.10}
\scalebox{0.83}{$
\begin{aligned}
\big(\mu_{\mathtt{jji}}^{\mathrm{H}_2\mathrm{H}_2,\mathrm{F}}\big)_X
\circ_{\mathsf{v}} f&=\underline{\mathbf{M}}_{\mathtt{ji}}\big((\mathrm{id}_{\mathrm{H}_2}\circ_{\mathsf{h}}\alpha_{\mathrm{F},\mathrm{C}^{X},\mathrm{F}^{\star}}^{\mone})
\circ_{\mathsf{h}}\mathrm{id}_{\mathrm{F}}\big)_X
\circ_{\mathsf{v}}
\underline{\mathbf{M}}_{\mathtt{ji}}(\alpha_{\mathrm{H}_2,\mathrm{F},\mathrm{C}^{X}\mathrm{F}^{\star}}\circ_{\mathsf{h}}\mathrm{id}_{\mathrm{F}})_X
\\&
\circ_{\mathsf{v}}
\underline{\mathbf{M}}_{\mathtt{ji}}\big((\alpha_{\mathrm{F}\mathrm{C}^{X},\mathrm{F}^{\star},\mathrm{F}}^{\mone}\circ_{\mathsf{h}}
\mathrm{id}_{\mathrm{C}^{X}\mathrm{F}^{\star}})\circ_{\mathsf{h}}\mathrm{id}_{\mathrm{F}}\big)_X
\\&
\circ_{\mathsf{v}}
\underline{\mathbf{M}}_{\mathtt{ji}}\Big(\big((\mathrm{id}_{\mathrm{F}\mathrm{C}^{X}}\circ_{\mathsf{h}}\mathrm{coev}_{\mathrm{F}})\circ_{\mathsf{h}}\mathrm{id}_{\mathrm{C}^{X}\mathrm{F}^{\star}}\big)\circ_{\mathsf{h}}\mathrm{id}_{\mathrm{F}}\Big)_X
\\&
\circ_{\mathsf{v}}
\underline{\mathbf{M}}_{\mathtt{ji}}\Big(\big((\runit_{\mathrm{FC}^{X}})^{\mone}\circ_{\mathsf{h}}\mathrm{id}_{\mathrm{C}^{X}\mathrm{F}^{\star}}\big)\circ_{\mathsf{h}}
\mathrm{id}_{\mathrm{F}}\Big)_X
\circ_{\mathsf{v}}
\underline{\mathbf{M}}_{\mathtt{ji}}\big(\alpha_{\mathrm{F}\mathrm{C}^{X},\mathrm{C}^{X},\mathrm{F}^{\star}}\circ_{\mathsf{h}}\mathrm{id}_{\mathrm{F}}\big)_X
\\&
\circ_{\mathsf{v}}
\underline{\mathbf{M}}_{\mathtt{ji}}\Big(\big(\alpha_{\mathrm{F},\mathrm{C}^{X},\mathrm{C}^{X}}^{\mone}\circ_{\mathsf{h}}\mathrm{id}_{\mathrm{F}^{\star}}\big)\circ_{\mathsf{h}}\mathrm{id}_{\mathrm{F}}\Big)_X
\circ_{\mathsf{v}}
\underline{\mathbf{M}}_{\mathtt{ji}}\Big(\big((\mathrm{id}_{\mathrm{F}}\circ_{\mathsf{h}}\delta_{X,X,X})\circ_{\mathsf{h}}\mathrm{id}_{\mathrm{F}^{\star}}\big)\circ_{\mathsf{h}}\mathrm{id}_{\mathrm{F}}\Big)_X
\\&
\circ_{\mathsf{v}}
\underline{\mathbf{M}}_{\mathtt{ji}}\big(\alpha_{\mathrm{FC}^{X},\mathrm{F}^{\star},\mathrm{F}}^{\mone}\big)_X
\circ_{\mathsf{v}}
\underline{\mathbf{M}}_{\mathtt{ji}}\big(\mathrm{id}_{\mathrm{FC}^{X}}\circ_{\mathsf{h}}\mathrm{coev}_{\mathrm{F}}\big)_X
\circ_{\mathsf{v}}
\underline{\mathbf{M}}_{\mathtt{ji}}\Big(\big(\runit_{\mathrm{FC}^{X}}\big)^{\mone}\Big)_X
\\&
\circ_{\mathsf{v}}
\big(\mu_{\mathtt{jii}}^{\mathrm{F},\mathrm{C}^{X}}\big)_X
\circ_{\mathsf{v}}
\big({\underline{\mathbf{M}}_{\mathtt{ji}}(\mathrm{F})}\mathrm{coev}_{X,X}\big).
\end{aligned}$}
\end{gather}

Now we prove that
$\delta_{\underline{\mathbf{M}}_{\mathtt{ji}}(\mathrm{F})X}$ is equal to
$\delta_{(\mathrm{FC}^{X})\mathrm{F}^{\star}}$, as defined in Lemma \ref{lem0.1}.
On one hand, the composite of the first three isomorphisms in \eqref{eq:0.4} sends $f$ to
\begin{gather}\label{eq:0.11}
\begin{aligned}
&\big(\mu_{\mathtt{iji}}^{\mathrm{F}^{\star},(\mathrm{H}_2\mathrm{H}_2)\mathrm{F}}\big)_X
\circ_{\mathsf{v}}
\Big({\underline{\mathbf{M}}_{\mathtt{ij}}(\mathrm{F}^{\star})}\big(\mu_{\mathtt{jji}}^{\mathrm{H}_2\mathrm{H}_2,\mathrm{F}}\big)_X\Big)
\circ_{\mathsf{v}}
\big({\underline{\mathbf{M}}_{\mathtt{ij}}(\mathrm{F}^{\star})}f\big)\\
&\circ_{\mathsf{v}}
\big(\mu_{\mathtt{iji}}^{\mathrm{F}^{\star},\mathrm{F}}\big)^{\mone}_X
\circ_{\mathsf{v}}
\underline{\mathbf{M}}_{\mathtt{ii}}(\mathrm{coev}_{\mathrm{F}})
\circ_{\mathsf{v}}
(\iota_{\mathtt{i}}^{\mone})_X
\end{aligned}
\end{gather}
in $\mathrm{Hom}_{\underline{\mathbf{M}}(\mathtt{i})}\Big(X,
\underline{\mathbf{M}}_{\mathtt{ii}}\Big(\mathrm{F}^{\star}\big((\mathrm{H}_2\mathrm{H}_2)\mathrm{F}\big)\Big)X\Big)$.

On the other hand, chasing the image of $\delta_{\mathrm{H}_2}$, where $\mathrm{H}_2=(\mathrm{FC}^{X})\mathrm{F}^{\star}$, the last isomorphism
\begin{gather*}
\mathrm{Hom}_{\underline{\ccC}}(\mathrm{H}_2,\mathrm{H}_2\mathrm{H}_2)\xrightarrow{\cong}
\mathrm{Hom}_{\underline{\ccC}}\big(\mathrm{FC}^{X},(\mathrm{H}_2\mathrm{H}_2)\mathrm{F}\big)
\end{gather*}
sends $\delta_{\mathrm{H}_2}$ to
\begin{gather}\label{eq:0.12}
\big(\delta_{\mathrm{H}_2}\circ_{\mathsf{h}}\mathrm{id}_{\mathrm{F}}\big)
\circ_{\mathsf{v}}
\alpha_{\mathrm{FC}^{X},\mathrm{F}^{\star},\mathrm{F}}^{\mone}
\circ_{\mathsf{v}}
\big(\mathrm{id}_{\mathrm{FC}^{X}}\circ_{\mathsf{h}}\mathrm{coev}_{\mathrm{F}}\big)
\circ_{\mathsf{v}}
(\runit_{\mathrm{FC}^{X}})^{\mone}.
\end{gather}
The fifth isomorphism
\begin{gather*}
\mathrm{Hom}_{\underline{\ccC}}\big(\mathrm{FC}^{X},(\mathrm{H}_2\mathrm{H}_2)\mathrm{F}\big)\cong
\mathrm{Hom}_{\underline{\ccC}}\Big(\mathrm{C}^{X},\mathrm{F}^{\star}\big((\mathrm{H}_2\mathrm{H}_2)\mathrm{F}\big)\Big)
\end{gather*}
sends \eqref{eq:0.12} to
\begin{gather}\label{eq:0.13}
\begin{aligned}
&\Big(\mathrm{id}_{\mathrm{F}^{\star}}\circ_{\mathsf{h}}\big(\delta_{\mathrm{H}_2}\circ_{\mathsf{h}}\mathrm{id}_{\mathrm{F}}\big)\Big)
\circ_{\mathsf{v}}
\big(\mathrm{id}_{\mathrm{F}^{\star}}\circ_{\mathsf{h}}\alpha_{\mathrm{FC}^{X},\mathrm{F}^{\star},\mathrm{F}}^{\mone}\big)
\circ_{\mathsf{v}}
\Big(\mathrm{id}_{\mathrm{F}^{\star}}\circ_{\mathsf{h}}\big(\mathrm{id}_{\mathrm{FC}^{X}}\circ_{\mathsf{h}}\mathrm{coev}_{\mathrm{F}}\big)\Big)
\\&
\circ_{\mathsf{v}}
\big(\mathrm{id}_{\mathrm{F}^{\star}}\circ_{\mathsf{h}}(\runit_{\mathrm{FC}^{X}})^{\mone}\big)
\circ_{\mathsf{v}}
\alpha_{\mathrm{F}^{\star},\mathrm{F},\mathrm{C}^{X}}^{\mone}
\circ_{\mathsf{v}}
\big(\mathrm{coev}_{\mathrm{F}}\circ_{\mathsf{h}}\mathrm{id}_{\mathrm{C}^{X}}\big)
\circ_{\mathsf{v}}
(\lunit_{\mathrm{C}^{X}})^{\mone}.
\end{aligned}
\end{gather}
The fourth isomorphism
\begin{gather*}
\mathrm{Hom}_{\underline{\ccC}}\Big(\mathrm{C}^{X},\mathrm{F}^{\star}\big((\mathrm{H}_2\mathrm{H}_2)\mathrm{F}\big)\Big)\xrightarrow{\cong}
\mathrm{Hom}_{\underline{\mathbf{M}}(\mathtt{i})}\bigg(X,
\underline{\mathbf{M}}_{\mathtt{ii}}\Big(\mathrm{F}^{\star}\big((\mathrm{H}_2\mathrm{H}_2)\mathrm{F}\big)\Big)X\bigg)
\end{gather*}
sends \eqref{eq:0.13} to
\begin{gather}\label{eq:0.14}
\begin{aligned}
&\underline{\mathbf{M}}_{\mathtt{ii}}\Big(\mathrm{id}_{\mathrm{F}^{\star}}\circ_{\mathsf{h}}\big(\delta_{\mathrm{H}_2}\circ_{\mathsf{h}}\mathrm{id}_{\mathrm{F}}\big)\Big)_X
\circ_{\mathsf{v}}
\underline{\mathbf{M}}_{\mathtt{ii}}\big(\mathrm{id}_{\mathrm{F}^{\star}}\circ_{\mathsf{h}}\alpha_{\mathrm{FC}^{X},\mathrm{F}^{\star},\mathrm{F}}^{\mone}\big)_X
\\&
\circ_{\mathsf{v}}
\underline{\mathbf{M}}_{\mathtt{ii}}\Big(\mathrm{id}_{\mathrm{F}^{\star}}\circ_{\mathsf{h}}\big(\mathrm{id}_{\mathrm{FC}^{X}}\circ_{\mathsf{h}}\mathrm{coev}_{\mathrm{F}}\big)\Big)_X
\circ_{\mathsf{v}}
\underline{\mathbf{M}}_{\mathtt{ii}}\big(\mathrm{id}_{\mathrm{F}^{\star}}\circ_{\mathsf{h}}(\runit_{\mathrm{FC}^{X}})^{\mone}\big)_X
\\&
\circ_{\mathsf{v}}
\underline{\mathbf{M}}_{\mathtt{ii}}\big(\alpha_{\mathrm{F}^{\star},\mathrm{F},\mathrm{C}^{X}}^{\mone}\big)_X
\circ_{\mathsf{v}}
\underline{\mathbf{M}}_{\mathtt{ii}}\big(\mathrm{coev}_{\mathrm{F}}\circ_{\mathsf{h}}\mathrm{id}_{\mathrm{C}^{X}}\big)_X
\circ_{\mathsf{v}}
\underline{\mathbf{M}}_{\mathtt{ii}}\big((\lunit_{\mathrm{C}^{X}})^{\mone}\big)_X
\\&
\circ_{\mathsf{v}}
\mathrm{coev}_{X,X}.
\end{aligned}
\end{gather}
Finally, consider the diagram
\begin{gather*}
\adjustbox{scale=.6,center}{%
\begin{tikzcd}[ampersand replacement=\&, column sep=4em, row sep=3.5em]
X
\ar[d,"{\mathrm{coev}_{X,X}}"description]
\ar[rr,"\big(\iota_{\mathtt{i}}\big)^{\mone}_{X}"]
\ar[rr, phantom,yshift=-8ex, "\circled{1}"]
\&\&
\underline{\mathbf{M}}_{\mathtt{ii}}(\mathbbm{1}_{\mathtt{i}})X
\ar[d,"{{\underline{\mathbf{M}}_{\mathtt{ii}}(\mathbbm{1}_{\mathtt{i}})}\mathrm{coev}_{X,X}}"description]
\ar[drr,"\underline{\mathbf{M}}_{\mathtt{ii}}\big(\mathrm{coev}_{\mathrm{F}}\big)_X"]
\ar[drr, phantom,yshift=-8ex, "\circled{4}"]
\&\&
\\[5ex]
\underline{\mathbf{M}}_{\mathtt{ii}}(\mathrm{C}^{X})X
\ar[d,"{\underline{\mathbf{M}}_{\mathtt{ii}}\big((\lunit_{\mathrm{C}^{X}})^{\mone}\big)_X}"description]
\ar[rr,"\big(\iota_{\mathtt{i}}\big)^{\mone}_{\underline{\mathbf{M}}_{\mathtt{ii}}(\mathrm{C}^{X})X}"]
\ar[rr, phantom,yshift=-4ex, xshift=-4ex,"\circled{2}"]
\&\&
\underline{\mathbf{M}}_{\mathtt{ii}}(\mathbbm{1}_{\mathtt{i}})\underline{\mathbf{M}}_{\mathtt{ii}}(\mathrm{C}^{X})X
\ar[d,"{\underline{\mathbf{M}}_{\mathtt{ii}}\big(\mathrm{coev}_{\mathrm{F}}\big)_{\underline{\mathbf{M}}_{\mathtt{ii}}(\mathrm{C}^{X})X}}"description]
\ar[dll,"\big(\mu_{\mathtt{iii}}^{\mathbbm{1}_{\mathtt{i}},\mathrm{C}^{X}}\big)_X",sloped]
\ar[dll, phantom,yshift=-8ex, "\circled{3}"]
\&\&
\underline{\mathbf{M}}_{\mathtt{ii}}(\mathrm{F}^{\star}\mathrm{F})X
\ar[dll,"{\underline{\mathbf{M}}_{\mathtt{ii}}(\mathrm{F}^{\star}\mathrm{F})}\mathrm{coev}_{X,X}"]
\ar[d,"\big(\mu_{\mathtt{iji}}^{\mathrm{F}^{\star},\mathrm{F}}\big)^{\mone}_X"]
\\[5ex]
\underline{\mathbf{M}}_{\mathtt{ii}}(\mathbbm{1}_{\mathtt{i}}\mathrm{C}^{X})X
\ar[d,"{\underline{\mathbf{M}}_{\mathtt{ii}}\big(\mathrm{coev}_{\mathrm{F}}\circ_{\mathsf{h}}\mathrm{id}_{\mathrm{C}^{X}}\big)_X}"description]
\&\&
\underline{\mathbf{M}}_{\mathtt{ii}}(\mathrm{F}^{\star}\mathrm{F})\underline{\mathbf{M}}_{\mathtt{ii}}(\mathrm{C}^{X})X
\ar[d,"{\big(\mu_{\mathtt{iji}}^{\mathrm{F}^{\star},\mathrm{F}}\big)^{\mone}_{\underline{\mathbf{M}}_{\mathtt{ii}}(\mathrm{C}^{X})X}}"description]
\ar[dll,"\big(\mu_{\mathtt{iii}}^{\mathrm{F}^{\star}\mathrm{F},\mathrm{C}^{X}}\big)_X",sloped]
\&\&
\underline{\mathbf{M}}_{\mathtt{ij}}(\mathrm{F}^{\star})\underline{\mathbf{M}}_{\mathtt{ji}}(\mathrm{F})X
\ar[dll,"{\underline{\mathbf{M}}_{\mathtt{ij}}(\mathrm{F}^{\star})\underline{\mathbf{M}}_{\mathtt{ji}}(\mathrm{F})}\mathrm{coev}_{X,X}"]
\ar[dll, phantom,yshift=8ex, "\circled{6}"]
\\[5ex]
\underline{\mathbf{M}}_{\mathtt{ii}}\big((\mathrm{F}^{\star}\mathrm{F})\mathrm{C}^{X}\big)X
\ar[d, "{\underline{\mathbf{M}}_{\mathtt{ii}}\big(\alpha_{\mathrm{F}^{\star},\mathrm{F},\mathrm{C}^{X}}^{\mone}\big)_X}"description]
\&\&
\underline{\mathbf{M}}_{\mathtt{ij}}(\mathrm{F}^{\star})\underline{\mathbf{M}}_{\mathtt{ji}}(\mathrm{F})\underline{\mathbf{M}}_{\mathtt{ii}}(\mathrm{C}^{X})X
\ar[d,"{{\underline{\mathbf{M}}_{\mathtt{ii}}(\mathrm{F}^{\star})}\big(\mu_{\mathtt{jii}}^{\mathrm{F},\mathrm{C}^{X}}\big)_X}"description]
\&\&
\\[5ex]
\underline{\mathbf{M}}_{\mathtt{ii}}\big(\mathrm{F}^{\star}(\mathrm{F}\mathrm{C}^{X})\big)X
\ar[d, "{\underline{\mathbf{M}}_{\mathtt{ii}}\big(\mathrm{id}_{\mathrm{F}^{\star}}\circ_{\mathsf{h}}(\runit_{\mathrm{FC}^{X}})^{\mone}\big)_X}"description]
\&\&
\underline{\mathbf{M}}_{\mathtt{ij}}(\mathrm{F}^{\star})\underline{\mathbf{M}}_{\mathtt{ji}}(\mathrm{F}\mathrm{C}^{X})X
\ar[ll,"\big(\mu_{\mathtt{iji}}^{\mathrm{F}^{\star},\mathrm{F}\mathrm{C}^{X}}\big)_X",swap]
\ar[d, "{{\underline{\mathbf{M}}_{\mathtt{ij}}(\mathrm{F}^{\star})}\underline{\mathbf{M}}_{\mathtt{ji}}\big((\runit_{\mathrm{FC}^{X}})^{\mone}\big)_X}"description]
\ar[ll, phantom,yshift=14ex, "\circled{5}"]
\ar[ll, phantom,yshift=-8ex, "\circled{7}"]
\&\&
\\[5ex]
\underline{\mathbf{M}}_{\mathtt{ii}}\Big(\mathrm{F}^{\star}\big((\mathrm{F}\mathrm{C}^{X})\mathbbm{1}_{\mathtt{i}}\big)\Big)X
\ar[d, "{\underline{\mathbf{M}}_{\mathtt{ii}}\Big(\mathrm{id}_{\mathrm{F}^{\star}}\circ_{\mathsf{h}}\big(\mathrm{id}_{\mathrm{FC}^{X}}\circ_{\mathsf{h}}\mathrm{coev}_{\mathrm{F}}\big)\Big)_X}"description]
\&\&
\underline{\mathbf{M}}_{\mathtt{ij}}(\mathrm{F}^{\star})\underline{\mathbf{M}}_{\mathtt{ji}}\big((\mathrm{F}\mathrm{C}^{X})\mathbbm{1}_{\mathtt{i}}\big)X
\ar[ll,"\big(\mu_{\mathtt{iji}}^{\mathrm{F}^{\star},(\mathrm{F}\mathrm{C}^{X})\mathbbm{1}_{\mathtt{i}}}\big)_X",swap]
\ar[d, "{{\underline{\mathbf{M}}_{\mathtt{ij}}(\mathrm{F}^{\star})}\underline{\mathbf{M}}_{\mathtt{ji}}\big(\mathrm{id}_{\mathrm{FC}^{X}}\circ_{\mathsf{h}}\mathrm{coev}_{\mathrm{F}}\big)_X}"description]
\ar[ll, phantom,yshift=-8ex, "\circled{8}"]
\&\&
\\[5ex]
\underline{\mathbf{M}}_{\mathtt{ii}}\Big(\mathrm{F}^{\star}\big((\mathrm{F}\mathrm{C}^{X})(\mathrm{F}^{\star}\mathrm{F})\big)\Big)X
\ar[d, "{\underline{\mathbf{M}}_{\mathtt{ii}}\big(\mathrm{id}_{\mathrm{F}^{\star}}\circ_{\mathsf{h}}\alpha_{\mathrm{FC}^{X},\mathrm{F}^{\star},\mathrm{F}}^{\mone}
\big)_X}"description]
\&\&
\underline{\mathbf{M}}_{\mathtt{ij}}(\mathrm{F}^{\star})\underline{\mathbf{M}}_{\mathtt{ji}}\big((\mathrm{F}\mathrm{C}^{X})(\mathrm{F}^{\star}\mathrm{F})\big)X
\ar[ll,"\big(\mu_{\mathtt{iji}}^{\mathrm{F}^{\star},(\mathrm{F}\mathrm{C}^{X})(\mathrm{F}^{\star}\mathrm{F})}\big)_X",swap]
\ar[d, "{{\underline{\mathbf{M}}_{\mathtt{ij}}(\mathrm{F}^{\star})}\underline{\mathbf{M}}_{\mathtt{ji}}\big(\alpha_{\mathrm{FC}^{X},\mathrm{F}^{\star},\mathrm{F}}^{\mone}\big)_X}"description]
\ar[ll, phantom,yshift=-8ex, "\circled{9}"]
\&\&
\\[5ex]
\underline{\mathbf{M}}_{\mathtt{ii}}\big(\mathrm{F}^{\star}(\mathrm{H}_2\mathrm{F})\big)X
\ar[d, "{\underline{\mathbf{M}}_{\mathtt{ii}}\big(\mathrm{id}_{\mathrm{F}^{\star}}\circ_{\mathsf{h}}(\delta_{\mathrm{H}_2}\circ_{\mathsf{h}}\mathrm{id}_{\mathrm{F}})
\big)_X}"description]
\&\&
\underline{\mathbf{M}}_{\mathtt{ij}}(\mathrm{F}^{\star})\underline{\mathbf{M}}_{\mathtt{ji}}(\mathrm{H}_2\mathrm{F})X
\ar[ll,"\big(\mu_{\mathtt{iji}}^{\mathrm{F}^{\star},\mathrm{H}_2\mathrm{F}}\big)_X",swap]
\ar[d, "{{\underline{\mathbf{M}}_{\mathtt{ij}}(\mathrm{F}^{\star})}\underline{\mathbf{M}}_{\mathtt{ji}}\big(\delta_{\mathrm{H}_2}\circ_{\mathsf{h}}\mathrm{id}_{\mathrm{F}}\big)_X}"description]
\ar[ll, phantom,yshift=-8ex, "\circled{10}"]
\&\&
\\[5ex]
\underline{\mathbf{M}}_{\mathtt{ii}}\Big(\mathrm{F}^{\star}\big((\mathrm{H}_2\mathrm{H}_2)\mathrm{F}\big)\Big)X
\&\&
\underline{\mathbf{M}}_{\mathtt{ij}}(\mathrm{F}^{\star})\underline{\mathbf{M}}_{\mathtt{ji}}\big((\mathrm{H}_2\mathrm{H}_2)\mathrm{F}\big)X
\ar[ll,"\big(\mu_{\mathtt{iji}}^{\mathrm{F}^{\star},(\mathrm{H}_2\mathrm{H}_2)\mathrm{F}}\big)_X"]
\&\&
\end{tikzcd}}
\end{gather*}
which commutes due to
\begin{itemize}
\item naturality of $\iota_{\mathtt{i}}^{\mone}$ for the facet labeled $1$;

\item the right diagram in \eqref{eq:birepresentation2} for the facet labeled $2$;

\item naturality of $\mu_{\mathtt{iii}}$ for the facet labeled $3$;

\item naturality of $\underline{\mathbf{M}}_{\mathtt{ii}}(\mathrm{coev}_{\mathrm{F}})$ for the facet labeled $4$;

\item the diagram in \eqref{eq:birepresentation3} for the facet labeled $5$;

\item naturality of $\big(\mu_{\mathtt{iji}}^{\mathrm{F}^{\star},\mathrm{F}}\big)^{\mone}$ for the facet labeled $6$;

\item naturality of $\mu_{\mathtt{iji}}$ for the facets labeled $7$, $8$, $9$ and $10$.
\end{itemize}
As above, commutativity guarantees that the $2$-morphisms corresponding to the paths along the boundary from northwest to southwest coincide. The path which goes straight down corresponds to the $2$-morphisms in \eqref{eq:0.14}, which is therefore equal to
\begin{gather*}\scalebox{0.83}{$
\begin{aligned}
&\big(\mu_{\mathtt{iji}}^{\mathrm{F}^{\star},(\mathrm{H}_2\mathrm{H}_2)\mathrm{F}}\big)_X
\circ_{\mathsf{v}}
\Big({\underline{\mathbf{M}}_{\mathtt{ij}}(\mathrm{F}^{\star})}\underline{\mathbf{M}}_{\mathtt{ji}}\big(\delta_{\mathrm{H}_2}\circ_{\mathsf{h}}\mathrm{id}_{\mathrm{F}}\big)_X\Big)
\circ_{\mathsf{v}}
\Big({\underline{\mathbf{M}}_{\mathtt{ij}}(\mathrm{F}^{\star})}\underline{\mathbf{M}}_{\mathtt{ji}}\big(\alpha_{\mathrm{FC}^{X},\mathrm{F}^{\star},\mathrm{F}}^{\mone}\big)_X\Big)
\\&
\circ_{\mathsf{v}}
\Big({\underline{\mathbf{M}}_{\mathtt{ij}}(\mathrm{F}^{\star})}\underline{\mathbf{M}}_{\mathtt{ji}}\big(\mathrm{id}_{\mathrm{FC}^{X}}\circ_{\mathsf{h}}\mathrm{coev}_{\mathrm{F}}\big)_X\Big)
\circ_{\mathsf{v}}
\Big({\underline{\mathbf{M}}_{\mathtt{ij}}(\mathrm{F}^{\star})}\underline{\mathbf{M}}_{\mathtt{ji}}\big((\runit_{\mathrm{FC}^{X}})^{\mone}\big)_X\Big)
\\&
\circ_{\mathsf{v}}
\Big({\underline{\mathbf{M}}_{\mathtt{ij}}(\mathrm{F}^{\star})}\big(\mu_{\mathtt{jii}}^{\mathrm{F},\mathrm{C}^{X}}\big)_X\Big)
\circ_{\mathsf{v}}
\Big({\underline{\mathbf{M}}_{\mathtt{ij}}(\mathrm{F}^{\star})\underline{\mathbf{M}}_{\mathtt{ji}}(\mathrm{F})}\mathrm{coev}_{X,X}\Big)
\\&
\circ_{\mathsf{v}}
\big(\mu_{\mathtt{iji}}^{\mathrm{F}^{\star},\mathrm{F}}\big)^{\mone}_X
\circ_{\mathsf{v}}
\underline{\mathbf{M}}_{\mathtt{ii}}\big(\mathrm{coev}_{\mathrm{F}}\big)_X
\circ_{\mathsf{v}}
\big(\iota_{\mathtt{i}}\big)^{\mone}_{X}
\end{aligned}$}
\end{gather*}
This expression coincides precisely with that in \eqref{eq:0.11},
considering \eqref{eq:0.10} and the definition of $\delta_{\mathrm{H}_2}=\delta_{(\mathrm{FC}^{X})\mathrm{F}^{\star}}$ as in Lemma \ref{lem0.1}. This completes the proof.
\end{proof}

As in \cite[Corollary 5.2]{MMMT} (see also \cite[Lemma 7.9.4]{EGNO}), the internal cohoms
$[X,\underline{\mathbf{M}}_{\mathtt{ji}}(\mathrm{F})\,X]$ and $[\underline{\mathbf{M}}_{\mathtt{ji}}(\mathrm{F})\,X,X]$ are the biinjective
$(\mathrm{FC}^{X})\mathrm{F}^{\star}\text{-}\mathrm{C}^{X}$-bicomodules respectively $\mathrm{C}^{X}\text{-}(\mathrm{FC}^{X})\mathrm{F}^{\star}$-bicomodules inducing this MT equivalence. Firstly, noting that
$\mathrm{C}^{X}=[X,X]$ and $(\mathrm{FC}^{X})
\mathrm{F}^{\star}\cong[\underline{\mathbf{M}}_{\mathtt{ji}}(\mathrm{F})\,X,\underline{\mathbf{M}}_{\mathtt{ji}}(\mathrm{F})\,X]$, we have
\begin{gather*}
\begin{aligned}
\mathrm{Hom}_{\underline{\ccC}(\mathtt{i},\mathtt{j})}\big([X,\underline{\mathbf{M}}_{\mathtt{ji}}(\mathrm{F})\,X],\mathrm{G}\big)
&\cong
\mathrm{Hom}_{\underline{\mathbf{M}}(\mathtt{j})}\big(\underline{\mathbf{M}}_{\mathtt{ji}}(\mathrm{F})\,X,\underline{\mathbf{M}}_{\mathtt{ji}}(\mathrm{G})\,X\big)
\\
&\cong\mathrm{Hom}_{\underline{\mathbf{M}}(\mathtt{i})}\big(X,\underline{\mathbf{M}}_{\mathtt{ij}}(\mathrm{F}^{\star})\underline{\mathbf{M}}_{\mathtt{ji}}(\mathrm{G})\,X\big)
\\
&\cong\mathrm{Hom}_{\underline{\mathbf{M}}(\mathtt{i})}\big(X,\underline{\mathbf{M}}_{\mathtt{ii}}(\mathrm{F}^{\star}\mathrm{G})\,X\big)
\\
&\cong\mathrm{Hom}_{\underline{\ccC}(\mathtt{i},\mathtt{i})}(\mathrm{C}^{X},\mathrm{F}^{\star}\mathrm{G})
\\
&\cong\mathrm{Hom}_{\underline{\ccC}(\mathtt{i},\mathtt{j})}(\mathrm{FC}^{X},\mathrm{G})
\end{aligned}
\end{gather*}
and
\begin{gather*}
\begin{aligned}
\mathrm{Hom}_{\underline{\ccC}(\mathtt{j},\mathtt{i})}
\big([\underline{\mathbf{M}}_{\mathtt{ji}}(\mathrm{F})\,X,X],\mathrm{H}\big)&\cong
\mathrm{Hom}_{\underline{\mathbf{M}}(\mathtt{i})}\big(X,\underline{\mathbf{M}}_{\mathtt{ij}}(\mathrm{H})\underline{\mathbf{M}}_{\mathtt{ji}}(\mathrm{F})\,X\big)
\\
&\cong
\mathrm{Hom}_{\underline{\mathbf{M}}(\mathtt{i})}\big(X,\underline{\mathbf{M}}_{\mathtt{ii}}(\mathrm{H}\mathrm{F})\,X\big)
\\
&\cong\mathrm{Hom}_{\underline{\ccC}(\mathtt{i},\mathtt{i})}(\mathrm{C}^{X},\mathrm{H}\mathrm{F})
\\
&\cong\mathrm{Hom}_{\underline{\ccC}(\mathtt{j},\mathtt{i})}(\mathrm{C}^{X}\mathrm{F}^{\star}, \mathrm{H})
\end{aligned}
\end{gather*}
for all $1$-morphisms $\mathrm{G}\in\underline{\cC}(\mathtt{i},\mathtt{j})$ and $\mathrm{H}\in\underline{\cC}(\mathtt{j},\mathtt{i})$. Therefore, there are isomorphisms of $1$-morphisms
\begin{gather}\label{eq:inthom}
[X,\underline{\mathbf{M}}_{\mathtt{ji}}(\mathrm{F})\,X]\cong\mathrm{FC}^{X}
\quad\text{and}\quad
[\underline{\mathbf{M}}_{\mathtt{ji}}(\mathrm{F})\,X,X]\cong\mathrm{C}^{X}
\mathrm{F}^{\star}.
\end{gather}

\begin{corollary}\label{corollary:mt} Under the same assumptions as in Theorem \ref{prop0.4}, the coalgebras $\mathrm{C}^{X}$ and
$(\mathrm{FC}^{X})\mathrm{F}^{\star}\cong\mathrm{F}(\mathrm{C}^{X}\mathrm{F}^{\star})$
are MT equivalent. Moreover, the MT equivalence is realized by the bicomodules $\mathrm{FC}^{X}$ and
$\mathrm{C}^{X}\mathrm{F}^{\star}$, whose right and left $\mathrm{C}^{X}$-comodule structures, respectively, are the canonical ones and whose left
and right $(\mathrm{FC}^{X})\mathrm{F}^{\star}$-comodule structures, respectively, are given by
\begin{gather*}
\begin{aligned}
&\alpha_{\mathrm{FC}^{X},\mathrm{F}^{\star},\mathrm{FC}^{X}}^{\mone}
\circ_{\mathsf{v}}
\big(\mathrm{id}_{\mathrm{FC}^{X}}\circ_{\mathsf{h}}\alpha_{\mathrm{F}^{\star},\mathrm{F},\mathrm{C}^{X}}\big)
\circ_{\mathsf{v}}
\big(\mathrm{id}_{\mathrm{FC}^{X}}\circ_{\mathsf{h}}(\mathrm{coev}_{\mathrm{F}}\circ_{\mathsf{h}}\mathrm{id}_{\mathrm{C}^{X}})\big)
\\&
\circ_{\mathsf{v}}
\big(\mathrm{id}_{\mathrm{FC}^{X}}\circ_{\mathsf{h}}(\lunit_{\mathrm{C}^{X}})^{\mone}\big)
\circ_{\mathsf{v}}
\alpha_{\mathrm{F},\mathrm{C}^{X},\mathrm{C}^{X}}^{\mone}
\circ_{\mathsf{v}}
\big(\mathrm{id}_{\mathrm{F}}\circ_{\mathsf{h}}\delta_{\mathrm{C}^{X}}\big)
\end{aligned}
\end{gather*}
and
\begin{gather*}
\begin{aligned}
&\big(\mathrm{id}_{\mathrm{C}^{X}\mathrm{F}^{\star}}\circ_{\mathsf{h}}\alpha_{\mathrm{F},\mathrm{C}^{X},\mathrm{F}^{\star}}^{\mone}\big)
\circ_{\mathsf{v}}
\alpha_{\mathrm{C}^{X}\mathrm{F}^{\star},\mathrm{F},\mathrm{C}^{X}\mathrm{F}^{\star}}
\circ_{\mathsf{v}}
\big(\alpha_{\mathrm{C}^{X},\mathrm{F}^{\star},\mathrm{F}}^{\mone}\circ_{\mathsf{h}}\mathrm{id}_{\mathrm{C}^{X}\mathrm{F}^{\star}}\big)
\\&
\circ_{\mathsf{v}}
\big((\mathrm{id}_{\mathrm{C}^{X}}\circ_{\mathsf{h}}\mathrm{coev}_{\mathrm{F}})\circ_{\mathsf{h}}\mathrm{id}_{\mathrm{C}^{X}\mathrm{F}^{\star}}\big)
\circ_{\mathsf{v}}
\big((\runit_{\mathrm{C}^{X}})^{\mone}\circ_{\mathsf{h}}\mathrm{id}_{\mathrm{C}^{X}\mathrm{F}^{\star}}\big)
\\&
\circ_{\mathsf{v}}
\alpha_{\mathrm{C}^{X},\mathrm{C}^{X},\mathrm{F}^{\star}}
\circ_{\mathsf{v}}
(\delta_{\mathrm{C}^{X}}\circ_{\mathsf{h}}\mathrm{id}_{\mathrm{F}^{\star}}).
\end{aligned}
\end{gather*}
\end{corollary}

\begin{proof}
The first statement follows immediately from Corollary \ref{corollary:MT} and Theorem \ref{prop0.4}.

The proof of the second statement is more involved. On the one hand, by the first isomorphism in \eqref{eq:associativity-cotensor2},
we have
\begin{gather*}
(\mathrm{F}\mathrm{C}^{X})\mathrm{F}^{\star}
\cong(\mathrm{F}\mathrm{C}^{X}\square_{\mathrm{C}^{X}}\mathrm{C}^{X})\mathrm{F}^{\star}
\cong(\mathrm{F}\mathrm{C}^{X})\square_{\mathrm{C}^{X}}(\mathrm{C}^{X}\mathrm{F}^{\star}).
\end{gather*}

On the other hand, to prove $(\mathrm{C}^{X}\mathrm{F}^{\star})\square_{(\mathrm{FC}^{X})\mathrm{F}^{\star}}(\mathrm{F}\mathrm{C}^{X})\cong \mathrm{C}^{X}$,
we consider the following diagram
\begin{gather}\label{diag0}
\adjustbox{scale=.6,center}{%
\begin{tikzcd}[ampersand replacement=\&,column sep=2em]
(\mathrm{C}^{X}\mathrm{F}^{\star})\square_{(\mathrm{FC}^{X})\mathrm{F}^{\star}}(\mathrm{F}\mathrm{C}^{X})
\ar[rr, hook, "\beta"]
\&\&
(\mathrm{C}^{X}\mathrm{F}^{\star})(\mathrm{F}\mathrm{C}^{X})
\arrow[rr, "\delta_{\mathrm{C}^{X}\mathrm{F}^{\star},(\mathrm{F}\mathrm{C}^{X})\mathrm{F}^{\star}}\circ_{\mathsf{h}}\mathrm{id}_{\mathrm{FC}^{X}}"]
\arrow[dr, "\mathrm{id}_{\mathrm{C}^{X}\mathrm{F}^{\star}}\circ_{\mathsf{h}}\delta_{(\mathrm{F}\mathrm{C}^{X})\mathrm{F}^{\star},\mathrm{FC}^{X}}"]
\&\&
\Big(\big(\mathrm{C}^{X}\mathrm{F}^{\star}\big)\big((\mathrm{F}\mathrm{C}^{X})\mathrm{F}^{\star}\big)\Big)\Big(\mathrm{F}\mathrm{C}^{X}\Big)
\arrow[dl, "\alpha_{\mathrm{C}^{X}\mathrm{F}^{\star},(\mathrm{F}\mathrm{C}^{X})\mathrm{F}^{\star},\mathrm{FC}^{X}}"]
\\[5ex]
\&
\mathrm{C}^{X}
\ar[ur,"\gamma", swap]
\ar[ul,"\theta",dashed]
\&\&
\big(\mathrm{C}^{X}\mathrm{F}^{\star}\big)\Big(\big((\mathrm{F}\mathrm{C}^{X})\mathrm{F}^{\star}\big)\big(\mathrm{F}\mathrm{C}^{X}\big)\Big)
\&
\end{tikzcd}}
\end{gather}
where $\delta_{(\mathrm{F}\mathrm{C}^{X})\mathrm{F}^{\star},\mathrm{FC}^{X}}$ and $\delta_{\mathrm{C}^{X}\mathrm{F}^{\star},(\mathrm{F}\mathrm{C}^{X})\mathrm{F}^{\star}}$ are
the left and right $(\mathrm{FC}^{X})\mathrm{F}^{\star}$-coaction $2$-morphisms, respectively,
and $\gamma$ is given by
\begin{gather*}
\begin{aligned}
&\alpha_{\mathrm{C}^{X}\mathrm{F}^{\star},\mathrm{F},\mathrm{C}^{X}}
\circ_{\mathsf{v}}
\big(\alpha_{\mathrm{C}^{X},\mathrm{F}^{\star},\mathrm{F}}^{\mone}\circ_{\mathsf{h}}\mathrm{id}_{\mathrm{C}^{X}}\big)
\circ_{\mathsf{v}}
\big((\mathrm{id}_{\mathrm{C}^{X}}\circ_{\mathsf{h}}\mathrm{coev}_{\mathrm{F}})\circ_{\mathsf{h}}\mathrm{id}_{\mathrm{C}^{X}}\big)
\\&
\circ_{\mathsf{v}}
\big((\runit_{\mathrm{C}^{X}})^{\mone}\circ_{\mathsf{h}}\mathrm{id}_{\mathrm{C}^{X}}\big)
\circ_{\mathsf{v}}
\delta_{\mathrm{C}^{X}}.
\end{aligned}
\end{gather*}
Now we claim that $\gamma$ equalizes the right triangle in diagram \eqref{diag0},
which, by the universal property of the equalizer, implies that there exists a unique $2$-morphism $\theta$ such that the left triangle in diagram \eqref{diag0} commutes.
Consider the diagram
\begin{gather*}
\scalebox{0.44}{
$\begin{tikzcd}[ampersand replacement=\&, column sep=2.9em, row sep=2.5em]
\mathrm{C}^{X}
\ar[dd,"{\delta_{\mathrm{C}^{X}}}"description]
\ar[rrr,"\delta_{\mathrm{C}^{X}}"]
\ar[rrr, phantom,yshift=-15ex, "\circled{1}"]
\&\&\&
\mathrm{C}^{X}\mathrm{C}^{X}
\ar[d,"{\delta_{\mathrm{C}^{X}}\circ_{\mathsf{h}}\mathrm{id}_{\mathrm{C}^{X}}}"description]
\ar[rr,"(\runit_{\mathrm{C}^{X}})^{\mone}\circ_{\mathsf{h}}\mathrm{id}_{\mathrm{C}^{X}}"]
\ar[rr, phantom,yshift=-7ex, "\circled{2}"]
\&\&
(\mathrm{C}^{X}\mathbbm{1}_{\mathtt{i}})\mathrm{C}^{X}
\ar[d, xshift=1ex,"{(\delta_{\mathrm{C}^{X}}\circ_{\mathsf{h}}\mathrm{id}_{\mathbbm{1}_{\mathtt{i}}})\circ_{\mathsf{h}}\mathrm{id}_{\mathrm{C}^{X}}}"description]
\ar[rrrrr, phantom,yshift=-7ex, "\circled{3}"]
\ar[rrrrr,"(\mathrm{id}_{\mathrm{C}^{X}}\circ_{\mathsf{h}}\mathrm{coev}_{\mathrm{F}})\circ_{\mathsf{h}}\mathrm{id}_{\mathrm{C}^{X}}"]
\&\&\&\&\&
\big(\mathrm{C}^{X}(\mathrm{F}^{\star}\mathrm{F})\big)\mathrm{C}^{X}
\ar[d,"{(\delta_{\mathrm{C}^{X}}\circ_{\mathsf{h}}\mathrm{id}_{\mathrm{F}^{\star}\mathrm{F}})\circ_{\mathsf{h}}\mathrm{id}_{\mathrm{C}^{X}}}"description]
\\[5ex]
\&\&\&
(\mathrm{C}^{X}\mathrm{C}^{X})\mathrm{C}^{X}
\ar[d,"\alpha_{\mathrm{C}^{X},\mathrm{C}^{X},\mathrm{C}^{X}}",swap]
\ar[dr,"{\mathrm{id}_{\mathrm{C}^{X}\mathrm{C}^{X}}\circ_{\mathsf{h}}(\lunit_{\mathrm{C}^{X}})^{\mone}}"description]
\ar[rr,"(\runit_{\mathrm{C}^{X}\mathrm{C}^{X}})^{\mone}\circ_{\mathsf{h}}\mathrm{id}_{\mathrm{C}^{X}}"]
\ar[rr, phantom,yshift=-5ex, "\circled{6}"]
\ar[dr, phantom,yshift=-8ex, "\circled{5}"]
\&\&
\big((\mathrm{C}^{X}\mathrm{C}^{X})\mathbbm{1}_{\mathtt{i}}\big)\mathrm{C}^{X}
\ar[d, xshift=1ex,"{\left(\left(\left(\runit_{\mathrm{C}^{X}}\right)^{\mone}\circ_{\mathsf{h}}\mathrm{id}_{\mathrm{C}^{X}}\right)\circ_{\mathsf{h}}\mathrm{id}_{\mathbbm{1}_{\mathtt{i}}}\right)\circ_{\mathsf{h}}\mathrm{id}_{\mathrm{C}^{X}}}"description, near end]
\ar[dl,"\alpha_{\mathrm{C}^{X}\mathrm{C}^{X},\mathbbm{1}_{\mathtt{i}},\mathrm{C}^{X}}^{\mone}",swap, near start]
\ar[rrrrr, phantom,yshift=-7ex, "\circled{8}"]
\ar[dl, phantom,yshift=-8ex, "\circled{7}"]
\ar[rrrrr,"(\mathrm{id}_{\mathrm{C}^{X}\mathrm{C}^{X}}\circ_{\mathsf{h}}\mathrm{coev}_{\mathrm{F}})\circ_{\mathsf{h}}\mathrm{id}_{\mathrm{C}^{X}}"]
\&\&\&\&\&
\big((\mathrm{C}^{X}\mathrm{C}^{X})(\mathrm{F}^{\star}\mathrm{F})\big)\mathrm{C}^{X}
\ar[d, "{\left(\left(\left(\runit_{\mathrm{C}^{X}}\right)^{\mone}\circ_{\mathsf{h}}\mathrm{id}_{\mathrm{C}^{X}}\right)\circ_{\mathsf{h}}\mathrm{id}_{\mathrm{F}^{\star}\mathrm{F}}\right)\circ_{\mathsf{h}}\mathrm{id}_{\mathrm{C}^{X}}}"description]
\\[5ex]
\mathrm{C}^{X}\mathrm{C}^{X}
\ar[rrr,"\mathrm{id}_{\mathrm{C}^{X}}\circ_{\mathsf{h}}\delta_{\mathrm{C}^{X}}"]
\ar[rrr, phantom,yshift=-15ex, "\circled{4}"]
\ar[dd,"{(\runit_{\mathrm{C}^{X}})^{\mone}\circ_{\mathsf{h}}\mathrm{id}_{\mathrm{C}^{X}}}"description]
\&\&\&
\mathrm{C}^{X}(\mathrm{C}^{X}\mathrm{C}^{X})
\ar[dr,xshift=-2ex,"{\mathrm{id}_{\mathrm{C}^{X}}\circ_{\mathsf{h}}\left(\mathrm{id}_{\mathrm{C}^{X}}\circ_{\mathsf{h}}\left(\lunit_{\mathrm{C}^{X}}\right)^{\mone}\right)}"description, near end]
\ar[dd,"{(\runit_{\mathrm{C}^{X}})^{\mone}\circ_{\mathsf{h}}\mathrm{id}_{\mathrm{C}^{X}\mathrm{C}^{X}}}"description, near end]
\&
(\mathrm{C}^{X}\mathrm{C}^{X})(\mathbbm{1}_{\mathtt{i}}\mathrm{C}^{X})
\ar[d,xshift=1ex, "\alpha_{\mathrm{C}^{X},\mathrm{C}^{X},\mathbbm{1}_{\mathtt{i}}\mathrm{C}^{X}}", near start, swap]
\ar[dr,"{\left(\left(\runit_{\mathrm{C}^{X}}\right)^{\mone}\circ_{\mathsf{h}}\mathrm{id}_{\mathrm{C}^{X}}\right)\circ_{\mathsf{h}}\mathrm{id}_{\mathbbm{1}_{\mathtt{i}}\mathrm{C}^{X}}}"description]
\ar[dr, phantom,yshift=-8ex, "\circled{11}"]
\&
\Big(\big((\mathrm{C}^{X}\mathbbm{1}_{\mathtt{i}})\mathrm{C}^{X}\big)\mathbbm{1}_{\mathtt{i}}\Big)\mathrm{C}^{X}
\ar[d,xshift=1ex,"\alpha_{(\mathrm{C}^{X}\mathbbm{1}_{\mathtt{i}})\mathrm{C}^{X},\mathbbm{1}_{\mathtt{i}},\mathrm{C}^{X}}"]
\ar[rrrrr,"(\mathrm{id}_{(\mathrm{C}^{X}\mathbbm{1}_{\mathtt{i}})\mathrm{C}^{X}}\circ_{\mathsf{h}}\mathrm{coev}_{\mathrm{F}})\circ_{\mathsf{h}}\mathrm{id}_{\mathrm{C}^{X}}"]
\ar[rrrrr, phantom,yshift=-7ex, "\circled{9}"]
\&\&\&\&\&
\Big(\big((\mathrm{C}^{X}\mathbbm{1}_{\mathtt{i}})\mathrm{C}^{X}\big)(\mathrm{F}^{\star}\mathrm{F})\Big)\mathrm{C}^{X}
\ar[d,"{\alpha_{(\mathrm{C}^{X}\mathbbm{1}_{\mathtt{i}})\mathrm{C}^{X},\mathrm{F}^{\star}\mathrm{F},\mathrm{C}^{X}}}"description]
\\[5ex]
\&\&\&\&
\mathrm{C}^{X}\big(\mathrm{C}^{X}(\mathbbm{1}_{\mathtt{i}}\mathrm{C}^{X})\big)
\ar[dr,"{(\runit_{\mathrm{C}^{X}})^{\mone}\circ_{\mathsf{h}}\mathrm{id}_{\mathrm{C}^{X}(\mathbbm{1}_{\mathtt{i}}\mathrm{C}^{X})}}"description]
\&
\big((\mathrm{C}^{X}\mathbbm{1}_{\mathtt{i}})\mathrm{C}^{X}\big)\big(\mathbbm{1}_{\mathtt{i}}\mathrm{C}^{X}\big)
\ar[d,xshift=1ex,"\alpha_{\mathrm{C}^{X}\mathbbm{1}_{\mathtt{i}},\mathrm{C}^{X},\mathbbm{1}_{\mathtt{i}}\mathrm{C}^{X}}"]
\ar[rrrrr,"\mathrm{id}_{(\mathrm{C}^{X}\mathbbm{1}_{\mathtt{i}})\mathrm{C}^{X}}\circ_{\mathsf{h}}(\mathrm{coev}_{\mathrm{F}}\circ_{\mathsf{h}}\mathrm{id}_{\mathrm{C}^{X}})"]
\ar[rrrrr, phantom,yshift=-7ex, "\circled{12}"]
\&\&\&\&\&
\big((\mathrm{C}^{X}\mathbbm{1}_{\mathtt{i}})\mathrm{C}^{X}\big)\big((\mathrm{F}^{\star}\mathrm{F})\mathrm{C}^{X}\big)
\ar[d,"{\alpha_{\mathrm{C}^{X}\mathbbm{1}_{\mathtt{i}},\mathrm{C}^{X},(\mathrm{F}^{\star}\mathrm{F})\mathrm{C}^{X}}}"description]
\\[5ex]
(\mathrm{C}^{X}\mathbbm{1}_{\mathtt{i}})\mathrm{C}^{X}
\ar[rrr,"\mathrm{id}_{\mathrm{C}^{X}\mathbbm{1}_{\mathtt{i}}}\circ_{\mathsf{h}}\delta_{\mathrm{C}^{X}}"]
\ar[rrr, phantom,yshift=-7ex, "\circled{13}"]
\ar[d,"{(\mathrm{id}_{\mathrm{C}^{X}}\circ_{\mathsf{h}}\mathrm{coev}_{\mathrm{F}})\circ_{\mathsf{h}}\mathrm{id}_{\mathrm{C}^{X}}}"description]
\&\&\&
(\mathrm{C}^{X}\mathbbm{1}_{\mathtt{i}})(\mathrm{C}^{X}\mathrm{C}^{X})
\ar[rr,"\mathrm{id}_{\mathrm{C}^{X}\mathbbm{1}_{\mathtt{i}}}\circ_{\mathsf{h}}\left(\mathrm{id}_{\mathrm{C}^{X}}\circ_{\mathsf{h}}\left(\lunit_{\mathrm{C}^{X}}\right)^{\mone}\right)"]
\ar[rr, phantom,yshift=-7ex, "\circled{14}"]
\ar[rr, phantom,yshift=8ex, xshift=-3ex, "\circled{10}"]
\ar[d,"{(\mathrm{id}_{\mathrm{C}^{X}}\circ_{\mathsf{h}}\mathrm{coev}_{\mathrm{F}})\circ_{\mathsf{h}}\mathrm{id}_{\mathrm{C}^{X}\mathrm{C}^{X}}}"description]
\&\&
\big(\mathrm{C}^{X}\mathbbm{1}_{\mathtt{i}}\big)\big(\mathrm{C}^{X}(\mathbbm{1}_{\mathtt{i}}\mathrm{C}^{X})\big)
\ar[rrrrr,"\mathrm{id}_{\mathrm{C}^{X}\mathbbm{1}_{\mathtt{i}}}\circ_{\mathsf{h}}\left(\mathrm{id}_{\mathrm{C}^{X}}\circ_{\mathsf{h}}\left(\mathrm{coev}_{\mathrm{F}}\circ_{\mathsf{h}}\mathrm{id}_{\mathrm{C}^{X}}\right)\right)"]
\ar[rrrrr, phantom,yshift=-7ex, "\circled{15}"]
\ar[d,xshift=1ex,"{(\mathrm{id}_{\mathrm{C}^{X}}\circ_{\mathsf{h}}\mathrm{coev}_{\mathrm{F}})\circ_{\mathsf{h}}\mathrm{id}_{\mathrm{C}^{X}(\mathbbm{1}_{\mathtt{i}}\mathrm{C}^{X})}}"description]
\&\&\&\&\&
\Big(\mathrm{C}^{X}\mathbbm{1}_{\mathtt{i}}\Big)\Big(\mathrm{C}^{X}\big((\mathrm{F}^{\star}\mathrm{F})\mathrm{C}^{X}\big)\Big)
\ar[d,"{(\mathrm{id}_{\mathrm{C}^{X}}\circ_{\mathsf{h}}\mathrm{coev}_{\mathrm{F}})\circ_{\mathsf{h}}\mathrm{id}_{\mathrm{C}^{X}\left(\left(\mathrm{F}^{\star}\mathrm{F}\right)\mathrm{C}^{X}\right)}}"description]
\\[5ex]
\big(\mathrm{C}^{X}(\mathrm{F}^{\star}\mathrm{F})\big)\mathrm{C}^{X}
\ar[rrr,"\mathrm{id}_{\mathrm{C}^{X}(\mathrm{F}^{\star}\mathrm{F})}\circ_{\mathsf{h}}\delta_{\mathrm{C}^{X}}"]
\ar[rrr, phantom,yshift=-7ex, "\circled{16}"]
\ar[d,"{\alpha_{\mathrm{C}^{X},\mathrm{F}^{\star},\mathrm{F}}^{\mone}\circ_{\mathsf{h}}\mathrm{id}_{\mathrm{C}^{X}}}"description]
\&\&\&
\big(\mathrm{C}^{X}(\mathrm{F}^{\star}\mathrm{F})\big)\big(\mathrm{C}^{X}\mathrm{C}^{X}\big)
\ar[rr,"\mathrm{id}_{\mathrm{C}^{X}(\mathrm{F}^{\star}\mathrm{F})}\circ_{\mathsf{h}}\left(\mathrm{id}_{\mathrm{C}^{X}}\circ_{\mathsf{h}}\left(\lunit_{\mathrm{C}^{X}}\right)^{\mone}\right)"]
\ar[rr, phantom,yshift=-7ex,"\circled{17}"]
\ar[d,"{\alpha_{\mathrm{C}^{X},\mathrm{F}^{\star},\mathrm{F}}^{\mone}\circ_{\mathsf{h}}\mathrm{id}_{\mathrm{C}^{X}\mathrm{C}^{X}}}"description]
\&\&
\big(\mathrm{C}^{X}(\mathrm{F}^{\star}\mathrm{F})\big)\big(\mathrm{C}^{X}(\mathbbm{1}_{\mathtt{i}}\mathrm{C}^{X})\big)
\ar[rrrrr,"\mathrm{id}_{\mathrm{C}^{X}(\mathrm{F}^{\star}\mathrm{F})}\circ_{\mathsf{h}}\left(\mathrm{id}_{\mathrm{C}^{X}}\circ_{\mathsf{h}}\left(\mathrm{coev}_{\mathrm{F}}\circ_{\mathsf{h}}\mathrm{id}_{\mathrm{C}^{X}}\right)\right)"]
\ar[rrrrr, phantom,yshift=-7ex,"\circled{18}"]
\ar[d,xshift=1ex,"{\alpha_{\mathrm{C}^{X},\mathrm{F}^{\star},\mathrm{F}}^{\mone}\circ_{\mathsf{h}}\mathrm{id}_{\mathrm{C}^{X}(\mathbbm{1}_{\mathtt{i}}\mathrm{C}^{X})}}"description]
\&\&\&\&\&
\Big(\mathrm{C}^{X}(\mathrm{F}^{\star}\mathrm{F})\Big)\Big(\mathrm{C}^{X}\big((\mathrm{F}^{\star}\mathrm{F})\mathrm{C}^{X}\big)\Big)
\ar[d,"{\alpha_{\mathrm{C}^{X},\mathrm{F}^{\star},\mathrm{F}}^{\mone}\circ_{\mathsf{h}}\mathrm{id}_{\mathrm{C}^{X}\left(\left(\mathrm{F}^{\star}\mathrm{F}\right)\mathrm{C}^{X}\right)}}"description]
\\[5ex]
\big((\mathrm{C}^{X}\mathrm{F}^{\star})\mathrm{F}\big)\mathrm{C}^{X}
\ar[rrr,"\mathrm{id}_{(\mathrm{C}^{X}\mathrm{F}^{\star})\mathrm{F}}\circ_{\mathsf{h}}\delta_{\mathrm{C}^{X}}"]
\ar[rrr, phantom,yshift=-7ex, "\circled{19}"]
\ar[d,"{\alpha_{\mathrm{C}^{X}\mathrm{F}^{\star},\mathrm{F},\mathrm{C}^{X}}}"description]
\&\&\&
\big((\mathrm{C}^{X}\mathrm{F}^{\star})\mathrm{F}\big)\big(\mathrm{C}^{X}\mathrm{C}^{X}\big)
\ar[rr,"\mathrm{id}_{(\mathrm{C}^{X}\mathrm{F}^{\star})\mathrm{F}}\circ_{\mathsf{h}}\left(\mathrm{id}_{\mathrm{C}^{X}}\circ_{\mathsf{h}}\left(\lunit_{\mathrm{C}^{X}}\right)^{\mone}\right)"]
\ar[rr, phantom,yshift=-7ex, "\circled{20}"]
\ar[d,"{\alpha_{\mathrm{C}^{X}\mathrm{F}^{\star},\mathrm{F},\mathrm{C}^{X}\mathrm{C}^{X}}}"description]
\&\&
\big((\mathrm{C}^{X}\mathrm{F}^{\star})\mathrm{F}\big)\big(\mathrm{C}^{X}(\mathbbm{1}_{\mathtt{i}}\mathrm{C}^{X})\big)
\ar[rrrrr,"\mathrm{id}_{(\mathrm{C}^{X}\mathrm{F}^{\star})\mathrm{F}}\circ_{\mathsf{h}}\left(\mathrm{id}_{\mathrm{C}^{X}}\circ_{\mathsf{h}}\left(\mathrm{coev}_{\mathrm{F}}\circ_{\mathsf{h}}\mathrm{id}_{\mathrm{C}^{X}}\right)\right)"]
\ar[rrrrr, phantom,yshift=-7ex, "\circled{21}"]
\ar[d,xshift=1ex,"{\alpha_{\mathrm{C}^{X}\mathrm{F}^{\star},\mathrm{F},\mathrm{C}^{X}(\mathbbm{1}_{\mathtt{i}}\mathrm{C}^{X})}}"description]
\&\&\&\&\&
\Big((\mathrm{C}^{X}\mathrm{F}^{\star})\mathrm{F}\Big)\Big(\mathrm{C}^{X}\big((\mathrm{F}^{\star}\mathrm{F})\mathrm{C}^{X}\big)\Big)
\ar[d,"{\alpha_{\mathrm{C}^{X}\mathrm{F}^{\star},\mathrm{F},\mathrm{C}^{X}\left(\left(\mathrm{F}^{\star}\mathrm{F}\right)\mathrm{C}^{X}\right)}}"description]
\\[5ex]
(\mathrm{C}^{X}\mathrm{F}^{\star})(\mathrm{F}\mathrm{C}^{X})
\ar[rrr,"\mathrm{id}_{\mathrm{C}^{X}\mathrm{F}^{\star}}\circ_{\mathsf{h}}(\mathrm{id}_{\mathrm{F}}\circ_{\mathsf{h}}\delta_{\mathrm{C}^{X}})"]
\&\&\&
\big(\mathrm{C}^{X}\mathrm{F}^{\star}\big)\big(\mathrm{F}(\mathrm{C}^{X}\mathrm{C}^{X})\big)
\ar[rr,"\mathrm{id}_{\mathrm{C}^{X}\mathrm{F}^{\star}}\circ_{\mathsf{h}}\left(\mathrm{id}_{\mathrm{F}}\circ_{\mathsf{h}}\left(\mathrm{id}_{\mathrm{C}^{X}}\circ_{\mathsf{h}}\left(\lunit_{\mathrm{C}^{X}}\right)^{\mone}\right)\right)"]
\ar[rr, phantom,yshift=-7ex, "\circled{22}"]
\ar[d,"{\mathrm{id}_{\mathrm{C}^{X}\mathrm{F}^{\star}}\circ_{\mathsf{h}}\alpha_{\mathrm{F},\mathrm{C}^{X},\mathrm{C}^{X}}^{\mone}}"description]
\&\&
\big(\mathrm{C}^{X}\mathrm{F}^{\star}\big)\Big(\mathrm{F}\big(\mathrm{C}^{X}(\mathbbm{1}_{\mathtt{i}}\mathrm{C}^{X})\big)\Big)
\ar[rrrrr,"\mathrm{id}_{\mathrm{C}^{X}\mathrm{F}^{\star}}\circ_{\mathsf{h}}\left(\mathrm{id}_{\mathrm{F}}\circ_{\mathsf{h}}\left(\mathrm{id}_{\mathrm{C}^{X}}\circ_{\mathsf{h}}\left(\mathrm{coev}_{\mathrm{F}}\circ_{\mathsf{h}}\mathrm{id}_{\mathrm{C}^{X}}\right)\right)\right)"]
\ar[rrrrr, phantom,yshift=-7ex, "\circled{23}"]
\ar[d,"{\mathrm{id}_{\mathrm{C}^{X}\mathrm{F}^{\star}}\circ_{\mathsf{h}}\alpha_{\mathrm{F},\mathrm{C}^{X},\mathbbm{1}_{\mathtt{i}}\mathrm{C}^{X}}^{\mone}}"description]
\&\&\&\&\&
\big(\mathrm{C}^{X}\mathrm{F}^{\star}\big)\bigg(\mathrm{F}\Big(\mathrm{C}^{X}\big((\mathrm{F}^{\star}\mathrm{F})\mathrm{C}^{X}\big)\Big)\bigg)
\ar[d,"{\mathrm{id}_{\mathrm{C}^{X}\mathrm{F}^{\star}}\circ_{\mathsf{h}}\alpha_{\mathrm{F},\mathrm{C}^{X},(\mathrm{F}^{\star}\mathrm{F})\mathrm{C}^{X}}^{\mone}}"description]
\\[5ex]
\&\&\&
\big(\mathrm{C}^{X}\mathrm{F}^{\star}\big)\big((\mathrm{F}\mathrm{C}^{X})\mathrm{C}^{X}\big)
\ar[rr,"\mathrm{id}_{\mathrm{C}^{X}\mathrm{F}^{\star}}\circ_{\mathsf{h}}\left(\mathrm{id}_{\mathrm{F}\mathrm{C}^{X}}\circ_{\mathsf{h}}\left(\lunit_{\mathrm{C}^{X}}\right)^{\mone}\right)"]
\&\&
\big(\mathrm{C}^{X}\mathrm{F}^{\star}\big)\big((\mathrm{F}\mathrm{C}^{X})(\mathbbm{1}_{\mathtt{i}}\mathrm{C}^{X})\big)
\ar[rrrrr,"\mathrm{id}_{\mathrm{C}^{X}\mathrm{F}^{\star}}\circ_{\mathsf{h}}\left(\mathrm{id}_{\mathrm{F}\mathrm{C}^{X}}\circ_{\mathsf{h}}\left(\mathrm{coev}_{\mathrm{F}}\circ_{\mathsf{h}}\mathrm{id}_{\mathrm{C}^{X}}\right)\right)"]
\&\&\&\&\&
\big(\mathrm{C}^{X}\mathrm{F}^{\star}\big)\Big(\big(\mathrm{F}\mathrm{C}^{X}\big)\big((\mathrm{F}^{\star}\mathrm{F})\mathrm{C}^{X}\big)\Big)
\end{tikzcd}$}
\end{gather*}
which commutes due to
\begin{itemize}
\item coassociativity of $\delta_{\mathrm{C}^{X}}$ for the facet labeled $1$;

\item naturality of $(\runit_{})^{\mone}$ for the facet labeled $2$;

\item the interchange law for the facets labeled $3$, $4$, $8$, $10$, $13$, $14$, $15$, $16$, $17$ and $18$;

\item naturality of $\alpha$ and $\alpha^{\mone}$ for the facets labeled $5$, $7$, $9$, $11$, $12$, $19$, $20$, $21$, $22$ and $23$;

\item the triangle coherence condition of the unitors for the facet labeled $6$,

\end{itemize}
and the diagram
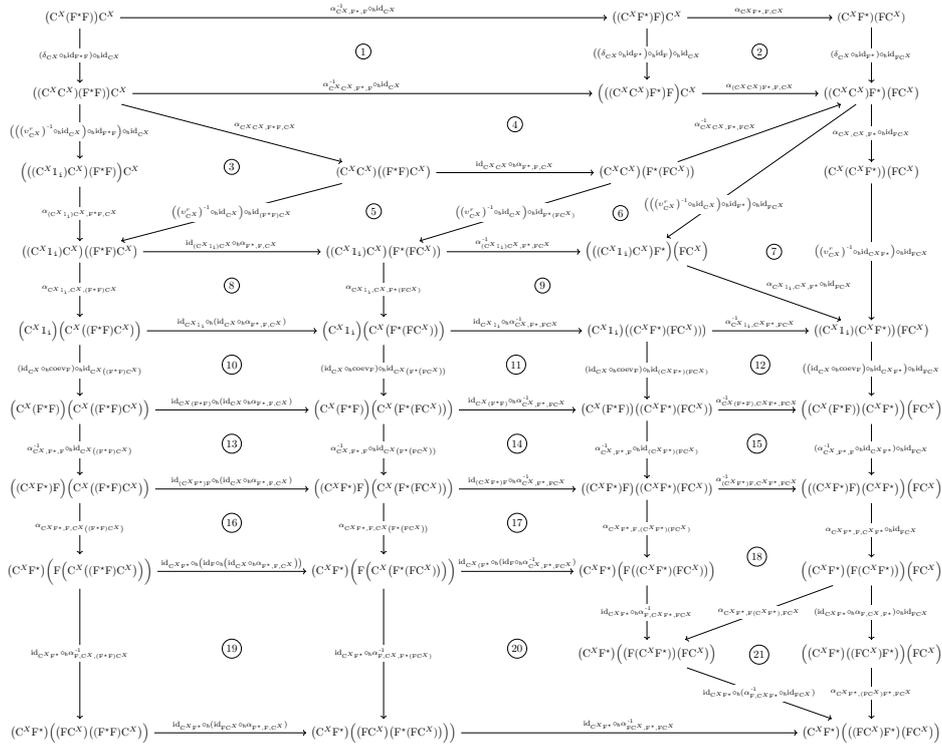
\begin{gather*}
\scalebox{0.41}{
$\begin{tikzcd}[ampersand replacement=\&, column sep=3.4em, row sep=2.5em]
\big(\mathrm{C}^{X}(\mathrm{F}^{\star}\mathrm{F})\big)\mathrm{C}^{X}
\ar[d,"{(\delta_{\mathrm{C}^{X}}\circ_{\mathsf{h}}\mathrm{id}_{\mathrm{F}^{\star}\mathrm{F}})\circ_{\mathsf{h}}\mathrm{id}_{\mathrm{C}^{X}}}"description]
\ar[rrrrrrr,"\alpha_{\mathrm{C}^{X},\mathrm{F}^{\star},\mathrm{F}}^{\mone}\circ_{\mathsf{h}}\mathrm{id}_{\mathrm{C}^{X}}"]
\ar[rrrrrrr, phantom,yshift=-7ex, "\circled{1}"]
\&\&\&\&\&\&\&
\big((\mathrm{C}^{X}\mathrm{F}^{\star})\mathrm{F}\big)\mathrm{C}^{X}
\ar[d,"{\left(\left(\delta_{\mathrm{C}^{X}}\circ_{\mathsf{h}}\mathrm{id}_{\mathrm{F}^{\star}}\right)\circ_{\mathsf{h}}\mathrm{id}_{\mathrm{F}}\right)\circ_{\mathsf{h}}\mathrm{id}_{\mathrm{C}^{X}}}"description]
\ar[rr,"\alpha_{\mathrm{C}^{X}\mathrm{F}^{\star},\mathrm{F},\mathrm{C}^{X}}"]
\ar[rr, phantom,yshift=-7ex, "\circled{2}"]
\&\&
(\mathrm{C}^{X}\mathrm{F}^{\star})(\mathrm{F}\mathrm{C}^{X})
\ar[d,"{(\delta_{\mathrm{C}^{X}}\circ_{\mathsf{h}}\mathrm{id}_{\mathrm{F}^{\star}})\circ_{\mathsf{h}}\mathrm{id}_{\mathrm{F}\mathrm{C}^{X}}}"description]
\\[5ex]
\big((\mathrm{C}^{X}\mathrm{C}^{X})(\mathrm{F}^{\star}\mathrm{F})\big)\mathrm{C}^{X}
\ar[d, "{\left(\left(\left(\runit_{\mathrm{C}^{X}}\right)^{\mone}\circ_{\mathsf{h}}\mathrm{id}_{\mathrm{C}^{X}}\right)\circ_{\mathsf{h}}\mathrm{id}_{\mathrm{F}^{\star}\mathrm{F}}\right)\circ_{\mathsf{h}}\mathrm{id}_{\mathrm{C}^{X}}}"description]
\ar[rrrrrrr,"\alpha_{\mathrm{C}^{X}\mathrm{C}^{X},\mathrm{F}^{\star},\mathrm{F}}^{\mone}\circ_{\mathsf{h}}\mathrm{id}_{\mathrm{C}^{X}}"]
\ar[drrrr,"\alpha_{\mathrm{C}^{X}\mathrm{C}^{X},\mathrm{F}^{\star}\mathrm{F},\mathrm{C}^{X}}"]
\ar[drrrr, phantom,yshift=-7ex, "\circled{3}"]
\&\&\&\&\&\&\&
\Big(\big((\mathrm{C}^{X}\mathrm{C}^{X})\mathrm{F}^{\star}\big)\mathrm{F}\Big)\mathrm{C}^{X}
\ar[rr,"\alpha_{(\mathrm{C}^{X}\mathrm{C}^{X})\mathrm{F}^{\star},\mathrm{F},\mathrm{C}^{X}}"]
\&\&
\big((\mathrm{C}^{X}\mathrm{C}^{X})\mathrm{F}^{\star}\big)\big(\mathrm{F}\mathrm{C}^{X}\big)
\ar[ddll,"{\left(\left(\left(\runit_{\mathrm{C}^{X}}\right)^{\mone}\circ_{\mathsf{h}}\mathrm{id}_{\mathrm{C}^{X}}\right)\circ_{\mathsf{h}}\mathrm{id}_{\mathrm{F}^{\star}}\right)\circ_{\mathsf{h}}\mathrm{id}_{\mathrm{F}\mathrm{C}^{X}}}"description, near end]
\ar[d,"{\alpha_{\mathrm{C}^{X},\mathrm{C}^{X},\mathrm{F}^{\star}}\circ_{\mathsf{h}}\mathrm{id}_{\mathrm{FC}^{X}}}"description]
\\[5ex]
\Big(\big((\mathrm{C}^{X}\mathbbm{1}_{\mathtt{i}})\mathrm{C}^{X}\big)(\mathrm{F}^{\star}\mathrm{F})\Big)\mathrm{C}^{X}
\ar[d,"{\alpha_{(\mathrm{C}^{X}\mathbbm{1}_{\mathtt{i}})\mathrm{C}^{X},\mathrm{F}^{\star}\mathrm{F},\mathrm{C}^{X}}}"description]
\&\&\&\&
(\mathrm{C}^{X}\mathrm{C}^{X})\big((\mathrm{F}^{\star}\mathrm{F})\mathrm{C}^{X}\big)
\ar[rrr,"\mathrm{id}_{\mathrm{C}^{X}\mathrm{C}^{X}}\circ_{\mathsf{h}}\alpha_{\mathrm{F}^{\star},\mathrm{F},\mathrm{C}^{X}}"]
\ar[dllll, "{\left(\left(\runit_{\mathrm{C}^{X}}\right)^{\mone}\circ_{\mathsf{h}}\mathrm{id}_{\mathrm{C}^{X}}\right)\circ_{\mathsf{h}}\mathrm{id}_{(\mathrm{F}^{\star}\mathrm{F})\mathrm{C}^{X}}}"description]
\ar[rrr, phantom,yshift=10ex, "\circled{4}"]
\ar[rrr, phantom,yshift=-8ex, xshift=-29ex,"\circled{5}"]
\&\&\&
\big(\mathrm{C}^{X}\mathrm{C}^{X}\big)\big(\mathrm{F}^{\star}(\mathrm{F}\mathrm{C}^{X})\big)
\ar[urr,"\alpha_{\mathrm{C}^{X}\mathrm{C}^{X},\mathrm{F}^{\star},\mathrm{F}\mathrm{C}^{X}}^{\mone}"]
\ar[dlll, "{\left(\left(\runit_{\mathrm{C}^{X}}\right)^{\mone}\circ_{\mathsf{h}}\mathrm{id}_{\mathrm{C}^{X}}\right)\circ_{\mathsf{h}}\mathrm{id}_{\mathrm{F}^{\star}(\mathrm{F}\mathrm{C}^{X})}}"description]
\&\&
\big(\mathrm{C}^{X}(\mathrm{C}^{X}\mathrm{F}^{\star})\big)\big(\mathrm{F}\mathrm{C}^{X}\big)
\ar[dd,"{\left(\left(\runit_{\mathrm{C}^{X}}\right)^{\mone}\circ_{\mathsf{h}}\mathrm{id}_{\mathrm{C}^{X}\mathrm{F}^{\star}}\right)\circ_{\mathsf{h}}\mathrm{id}_{\mathrm{F}\mathrm{C}^{X}}}"description]
\ar[dd, phantom,xshift=-20ex, "\circled{7}"]
\\[5ex]
\big((\mathrm{C}^{X}\mathbbm{1}_{\mathtt{i}})\mathrm{C}^{X}\big)\big((\mathrm{F}^{\star}\mathrm{F})\mathrm{C}^{X}\big)
\ar[d,"{\alpha_{\mathrm{C}^{X}\mathbbm{1}_{\mathtt{i}},\mathrm{C}^{X},(\mathrm{F}^{\star}\mathrm{F})\mathrm{C}^{X}}}"description]
\ar[rrrr,"\mathrm{id}_{(\mathrm{C}^{X}\mathbbm{1}_{\mathtt{i}})\mathrm{C}^{X}}\circ_{\mathsf{h}}\alpha_{\mathrm{F}^{\star},\mathrm{F},\mathrm{C}^{X}}"]
\ar[rrrr, phantom,yshift=-7ex, "\circled{8}"]
\&\&\&\&
\big((\mathrm{C}^{X}\mathbbm{1}_{\mathtt{i}})\mathrm{C}^{X}\big)\big(\mathrm{F}^{\star}(\mathrm{F}\mathrm{C}^{X})\big)
\ar[rrr,"\alpha_{(\mathrm{C}^{X}\mathbbm{1}_{\mathtt{i}})\mathrm{C}^{X},\mathrm{F}^{\star},\mathrm{F}\mathrm{C}^{X}}^{\mone}"]
\ar[rrr, phantom,yshift=8ex, xshift=22ex, "\circled{6}"]
\ar[d,"{\alpha_{\mathrm{C}^{X}\mathbbm{1}_{\mathtt{i}},\mathrm{C}^{X},\mathrm{F}^{\star}(\mathrm{F}\mathrm{C}^{X})}}"description]
\ar[rrr, phantom,yshift=-7ex,xshift=6ex, "\circled{9}"]
\&\&\&
\Big(\big((\mathrm{C}^{X}\mathbbm{1}_{\mathtt{i}})\mathrm{C}^{X}\big)\mathrm{F}^{\star}\Big)\Big(\mathrm{F}\mathrm{C}^{X}\Big)
\ar[drr,"\alpha_{\mathrm{C}^{X}\mathbbm{1}_{\mathtt{i}},\mathrm{C}^{X},\mathrm{F}^{\star}}\circ_{\mathsf{h}}\mathrm{id}_{\mathrm{F}\mathrm{C}^{X}}"]
\&\&
\\[5ex]
\Big(\mathrm{C}^{X}\mathbbm{1}_{\mathtt{i}}\Big)\Big(\mathrm{C}^{X}\big((\mathrm{F}^{\star}\mathrm{F})\mathrm{C}^{X}\big)\Big)
\ar[d,"{(\mathrm{id}_{\mathrm{C}^{X}}\circ_{\mathsf{h}}\mathrm{coev}_{\mathrm{F}})\circ_{\mathsf{h}}\mathrm{id}_{\mathrm{C}^{X}\left(\left(\mathrm{F}^{\star}\mathrm{F}\right)\mathrm{C}^{X}\right)}}"description]
\ar[rrrr,"\mathrm{id}_{\mathrm{C}^{X}\mathbbm{1}_{\mathtt{i}}}\circ_{\mathsf{h}}(\mathrm{id}_{\mathrm{C}^{X}}\circ_{\mathsf{h}}\alpha_{\mathrm{F}^{\star},\mathrm{F},\mathrm{C}^{X}})"]
\ar[rrrr, phantom,yshift=-7ex, "\circled{10}"]
\&\&\&\&
\Big(\mathrm{C}^{X}\mathbbm{1}_{\mathtt{i}}\Big)\Big(\mathrm{C}^{X}\big(\mathrm{F}^{\star}(\mathrm{F}\mathrm{C}^{X})\big)\Big)
\ar[rrr,"\mathrm{id}_{\mathrm{C}^{X}\mathbbm{1}_{\mathtt{i}}}\circ_{\mathsf{h}}\alpha_{\mathrm{C}^{X},\mathrm{F}^{\star},\mathrm{F}\mathrm{C}^{X}}^{\mone}"]
\ar[d,"{(\mathrm{id}_{\mathrm{C}^{X}}\circ_{\mathsf{h}}\mathrm{coev}_{\mathrm{F}})\circ_{\mathsf{h}}\mathrm{id}_{\mathrm{C}^{X}\left(\mathrm{F}^{\star}\left(\mathrm{F}\mathrm{C}^{X}\right)\right)}}"description]
\ar[rrr, phantom,yshift=-7ex, "\circled{11}"]
\&\&\&
\big(\mathrm{C}^{X}\mathbbm{1}_{\mathtt{i}}\big)\big((\mathrm{C}^{X}\mathrm{F}^{\star})(\mathrm{F}\mathrm{C}^{X}))\big)
\ar[rr,"\alpha_{\mathrm{C}^{X}\mathbbm{1}_{\mathtt{i}},\mathrm{C}^{X}\mathrm{F}^{\star},\mathrm{F}\mathrm{C}^{X}}^{\mone}"]
\ar[rr, phantom,yshift=-7ex,"\circled{12}"]
\ar[d,"{(\mathrm{id}_{\mathrm{C}^{X}}\circ_{\mathsf{h}}\mathrm{coev}_{\mathrm{F}})\circ_{\mathsf{h}}\mathrm{id}_{(\mathrm{C}^{X}\mathrm{F}^{\star})(\mathrm{F}\mathrm{C}^{X})}}"description]
\&\&
\big((\mathrm{C}^{X}\mathbbm{1}_{\mathtt{i}})(\mathrm{C}^{X}\mathrm{F}^{\star})\big)\big(\mathrm{F}\mathrm{C}^{X}\big)
\ar[d,"{\left(\left(\mathrm{id}_{\mathrm{C}^{X}}\circ_{\mathsf{h}}\mathrm{coev}_{\mathrm{F}}\right)\circ_{\mathsf{h}}\mathrm{id}_{\mathrm{C}^{X}\mathrm{F}^{\star}}\right)\circ_{\mathsf{h}}\mathrm{id}_{\mathrm{F}\mathrm{C}^{X}}}"description]
\\[5ex]
\Big(\mathrm{C}^{X}(\mathrm{F}^{\star}\mathrm{F})\Big)\Big(\mathrm{C}^{X}\big((\mathrm{F}^{\star}\mathrm{F})\mathrm{C}^{X}\big)\Big)
\ar[d,"{\alpha_{\mathrm{C}^{X},\mathrm{F}^{\star},\mathrm{F}}^{\mone}\circ_{\mathsf{h}}\mathrm{id}_{\mathrm{C}^{X}\left(\left(\mathrm{F}^{\star}\mathrm{F}\right)\mathrm{C}^{X}\right)}}"description]
\ar[rrrr,"\mathrm{id}_{\mathrm{C}^{X}(\mathrm{F}^{\star}\mathrm{F})}\circ_{\mathsf{h}}(\mathrm{id}_{\mathrm{C}^{X}}\circ_{\mathsf{h}}\alpha_{\mathrm{F}^{\star},\mathrm{F},\mathrm{C}^{X}})"]
\ar[rrrr, phantom,yshift=-7ex, "\circled{13}"]
\&\&\&\&
\Big(\mathrm{C}^{X}(\mathrm{F}^{\star}\mathrm{F})\Big)\Big(\mathrm{C}^{X}\big(\mathrm{F}^{\star}(\mathrm{F}\mathrm{C}^{X})\big)\Big)
\ar[rrr,"\mathrm{id}_{\mathrm{C}^{X}(\mathrm{F}^{\star}\mathrm{F})}\circ_{\mathsf{h}}\alpha_{\mathrm{C}^{X},\mathrm{F}^{\star},\mathrm{F}\mathrm{C}^{X}}^{\mone}"]
\ar[d,"{\alpha_{\mathrm{C}^{X},\mathrm{F}^{\star},\mathrm{F}}^{\mone}\circ_{\mathsf{h}}\mathrm{id}_{\mathrm{C}^{X}\left(\mathrm{F}^{\star}\left(\mathrm{F}\mathrm{C}^{X}\right)\right)}}"description]
\ar[rrr, phantom,yshift=-7ex, "\circled{14}"]
\&\&\&
\big(\mathrm{C}^{X}(\mathrm{F}^{\star}\mathrm{F})\big)\big((\mathrm{C}^{X}\mathrm{F}^{\star})(\mathrm{F}\mathrm{C}^{X})\big)
\ar[rr,"\alpha_{\mathrm{C}^{X}(\mathrm{F}^{\star}\mathrm{F}),\mathrm{C}^{X}\mathrm{F}^{\star},\mathrm{F}\mathrm{C}^{X}}^{\mone}"]
\ar[rr, phantom,yshift=-7ex, "\circled{15}"]
\ar[d,"{\alpha_{\mathrm{C}^{X},\mathrm{F}^{\star},\mathrm{F}}^{\mone}\circ_{\mathsf{h}}\mathrm{id}_{(\mathrm{C}^{X}\mathrm{F}^{\star})(\mathrm{F}\mathrm{C}^{X})}}"description]
\&\&
\Big(\big(\mathrm{C}^{X}(\mathrm{F}^{\star}\mathrm{F})\big)\big(\mathrm{C}^{X}\mathrm{F}^{\star}\big)\Big)\Big(\mathrm{F}\mathrm{C}^{X}\Big)
\ar[d,"{(\alpha_{\mathrm{C}^{X},\mathrm{F}^{\star},\mathrm{F}}^{\mone}\circ_{\mathsf{h}}\mathrm{id}_{\mathrm{C}^{X}\mathrm{F}^{\star}})\circ_{\mathsf{h}}\mathrm{id}_{\mathrm{F}\mathrm{C}^{X}}}"description]
\\[5ex]
\Big((\mathrm{C}^{X}\mathrm{F}^{\star})\mathrm{F}\Big)\Big(\mathrm{C}^{X}\big((\mathrm{F}^{\star}\mathrm{F})\mathrm{C}^{X}\big)\Big)
\ar[d,"{\alpha_{\mathrm{C}^{X}\mathrm{F}^{\star},\mathrm{F},\mathrm{C}^{X}\left(\left(\mathrm{F}^{\star}\mathrm{F}\right)\mathrm{C}^{X}\right)}}"description]
\ar[rrrr,"\mathrm{id}_{(\mathrm{C}^{X}\mathrm{F}^{\star})\mathrm{F}}\circ_{\mathsf{h}}(\mathrm{id}_{\mathrm{C}^{X}}\circ_{\mathsf{h}}\alpha_{\mathrm{F}^{\star},\mathrm{F},\mathrm{C}^{X}})"]
\ar[rrrr, phantom,yshift=-7ex, "\circled{16}"]
\&\&\&\&
\Big((\mathrm{C}^{X}\mathrm{F}^{\star})\mathrm{F}\Big)\Big(\mathrm{C}^{X}\big(\mathrm{F}^{\star}(\mathrm{F}\mathrm{C}^{X})\big)\Big)
\ar[d,"{\alpha_{\mathrm{C}^{X}\mathrm{F}^{\star},\mathrm{F},\mathrm{C}^{X}\left(\mathrm{F}^{\star}\left(\mathrm{F}\mathrm{C}^{X}\right)\right)}}"description]
\ar[rrr, phantom,yshift=-7ex, "\circled{17}"]
\ar[rrr,"\mathrm{id}_{(\mathrm{C}^{X}\mathrm{F}^{\star})\mathrm{F}}\circ_{\mathsf{h}}\alpha_{\mathrm{C}^{X},\mathrm{F}^{\star},\mathrm{F}\mathrm{C}^{X}}^{\mone}"]
\&\&\&
\big((\mathrm{C}^{X}\mathrm{F}^{\star})\mathrm{F}\big)\big((\mathrm{C}^{X}\mathrm{F}^{\star})(\mathrm{F}\mathrm{C}^{X})\big)
\ar[d,"{\alpha_{\mathrm{C}^{X}\mathrm{F}^{\star},\mathrm{F},(\mathrm{C}^{X}\mathrm{F}^{\star})(\mathrm{F}\mathrm{C}^{X})}}"description]
\ar[rr,"\alpha_{(\mathrm{C}^{X}\mathrm{F}^{\star})\mathrm{F},\mathrm{C}^{X}\mathrm{F}^{\star},\mathrm{F}\mathrm{C}^{X}}^{\mone}"]
\ar[rr, phantom,yshift=-14ex, "\circled{18}"]
\&\&
\Big(\big((\mathrm{C}^{X}\mathrm{F}^{\star})\mathrm{F}\big)\big(\mathrm{C}^{X}\mathrm{F}^{\star}\big)\Big)\Big(\mathrm{F}\mathrm{C}^{X}\Big)
\ar[d,"{\alpha_{\mathrm{C}^{X}\mathrm{F}^{\star},\mathrm{F},\mathrm{C}^{X}\mathrm{F}^{\star}}\circ_{\mathsf{h}}\mathrm{id}_{\mathrm{F}\mathrm{C}^{X}}}"description]
\\[5ex]
\big(\mathrm{C}^{X}\mathrm{F}^{\star}\big)\bigg(\mathrm{F}\Big(\mathrm{C}^{X}\big((\mathrm{F}^{\star}\mathrm{F})\mathrm{C}^{X}\big)\Big)\bigg)
\ar[dd,"{\mathrm{id}_{\mathrm{C}^{X}\mathrm{F}^{\star}}\circ_{\mathsf{h}}\alpha_{\mathrm{F},\mathrm{C}^{X},(\mathrm{F}^{\star}\mathrm{F})\mathrm{C}^{X}}^{\mone}}"description]
\ar[rrrr,"\mathrm{id}_{\mathrm{C}^{X}\mathrm{F}^{\star}}\circ_{\mathsf{h}}\left(\mathrm{id}_{\mathrm{F}}\circ_{\mathsf{h}}\left(\mathrm{id}_{\mathrm{C}^{X}}\circ_{\mathsf{h}}\alpha_{\mathrm{F}^{\star},\mathrm{F},\mathrm{C}^{X}}\right)\right)"]
\ar[rrrr, phantom,yshift=-16ex, "\circled{19}"]
\&\&\&\&
\big(\mathrm{C}^{X}\mathrm{F}^{\star}\big)\bigg(\mathrm{F}\Big(\mathrm{C}^{X}\big(\mathrm{F}^{\star}(\mathrm{F}\mathrm{C}^{X})\big)\Big)\bigg)
\ar[dd,"{\mathrm{id}_{\mathrm{C}^{X}\mathrm{F}^{\star}}\circ_{\mathsf{h}}\alpha_{\mathrm{F},\mathrm{C}^{X},\mathrm{F}^{\star}(\mathrm{F}\mathrm{C}^{X})}^{\mone}}"description]
\ar[rrr, phantom,yshift=-16ex, "\circled{20}"]
\ar[rrr,"\mathrm{id}_{\mathrm{C}^{X}(\mathrm{F}^{\star}}\circ_{\mathsf{h}}(\mathrm{id}_{\mathrm{F}}\circ_{\mathsf{h}}\alpha_{\mathrm{C}^{X},\mathrm{F}^{\star},\mathrm{F}\mathrm{C}^{X}}^{\mone})"]
\&\&\&
\big(\mathrm{C}^{X}\mathrm{F}^{\star}\big)\Big(\mathrm{F}\big((\mathrm{C}^{X}\mathrm{F}^{\star})(\mathrm{F}\mathrm{C}^{X})\big)\Big)
\ar[d,"{\mathrm{id}_{\mathrm{C}^{X}\mathrm{F}^{\star}}\circ_{\mathsf{h}}\alpha_{\mathrm{F},\mathrm{C}^{X}\mathrm{F}^{\star},\mathrm{F}\mathrm{C}^{X}}^{\mone}}"description]
\&\&
\Big(\big(\mathrm{C}^{X}\mathrm{F}^{\star}\big)\big(\mathrm{F}(\mathrm{C}^{X}\mathrm{F}^{\star})\big)\Big)\Big(\mathrm{F}\mathrm{C}^{X}\Big)
\ar[d,"{(\mathrm{id}_{\mathrm{C}^{X}\mathrm{F}^{\star}}\circ_{\mathsf{h}}\alpha_{\mathrm{F},\mathrm{C}^{X},\mathrm{F}^{\star}})\circ_{\mathsf{h}}\mathrm{id}_{\mathrm{F}\mathrm{C}^{X}}}"description]
\ar[dll,"{\alpha_{\mathrm{C}^{X}\mathrm{F}^{\star},\mathrm{F}(\mathrm{C}^{X}\mathrm{F}^{\star}),\mathrm{F}\mathrm{C}^{X}}}"description]
\\[5ex]
\&\&\&\&\&\&\&
\big(\mathrm{C}^{X}\mathrm{F}^{\star}\big)\Big(\big(\mathrm{F}(\mathrm{C}^{X}\mathrm{F}^{\star})\big)\big(\mathrm{F}\mathrm{C}^{X}\big)\Big)
\ar[drr,"{\mathrm{id}_{\mathrm{C}^{X}\mathrm{F}^{\star}}\circ_{\mathsf{h}}(\alpha_{\mathrm{F},\mathrm{C}^{X}\mathrm{F}^{\star}}^{\mone}\circ_{\mathsf{h}}\mathrm{id}_{\mathrm{F}\mathrm{C}^{X}})}"description]
\ar[drr, phantom,yshift=8ex, "\circled{21}"]
\&\&
\Big(\big(\mathrm{C}^{X}\mathrm{F}^{\star}\big)\big((\mathrm{F}\mathrm{C}^{X})\mathrm{F}^{\star}\big)\Big)\Big(\mathrm{F}\mathrm{C}^{X}\Big)
\ar[d,"{\alpha_{\mathrm{C}^{X}\mathrm{F}^{\star},(\mathrm{F}\mathrm{C}^{X})\mathrm{F}^{\star},\mathrm{F}\mathrm{C}^{X}}}"description]
\\[5ex]
\big(\mathrm{C}^{X}\mathrm{F}^{\star}\big)\Big(\big(\mathrm{F}\mathrm{C}^{X}\big)\big((\mathrm{F}^{\star}\mathrm{F})\mathrm{C}^{X}\big)\Big)
\ar[rrrr,"\mathrm{id}_{\mathrm{C}^{X}\mathrm{F}^{\star}}\circ_{\mathsf{h}}(\mathrm{id}_{\mathrm{F}\mathrm{C}^{X}}\circ_{\mathsf{h}}\alpha_{\mathrm{F}^{\star},\mathrm{F},\mathrm{C}^{X}})"]
\&\&\&\&
\big(\mathrm{C}^{X}\mathrm{F}^{\star}\big)\Big(\big(\mathrm{F}\mathrm{C}^{X}\big)\big(\mathrm{F}^{\star}(\mathrm{F}\mathrm{C}^{X})\big)\Big)\Big)
\ar[rrrrr,"\mathrm{id}_{\mathrm{C}^{X}\mathrm{F}^{\star}}\circ_{\mathsf{h}}\alpha_{\mathrm{F}\mathrm{C}^{X},\mathrm{F}^{\star},\mathrm{F}\mathrm{C}^{X}}^{\mone}"]
\&\&\&\&\&
\big(\mathrm{C}^{X}\mathrm{F}^{\star}\big)\Big(\big((\mathrm{F}\mathrm{C}^{X})\mathrm{F}^{\star}\big)\big(\mathrm{F}\mathrm{C}^{X}\big)\Big)
\end{tikzcd}$}
\end{gather*}
which commutes due to
\begin{itemize}
\item naturality of $\alpha$ and $\alpha^{\mone}$ for the facets labeled $1$, $2$, $3$, $6$, $7$, $8$, $12$, $15$, $16$, $17$, $19$ and $21$;

\item the pentagon coherence condition of the associator for the facets labeled $4$, $9$, $18$ and $20$;

\item the interchange law for the facets labeled $5$, $10$, $11$, $13$ and $14$.

\end{itemize}
The last column of the former diagram coincides with the first column of the latter one,
so we can glue the above two diagrams from left to right. Commutativity of the resulting big diagram proves the claim, as the two paths along its boundary from northwest to southeast correspond precisely to the two paths along the boundary of the right triangle
in \eqref{diag0} precomposed with $\gamma$.

By applying $\mathrm{F}$ to the diagram in \eqref{diag0} from the left, we obtain the upper part of the diagram
\begin{gather*}
\adjustbox{scale=.55,center}{%
\begin{tikzcd}[ampersand replacement=\&,column sep=2em, row sep=2em]
\mathrm{F}\big((\mathrm{C}^{X}\mathrm{F}^{\star})\square_{(\mathrm{FC}^{X})\mathrm{F}^{\star}}(\mathrm{F}\mathrm{C}^{X})\big)
\ar[rr, hook, "\mathrm{id}_{\mathrm{F}\circ_{\mathsf{h}}}\beta"]
\ar[ddd,phantom, xshift=2ex,"\cong"]
\ar[ddd, "\varphi",swap]
\&\&
\mathrm{F}\big((\mathrm{C}^{X}\mathrm{F}^{\star})(\mathrm{F}\mathrm{C}^{X})\big)
\arrow[rr, "\mathrm{id}_{\mathrm{F}}\circ_{\mathsf{h}}(\delta_{\mathrm{C}^{X}\mathrm{F}^{\star},(\mathrm{F}\mathrm{C}^{X})\mathrm{F}^{\star}}\circ_{\mathsf{h}}\mathrm{id}_{\mathrm{FC}^{X}})"]
\arrow[dr, "{\mathrm{id}_{\mathrm{F}}\circ_{\mathsf{h}}(\mathrm{id}_{\mathrm{C}^{X}\mathrm{F}^{\star}}\circ_{\mathsf{h}}\delta_{(\mathrm{F}\mathrm{C}^{X})\mathrm{F}^{\star},\mathrm{FC}^{X}})}"description]
\arrow[dr,phantom,yshift=-20ex, xshift=4ex,"\phantom{aaa}\circled{2}\text{\tiny(front and back)}"]
\ar[ddd,"{\alpha_{\mathrm{F},\mathrm{C}^{X}\mathrm{F}^{\star},\mathrm{F}\mathrm{C}^{X}}^{\mone}}"description]
\&\&
\mathrm{F}\bigg(\Big(\big(\mathrm{C}^{X}\mathrm{F}^{\star}\big)\big((\mathrm{F}\mathrm{C}^{X})\mathrm{F}^{\star}\big)\Big)\Big(\mathrm{F}\mathrm{C}^{X}\Big)\bigg)
\arrow[dl, "{\mathrm{id}_{\mathrm{F}}\circ_{\mathsf{h}}\alpha_{\mathrm{C}^{X}\mathrm{F}^{\star},(\mathrm{F}\mathrm{C}^{X})\mathrm{F}^{\star},\mathrm{FC}^{X}}}"description]
\arrow[dl,phantom,yshift=-12ex,xshift=3ex,"\circled{3}"]
\ar[dd,"{\alpha_{\mathrm{F},\left(\mathrm{C}^{X}\mathrm{F}^{\star}\right)\left(\left(\mathrm{F}\mathrm{C}^{X}\right)\mathrm{F}^{\star}\right),\mathrm{F}\mathrm{C}^{X}}^{\mone}}"description]
\\[5ex]
\&
\mathrm{F}\mathrm{C}^{X}
\ar[ur,"\mathrm{id}_{\mathrm{F}}\circ_{\mathsf{h}}\gamma", swap]
\ar[ul,"\mathrm{id}_{\mathrm{F}}\circ_{\mathsf{h}}\theta",dashed]
\&\&
\mathrm{F}\bigg(\big(\mathrm{C}^{X}\mathrm{F}^{\star}\big)\Big(\big((\mathrm{F}\mathrm{C}^{X})\mathrm{F}^{\star}\big)\big(\mathrm{F}\mathrm{C}^{X}\big)\Big)\bigg)
\&
\\[5ex]
\&\&\&\&
\bigg(\mathrm{F}\Big(\big(\mathrm{C}^{X}\mathrm{F}^{\star}\big)\big((\mathrm{F}\mathrm{C}^{X})\mathrm{F}^{\star}\big)\Big)\bigg)\bigg(\mathrm{F}\mathrm{C}^{X}\bigg)
\ar[d,"{\alpha_{\mathrm{F},\mathrm{C}^{X}\mathrm{F}^{\star},(\mathrm{F}\mathrm{C}^{X})\mathrm{F}^{\star}}^{\mone}\circ_{\mathsf{h}}\mathrm{id}_{\mathrm{F}\mathrm{C}^{X}}}"description]
\\[5ex]
\big(\mathrm{F}(\mathrm{C}^{X}\mathrm{F}^{\star})\big)\square_{(\mathrm{FC}^{X})\mathrm{F}^{\star}}\big(\mathrm{F}\mathrm{C}^{X}\big)
\ar[rr, hook, "\xi"]
\ar[rr, phantom,yshift=20ex, xshift=-2ex,"\circled{1}"]
\&\&
\big(\mathrm{F}(\mathrm{C}^{X}\mathrm{F}^{\star})\big)\big(\mathrm{F}\mathrm{C}^{X}\big)
\ar[rr,"\delta_{\mathrm{F}(\mathrm{C}^{X}\mathrm{F}^{\star}),(\mathrm{F}\mathrm{C}^{X})\mathrm{F}^{\star}}\circ_{\mathsf{h}}\mathrm{id}_{\mathrm{F}\mathrm{C}^{X}}", near start, swap]
\ar[dr,"\mathrm{id}_{\mathrm{F}(\mathrm{C}^{X}\mathrm{F}^{\star})}\circ_{\mathsf{h}}\delta_{(\mathrm{F}\mathrm{C}^{X})\mathrm{F}^{\star},\mathrm{FC}^{X}}", swap]
\arrow[rr,phantom, xshift=9ex,yshift=5ex,"\phantom{aaa}\circled{4}\text{\tiny(back)}"]
\ar[urr,"{\ \ (\mathrm{id}_{\mathrm{F}}\circ_{\mathsf{h}}\delta_{\mathrm{C}^{X}\mathrm{F}^{\star},(\mathrm{F}\mathrm{C}^{X})\mathrm{F}^{\star}})\circ_{\mathsf{h}}\mathrm{id}_{\mathrm{FC}^{X}}}"description, near end]
\&\&
\Big(\big(\mathrm{F}(\mathrm{C}^{X}\mathrm{F}^{\star})\big)\big((\mathrm{F}\mathrm{C}^{X})\mathrm{F}^{\star}\big)\Big)\Big(\mathrm{F}\mathrm{C}^{X}\Big)
\ar[dl,"\alpha_{\mathrm{F}(\mathrm{C}^{X}\mathrm{F}^{\star}),(\mathrm{F}\mathrm{C}^{X})\mathrm{F}^{\star},\mathrm{F}\mathrm{C}^{X}}"]
\\[5ex]
\&\&\&
\Big(\mathrm{F}(\mathrm{C}^{X}\mathrm{F}^{\star})\Big)\Big(\big((\mathrm{F}\mathrm{C}^{X})\mathrm{F}^{\star}\big)\big(\mathrm{F}\mathrm{C}^{X}\big)\Big)
\arrow[from=uuu,crossing over,"\alpha_{\mathrm{F},\mathrm{C}^{X}\mathrm{F}^{\star},\left(\left(\mathrm{F}\mathrm{C}^{X}\right)\mathrm{F}^{\star}\right)\left(\mathrm{F}\mathrm{C}^{X}\right)}^{\mone}",very near start]
\&
\end{tikzcd}}
\end{gather*}
Note that $\mathrm{F}\big((\mathrm{C}^{X}\mathrm{F}^{\star})\square_{(\mathrm{FC}^{X})\mathrm{F}^{\star}}(\mathrm{F}\mathrm{C}^{X})\big)$ is the equalizer of the right upper triangle, since the functor $\mathrm{F}$ is left exact when acting on $\mathrm{comod}_{\underline{\ccC}}(\mathrm{C}^{X})$. As in the proof of Lemma \ref{lemma:associator-cotensor-product}, the
associator $\alpha_{\mathrm{F},\mathrm{C}^{X}\mathrm{F}^{\star},\mathrm{F}\mathrm{C}^{X}}^{\mone}$ induces the $2$-isomorphism $\varphi$ (see the second isomorphism in \eqref{eq:associativity-cotensor2}), such that the pentagon labeled $1$ commutes. All other vertical facets also commute: the facet labeled $2$ by naturality of the associator, the one labeled $3$ by the pentagon
coherence condition for the associator, and the triangle labeled $4$ by definition of
$\delta_{\mathrm{F}(\mathrm{C}^{X}\mathrm{F}^{\star}),(\mathrm{F}\mathrm{C}^{X})\mathrm{F}^{\star}}$. Further,
$\alpha_{\mathrm{F},\mathrm{C}^{X}\mathrm{F}^{\star},\mathrm{F}\mathrm{C}^{X}}^{\mone}\circ_{\mathsf{v}}(\mathrm{id}_{\mathrm{F}}\circ_{\mathsf{h}}\gamma)$ equalizes the right bottom triangle, whence, by the universal property of kernels, the $2$-morphism
$\varphi\circ_{\mathsf{v}}(\mathrm{id}_{\mathrm{F}}\circ_{\mathsf{h}}\theta)$ provides a
$2$-isomorphism
\begin{gather*}
\big(\mathrm{F}(\mathrm{C}^{X}\mathrm{F}^{\star})\big)\square_{(\mathrm{FC}^{X})\mathrm{F}^{\star}}\big(\mathrm{F}\mathrm{C}^{X}\big)\cong
(\mathrm{FC}^{X})\mathrm{F}^{\star}\square_{(\mathrm{FC}^{X})\mathrm{F}^{\star}}\big(\mathrm{F}\mathrm{C}^{X}\big)\cong \mathrm{F}\mathrm{C}^{X}.
\end{gather*}
In particular, $\mathrm{id}_{\mathrm{F}}\circ_{\mathsf{h}}\theta$ provides a $2$-isomorphism
\[\mathrm{F}\big((\mathrm{C}^{X}\mathrm{F}^{\star})\square_{(\mathrm{FC}^{X})\mathrm{F}^{\star}}(\mathrm{F}\mathrm{C}^{X})\big)\cong\mathrm{FC}^{X}.\]
Analogously, for any $\mathrm{H}\in\cC$, we have the $2$-isomorphisms
\begin{gather*}\scalebox{0.91}{$
\begin{aligned}
\big(\mathrm{H}\mathrm{F}\big)\big((\mathrm{C}^{X}\mathrm{F}^{\star})\square_{(\mathrm{FC}^{X})\mathrm{F}^{\star}}(\mathrm{F}\mathrm{C}^{X})\big)&\cong
\mathrm{H}\Big(\mathrm{F}\big((\mathrm{C}^{X}\mathrm{F}^{\star})\square_{(\mathrm{FC}^{X})\mathrm{F}^{\star}}(\mathrm{F}\mathrm{C}^{X})\big)\Big)\\
&\cong \mathrm{H}(\mathrm{F}\mathrm{C}^{X})\\
&\cong (\mathrm{H}\mathrm{F})\mathrm{C}^{X}
\end{aligned}$}
\end{gather*}
where the first and third $2$-isomorphisms are given by the associator and the second
$2$-isomorphism is induced by $\mathrm{id}_{\mathrm{H}}(\mathrm{id}_{\mathrm{F}}\circ_{\mathsf{h}}\theta)$. By naturality of the associator, the composite of the above three
$2$-isomorphisms equals
$\mathrm{id}_{\mathrm{HF}}\circ\theta$.
Therefore we have
\begin{gather*}
\mathrm{K}\big((\mathrm{C}^{X}\mathrm{F}^{\star})\square_{(\mathrm{FC}^{X})\mathrm{F}^{\star}}(\mathrm{F}\mathrm{C}^{X})\big)\cong \mathrm{KC}^{X}.
\end{gather*}
for any direct summand $\mathrm{K}$ of $\mathrm{HF}$ with some $\mathrm{H}\in\cC$,
and this is functorial. Note that
\begin{gather*}
\begin{aligned}
(\mathrm{K}\mathrm{C}^{X})\square_{\mathrm{C}^{X}}\big((\mathrm{C}^{X}\mathrm{F}^{\star})\square_{(\mathrm{FC}^{X})\mathrm{F}^{\star}}(\mathrm{F}\mathrm{C}^{X})\big)
&\cong
\big((\mathrm{K}\mathrm{C}^{X})\square_{\mathrm{C}^{X}}(\mathrm{C}^{X}\mathrm{F}^{\star})\big)\square_{(\mathrm{FC}^{X})\mathrm{F}^{\star}}(\mathrm{F}\mathrm{C}^{X})\\
&\cong
\big(\mathrm{K}(\mathrm{C}^{X}\mathrm{F}^{\star})\big)\square_{(\mathrm{FC}^{X})\mathrm{F}^{\star}}(\mathrm{F}\mathrm{C}^{X})\\
&\cong
\mathrm{K}\big((\mathrm{C}^{X}\mathrm{F}^{\star})\square_{(\mathrm{FC}^{X})\mathrm{F}^{\star}}(\mathrm{F}\mathrm{C}^{X})\big)\\
&\cong \mathrm{KC}^{X}
\end{aligned}
\end{gather*}
where the first $2$-isomorphism is given by the associator, the second and third ones are due to the second $2$-isomorphism in \eqref{eq:associativity-cotensor2}, and the fourth one is induced by $\theta$.
Combining this with the fact that $\mathrm{FC}^{X}$ generates
$\mathrm{inj}_{\underline{\ccC}}(\mathrm{C}^{X})$, we see that the natural transformation
${}_{-}\square_{\mathrm{C}^{X}}\theta\colon
{}_{-}\square_{\mathrm{C}^{X}}\mathrm{C}^{X}\to
{}_{-}\square_{\mathrm{C}^{X}}\big((\mathrm{C}^{X}\mathrm{F}^{\star})\square_{(\mathrm{FC}^{X})\mathrm{F}^{\star}}(\mathrm{F}\mathrm{C}^{X})\big)$
is an isomorphism. This implies that $\theta$ is a $2$-isomorphism.
\end{proof}

%%%%%%%%%%%%%%%%%%%%%%%%%%%%%%%%%%%%%%%%%

\subsection{Avoiding abelianizations}\label{subsection:no-abelian}

%%%%%%%%%%%%%%%%%%%%%%%%%%%%%%%%%%%%%%%%%

\begin{proposition}\label{prop5.5}
Let $\cC$ be a $\mathcal{J}$-simple quasi multifiab bicategory and $\mathrm{F}\in\mathcal{J}$.
The pseudofunctor
\begin{gather*}
(\mathrm{F}{}_{-})\mathrm{F}^{\star}\colon
\underline{\cC}\to\underline{\cC},\quad
\mathrm{G}\mapsto (\mathrm{F}\mathrm{G})\mathrm{F}^{\star}
\end{gather*}
takes values in $\mathrm{inj}(\underline{\cC})\cong\cC$.
\end{proposition}

\begin{proof}
Let us consider the $\mathcal{J}\boxtimes\mathcal{J}^{\,\mathrm{op}}$-simple fiab bicategory
\begin{gather*}
\cC^{\,\mathrm{e}}=
\cC\boxtimes\cC^{\mathrm{op}},
\end{gather*}
cf. \cite[Section 6 and Proposition 21]{MM6}. Note that
$\cC$ is a birepresentation of $\cC^{\,\mathrm{e}}$,
and thus, by $\mathcal{J}$-simplicity,
$\mathrm{add}(\mathcal{J})$
is
a simple transitive birepresentation of $\cC^{\,\mathrm{e}}$.
By construction, $\mathrm{add}(\mathcal{J})$ has
apex $\mathcal{J}\boxtimes\mathcal{J}^{\mathrm{op}}$ in $\cC^{\,\mathrm{e}}$. From the straightforward
generalization of \cite[Theorem 2]{KMMZ} to bicategories, we know that $(\mathrm{FX})\mathrm{F}^{\star}$ is injective
in $\underline{\mathrm{add}(\mathcal{J})}$ for any $\mathrm{X}\in\underline{\mathrm{add}(\mathcal{J})}$.
Finally, since for any simple $1$-morphism $\mathrm{L}$ in
$\underline{\cC}$ we have
\begin{gather*}
(\mathrm{F}\mathrm{L})\mathrm{F}^{\star}=0
\quad\Leftrightarrow\quad
\mathrm{L}\text{ is not supported in }\mathcal{J},
\end{gather*}
cf. \cite[Proposition 26]{MM6}, the claim follows.
\end{proof}

\begin{theorem}\label{thm5.5}
Let $\cC$ be a $\mathcal{J}$-simple quasi multifiab bicategory and $\mathbf{M}$ a transitive
birepresenta\-tion of $\cC$ with apex $\mathcal{J}$. Then, for any
$X\in\mathbf{M}(\mathtt{i}), Y\in\mathbf{M}(\mathtt{j})$,
the $1$-morphism $[X,Y]$ belongs to $\cC(\mathtt{i},\mathtt{j})$ (not only to
$\underline{\cC}(\mathtt{i},\mathtt{j})$).
\end{theorem}

\begin{proof}
Let $X\in\mathbf{M}(\mathtt{i}), Y\in\mathbf{M}(\mathtt{j})$, fix an arbitrary
$\mathcal{H}$-cell $\mathcal{H}$ inside $\mathcal{J}$ and denote $\mathtt{k}:=\mathtt{i}_{s(\mathcal{H})}$ and $\mathtt{t}:=\mathtt{i}_{t(\mathcal{H})}$.
By Lemma \ref{lemma:cyclicbirep}, we can choose a generator $Z\in\mathbf{M}(\mathtt{k})$ such that, for any $\mathrm{F}\in\mathcal{H}$, $\mathbf{M}_{\mathtt{tk}}(\mathrm{F})Z$ also generates $\mathbf{M}$. Therefore, there exist $1$-morphisms
$\mathrm{G}\in\cC(\mathtt{t},\mathtt{i}),\mathrm{H}\in\cC(\mathtt{t},\mathtt{j})$ such that
\begin{gather*}
\mathbf{M}_{\mathtt{ik}}(\mathrm{G}\mathrm{F})\,Z\cong\mathbf{M}_{\mathtt{it}}(\mathrm{G})\mathbf{M}_{\mathtt{tk}}(\mathrm{F})\,Z\cong X\oplus X^{\prime},\\
\mathbf{M}_{\mathtt{jk}}(\mathrm{H}\mathrm{F})\,Z\cong\mathbf{M}_{\mathtt{jt}}(\mathrm{H})\mathbf{M}_{\mathtt{tk}}(\mathrm{F})\,Z\cong Y\oplus Y^{\prime}.
\end{gather*}
for some $X^{\prime}\in\mathbf{M}(\mathtt{i})$ and $Y^{\prime}\in\mathbf{M}(\mathtt{j})$.

Next, recall from \eqref{eq:inthom} that
\begin{gather*}
[\mathbf{M}_{\mathtt{ik}}(\mathrm{G}\mathrm{F})\,Z,\mathbf{M}_{\mathtt{jk}}(\mathrm{H}\mathrm{F})\,Z]\cong \big((\mathrm{HF})[Z,Z]\big)(\mathrm{GF})^{\star}\cong \Big(\mathrm{H}\big((\mathrm{F}[Z,Z])\mathrm{F}^{\star}\big)\Big)\mathrm{G}^{\star},
\end{gather*}
where the last isomorphism is obtained by using the associator several times and $(\mathrm{GF})^{\star}=\mathrm{F}^{\star}\mathrm{G}^{\star}$.
By Proposition \ref{prop5.5},
we know that $(\mathrm{F}[Z,Z])\mathrm{F}^{\star}$ belongs to
$\cC$ for all $1$-morphisms $\mathrm{F}$ and thus $[\mathbf{M}_{\mathtt{ik}}(\mathrm{G}\mathrm{F})\,Z,\mathbf{M}_{\mathtt{jk}}(\mathrm{H}\mathrm{F})\,Z]$ also belongs to $\cC$.
Since the internal cohom is additive in both entries, we
see that $[X,Y]$ is a direct summand of $[\mathbf{M}_{\mathtt{ik}}(\mathrm{G}\mathrm{F})\,Z,\mathbf{M}_{\mathtt{jk}}(\mathrm{H}\mathrm{F})\,Z]$ and therefore, it belongs to $\cC$ as well.
\end{proof}

\begin{example}\label{example.2b}
For any coalgebra $1$-morphism $\mathrm{C}$ in $\underline{\cC}$ we have
\begin{gather*}
\mathrm{C}\cong[\mathrm{C},\mathrm{C}],
\end{gather*}
as follows e.g. from \cite[Lemma 3]{ChMi}. However,
this does not contradict Theorem \ref{thm5.5}, since a coalgebra $\mathrm{C}$ which is strictly
in $\underline{\cC}$ will correspond to a birepresentation $\mathbf{M}$ that is either not transitive or has smaller apex.
\end{example}

%%%%%%%%%%%%%%%%%%%%%%%%%%%%%%%%%%%%%%%%%

\subsection{Simple transitive birepresentations and coalgebras}\label{subsection:st-and-coalgebras}

%%%%%%%%%%%%%%%%%%%%%%%%%%%%%%%%%%%%%%%%%

Simple transitive birepresentations correspond to particularly nice coalgebras (compare \cite[Corollary 4.9]{MMMT} and \cite[Corollary 12]{MMMZ}).

\begin{proposition}\label{cor:cosimplicity} Let $\cC$ be a quasi multifiab bicategory and $\mathcal{J}$ a two-sided cell. If $\mathbf{M}\in\cC\text{-}\mathrm{afmod}_{\mathcal{J}}$,  then, for any fixed left cell $\mathcal{L}$ inside $\mathcal{J}$, there is a $1$-morphism $\mathrm{C}\in\underline{\mathrm{add}\big(\mathcal{H}(\mathcal{L})\big)}\subseteq\underline{\mathrm{add}(\mathcal{J})}$ which has a coalgebra structure in $\underline{\cC}$ such that
\begin{gather}\label{eq:cosimplicity1}
\mathbf{M}\simeq\mathbf{inj}_{\underline{\ccC}}(\mathrm{C}).
\end{gather}
If, moreover, $\cC$ is $\mathcal{J}$-simple, then we can choose $\mathrm{C}\in\mathrm{add}\big(\mathcal{H}(\mathcal{L})\big)\subseteq\mathrm{add}(\mathcal{J})$.

If $\mathbf{M}\in\cC\text{-}\mathrm{stmod}_{\mathcal{J}}$, then such a coalgebra $\mathrm{C}$ is cosimple.
Conversely, if $\mathrm{C}\in\underline{\mathrm{add}(\mathcal{J})}$ is a cosimple coalgebra in $\underline{\cC}$, then $\mathbf{inj}_{\underline{\ccC}}(\mathrm{C})$ is a simple transitive
birepresentation of $\cC$ with apex $\mathcal{J}$.
\end{proposition}

\begin{proof} Due to the biequivalence \eqref{eq:pullback1-5} for $\mathcal{J}^{\prime}=\mathcal{J}$,
without loss of generality, we may assume that $\mathcal{J}$ is the unique maximal two-sided cell of $\cC$.

Set $\mathcal{H}:=\mathcal{H}(\mathcal{L})$ and let $\mathtt{i}$ be the source of $\mathcal{H}$.
By Lemma \ref{lemma:cyclicbirep}, for any $\mathbf{M}\in\cC\text{-}\mathrm{afmod}_{\mathcal{J}}$, there is
a generator $X\in\mathbf{M}(\mathtt{i})$ of  $\mathbf{M}$ such that, for any $\mathrm{F}\in\mathcal{H}$, $\mathbf{M}_{\mathtt{i}\mathtt{i}}(\mathrm{F})X$ also generates $\mathbf{M}$. By Theorem \ref{theorem:generator}, there is a biequivalence
\begin{gather*}
\mathbf{M}\simeq\mathbf{inj}_{\underline{\ccC}}(\mathrm{C}^{X}),
\end{gather*}
where $\mathrm{C}^{X}\in\underline{\cC}(\mathtt{i},\mathtt{i})$.
By Corollary \ref{corollary:MT} and Theorem \ref{prop0.4}, for any $\mathrm{F}\in\mathcal{H}$, the coalgebra $\mathrm{C}^{X}$ is MT equivalent to $\mathrm{C}^{\mathbf{M}_{\mathtt{i}\mathtt{i}}(\mathrm{F})X}\cong (\mathrm{F}\mathrm{C}^{X} )\mathrm{F}^{\star}
\in\underline{\cC}(\mathtt{i},\mathtt{i})$.
Suppose that $\mathrm{C}^{X}$ is given by $\mathrm{C}^{X}_1\xrightarrow{\beta}\mathrm{C}^{X}_2$ in $\underline{\cC}(\mathtt{i},\mathtt{i})$. Then $(\mathrm{F}\mathrm{C}^{X})\mathrm{F}^{\star}$ is given by
\begin{gather*}
(\mathrm{F}\mathrm{C}^{X}_{1})\mathrm{F}^{\star} \xrightarrow{(\mathrm{id}_{\mathrm{F}}\circ_{\mathsf{h}}\beta)\circ_{\mathsf{h}}\mathrm{id}_{\mathrm{F}^{\star}}}(\mathrm{F}\mathrm{C}^{X}_{2})\mathrm{F}^{\star}.
\end{gather*}
Since $\mathcal{J}$ is the unique maximal two-sided cell of $\cC$,
the $1$-morphisms $(\mathrm{F}\mathrm{C}^{X}_1)\mathrm{F}^{\star}$ and
$(\mathrm{F}\mathrm{C}^{X}_2)\mathrm{F}^{\star}$ belong to $\mathrm{add}(\mathcal{H})$, whence $\mathrm{C}:=(\mathrm{F}\mathrm{C}^{X} )\mathrm{F}^{\star}$
belongs to $\underline{\mathrm{add}(\mathcal{H})}$. This proves the first claim of the proposition.

If $\cC$ is $\mathcal{J}$-simple, then $\mathrm{C}$ already belongs to $\mathrm{add}(\mathcal{H})$ as a result of Theorem \ref{thm5.5}.

If $\mathbf{M}\in\cC\text{-}\mathrm{stmod}_{\mathcal{J}}$, then the coalgebra $\mathrm{C}$ satisfying \eqref{eq:cosimplicity1} is cosimple by the generalization
of \cite[Corollary 12]{MMMZ} to bicategories.

For the converse statement, first observe that for cosimple $\mathrm{C}$, the birepresentation $\mathbf{inj}_{\underline{\ccC}}(\mathrm{C})$ is transitive by the generalization of \cite[Theorem 20~(ii)]{ChMi} to bicategories. The generalization of \cite[Corollary~12]{MMMZ} to bicategories then implies that it is simple transitive. If $\mathbf{inj}_{\underline{\ccC}}(\mathrm{C})$ annihilates $\mathcal{J}$, then we obtain $\mathrm{C}\mathrm{C}=0$ since $\mathrm{C}\in\underline{\mathrm{add}(\mathcal{J})}$, which is a contradiction. Therefore
the apex of $\mathbf{inj}_{\underline{\ccC}}(\mathrm{C})$ being $\mathcal{J}$ follows from the maximality of $\mathcal{J}$.
\end{proof}

\begin{proposition}\label{prop:J-simple-descend-stmod}
If $\cC$ is quasi multifiab,
then, for any two-sided cell $\mathcal{J}$ in $\cC$, there is a biequivalence
\begin{gather*}
\cC_{\leq\mathcal{J}}\text{-}\mathrm{stmod}_{\mathcal{J}}\xrightarrow{\simeq}\cC\text{-}\mathrm{stmod}_{\mathcal{J}}.
\end{gather*}
\end{proposition}

\begin{proof}
By \eqref{eq:pullback3-4} for $\mathcal{J}^{\prime}=\mathcal{J}$, we already know that there is a
local equivalence
\begin{gather*}
\cC_{\leq\mathcal{J}}\text{-}\mathrm{stmod}_{\mathcal{J}}\xrightarrow{} \cC\text{-}\mathrm{stmod}_{\mathcal{J}}.
\end{gather*}
It remains to prove that any simple transitive birepresentation of $\cC$ with apex $\mathcal{J}$ descends to $\cC_{\leq\mathcal{J}}$. Due to the biequivalence \eqref{eq:pullback1-4} for $\mathcal{J}^{\prime}=\mathcal{J}$,
without loss of generality, we can assume that $\mathcal{J}$ is the unique maximal two-sided cell of $\cC$, i.e. that $\cC\simeq\cC/\mathcal{I}_{\not\leq\mathcal{J}}$.
Let $\mathcal{I}$ be the biideal of $\cC$ such that $\cC_{\leq\mathcal{J}}\simeq\cC/\mathcal{I}$, i.e. $\mathcal{I}$ is the maximal biideal of $\cC$ not containing
$\mathrm{id}_{\mathrm{X}}$ for any $\mathrm{X}\in\mathcal{J}$.

Now, suppose that $\mathbf{M}\in\cC\text{-}\mathrm{stmod}_{\mathcal{J}}$. Since $\mathrm{apex}(\mathbf{M})=\mathcal{J}$, the annihilator of
$\mathbf{M}$ is contained in $\mathcal{I}$.
We need to show that this inclusion is an equality, so
suppose that $\alpha\colon \mathrm{X}\to \mathrm{Y}$ is a $2$-morphism in $\cC$ not belonging to  the annihilator of $\mathbf{M}$.
By Proposition \ref{cor:cosimplicity}, there is a coalgebra $\mathrm{C}\in\underline{\mathrm{add}(\mathcal{J})}\subseteq\underline{\cC}$ such that
\begin{gather*}
\mathbf{M}\simeq\mathbf{inj}_{\underline{\ccC}}(\mathrm{C}).
\end{gather*}
By this equivalence of birepresentations,
there exists a $1$-morphism $\mathrm{F}\in\mathcal{J}$ such that $\alpha\circ_{\mathsf{h}}\mathrm{id}_{\mathrm{FC}}\colon \mathrm{X(FC)}\to \mathrm{Y(FC)}$ is non-zero in $\mathrm{inj}_{\underline{\ccC}}(\mathrm{C})$, whence
the left $\cC$-stable ideal in $\mathrm{inj}_{\underline{\ccC}}(\mathrm{C})$ generated by
$\alpha\circ_{\mathsf{h}}\mathrm{id}_{\mathrm{FC}}$ is equal to
$\mathrm{inj}_{\underline{\ccC}}(\mathrm{C})$ by simple transitivity.
In particular, this left $\cC$-stable ideal contains some $\mathrm{id}_{\mathrm{G}}$ with $\mathrm{G}\in\underline{\mathrm{add}(\mathcal{J})}$. We claim that therefore
$\alpha\not\in\mathcal{I}$.
To prove this claim, assume that $\mathrm{C}$ is given by $\mathrm{C}_1\xrightarrow{\beta}\mathrm{C}_2 \in\underline{\cC}$. Then $\alpha\circ_{\mathsf{h}}\mathrm{id}_{\mathrm{FC}}$ is given by the commutative square
\begin{gather*}
\begin{tikzcd}[ampersand replacement=\&]
\mathrm{X}(\mathrm{FC}_{1})
\ar[rr,"\mathrm{id}_{\mathrm{X}}\circ_{\mathsf{h}} (\mathrm{id}_{\mathrm{F}}\circ_{\mathsf{h}}\beta)"]
\ar[d,"\alpha\circ_{\mathsf{h}} \mathrm{id}_{\mathrm{F}\mathrm{C}_1}",swap]
\& [1em]\&
\mathrm{X}(\mathrm{FC}_{2})
\ar[d,"\alpha\circ_{\mathsf{h}} \mathrm{id}_{\mathrm{F}\mathrm{C}_2}"]
\\
\mathrm{Y}(\mathrm{FC}_{1})
\ar[rr,"\mathrm{id}_{\mathrm{Y}}\circ_{\mathsf{h}} (\mathrm{id}_{\mathrm{F}}\circ_{\mathsf{h}}\beta)",swap]
\& [1em] \&
\mathrm{Y}(\mathrm{FC}_{2})
\end{tikzcd}.
\end{gather*}
Since the left $\cC$-stable ideal in $\underline{\cC}$ generated by $\alpha\circ_{\mathsf{h}}\mathrm{id}_{\mathrm{FC}}$ contains $\mathrm{id}_{\mathrm{G}}$ with
$\mathrm{G}\in\underline{\mathrm{add}(\mathcal{J})}$, the left $\cC$-stable ideal in $\cC$ generated by $\alpha\circ_{\mathsf{h}}\mathrm{id}_{\mathrm{F}\mathrm{C}_{1}}$
contains some $\mathrm{id}_{\mathrm{K}}$ with $\mathrm{K}\in\mathrm{add}(\mathcal{J})$. The latter left $\cC$-stable ideal is contained in the biideal of $\cC$ generated by $\alpha$, whence
$\alpha\not\in\mathcal{I}$. We conclude that $\mathrm{ann}(\mathbf{M})=\mathcal{I}$, which is what we had to prove.
\end{proof}

%%%%%%%%%%%%%%%%%%%%%%%%%%%%%%%%%%%%%%%%%

\subsection{Bicomodules and birepresentations}\label{subsection:bimodule-bicats}

%%%%%%%%%%%%%%%%%%%%%%%%%%%%%%%%%%%%%%%%%

Let $\cC$ be a multifinitary bicategory.

\begin{definition}\label{definition:comod-bicat}
We define
$\BB_{\underline{\ccC}}$ to be
the \emph{bicategory of biinjective bicomodules over coalgebras in $\underline{\cC}$},
whose objects, $1$-morphisms and $2$-morphisms are coalgebras,
biinjective bicomodules and bicomodule homomorphisms in
$\underline{\cC}$, respectively. Horizontal composition is defined by the cotensor
product over coalgebras and vertical composition is defined by the composition of bicomodule homomorphisms.
For each object
$\mathrm{C}$ in $\BB_{\underline{\ccC}}$, the identity $1$-morphism $\mathbbm{1}_{\mathrm{C}}$
is given by $\mathrm{C}$, seen as
a $\mathrm{C}\text{-}\mathrm{C}$-bicomodule. For each $1$-morphism $\mathrm{M}$ in
$\BB_{\underline{\ccC}}$, the
identity $2$-morphism is simply the identity bicomodule
endomorphism of $\mathrm{M}$.
\end{definition}
By \eqref{eq:associativity-cotensor} and the explanations above it, as well as the
fact that the cotensor product over a coalgebra
of two biinjective comodules is again biinjective, $\BB_{\underline{\ccC}}$ is indeed a bicategory.

\begin{definition}\label{def:excfmodstmod}
For various $1,2$-full $2$-subcategories $\cD$ of $\cC\text{-}\mathrm{afmod}$ appearing in Definitions \ref{def:cfmodstmod} and \ref{def:mod-apexJ}, we define the associated $2$-subcategories $\cD^{\mathrm{ex}}$ with the objects being the same as those of $\cD$,
the $1$-morphisms being the exact morphisms of finitary birepresentations and $2$-morphisms being all modifications.
\end{definition}
We have
\begin{gather*}
\cC\text{-}\mathrm{stmod}^{\mathrm{ex}}\subset\cC\text{-}\mathrm{tfmod}^{\mathrm{ex}}\subset\cC\text{-}\mathrm{cfmod}^{\mathrm{ex}}\subset\cC\text{-}\mathrm{afmod}^{\mathrm{ex}}
\end{gather*}
and, due to Lemma \ref{lemma:cyclicbirep},
\begin{gather*}
\cC\text{-}\mathrm{stmod}_{\mathcal{J}}^{\mathrm{ex}}\subset\cC\text{-}\mathrm{tfmod}_{\mathcal{J}}^{\mathrm{ex}}\subset\cC\text{-}\mathrm{cfmod}_{\mathcal{J}}^{\mathrm{ex}}=\cC\text{-}\mathrm{afmod}_{\mathcal{J}}^{\mathrm{ex}}.
\end{gather*}

Note that all finitary birepresentations of $\cC^{\,\oplus}$ are cyclic.
Furthermore, all morphisms between simple transitive birepresentations with the same apex
of a given fiab bicategory are exact,
as the following proposition shows. This is the analog of \cite[Proposition 7.6.9]{EGNO}
in our context and its proof follows the same reasoning, except that we have to invoke \cite[Theorem 2]{KMMZ} at some point.

\begin{proposition}\label{prop:exactness}
Suppose that $\cC$ is quasi (multi)fiab. For any two-sided cell $\mathcal{J}$
of $\cC$, the bicategories $\cC\text{-}\mathrm{stmod}_{\mathcal{J}}^{\mathrm{ex}}$
and $\cC\text{-}\mathrm{stmod}_{\mathcal{J}}$ are equal.
\end{proposition}

\begin{proof}
Let $\mathbf{M}$, $\mathbf{N}$ be two simple transitive birepresentations of $\cC$ with apex $\mathcal{J}$ and let $\Phi\colon\mathbf{M}\to\mathbf{N}$ be a $\Bbbk$-linear homomorphism of birepresentations. We have to show that its extension $\underline{\Phi}
\colon \underline{\mathbf{M}}\to \underline{\mathbf{N}}$ is exact.

Before we do that, we first prove an auxiliary result. For $\mathtt{i,j}\in\cC$, let
\begin{gather*}
\mathrm{C}_{\mathtt{j,i}}:=\bigoplus_{\mathrm{X}\in\ccC(\mathtt{i,j})\cap\mathcal{J}}\mathrm{X}\;\in\mathrm{add}(\mathcal{J})
\end{gather*}
and notice that, by adjunction, $\mathrm{C}_{\mathtt{j,i}}^{\star}\cong\mathrm{C}_{\mathtt{i,j}}$.

\textbf{Claim}. The endofunctors $\underline{\mathbf{M}}(\mathrm{C}_{\mathtt{j,i}})$
and $\underline{\mathbf{N}}(\mathrm{C}_{\mathtt{j,i}})$ are both projective and
injective in the category of left exact endofunctors and they do not annihilate any objects in $\underline{\mathbf{M}(\mathtt{i})}$ and $\underline{\mathbf{N}(\mathtt{i})}$, respectively.

The first part of the claim follows from simple transitivity of $\mathbf{M}$, $\mathbf{N}$ and
\cite[Theorem 2]{KMMZ}. Let us show the second part of the claim for
$\underline{\mathbf{M}}(\mathrm{C}_{\mathtt{j,i}})$, the argument for
$\underline{\mathbf{N}}(\mathrm{C}_{\mathtt{j,i}})$ being analogous. Suppose to the contrary that $L\in\underline{\mathbf{M}(\mathtt{i})}$ is a simple object such that $\underline{\mathbf{M}}(\mathrm{C}_{\mathtt{j,i}})L=0$. Let $Q$ be the direct sum of all indecomposable injectives in $\underline{\mathbf{M}(\mathtt{j})}$, i.e. the direct sum of all indecomposables in
$\mathbf{M}(\mathtt{j})$. By adjunction, we have
\begin{gather*}
\mbox{Hom}_{\underline{\mathbf{M}}}\big(L,\underline{\mathbf{M}}(\mathrm{C}_{\mathtt{i,j}})Q\big)\cong\mbox{Hom}_{\mathbf{M}}\big(\underline{\mathbf{M}}(\mathrm{C}_{\mathtt{j,i}})L,Q\big)=0.
\end{gather*}
However, this means that the injective hull of $L$ has multiplicity zero in the decomposition of
$\underline{\mathbf{M}}(\mathrm{C}_{\mathtt{i,j}})Q=
\mathbf{M}(\mathrm{C}_{\mathtt{i,j}})Q$, which contradicts transitivity of $\mathbf{M}$. This completes the proof of the claim.

Now, suppose that the above homomorphism $\underline{\Phi}$ is not exact. Then there exists an object $\mathtt{i}$ and a short exact sequence of objects in $\underline{\mathbf{M}(\mathtt{i})}$
\begin{gather*}
0\xrightarrow{}X\xrightarrow{}Y\xrightarrow{}Z\xrightarrow{}0
\end{gather*}
such that its image under $\underline{\Phi}$
\begin{gather*}
0\xrightarrow{}\underline{\Phi}(X)\xrightarrow{}\underline{\Phi}(Y)
\xrightarrow{}\underline{\Phi}(Z)\xrightarrow{}0
\end{gather*}
is not exact in $\underline{\mathbf{N}(\mathtt{i})}$. The claim implies that
\begin{gather*}
0\xrightarrow{}\underline{\mathbf{M}}(\mathrm{C}_{\mathtt{j,i}})X\xrightarrow{}
\underline{\mathbf{M}}(\mathrm{C}_{\mathtt{j,i}})Y\xrightarrow{}
\underline{\mathbf{M}}(\mathrm{C}_{\mathtt{j,i}})Z\xrightarrow{}0
\end{gather*}
is split exact, while
\begin{gather*}
0\xrightarrow{}\underline{\mathbf{N}}(\mathrm{C}_{\mathtt{j,i}})\underline{\Phi}(X)\xrightarrow{}\underline{\mathbf{N}}(\mathrm{C}_{\mathtt{j,i}})\underline{\Phi}(Y)\xrightarrow{}
\underline{\mathbf{N}}(\mathrm{C}_{\mathtt{j,i}})\underline{\Phi}(Z)\xrightarrow{}0
\end{gather*}
is not exact. But this is a contradiction, since the latter sequence is isomorphic to
\begin{gather*}
0\xrightarrow{}\underline{\Phi}\big(\underline{\mathbf{M}}(\mathrm{C}_{\mathtt{j,i}})X\big)\xrightarrow{}\underline{\Phi}\big(\underline{\mathbf{M}}(\mathrm{C}_{\mathtt{j,i}})Y\big)\xrightarrow{}\underline{\Phi} \big(\underline{\mathbf{M}}(\mathrm{C}_{\mathtt{j,i}})Z\big)\xrightarrow{}0
\end{gather*}
and $\underline{\Phi}$ preserves split exactness.

This shows that $\underline{\Phi}$ is exact and completes the proof of the proposition.
\end{proof}

\begin{theorem}\label{theorem:MT2}
Let $\cC$ be a quasi multifiab bicategory.
The assignment
\begin{align*}
\mathrm{C}&\mapsto\boldsymbol{\mathrm{inj}}_{\underline{\ccC}}(\mathrm{C}),
\\
\mathrm{M}&\mapsto
\boldsymbol{\mathrm{inj}}_{\underline{\ccC}}(\mathrm{C})\xrightarrow{{}_{-}\s\mathrm{M}}\boldsymbol{\mathrm{inj}}_{\underline{\ccC}}(\mathrm{D}),
\\
(\mathrm{M}\xrightarrow{f}\mathrm{N})&\mapsto({}_{-}\,\s\mathrm{M}\xrightarrow{{}_{-}\s f}{}_{-}\,\s\mathrm{N}),
\end{align*}
defines a biequivalence
\begin{gather}\label{eq:MT2bieq1}
\BB_{\underline{\ccC^{\oplus}}}\simeq\cC^{\,\oplus}\text{-}\mathrm{afmod}^{\mathrm{ex}},
\end{gather}
which restricts to a biequivalence
\begin{gather}\label{eq:MT2bieq2}
\BB_{\underline{\ccC}}\simeq\cC\text{-}\mathrm{cfmod}^{\mathrm{ex}}.
\end{gather}
\end{theorem}

\begin{proof}
The pseudofunctor $\BB_{\underline{\ccC}^{\oplus}}\to
\cC^{\,\oplus}\text{-}\mathrm{afmod}^{\mathrm{ex}}$ is well-defined by Lemmas \ref{lemma:associator-cotensor-product} and
\ref{lemma:biinjective-vs-exact}. It is a biequivalence due to Theorems \ref{theorem:MT} and  \ref{theorem:generator}.
When we restrict to coalgebras and biinjective bicomodules in $\underline{\cC}$ on one side of the biequivalence, we have to restrict to cyclic birepresentations of $\cC$ on the other side, because
we need a generator $X\in\mathbf{M}(\mathtt{i})$, for some $\mathtt{i}\in\cC$, in order to define $\mathrm{C}^{X}$ via the internal cohom construction.
\end{proof}

The following corollary follows immediately from Proposition \ref{proposition:bieq-additive} and \eqref{eq:MT2bieq1}.

\begin{corollary}\label{corollary:MT2}
Let $\cC$ be a quasi multifiab bicategory. Then there is a biequivalence
\begin{gather*}
\BB_{\underline{\ccC^{\oplus}}}\simeq\cC\text{-}\mathrm{afmod}^{\mathrm{ex}}.
\end{gather*}
\end{corollary}

Let $\mathrm{add}_{\ccC_{\leq\mathcal{J}}}(\mathcal{J})$ be the additive closure of $\mathcal{J}$ inside $\cC_{\leq\mathcal{J}}$ and
let $\BB_{\mathrm{add}_{\cccC_{\leq\mathcal{J}}}(\mathcal{J})}^{\mathrm{cos}}$ be the $1$-dense and $2$-full subbicategory of $\BB_{\ccC_{\leq\mathcal{J}}}$ of biinjective bicomodules over
cosimple coalgebras in $\mathrm{add}_{\ccC_{\leq\mathcal{J}}}(\mathcal{J})$. For the multifinitary bicategory $\cC_{\mathcal{J}}$,
one can also define $\mathrm{add}_{\ccC_{\mathcal{J}}}(\mathcal{J})$ and $\BB_{\mathrm{add}_{\cccC_{\mathcal{J}}}(\mathcal{J})}^{\mathrm{cos}}$.
Since $\cC_{\mathcal{J}}$  is a $2$-full subbicategory of $\cC_{\leq\mathcal{J}}$,
we have
\begin{gather}\label{eq:0}
\mathrm{add}_{\ccC_{\mathcal{J}}}(\mathcal{J})=\mathrm{add}_{\ccC_{\leq\mathcal{J}}}(\mathcal{J}) \quad\text{and}\quad
\BB_{\mathrm{add}_{\cccC_{\mathcal{J}}}(\mathcal{J})}^{\mathrm{cos}}=\BB_{\mathrm{add}_{\cccC_{\leq\mathcal{J}}}(\mathcal{J})}^{\mathrm{cos}}.
\end{gather}

\begin{theorem}\label{cor:MT2}
If $\cC$ is quasi multifiab, then there are the following biequivalences:
\begin{gather*}
\cC\text{-}\mathrm{stmod}_{\mathcal{J}}\simeq\cC_{\leq\mathcal{J}}\text{-}\mathrm{stmod}_{\mathcal{J}} \simeq
\cC_{\mathcal{J}}\text{-}\mathrm{stmod}_{\mathcal{J}}\simeq
\BB_{\mathrm{add}_{\cccC_{\leq\mathcal{J}}}(\mathcal{J})}^{\mathrm{cos}}.
\end{gather*}
\end{theorem}

\begin{proof}
By Proposition \ref{prop:exactness}, all instances of $\mathrm{stmod}$ in
this theorem are equal to $\mathrm{stmod}^{\mathrm{ex}}$. Bearing this in mind, the first biequivalence
is due to the restriction of the biequivalence in Proposition \ref{prop:J-simple-descend-stmod}. By Proposition \ref{cor:cosimplicity}, therefore,
the biequivalence in Theorem \ref{theorem:MT2} restricts to a biequivalence
\begin{gather}\label{eq:00}
\BB_{\mathrm{add}_{\cccC_{\leq\mathcal{J}}}(\mathcal{J})}^{\mathrm{cos}}\simeq\cC_{\leq\mathcal{J}}\text{-}\mathrm{stmod}_{\mathcal{J}},
\end{gather}
By \eqref{eq:0} and $\mathcal{J}$-simplicity of $\cC_{\mathcal{J}}$, we also have a biequivalence
\begin{gather*}
\BB_{\mathrm{add}_{\cccC_{\leq\mathcal{J}}}(\mathcal{J})}^{\mathrm{cos}}=\BB_{\mathrm{add}_{\cccC_{\mathcal{J}}}(\mathcal{J})}^{\mathrm{cos}}\simeq\cC_{\mathcal{J}}\text{-}\mathrm{stmod}_{\mathcal{J}},
\end{gather*}
which is indeed a restriction of \eqref{eq:00}.
\end{proof}

\begin{remark}
Note that Theorems \ref{theorem:MT2} and \ref{cor:MT2} also prove that
\begin{gather*}
\BB_{\underline{\mathrm{add}(\mathcal{J})}}^{\mathrm{cos}}\simeq\BB_{\underline{\mathrm{add}_{\cccC_{\leq\mathcal{J}}}(\mathcal{J})}}^{\mathrm{cos}}\simeq\BB_{\mathrm{add}_{\cccC_{\leq\mathcal{J}}}(\mathcal{J})}^{\mathrm{cos}}.
\end{gather*}
\end{remark}

%%%%%%%%%%%%%%%%%%%%%%%%%%%%%%%%%%%%%%%%%

\subsection{Strong $\mathcal{H}$-reduction}\label{section:h-reduction}

%%%%%%%%%%%%%%%%%%%%%%%%%%%%%%%%%%%%%%%%%

Let $\cC$ be a multifiab bicategory, $\mathcal{J}$ a two-sided cell in $\cC$ and $\mathcal{H}$ a diagonal $\mathcal{H}$-cell inside $\mathcal{J}$.
Assume that $\mathtt{i}$ is the source of $\mathcal{H}$. Recall from Lemma \ref{lemma:cyclicbirep} that
\begin{gather*}
\cC_{(\mathcal{J})}\text{-}\mathrm{afmod}_{\mathcal{J}}=\cC_{(\mathcal{J})}\text{-}\mathrm{cfmod}_{\mathcal{J}} \quad\text{and}\quad \cC_{(\mathcal{H})}\text{-}\mathrm{afmod}_\mathcal{H}=\cC_{(\mathcal{H})}\text{-}\mathrm{cfmod}_\mathcal{H},
\end{gather*}
which implies that
\begin{gather*}
\cC_{(\mathcal{J})}\text{-}\mathrm{afmod}_{\mathcal{J}}^{\mathrm{ex}}=\cC_{(\mathcal{J})}\text{-}\mathrm{cfmod}_{\mathcal{J}}^{\mathrm{ex}}\quad\text{and}\quad \cC_{(\mathcal{H})}\text{-}\mathrm{afmod}_\mathcal{H}^{\mathrm{ex}}=\cC_{(\mathcal{H})}\text{-}\mathrm{cfmod}_\mathcal{H}^{\mathrm{ex}}.
\end{gather*}

We denote by  $\BB_{\underline{\ccC_{(\mathcal{J})}}, \underline{\mathrm{add}(\mathcal{H})}}$, respectively $\BB_{\underline{\ccC_{(\mathcal{H})}}, \underline{\mathrm{add}(\mathcal{H})}}$, the bicategory of coalgebras, biinjective bicomodules and comodule homomorphisms in $\underline{\mathrm{add}(\mathcal{H})}$ inside $\underline{\cC_{(\mathcal{J})}}$, respectively $\underline{\cC_{(\mathcal{H})}}$.

\begin{theorem}\label{theorem:H-reduction1}
There are biequivalences
\begin{gather*}
\begin{split}
\cC_{(\mathcal{J})}\text{-}\mathrm{afmod}_{\mathcal{J}}^{\mathrm{ex}} &=\cC_{(\mathcal{J})}\text{-}\mathrm{cfmod}_{\mathcal{J}}^{\mathrm{ex}} \simeq\BB_{\underline{\ccC_{(\mathcal{J})}}, \underline{\mathrm{add}(\mathcal{H})}}\\ &
\simeq \BB_{\underline{\ccC_{(\mathcal{H})}}, \underline{\mathrm{add}(\mathcal{H})}} \simeq\cC_{(\mathcal{H})}\text{-}\mathrm{cfmod}_\mathcal{H}^{\mathrm{ex}}
=\cC_{(\mathcal{H})}\text{-}\mathrm{afmod}_\mathcal{H}^{\mathrm{ex}}.
\end{split}
\end{gather*}
\end{theorem}

\begin{proof}
We claim that the last biequivalence in the first row and the middle biequivalence in the second row are obtained by restricting the biequivalence in \eqref{eq:MT2bieq2}.
Indeed, \eqref{eq:MT2bieq2} provides biequivalences
\begin{gather*}
\BB_{\underline{\ccC_{(\mathcal{J})}}} \simeq\cC_{(\mathcal{J})}\text{-}\mathrm{cfmod}^{\mathrm{ex}} \quad \text{and} \quad \BB_{\underline{\ccC_{(\mathcal{H})}}} \simeq\cC_{(\mathcal{H})}\text{-}\mathrm{cfmod}^{\mathrm{ex}} \end{gather*}
and Proposition \ref{cor:cosimplicity} guarantees that the restriction of the pseudofunctor in Theorem \ref{theorem:MT2} to coalgebras in $\underline{\mathrm{add}(\mathcal{H})}$ is still essentially surjective on objects when corestricting to $\cC_{(\mathcal{J})}\text{-}\mathrm{cfmod}_{\mathcal{J}}^{\mathrm{ex}}$ respectively $\cC_{(\mathcal{H})}\text{-}\mathrm{cfmod}_\mathcal{H}^{\mathrm{ex}} $.

Finally, $\BB_{\underline{\ccC_{(\mathcal{H})}}, \underline{\mathrm{add}(\mathcal{H})}} $ is naturally isomorphic to $\BB_{\underline{\ccC_{(\mathcal{J})}}, \underline{\mathrm{add}(\mathcal{H})}}$, consisting of the same objects, $1$-morphisms and $2$-morphisms, just considered in different ambient bicategories.
\end{proof}

Passing to the $\mathcal{J}$-simple and $\mathcal{H}$-simple quotients $\cC_{\mathcal{J}}$
and $\cC_\mathcal{H}$, respectively, and defining $\BB_{\underline{\ccC_{\mathcal{J}}}, {\mathrm{add}(\mathcal{H})}}$ and $\BB_{\underline{\ccC_{\mathcal{H}}}, {\mathrm{add}(\mathcal{H})}}$ as the bicategories of coalgebras, bicomodules and comodule homomorphisms in ${\mathrm{add}(\mathcal{H})}$ inside $\underline{\cC_{\mathcal{J}}}$ and $\underline{\cC_{\mathcal{H}}}$, respectively, we obtain the following.

\begin{theorem}\label{theorem:H-reduction2}
There are biequivalences
\begin{gather*}
\begin{aligned}
\cC_{\mathcal{J}}\text{-}\mathrm{afmod}_{\mathcal{J}}^{\mathrm{ex}}  &=\cC_{\mathcal{J}}\text{-}\mathrm{cfmod}_{\mathcal{J}}^{\mathrm{ex}}  \simeq\BB_{\underline{\ccC_{\mathcal{J}}}, {\mathrm{add}(\mathcal{H})}}\\ &\simeq \BB_{\underline{\ccC_{\mathcal{H}}}, {\mathrm{add}(\mathcal{H})}}\simeq\cC_{\mathcal{H}}\text{-}\mathrm{cfmod}_\mathcal{H}^{\mathrm{ex}}=\cC_{\mathcal{H}}\text{-}\mathrm{afmod}_\mathcal{H}^{\mathrm{ex}} .
\end{aligned}
\end{gather*}
\end{theorem}

\begin{proof}
The only thing to note is that, under the assumption of $\mathcal{J}$-simplicity,
the coalgebra $\mathrm{C}^{\mathbf{M}_{\mathtt{i}\mathtt{i}}(\mathrm{F})X} \cong (\mathrm{F}\mathrm{C}^{X})\mathrm{F}^{\star}$ in Proposition \ref{cor:cosimplicity} belongs indeed to $\mathrm{add}(\mathcal{H})$ by Proposition \ref{prop5.5}.
\end{proof}

We deduce the following consequence, which we (also) call
\emph{strong $\mathcal{H}$-reduction}.

\begin{theorem}\label{theorem:strongH}
Let $\cC$ be a fiab bicategory, and fix a two-sided cell $\mathcal{J}$ of $\cC$ as well as diagonal $\mathcal{H}$-cell
$\mathcal{H}\subset\mathcal{J}$. Then there is a biequivalence
\begin{gather*}
\cC\text{-}\mathrm{stmod}_{\mathcal{J}}\simeq
\cC_{\mathcal{H}}\text{-}\mathrm{stmod}_{\mathcal{H}}.
\end{gather*}
\end{theorem}

\begin{proof}
Bearing Proposition \ref{prop:exactness} in mind,
the statement follows by Theorems \ref{cor:MT2} and \ref{theorem:H-reduction2}.
\end{proof}

%%%%%%%%%%%%%%%%%%%%%%%%%%%%%%%%%%%%

\subsection{An extra biequivalence}

%%%%%%%%%%%%%%%%%%%%%%%%%%%%%%%%%%%%

The goal of this subsection is to prove that the $2$-functor in \eqref{eq:pullback4-3} is a local equivalence.

\begin{theorem}\label{theorem:00}
Let $\cC$ be a quasi multifiab bicategory and $\mathcal{J}$ a two-sided cell in $\cC$. Then the $2$-functor
\begin{gather*}
\cC_{\leq\mathcal{J}}\text{-}\mathrm{afmod}_{\mathcal{J}}\to \cC_{\mathcal{J}}\text{-}\mathrm{afmod}_{\mathcal{J}},
\end{gather*}
defined in \eqref{eq:pullback4-3}, is a local equivalence.
\end{theorem}

\begin{proof}
Recall from Subsection \ref{quotientbicat} that \eqref{eq:pullback4-3} is well-defined.
Let $\mathbf{M}$ and $\mathbf{N}$ be two arbitrary birepresentations in $\cC_{\leq\mathcal{J}}\text{-}\mathrm{afmod}_{\mathcal{J}}$.
Since $2$-faithfulness of \eqref{eq:pullback4-3} is obvious, it suffices to prove $1$- and $2$-fullness, or in other words essential surjectivity and fullness of the induced functor
\begin{gather}\label{eq:localfunctor}
\mathrm{Hom}_{\ccC_{\leq_{\mathcal{J}}}\text{-}\mathrm{afmod}_{\mathcal{J}}}(\mathbf{M},\mathbf{N})\to\mathrm{Hom}_{\ccC_{\mathcal{J}}\text{-}\mathrm{afmod}_{\mathcal{J}}}(\mathbf{M},\mathbf{N}).
\end{gather}

By the abelianized version of Proposition \ref{cor:cosimplicity} and the fact that $\cC_{\leq\mathcal{J}}$ is
$\mathcal{J}$-simple,
there exist two $1$-morphisms $\mathrm{C}$ and $\mathrm{D}$ in $\mathrm{add}(\mathcal{J})$ which have coalgebra structures in $\cC_{\leq\mathcal{J}}$
such that
\begin{gather*}
\underline{\mathbf{M}}\simeq\mathbf{comod}_{\underline{\ccC_{\leq\mathcal{J}}}}(\mathrm{C}) \quad\text{and}\quad \underline{\mathbf{N}}\simeq\mathbf{comod}_{\underline{\ccC_{\leq\mathcal{J}}}}(\mathrm{D}),
\end{gather*}
as birepresentations of $\cC_{\leq\mathcal{J}}$. Both equivalences can be restricted to equivalences
\begin{gather*}
\mathbf{M}\simeq\mathbf{inj}_{\underline{\ccC_{\leq\mathcal{J}}}}(\mathrm{C}) \quad\text{and}\quad \mathbf{N}\simeq\mathbf{inj}_{\underline{\leq\ccC_{\mathcal{J}}}}(\mathrm{D}),
\end{gather*}
as birepresentations of $\cC_{\leq\mathcal{J}}$. Since $\cC_{\mathcal{J}}$ is a $2$-full subbicategory of $\cC_{\leq\mathcal{J}}$,
the coalgebra structures of $\mathrm{C}$ and $\mathrm{D}$ in $\mathrm{add}(\mathcal{J})$ both restrict to $\cC_{\mathcal{J}}$.
Via the $2$-functor in \eqref{eq:pullback4-3}, the above equivalences descend to
\begin{gather*}
\underline{\mathbf{M}}\simeq\mathbf{comod}_{\underline{\ccC_{\mathcal{J}}}}(\mathrm{C})
\quad\text{and}\quad
\underline{\mathbf{N}}\simeq\mathbf{comod}_{\underline{\ccC_{\mathcal{J}}}}(\mathrm{D}),
\end{gather*}
as birepresentations of $\cC_{\mathcal{J}}$, respectively, and
\begin{gather*}
\mathbf{M}\simeq\mathbf{inj}_{\underline{\ccC_{\mathcal{J}}}}(\mathrm{C})
\quad\text{and}\quad
\mathbf{N}\simeq\mathbf{inj}_{\underline{\ccC_{\mathcal{J}}}}(\mathrm{D}),
\end{gather*}
as birepresentations of $\cC_{\mathcal{J}}$.

Let $\Phi$ be any morphism of birepresentations in
$\mathrm{Hom}_{\ccC_{\mathcal{J}}\text{-}\mathrm{afmod}_{\mathcal{J}}}(\mathbf{M},\mathbf{N})$. Then the induced morphism $\underline{\Phi}$ from $\underline{\mathbf{M}}$ to $\underline{\mathbf{N}}$ is left exact by definition.
Hence the functor underlying $\underline{\Phi}$ can be represented by cotensoring
with some $\mathrm{C}\text{-}\mathrm{D}$-bicomodule $\mathrm{X}\in\underline{\cC_{\mathcal{J}}}\subseteq\underline{\cC_{\leq\mathcal{J}}}$, which is injective as a right $\mathrm{D}$-comodule since it sends injective right $\mathrm{C}$-comodules to injective right $\mathrm{D}$-comodules.
It is clear that the functor ${}_{-}\s\mathrm{X}$ is an element of
$\mathrm{Hom}_{\ccC_{\leq_{\mathcal{J}}}\text{-}\mathrm{afmod}_{\mathcal{J}}}\big(\mathbf{inj}_{\underline{\ccC_{\mathcal{J}}}}(\mathrm{C}),\mathbf{inj}_{\underline{\ccC_{\mathcal{J}}}}(\mathrm{D})\big)$
by Lemma \ref{lemma:associator-cotensor-product}. This implies that $\Phi\in\mathrm{Hom}_{\ccC_{\leq_{\mathcal{J}}}\text{-}\mathrm{afmod}_{\mathcal{J}}}(\mathbf{M},\mathbf{N})$ and our functor \eqref{eq:localfunctor} are essentially surjective. Since modifications correspond to homomorphisms of $\mathrm{C}\text{-}\mathrm{D}$-bicomodules, fullness of \eqref{eq:localfunctor} is also clear and the statement is proved.
\end{proof}

\begin{remark}\label{remark0.0}
In fact, the local equivalence in Theorem~\ref{theorem:00} restricts to a biequivalence $\cC_{\leq\mathcal{J}}\text{-}\mathrm{stmod}_{\mathcal{J}} \simeq
\cC_{\mathcal{J}}\text{-}\mathrm{stmod}_{\mathcal{J}}$, cf Theorem~\ref{cor:MT2}.
\end{remark}

%%%%%%%%%%%%%%%%%%%%%%%%%%%%%%%%%%%%%%%%%%%

\section{The double centralizer theorem}\label{section:doublecentralizer}

%%%%%%%%%%%%%%%%%%%%%%%%%%%%%%%%%%%%%%%%%%%

Throughout this section, let $\cC$ be a fiab bicategory, $\mathcal{H}$ a diagonal
$\mathcal{H}$-cell and $\mathbf{M}$ a simple transitive birepresentation of
$\cCH$ with apex $\mathcal{H}$. By Proposition \ref{cor:cosimplicity}
there is a cosimple coalgebra
$\mathrm{C}\in\mathrm{add}(\mathcal{H})$ such that
\begin{gather*}
\mathbf{M}\simeq\mathbf{inj}_{\underline{\ccCH}}(\mathbf{C}),
\end{gather*}
and by Lemma \ref{lemma:addGX} we have
\begin{gather}\label{eq:comodequalsinj}
\mathrm{inj}_{\underline{\ccCH}}(\mathbf{C})=\mathrm{add}\{\mathrm{GC}\mid\mathrm{G}\in\cCH\},
\end{gather}
where the additive closure is taken inside $\mathrm{comod}_{\underline{\ccCH}}(\mathrm{C})$.

Let $\cEnd_{\ccCH}(\mathbf{M})$ denote the one-object $2$-category of endomorphisms (of finitary birepresentations) of $\mathbf{M}$ and recall that $\cEnd_{\ccCH}(\mathbf{M})=\cEnd_{\ccCH}^{\mathrm{ex}}(\mathbf{M})$ by Proposition \ref{prop:exactness}.
Further, let $\cBM:=(\mathrm{C})\Bi_{\underline{\ccCH}}(\mathrm{C})$ denote the one-object bicategory of biinjective
$\mathrm{C}$-bicomodules in $\underline{\cCH}$, with the horizontal composition being
given by ${}_{-}\s{}_{-}$. By Proposition \ref{cor:cosimplicity},
there is a biequivalence
\begin{gather}\label{eq:endbicomodules}
\cEnd_{\ccCH}(\mathbf{M})\simeq\cBMop,
\end{gather}
where the right biaction of $\cBM$ on $\mathbf{inj}_{\underline{\ccCH}}(\mathbf{C})$ is
given by ${}_{-}\s{}_{-}$. By \eqref{eq:endbicomodules}, $\mathbf{M}$ can be viewed as a left birepresentation of $\cBMop$ or, equivalently, a right birepresentation of $\cBM$.

\begin{lemma}\label{lemma:endtrans}
$\mathbf{M}$ is transitive as a birepresentation of $\cBMop$.
\end{lemma}

\begin{proof}
For every $\mathrm{X}\in\mathrm{inj}_{\underline{\ccCH}}(\mathbf{C})$ and $\mathrm{Y}\in\mathrm{add}(\mathcal{H})$, we have $\mathrm{C(YC)}\in\cBM$ and
\begin{gather*}
\mathrm{X}\s\big(\mathrm{C}(\mathrm{YC})\big)\cong(\mathrm{X}\s\mathrm{C})(\mathrm{YC})\cong\mathrm{X(YC)}\cong\mathrm{(XY)C}.
\end{gather*}
The result now follows from \eqref{eq:comodequalsinj} and the fact that $\mathcal{H}$ is a
right cell of $\cCH$.
\end{proof}

\begin{remark} 
In general, $\cEnd_{\ccCH}(\mathbf{M})$ may not be finitary, 
but for our purpose that does not cause any serious problems fortunately. It always has a finitary $\mathcal{I}$-simple subquotient 
$\cEnd_{\ccCH}(\mathbf{M})_{\mathcal{I}}$, where 
$\mathcal{I}$ is the unique maximal two-sided cell of indecomposable injective 
endomorphisms. Recall that we call an endofunctor of an additive category injective if it is 
injective in the category of endofunctors of the injective abelianization and note that $\mathcal{I}$ 
is finite since any indecomposable injective $\mathrm{X}\in\cBM$ 
is a direct summand of one of the form $\mathrm{C(YC)}$ for some 
$\mathrm{Y}\in\mathrm{add}(\mathcal{H})$. Moreover, left and right 
adjoints define a quasi fiab structure on $\cEnd_{\ccCH}(\mathbf{M})$ 
which preserves $\mathcal{I}$ and, therefore, induces a quasi fiab structure on $\cEnd_{\ccCH}(\mathbf{M})_{\mathcal{I}}$. 
Finally, the proof of Lemma \ref{lemma:endtrans} shows that $\mathbf{M}$ restricts to 
a transitive birepresentation of $\cEnd_{\ccCH}(\mathbf{M})_{\mathcal{I}}$ with apex 
$\mathcal{I}$. Of course, $\cBMop$ has a biequivalent quasi fiab subquotient $(\cBMop)_{\mathcal{I}}$. 
Strictly speaking, we will be using these $\mathcal{I}$-simple 
subquotients below, but to avoid cluttering the notation, we will always suppress the 
subscript $\mathcal{I}$. 
\end{remark}

Since the biaction of $\cCH$ and $\cEnd_{\ccCH}(\mathbf{M})$ on $\mathbf{M}$ weakly commute by definition, there is a \emph{canonical pseudofunctor}
\begin{gather*}
\mathrm{can}\colon\cCH\to \cEnd^{\mathrm{ex}}_{\ccEnd_{\cccCH}(\mathbf{M})}(\mathbf{M}).
\end{gather*}
For every $\mathrm{X}\in\mathrm{add}(\mathcal{H})$, the endofunctor
$\mathbf{M}(\mathrm{X})$ of $\mathbf{M}(\mathtt{i})$ is injective by the dual version of \cite[Theorem 2]{KMMZ}, so in particular it is exact.
The identity $\mathbbm{1}_{\mathtt{i}}$ in $\cCH$, where
$\mathtt{i}$ is the source of $\mathcal{H}$, acts by
the identity functor, which is not injective but is, of course, exact.
The following theorem, which we call the \emph{double centralizer theorem}, is the analog of \cite[Theorem 7.12.11]{EGNO} for fiab bicategories
and simple transitive birepresentations. In its formulation, the superscript $\mathrm{inj}$ denotes
the injective morphisms, i.e. those realized by injective endofunctors.

\begin{theorem}\label{thm:double-centralizer}
The canonical pseudofunctor is fully faithful on $2$-morphisms and essentially surjective
on $1$-morphisms when restricted to $\mathrm{add}(\mathcal{H})$ and
corestricted to $\cEnd^{\mathrm{inj}}_{\ccEnd_{\cccCH}(\mathbf{M})}(\mathbf{M})$.
\end{theorem}

The proof follows similar reasoning as the proof of \cite[Theorem 7.12.11]{EGNO}, but
we have to adapt some of the arguments to our setting, because $\cCH$ is not
abelian and $\mathbbm{1}_{\mathtt{i}}$ does not act on $\mathbf{M}(\mathtt{i})$ by an injective endofunctor, as already remarked.

Before we give the proof of Theorem \ref{thm:double-centralizer}, let us recall some general facts about duality and coactions and point out
some consequences. Since these facts are
well-known and not difficult to check, we omit their proofs, see also Remark \ref{remark:diagrams}. Suppose that $\C$ is a coalgebra
in $\cCH$ and let $\mathrm{Y}\in\mathrm{inj}_{\underline{\ccCH}}(\C)$. Then
$\mathrm{Y}^{\star}\in(\C)\mathrm{inj}_{\underline{\ccCH}}$, with the left
$\C$-coaction $\delta_{\C,\mathrm{Y}^{\star}}$ being defined as the composite of
(recall $\mathrm{ev}^{\prime}$ and $\mathrm{coev}^{\prime}$ from
below Definition \ref{definition:fiab})
\begin{gather*}
\hspace*{-1.6cm}
\begin{tikzcd}[ampersand replacement=\&,column sep=3.75em]
\mathrm{Y}^{\star}
\arrow[d,xshift=-2cm,phantom,""{coordinate, name=Z}]
\arrow[r,"(\runit_{\mathrm{Y}^{\star}})^{\mone}"]
\&
\mathrm{Y}^{\star}\mathbbm{1}_{\mathtt{i}}
\arrow[r,"\mathrm{id}_{\mathrm{Y}^{\star}}\circ_{\mathsf{h}}\mathrm{coev}^{\prime}_{\mathrm{Y}}"]
\&
\mathrm{Y}^{\star}(\mathrm{Y}\mathrm{Y}^{\star})
\arrow[rr, "\mathrm{id}_{\mathrm{Y}^{\star}}\circ_{\mathsf{h}}
(\delta_{\mathrm{Y},\C}\circ_{\mathsf{h}}\mathrm{id}_{\mathrm{Y}^{\star}})"]
\ar[draw=none]{d}[name=X, anchor=center]{}
\& \&
\mathrm{Y}^{\star}\big((\mathrm{Y}\C)\mathrm{Y}^{\star}\big)
\ar[rounded corners,
to path={ -- ([xshift=2ex]\tikztostart.east)
|- (X.center) \tikztonodes
-| ([xshift=-2ex]\tikztotarget.west)
-- (\tikztotarget)},swap]{dlll}[at end]{\alpha^{\mone}_{\mathrm{Y}^{\star},\mathrm{Y}, \C \mathrm{Y}^{\star}}\circ_{\mathsf{v}} (\mathrm{id}_{\mathrm{Y}^{\star}}\circ_{\mathsf{h}}\alpha_{\mathrm{Y},\C,\mathrm{Y}^{\star}})}
\\
{} \& (\mathrm{Y}^{\star}\mathrm{Y})(\C \mathrm{Y}^{\star})
\arrow[r,"\mathrm{ev}^{\prime}_{\mathrm{Y}}\circ_{\mathsf{h}}\mathrm{id}_{\C \mathrm{Y}^{\star}}",swap]
\&
\mathbbm{1}_{\mathtt{i}}(\C \mathrm{Y}^{\star})
\arrow[rr,"\lunit_{\C\mathrm{Y}^{\star}}",swap]
\&\&
\C\mathrm{Y}^{\star}.
\end{tikzcd}
\end{gather*}
This implies that $\mathrm{Y}^{\star}\mathrm{Y}\in\cBM$ and that the following
diagrams commute:
\begin{gather}\label{eq:generalfacts1}
\begin{tikzcd}[ampersand replacement=\&,column sep=4.1em]
\mathrm{Y}^{\star}\mathrm{Y}
\arrow[r,"\delta_{\C,\mathrm{Y}^{\star}}\circ_{\mathsf{h}}
\mathrm{id}_{\mathrm{Y}}"]
\arrow[d, equal]
\&
(\C\mathrm{Y}^{\star})\mathrm{Y}
\arrow[r,"\alpha_{\C,\mathrm{Y}^{\star},\mathrm{Y}}"]
\&
\C (\mathrm{Y}^{\star}\mathrm{Y})
\arrow[r,"\mathrm{id}_{\C}\circ_{\mathsf{h}}\mathrm{ev}^{\prime}_{\mathrm{Y}}"]
\&
\C\mathbbm{1}_{\mathtt{i}}
\arrow[r,"\runit_{\C}"]
\&
\C
\arrow[d, equal]
\\
\mathrm{Y}^{\star}\mathrm{Y}
\arrow[r,"\mathrm{id}_{\mathrm{Y}}\circ_{\mathsf{h}}\delta_{\mathrm{Y},\C}",swap]
\&
\mathrm{Y}^{\star}(\mathrm{Y}\C)
\arrow[r,"\alpha^{\mone}_{\mathrm{Y}^{\star}, \mathrm{Y},\C}",swap]
\&
(\mathrm{Y}^{\star}\mathrm{Y})\C\arrow[r,"\mathrm{ev}^{\prime}_{\mathrm{Y}}\circ_{\mathsf{h}}\mathrm{id}_{\C}",swap]
\&
\mathbbm{1}_{\mathtt{i}}\C\arrow[r,"\lunit_{\C}",swap]
\&
\C
\end{tikzcd}
,
\end{gather}

\begin{gather}\label{eq:generalfacts2}
\begin{tikzcd}[ampersand replacement=\&,column sep=4em]
\mathbbm{1}_{\mathtt{i}}
\arrow[r,"\mathrm{coev}^{\prime}_{\mathrm{Y}}"]
\&
\mathrm{Y}\mathrm{Y}^{\star}
\arrow[r,"\delta_{\mathrm{Y},\C}\circ_{\mathsf{h}} \mathrm{id}_{\mathrm{Y}^{\star}}"]
\arrow[dr,"\mathrm{id}_{\mathrm{Y}}\circ_{\mathsf{h}} \delta_{\C,\mathrm{Y}^{\star}}", swap]
\&
(\mathrm{Y}\C) \mathrm{Y}^{\star}
\arrow[d,"\alpha_{\mathrm{Y}, \C, \mathrm{Y}^{\star}}"]
\\
\&\&
\mathrm{Y}(\C \mathrm{Y}^{\star})
\end{tikzcd}
.
\end{gather}

Now, let $\mathrm{X}\in\cCH$,
$\mathrm{Y}\in\mathrm{inj}_{\underline{\ccCH}}(\C)$ and
$\mathrm{Z}\in\cBM$. For any $f\in\mathrm{Hom}_{\underline{\ccC_{\mathcal{H}}}}(\mathrm{X},\mathrm{Y}\s\mathrm{Z})$, define $\tilde{f}:= \iota_{\mathrm{Y},\mathrm{Z}}\circ_{\mathsf{h}} f \in\mathrm{Hom}_{\underline{\ccC_{\mathcal{H}}}}(\mathrm{X},\mathrm{Y}\mathrm{Z})$, where $\iota_{\mathrm{Y},\mathrm{Z}}
\colon \mathrm{Y}\square_{\C}\mathrm{Z}\xhookrightarrow{}
\mathrm{Y}\mathrm{Z}$ is the canonical embedding. Then
$g\in\mathrm{Hom}_{\underline{\ccC_{\mathcal{H}}}}(\mathrm{X},\mathrm{Y}\mathrm{Z})$ satisfies $g=\tilde{f}$, for some
$f\in\mathrm{Hom}_{\underline{\ccC_{\mathcal{H}}}}(\mathrm{X},\mathrm{Y}\square_{\C}\mathrm{Z})$, if and only if
\begin{gather}\label{eq:generalfacts3}
\alpha_{\mathrm{Y}, \mathrm{C},\mathrm{Z}}\circ_{\mathsf{v}}(\delta_{\mathrm{Y},\mathrm{C}}\circ_{\mathsf{h}}\mathrm{id}_{\mathrm{Z}})
\circ_{\mathsf{v}} g= (\mathrm{id}_{\mathrm{Y}}\circ_{\mathsf{h}}
\delta_{\mathrm{C},\mathrm{Z}})
\circ_{\mathsf{v}} g.
\end{gather}
Taking $\mathrm{X}=\mathbbm{1}_{\mathtt{i}}$ and $\mathrm{Z}=\mathrm{Y}^{\star}$,
we see that commutativity of the diagram in \eqref{eq:generalfacts2} means that
$\mathrm{coev}^{\prime}_{\mathrm{Y}}$ factors through $\mathrm{Y}\s\mathrm{Y}^{\star}$,
i.e.
\begin{gather*}
\mathrm{coev}^{\prime}_{\mathrm{Y}}=\widetilde{\mathrm{coev}^{\C,\prime}_{\mathrm{Y}}},
\end{gather*}
where $\mathrm{coev}^{\C,\prime}_{\mathrm{Y}}\in\mathrm{Hom}_{\underline{\ccC_{\mathcal{H}}}}(\mathbbm{1}_{\mathtt{i}},\mathrm{Y}\square_{\C}\mathrm{Y}^{\star})$. This, in turn, implies that
$\mathrm{Y}^{\star}\mathrm{Y}$ is a coalgebra in $\cBM$, with comultiplication
$\delta^{\C}_{\mathrm{Y}^{\star}\mathrm{Y}}$ being the
composite of
\begin{gather}\label{eq:generalfacts5}
\mathrm{Y}^{\star}\mathrm{Y}
\xrightarrow{(\runit_{\mathrm{Y}^{\star}})^{\mone}\circ_{\mathsf{h}}
\mathrm{id}_{\mathrm{Y}}}
(\mathrm{Y}^{\star}\mathbbm{1}_{\mathtt{i}})\mathrm{Y}
\xrightarrow{\mathrm{id}_{\mathrm{Y}^{\star}}\circ_{\mathsf{h}} \mathrm{coev}^{\C,\prime}_{\mathrm{Y}}\circ_{\mathsf{h}}\mathrm{id}_{\mathrm{Y}}}
\big(\mathrm{Y}^{\star}(\mathrm{Y}\s\mathrm{Y}^{\star})\big)\mathrm{Y}
\xrightarrow{\cong}
(\mathrm{Y}^{\star}\mathrm{Y})\s(\mathrm{Y}^{\star}\mathrm{Y})
\end{gather}
and counit $\epsilon^{\C}_{\mathrm{Y}^{\star}\mathrm{Y}}$ being the composite
of either one of the rows in \eqref{eq:generalfacts1}. Checking coassociativity and counitality is
an easy but tedious exercise in diagram-chasing, which we leave to the reader. We only note that to check counitality, one has to use commutativity of \eqref{eq:generalfacts1}.

Finally, let
\begin{gather}\label{eq:generalfacts6}
\delta_{\mathrm{Y}^{\star}\mathrm{Y}}:=
\widetilde{\delta^{\C}_{\mathrm{Y}^{\star}\mathrm{Y}}}\quad
\epsilon_{\mathrm{Y}^{\star}\mathrm{Y}}:=\epsilon_{\C}\circ_{\mathsf{v}} \epsilon^{\C}_{\mathrm{Y}^{\star}\mathrm{Y}}.
\end{gather}
Then $(\mathrm{Y}^{\star}\mathrm{Y}, \delta_{\mathrm{Y}^{\star}\mathrm{Y}},
\epsilon_{\mathrm{Y}^{\star}\mathrm{Y}})$ is a coalgebra in $\cCH$. As a matter of fact,
it is exactly the coalgebra structure on $\mathrm{Y}^{\star}\mathrm{Y}$ which we defined
in Lemma \ref{lem0.1}, if we consider $\mathrm{Y}^{\star}\mathrm{Y}$ as the framing of
the coalgebra $\mathbbm{1}_{\mathtt{i}}$ by $\mathrm{Y}^{\star}$ in $\cCH$, namely,
\begin{gather*}
\begin{gathered}
\delta_{\mathrm{Y}^{\star}\mathrm{Y}}=
\alpha_{\mathrm{Y}^{\star}\mathrm{Y},\mathrm{Y}^{\star},\mathrm{Y}}\circ_{\mathsf{v}}
(\alpha_{\mathrm{Y}^{\star},\mathrm{Y},\mathrm{Y}^{\star}}^{\mone}\circ_{\mathsf{h}}\mathrm{id}_{\mathrm{Y}})\circ_{\mathsf{v}}
\\
\big((\mathrm{id}_{\mathrm{Y}^{\star}}\circ_{\mathsf{h}}
\mathrm{coev}^{\prime}_{\mathrm{Y}})\circ_{\mathsf{h}}\mathrm{id}_{\mathrm{Y}}\big)\circ_{\mathsf{v}}\big((\runit_{\mathrm{Y}^{\star}})^{\mone}\circ_{\mathsf{h}}\mathrm{id}_{\mathrm{Y}}\big)
\end{gathered}
,
\\
\epsilon_{\mathrm{Y}^{\star}\mathrm{Y}}=\mathrm{ev}^{\prime}_{\mathrm{Y}}.
\end{gather*}

\begin{remark}\label{remark:diagrams}
The facts above are easy to see in the strict setting using string diagrams. For example, \eqref{eq:generalfacts1} reads as
\begin{gather*}
\begin{tikzpicture}[anchorbase,scale=1]
\draw[dstrand,directed=0.5] (1,0)node[below]{$\mathrm{Y}$} to[out=90,in=0] (0.5,0.5) to[out=180,in=90] (0,0)node[below]{$\mathrm{Y}^{\star}$};
\draw[cstrand] (0.07,0.2) to[out=180,in=270] (-0.5,0.75)node[above]{$\C$};
\end{tikzpicture}
=
\begin{tikzpicture}[anchorbase,scale=1]
\draw[dstrand,directed=0.5] (1,0)node[below]{$\mathrm{Y}$} to[out=90,in=0] (0.5,0.5) to[out=180,in=90] (0,0)node[below]{$\mathrm{Y}^{\star}$};
\draw[cstrand] (0.93,0.2) to[out=0,in=270] (1.5,0.75)node[above]{$\C$};
\end{tikzpicture}
,
\end{gather*}
while $\delta_{\mathrm{Y}^{\star}\mathrm{Y}}$ and $\epsilon_{\mathrm{Y}^{\star}\mathrm{Y}}$ can be depicted as
\begin{gather*}
\delta_{\mathrm{Y}^{\star}\mathrm{Y}}=
\begin{tikzpicture}[anchorbase,scale=1]
\draw[dstrand,opdirected=0.5] (-1,0)node[above]{$\mathrm{Y}$} to[out=270,in=180] (-0.5,-0.5) to[out=0,in=270] (0,0)node[above]{$\mathrm{Y}^{\star}$};
\draw[dstrand,opdirected=0.5] (1,0)node[above]{$\mathrm{Y}$} to[out=270,in=90] (0,-1)node[below]{$\mathrm{Y}$};
\draw[dstrand,opdirected=0.5] (-1,-1)node[below]{$\mathrm{Y}^{\star}$} to[out=90,in=270] (-2,0)node[above]{$\mathrm{Y}^{\star}$};
\end{tikzpicture}
\,,\quad
\epsilon_{\mathrm{Y}^{\star}\mathrm{Y}}=
\begin{tikzpicture}[anchorbase,scale=1]
\draw[dstrand,opdirected=0.5] (0,0)node[below]{$\mathrm{Y}^{\star}$} to[out=90,in=180] (0.5,0.5) to[out=0,in=90] (1,0)node[below]{$\mathrm{Y}$};
\end{tikzpicture}
\,
\end{gather*}
for which coassociativity and counitality are easy exercises in planar topology. Similarly, the various constructions which we will use below also have string diagrammatic interpretations. For example,
\begin{gather*}
\begin{tikzpicture}[anchorbase,scale=1]
\draw[dstrand,white] (0,-0.45) node[left,yshift=0.225cm]{$\mathrm{F}$} to (0,0.05);
\draw[dstrand,white] (0,0) to (0,0.01) node[above,black,box]{\raisebox{-0.025cm}{\hspace*{0.35cm}$f$\hspace*{0.35cm}}};
\draw[dstrand,directed=0.5] (0,-0.4)node[below]{$\mathrm{X}$} to (0,0);
\draw[dstrand,directed=0.5] (-0.4,0.6) to (-0.4,1)node[above]{$\mathrm{Y}$};
\draw[dstrand,directed=0.5] (0.4,0.6) to (0.4,1)node[above]{$\mathrm{Z}$};
\end{tikzpicture}
\mapsto
\begin{tikzpicture}[anchorbase,scale=1]
\draw[dstrand,white] (0,-0.45) node[left,yshift=0.225cm]{$\mathrm{F}$} to (0,0.05);
\draw[dstrand,white] (0,0) to (0,0.01) node[above,black,box]{\raisebox{-0.025cm}{\hspace*{0.35cm}$f$\hspace*{0.35cm}}};
\draw[dstrand,directed=0.5] (0,-0.4)node[below]{$\mathrm{X}$} to (0,0);
\draw[dstrand,directed=0.5] (-0.4,0.6) to[out=90,in=0] (-0.6,1) to[out=180,in=90] (-0.8,0.6) to (-0.8,-0.4) node[below]{$\mathrm{Y}^{\star}$};
\draw[dstrand,directed=0.5] (0.4,0.6) to (0.4,1)node[above]{$\mathrm{Z}$};
\end{tikzpicture}
\,,
\end{gather*}
is the picture of the map in \eqref{eq:fiso}.
\end{remark}	

\begin{proof}[Proof of Theorem \ref{thm:double-centralizer}]
By \eqref{eq:endbicomodules} we can interpret the canonical pseudofunctor
as a pseudofunctor (with the same name)
$\mathrm{can}\colon\cCH\to\cEnd^{\mathrm{ex}}_{{\ccB}^{\mathrm{op}}_{\mathbf{M}}}(\mathbf{M})$.

By Lemma \ref{lemma:endtrans} and the internal cohom construction, there is an
equivalence of right birepresentations of $\cBM$
\begin{gather*}
\mathbf{inj}_{\underline{\ccCH}}(\C)\simeq\big([\C,\C]\big)\mathbf{inj}_{\underline{\ccB_{\mathbf{M}}}}.
\end{gather*}
A crucial ingredient is the following claim.

\textbf{Claim 1.} We have
\begin{gather}\label{eq:double-centralizer1}
[\C,\C]\cong\C^{\star}\C\;\text{in}\;\cBM
\end{gather}
and the implied equivalence
\begin{gather}\label{eq:double-centralizer2}
\mathbf{inj}_{\underline{\ccCH}}(\C)\simeq(\C^{\star}\C)\mathbf{inj}_{\underline{\ccB_{\mathbf{M}}}}
\end{gather}
is given explicitly by $\mathrm{X}\mapsto\C^{\star}\mathrm{X}$.

Claim 1 then implies that
\begin{gather*}
\cEnd^{\mathrm{ex}}_{\ccEnd_{\cccCH}(\mathbf{M})}(\mathbf{M})\simeq(\C^{\star} \C)\Bi_{\underline{\ccB_{\mathbf{M}}}}(\C^{\star}\C)
\end{gather*}
and $1$-morphisms in $\cEnd^{\mathrm{inj}}_{\ccEnd_{\cccCH}(\mathbf{M})}(\mathbf{M})$, under this equivalence, correspond to
injective $(\C^{\star}\C)$-bicomodules in $\underline{\cBM}$, which all live
in $\cBM$. We can then trace the
canonical pseudofunctor through the equivalences to a pseudofunctor
\begin{gather*}
\mathrm{can}\colon\cCH\to(\C^{\star} \C)\Bi_{\underline{\ccB_{\mathbf{M}}}}(\C^{\star}\C).
\end{gather*}
It is on that level that we prove fully faithfulness on $2$-morphisms and essential surjectivity on $1$-morphisms when restricting to $\mathrm{add}(\mathcal{H})$ and corestricting to
injective $(\C^{\star}\C)$-bicomodules.

We now proceed to prove Claim 1, i.e. the isomorphism \eqref{eq:double-centralizer1} and the explicit description of the equivalence it implies. Both follow from
the natural isomorphism
\begin{gather*}
\mathrm{Hom}_{\C}(\mathrm{X},\mathrm{Y}\s\mathrm{Z})
\cong
\mathrm{Hom}_{\underline{{\ccB}_{\mathbf{M}}}}(\mathrm{Y}^{\star}\mathrm{X},\mathrm{Z})
\end{gather*}
for $\mathrm{X},\mathrm{Y}\in\mathbf{inj}_{\underline{\ccC_{\mathcal{H}}}}(\C)$ and
$\mathrm{Z}\in\underline{\cBM}$, which we claim is given by
\begin{gather}\label{eq:fiso}
f\mapsto \lunit_{\mathrm{Z}}\circ_{\mathsf{v}}(\mathrm{ev}_{\mathrm{Y}}^{\prime}
\circ_{\mathsf{v}}\circ_{\mathsf{h}}\mathrm{id}_{\mathrm{Z}})\circ_{\mathsf{v}}
\alpha^{\mone}_{\mathrm{Y}^{\star},\mathrm{Y},\mathrm{Z}}\circ_{\mathsf{v}}
(\mathrm{id}_{\mathrm{Y}^{\star}}\circ_{\mathsf{h}}\tilde{f}).
\end{gather}
To verify this, we have to show that the natural transformation is well-defined and that it is an isomorphism.

Since $f$ is assumed to be a right $\C$-comodule homomorphism, it is clear
from the definition of the natural transformation that the image of $f$ is also
a right $\C$-comodule homomorphism. The fact that
the image of $f$ is also a left $\C$-comodule homomorphism follows
from the assumption that the target of $f$ is $\mathrm{Y}\s\mathrm{Z}$, as the following
commutative diagram shows:
\begin{gather*}
\adjustbox{scale=0.57,center}{%
\begin{tikzcd}[ampersand replacement=\&,column sep=2.05em]
\mathrm{Y}^{\star} \mathrm{X}
\arrow[rrr,"\delta_{\C,\mathrm{Y}^{\star}}\circ_{\mathsf{h}}\mathrm{id}_{\mathrm{X}}"]
\arrow[d,equal]
\&\&\&
(\C\mathrm{Y}^{\star})\mathrm{X}
\arrow[r,"\alpha_{\C,\mathrm{Y}^{\star},\mathrm{X}}"]
\&
\C(\mathrm{Y}^{\star}\mathrm{X})
\arrow[rrr,"\mathrm{id}_{\C}(\mathrm{id}_{\mathrm{Y}^{\star}}\circ_{\mathsf{h}}\tilde{f})"]
\ar[rrr,phantom,yshift=-3.5ex, xshift=-9ex,"\circled{1}"]
\&\&\&
\C\big(\mathrm{Y}^{\star}(\mathrm{YZ})\big)
\arrow[r,white,"\mathrm{id}_{\C}\circ_{\mathsf{h}}\alpha_{\mathrm{Y}^{\star},\mathrm{Y},\mathrm{Z}}^{\mone}",yshift=0.1cm]
\arrow[r]
\&
\C\big((\mathrm{Y}^{\star}\mathrm{Y})\mathrm{Z}\big)
\arrow[rrr,"\mathrm{id}_{\C}\circ_{\mathsf{h}}(\mathrm{ev}^{\prime}_{\mathrm{Y}}
\mathrm{id}_{\mathrm{Z}})"]
\arrow[d,"(\alpha_{\C,\mathrm{Y}^{\star},\mathrm{Y}}^{\mone}\circ_{\mathsf{h}}\mathrm{id}_{\mathrm{Z}})\circ_{\mathsf{v}}\alpha_{\C,\mathrm{Y}^{\star}\mathrm{Y},\mathrm{Z}}^{\mone}",swap]
\&\&\&
\C(\mathbbm{1}_{\mathtt{i}}\mathrm{Z})
\arrow[ddrr, bend left,"\mathrm{id}_{\C}\circ_{\mathsf{h}}\lunit_{\mathrm{Z}}"]
\arrow[d,"\alpha_{\C,\mathbbm{1}_{\mathtt{i}},\mathrm{Z}}^{\mone}",swap]
\ar[d,phantom,yshift=-3.5ex,xshift=8ex,"\circled{3}"]
\&\&
\\
\mathrm{Y}^{\star} \mathrm{X}
\arrow[rrr,"\mathrm{id}_{\mathrm{Y}^{\star}}\circ_{\mathsf{h}}\tilde{f}",swap]
\arrow[dddd,equal]
\ar[dddd,phantom,xshift=20ex,"\circled{4}"]
\&\&\&
\mathrm{Y}^{\star}(\mathrm{YZ})
\arrow[r,"\alpha_{\mathrm{Y}^{\star},\mathrm{Y},\mathrm{Z}}^{\mone}",swap]
\arrow[ddr,equal]
\&
(\mathrm{Y}^{\star}\mathrm{Y})\mathrm{Z}
\arrow[rrrr,"(\delta_{\C,\mathrm{Y}^{\star}}\circ_{\mathsf{h}}\mathrm{id}_{\mathrm{Y}})
\circ_{\mathsf{h}}\mathrm{id}_{\mathrm{Z}}",swap]
\arrow[d,equal]
\&\&\&
\&
\big((\C\mathrm{Y}^{\star})\mathrm{Y}\big)\mathrm{Z}
\ar[rrr,phantom,"\circled{2}"]
\&\&\&
(\C\mathbbm{1}_{\mathtt{i}})\mathrm{Z}
\arrow[drr,"\runit_{\C}\circ_{\mathsf{h}}\mathrm{id}_{\mathrm{Z}}",swap,near start]
\&\&
\\
\&\&\&
\&
(\mathrm{Y}^{\star}\mathrm{Y})\mathrm{Z}
\arrow[rrr,"(\mathrm{id}_{\mathrm{Y}^{\star}}\circ_{\mathsf{h}}\delta_{\mathrm{Y},\C})
\circ_{\mathsf{h}}\mathrm{id}_{\mathrm{Z}}"]
\ar[d,"\alpha_{\mathrm{Y}^{\star}, \mathrm{Y},\mathrm{Z}}"]
\ar[rrr,phantom, yshift=-3ex, "\circled{5}"]
\&\&\&
\big(\mathrm{Y}^{\star}(\mathrm{Y}\C)\big)\mathrm{Z}
\arrow[r,white,"\alpha_{\mathrm{Y}^{\star},\mathrm{Y},\C}^{\mone}\circ_{\mathsf{h}}\mathrm{id}_{\mathrm{Z}}",yshift=0.1cm]
\ar[r]
\arrow[d,"\alpha_{\mathrm{Y}^{\star},\mathrm{Y}\C,\mathrm{Z}}"]
\&
\big((\mathrm{Y}^{\star}\mathrm{Y})\C\big)\mathrm{Z}
\arrow[rrr,"(\mathrm{ev}^{\prime}_{\mathrm{Y}}\circ_{\mathsf{h}}\mathrm{id}_{\C})\circ_{\mathsf{h}}\mathrm{id}_{\mathrm{Z}}"]
\ar[rrr,phantom, yshift=-9ex, xshift=-9ex, "\circled{6}"]
\&\&\&
(\mathbbm{1}_{\mathtt{i}}\C)\mathrm{Z}
\arrow[rr,"\lunit_{\C}\circ_{\mathsf{h}}\mathrm{id}_{\mathrm{Z}}"]
\&\&
\C\mathrm{Z}
\\
\&\&\&
\&
\mathrm{Y}^{\star} (\mathrm{Y}\mathrm{Z})
\ar[rrr,"\mathrm{id}_{\mathrm{Y}^{\star}}\circ_{\mathsf{h}}(\delta_{\mathrm{Y},\C}
\circ_{\mathsf{h}}\mathrm{id}_{\mathrm{Z}})"]
\&\&\&
\mathrm{Y}^{\star}\big((\mathrm{Y}\C)\mathrm{Z}\big)
\arrow[d,"\mathrm{id}_{\mathrm{Y}^{\star}}\circ_{\mathsf{h}}\alpha_{\mathrm{Y},\C,\mathrm{Z}}"]
\&
\&\&\&
\mathbbm{1}_{\mathtt{i}}(\C\mathrm{Z})
\arrow[u,"\alpha_{\mathbbm{1}_{\mathtt{i}},\C,\mathrm{Z}}^{\mone}"]
\&\&
\\
\&\&\&
\&
\&\&\&
\mathrm{Y}^{\star}\big(\mathrm{Y}(\C\mathrm{Z})\big)
\arrow[rr,"\alpha_{\mathrm{Y}^{\star},\mathrm{Y},\C\mathrm{Z}}^{\mone}"]
\ar[rr,phantom, yshift=-3ex, "\circled{7}"]
\&\&
(\mathrm{Y}^{\star}\mathrm{Y})(\C\mathrm{Z})
\arrow[urr,"\mathrm{ev}^{\prime}_{\mathrm{Y}}\circ_{\mathsf{h}}
\mathrm{id}_{\mathrm{CZ}}"]
\ar[urr,phantom, yshift=-4ex, "\circled{8}"]
\&\&
\mathbbm{1}_{\mathtt{i}}\mathrm{Z}
\arrow[rr,"\lunit_{\mathrm{Z}}",swap]
\arrow[u,"\mathrm{id}_{\mathbbm{1}_{\mathtt{i}}}\circ_{\mathsf{h}}\delta_{\C,\mathrm{Z}}",swap]
\&\&
\mathrm{Z}
\ar[uu,"\delta_{\C,\mathrm{Z}}",swap]
\ar[uu,phantom,xshift=-7ex,"\circled{9}"]
\\
\mathrm{Y}^{\star}\mathrm{X}
\arrow[rrrrrrr,"\mathrm{id}_{\mathrm{Y}^{\star}}\circ_{\mathsf{h}}\tilde{f}",swap]
\&\&\&
\&
\&\&\&
\mathrm{Y}^{\star} (\mathrm{Y}\mathrm{Z})
\arrow[rr,"\alpha_{\mathrm{Y}^{\star},\mathrm{Y},\mathrm{Z}}^{\mone}",swap]
\arrow[u,"\mathrm{id}_{\mathrm{Y}^{\star}}\circ_{\mathsf{h}}(\mathrm{id}_{\mathrm{Y}}\circ_{\mathsf{h}}\delta_{\C,\mathrm{Z}})"]
\&\&
(\mathrm{Y}^{\star} \mathrm{Y})\mathrm{Z}
\arrow[u,"\mathrm{id}_{\mathrm{Y}^{\star}\mathrm{Y}}\circ_{\mathsf{h}}
\delta_{\C,\mathrm{Z}}"]
\arrow[urr, "\mathrm{ev}^{\prime}_{\mathrm{Y}}\circ_{\mathsf{h}}\mathrm{id}_{\mathrm{Z}}",swap]
\&\&
\&\&
\end{tikzcd}
}
\hspace*{-0.25cm}
\end{gather*}
Commutativity follows from
\begin{itemize}

\item the interchange law, naturality of the associator and the pentagon coherence condition of the associator for the facet labeled $1$;

\item commutativity of \eqref{eq:generalfacts1} and naturality of the associator for the facet labeled $2$;

\item the triangle coherence condition of the unitors for the facet labeled $3$;

\item equation \eqref{eq:generalfacts3} for the facet labeled $4$;

\item naturality of the associator for the facets labeled $5$ and $7$;

\item naturality of the associator and the pentagon coherence condition of the associator for the facet labeled $6$;

\item the interchange law for the facet labeled $8$;

\item naturality of the unitor and \eqref{eq:0.00} for the facet labeled $9$.
\end{itemize}
This finishes the proof that the image of $f$ belongs to $\mathrm{Hom}_{\underline{\ccB_{\mathbf{M}}}}(\mathrm{Y}^{\star}\mathrm{X},\mathrm{Z})$, in other words, that the natural transformation
above is well-defined.

Its inverse is given by sending any $h\in\mathrm{Hom}_{\underline{\ccB_{\mathbf{M}}}}(\mathrm{Y}^{\star}\mathrm{X},\mathrm{Z})$ to the unique $2$-morphism
$h^{\prime}\in\mathrm{Hom}_{\C}(\mathrm{X},\mathrm{Y}\s\mathrm{Z})$ such that
\begin{gather*}
\widetilde{h^{\prime}}=(\mathrm{id}_{\mathrm{Y}}\circ_{\mathsf{h}}h)\circ_{\mathsf{v}}\alpha_{\mathrm{Y},\mathrm{Y}^{\star},\mathrm{X}}\circ_{\mathsf{v}}(\mathrm{coev}^{\prime}_{\mathrm{Y}}\circ_{\mathsf{h}}\mathrm{id}_{\mathrm{X}})\circ_{\mathsf{v}}(\lunit_{\mathrm{X}})^{\mone}.
\end{gather*}
The fact that the latter composite satisfies
\eqref{eq:generalfacts3} implies well-definition. The statement that the map $h\mapsto h^{\prime}$ 
is the inverse of \eqref{eq:fiso} follows from naturality of the unitors and associator, the 
interchange law, the pentagon coherence condition for the associator, the triangle 
coherence condition of the unitors and the zigzag relation for coevaluations and evaluations.

This completes the proof of Claim 1, since it shows that 
$[\mathrm{Y},\mathrm{X}]\cong\mathrm{Y}^{\star}\mathrm{X}$ which, setting $\mathrm{X}=\mathrm{Y}=\C$, 
yields \eqref{eq:double-centralizer1}. Moreover, the equivalence in \eqref{eq:double-centralizer2} 
is then given by $\mathrm{X}\mapsto[\C,\mathrm{X}]\cong\C^{\star}\mathrm{X}$.

It is not hard to see that the coalgebra structure of $\C^{\star}\C\cong [\C,\C]$ in $\cBM$
given by the internal cohom construction coincides with the one given
in \eqref{eq:generalfacts5} and the text above and below it.
As we explained in \eqref{eq:generalfacts6} and the text above and below it,
this coalgebra structure in $\cBM$ induces one in $\cCH$.

It follows that $\C^{\star}$ is naturally a left $\C^{\star}\C$-comodule in
$(\C)\mathrm{inj}_{\underline{\ccCH}}$, with
coaction defined by
\begin{gather*}
\delta_{\C^{\star}\C,\C^{\star}}:=\alpha^{\mone}_{\C^{\star},\C,\C^{\star}}\circ_{\mathsf{v}}(\mathrm{id}_{\C^{\star}}\circ_{\mathsf{h}}
\mathrm{coev}^{\prime}_{\C})\circ_{\mathsf{v}}(\runit_{\C^{\star}})^{\mone},
\end{gather*}
and $\C$ is naturally a right $\C^{\star}\C$-comodule, with coaction defined by
\begin{gather*}
\delta_{\C,\C^{\star}\C}:=
\alpha_{\C,\C^{\star},\C}
\circ_{\mathsf{v}}
(\mathrm{coev}^{\prime}_{\C}\circ_{\mathsf{h}}\mathrm{id}_{\C})
\circ_{\mathsf{v}}(\lunit_{\C})^{\mone}.
\end{gather*}
Coassociativity follows from the interchange law, naturality of the associator and unitors,
and the pentagon coherence condition of the associator. Counitality follows from the zigzag relations.

To show that the canonical pseudofunctor
$\mathrm{can}\colon\cCH\to(\C^{\star}\C)\Bi_{\underline{\ccB_{\mathbf{M}}}}(\C^{\star}\C)$ 
is fully faithful on $2$-morphisms and essentially surjective on $1$-morphisms,
when restricted to $\mathrm{add}(\mathcal{H})$, we need the following.

\textbf{Claim 2.}
The canonical pseudofunctor $\mathrm{can}\colon\cCH\to(\C^{\star}\C)\Bi_{\underline{\ccB_{\mathbf{M}}}}(\C^{\star}\C)$ is explicitly given by
\begin{gather*}
\mathrm{F}\mapsto\C^{\star}(\mathrm{F}\C).
\end{gather*}
We already know that the equivalence in \eqref{eq:double-centralizer2} is given by
$\mathrm{X}\to\C^{\star}\mathrm{X}$. Since $\mathrm{F}$, viewed as an endomorphism 
of $\mathbf{inj}_{\underline{\ccC_{\mathcal{H}}}}(\C)$, sends $\mathrm{X}$ to $\mathrm{FX}$, 
it follows that under the equivalence in \eqref{eq:double-centralizer2}, it sends 
$\C^{\star}\mathrm{X}$ to $\C^{\star}(\mathrm{FX})$. We therefore have to show that there is a
natural isomorphism
\begin{gather*}
\mathrm{X}\cong(\C\square_{\C^{\star}\C}\C^{\star})\mathrm{X},
\end{gather*}
for $\mathrm{X}\in\mathbf{inj}_{\underline{\ccC_{\mathcal{H}}}}(\C)$. By
Lemma \ref{lemma:endtrans}, it suffices to prove that there is a natural isomorphism
in $\cCH$
\begin{gather}\label{eq:double-centralizer5.51}
\mathrm{G}\cong(\C\square_{\C^{\star}\C}\C^{\star})\mathrm{G},
\end{gather}
for $\mathrm{G}\in\mathrm{add}(\mathcal{H})$. Finally, since $\mathcal{H}$ is a right
cell of $\cCH$, it suffices to prove $\eqref{eq:double-centralizer5.51}$ when $\mathrm{G}=
\mathrm{CH}$, for some $\mathrm{H}\in\mathrm{add}(\mathcal{H})$. In this case,
the isomorphism in \eqref{eq:double-centralizer5.51} is immediate,
since $\C\cong\C\square_{\C^{\star}\C}(\C^{\star}\C)\cong
(\C\square_{\C^{\star}\C}\C^{\star})\C$. Note that the
isomorphism
\begin{gather*}
\mathrm{CH}\cong(\C\square_{\C^{\star}\C}\C^{\star})(\mathrm{CH}),
\end{gather*}
is natural in $\mathrm{CH}$, as it is given by
\begin{gather*}
(\mathrm{coev}^{\C^{\star}\C,\prime}_{\mathrm{C}}\circ_{\mathsf{h}}\mathrm{id}_{\mathrm{CH}})\circ_{\mathsf{v}}(\lunit_{\mathrm{CH}})^{\mone},
\end{gather*}
where $\mathrm{coev}^{\C^{\star}\C,\prime}_{\mathrm{C}}$ is defined just as
$\mathrm{coev}^{\C,\prime}_{\mathrm{Y}}$ below \eqref{eq:generalfacts3}.
This completes the proof of Claim 2.

We are now ready to complete the proof of Theorem \ref{thm:double-centralizer}. The assignment
\begin{gather*}
\mathrm{can}^{\mone}\colon \mathrm{Z}\mapsto\C\square_{\C^{\star}\C}(\mathrm{Z}
\square_{\C^{\star}\C}\C^{\star})
\end{gather*}
defines a pseudofunctor in the opposite direction. Clearly,
$\mathrm{can}\circ\mathrm{can}^{\mone}$ is naturally isomorphic to the identity
on $(\C^{\star}\C)\Bi_{\underline{\ccB_{\mathbf{M}}}}(\C^{\star}\C)$, as follows from
\begin{gather*}
(\C^{\star}\C)\square_{\C^{\star}\C}(\mathrm{Z}
\square_{\C^{\star}\C}(\C^{\star}\C))\cong\mathrm{Z}
\end{gather*}
for any $\mathrm{Z}\in(\C^{\star}\C)\Bi_{\underline{\ccB_{\mathbf{M}}}}(\C^{\star}\C)$. 
On the other hand, $\mathrm{can}^{\mone}\circ\mathrm{can}$ is naturally isomorphic to the identity on $\mathrm{add}(\mathcal{H})$, due to
\eqref{eq:double-centralizer5.51} and its analog for tensoring with $\C \square_{\C^{\star}\C}
\C^{\star}$ on the right in $\mathrm{add}(\mathcal{H})$. Since we already know that $\mathrm{can}$ takes values in injective
$(\C^{\star}\C)$-bicomodules when restricted to $\mathrm{add}(\mathcal{H})$,
these two natural isomorphisms imply that $\mathrm{can}^{\mone}$ takes values in
$\mathrm{add}(\mathcal{H})$ when restricted to injective $(\C^{\star}\C)$-bicomodules
and, moreover, that $\mathrm{can}$ is fully faithful on $2$-morphisms and
essentially surjective on $1$-morphisms when
restricted to $\mathrm{add}(\mathcal{H})$ and corestricted to injective $(\C^{\star}\C)$-bicomodules.
\end{proof}

%%%%%%%%%%%%%%%%%%%%%%%%%%%%%%%%%%%%%%%%%%%%%%%%%%%%%

\vspace{2mm}

M.M.: Center for Mathematical Analysis, Geometry, and Dynamical Systems, Departamento de Matem{\'a}tica,
Instituto Superior T{\'e}cnico, 1049-001 Lisboa, PORTUGAL \& Departamento de Matem{\'a}tica, FCT,
Universidade do Algarve, Campus de Gambelas, 8005-139 Faro, PORTUGAL
\newline email: {\tt mmackaay\symbol{64}ualg.pt}

Vo.Ma.: Department of Mathematics, Uppsala University, Box. 480,
SE-75106, Uppsala, SWEDEN
\newline email: {\tt mazor\symbol{64}math.uu.se}

Va.Mi.: School of Mathematics, University of East Anglia,
Norwich NR4 7TJ, UK
\newline email: {\tt v.miemietz\symbol{64}uea.ac.uk}

D.T.: The University of Sydney, School of Mathematics and Statistics F07, Office Carslaw 827, NSW 2006, AUSTRALIA
\newline email: {\tt daniel.tubbenhauer\symbol{64}sydney.edu.au}, web: {\tt www.dtubbenhauer.com}

X.Z.: Beijing Advanced Innovation Center for Imaging Theory and Technology, Academy for Multidisciplinary Studies, Capital Normal University, Beijing 100048, CHINA
\newline email: {\tt xiaoting.zhang09\symbol{64}hotmail.com}

\end{document}